%% file: main.tex
\documentclass[table]{book}
\pdfoutput=1 
\input{package.tex}
\input{macro.tex}

\title{Sparse Polynomial Optimization: Theory and Practice}
\author{Victor Magron \and Jie Wang}
\date{\today}
\begin{document}
\maketitle
\tableofcontents
\include{acronym}
\input{notation.tex}
\input{preface.tex}
\input{sdp.tex}

\input{pop.tex}
\input{cs.tex}
\input{roundoff.tex}

\input{lip.tex}

\input{ncsparse.tex}

\input{ts.tex}

\input{cstssos.tex}
\input{opf.tex}

\input{nctssos.tex}
\input{jsr.tex}
\input{misc.tex}

\input{appendix.tex}

\end{document}

%% file: package.tex
\usepackage[T1]{fontenc}
\usepackage[utf8]{inputenc}
\usepackage[frozencache,outputdir=.]{minted}
\usepackage{hyperref}
\usepackage[english]{babel}
\usepackage{mathpazo}
\usepackage{float}
\usepackage{booktabs}
\usepackage[resetlabels]{multibib}
\usepackage{morewrites}
\usepackage{mathtools}
\usepackage{longtable}
\usepackage{graphics}
\newcites{preface}{Bibliography}
\newcites{sdp}{Bibliography}
\newcites{pop}{Bibliography}
\newcites{cs}{Bibliography}
\newcites{roundoff}{Bibliography}
\newcites{lip}{Bibliography}
\newcites{ncsparse}{Bibliography}
\newcites{ts}{Bibliography}
\newcites{cstssos}{Bibliography}
\newcites{opf}{Bibliography}
\newcites{ncts}{Bibliography}
\newcites{jsr}{Bibliography}
\newcites{misc}{Bibliography}
\newcites{app}{Bibliography}
\usepackage{ifdraft}
\usepackage{amsmath,amssymb}
\usepackage{framed}
\usepackage[standard,framed]{ntheorem}
\usepackage[shortlabels]{enumitem}
\usepackage{mathrsfs}
\usepackage{multirow}
\usepackage{aecompl}
\usepackage{url}
\usepackage{pgfplots}
\pgfplotsset{compat=1.11}
\usepackage{tikz}
\usetikzlibrary{patterns,decorations.pathreplacing}
\usetikzlibrary{cd,decorations.pathmorphing,decorations.pathreplacing}
\usetikzlibrary{fit,shadows,chains,matrix,arrows,calc,positioning}

\definecolor{linkcol}{rgb}{0,0,0.4}
\definecolor{citecol}{rgb}{0.5,0,0}
\usepackage{subcaption}
\hypersetup
{
bookmarksopen=true,
pdftitle=SparseBook,
pdfauthor="Victor Magron and Jie Wang",
pdfsubject="SparseBook", 
pdfmenubar=true, 
pdfhighlight=/O, 
colorlinks=true, 
pdfpagemode=UseNone, 
pdfpagelayout=SinglePage, 
pdffitwindow=true, 
linkcolor=linkcol, 
citecolor=citecol, 
urlcolor=linkcol 
}

\newtheorem{assumption}[theorem]{Assumption}
\renewcommand{\epsilon}{\varepsilon}

\usepackage{fancyhdr}
\usepackage[paperwidth=6in, paperheight=9in, total={4.5in, 7.35in}]{geometry}

\let\headruleORIG\headrule
\renewcommand{\headrule}{\color{black} \headruleORIG}

\usepackage{xcolor}
\arrayrulecolor{black}

\fancypagestyle{plain}{
  \fancyhead{}
  \fancyfoot{}
  
}

\makeatletter

\def\cleardoublepage{\clearpage\if@twoside \ifodd\c@page\else%
  \hbox{}%
  \thispagestyle{empty}
  \newpage%
  \if@twocolumn\hbox{}\newpage\fi\fi\fi}

\makeatother

\usepackage[printonlyused,withpage]{acronym}
\usepackage{algorithm}
\usepackage{algpseudocode}
\newcommand{\qed}{\hfill \ensuremath{\Box}} 

\usepackage[framed,numbered,autolinebreaks,useliterate]{mcode}

%% file: macro.tex
\newif\ifdraft
\drafttrue
\def\a{{\alpha}}
\def\b{\beta}
\def\g{\gamma}
\def\d{\delta}
\def\TT{{\mathscr{T}}}

\def\ba{\mathbf{a}}
\def\cb{\mathbf{c}}
\def\e{\mathbf{e}}

\def\f{\mathbf{f}}
\def\frakg{\mathfrak{g}}

\def\lowerb{b}
\def\oneb{\mathbf{1}}

\def\q{\mathbf{q}}

\newcommand{\bs}{\mathbf{s}}
\def\vb{\mathbf{v}}
\def\bv{\mathbf{v}}
\def\bu{\mathbf{u}}
\def\bw{\mathbf{w}}
\def\w{\mathbf{w}}
\def\x{\mathbf{x}}
\def\xb{\mathbf{x}}

\def\y{\mathbf{y}}
\def\z{\mathbf{z}}
\def\A{\mathbf{A}}
\def\B{\mathbf{B}}
\def\Cb{\mathbf{C}}
\def\CC{{\mathscr{C}}}
\def\D{\mathbf{D}}
\def\E{\mathbf{E}}
\def\F{\mathbf{F}}
\def\G{\mathbf{G}}
\def\K{\mathbf{K}}

\def\M{\mathbf{M}}
\def\Nb{\mathbf{N}}

\def\P{\mathbf{P}}
\def\Qb{\mathbf{Q}}
\def\Rb{\mathbf{R}}

\def\Tb{\mathbf{T}}

\def\U{\mathbf{U}}
\def\V{\mathbf{V}}
\def\X{\mathbf{X}}
\def\Y{\mathbf{Y}}
\def\Zb{\mathbf{Z}}
\def\I{\mathbf{I}}

\def\RR{{\mathscr{R}}}
\def\o{{\mathbf{o}}}

\def\i{\hbox{\bf{i}}}

\newcommand{\st}{\text{s.t.~}}

\newcommand{\resp}{resp.~}

\def\Mbb{\mathbb{M}}
\def\Sbb{\mathbb{S}}
\def\Hbb{\mathbb{H}}
\def\Bbb{\mathbb{B}}
\def\Dbb{\mathbb{D}}

\DeclareMathOperator{\cdeg}{cdeg}
\def\cyc{\overset{\text{cyc}}{\sim}}

\DeclareMathOperator{\II}{II}
\DeclareMathOperator{\Trace}{tr}
\DeclareMathOperator{\cov}{cov}
\DeclareMathOperator{\trace}{trace}

\def\RX{\R \langle \underline{x} \rangle}
\newcommand{\RXI}[1]{\R \langle \underline{x}, I_{#1} \rangle }
\newcommand{\SymRXI}[1]{\Sym \R \langle \underline{x}, I_{#1} \rangle }
\def\SigmaX{\Sigma \langle \underline{x} \rangle}
\newcommand{\SigmaXI}[1]{\Sigma \langle \underline{x},I_{#1} \rangle }

\def\W{\mathbf{W}}

\def\ux{\underline x}

\def\SymRX{\Sym \RX}

\def\cA{\mathcal A}

\def\cD{\mathcal{D}}

\def\cH{\mathcal H}
\def\cM{\mathcal{M}}

\def\sA{\mathscr A}
\def\sC{\mathscr C}
\def\AA{{\mathcal{A}}}

\def\bsigma{{\boldsymbol{\sigma}}}

\newcommand{\R}{\mathbb{R}}
\newcommand{\Q}{\mathbb{Q}}
\newcommand{\C}{\mathbb{C}}
\newcommand{\N}{\mathbb{N}}
\newcommand{\Z}{\mathbb{Z}}

\DeclareMathOperator{\Basis}{\mathscr{B}}

\DeclareMathOperator{\bigo}{\mathcal{O}}

\DeclareMathOperator{\cs}{cs}
\DeclareMathOperator{\csts}{cs-ts}
\DeclareMathOperator{\opt}{opt}
\DeclareMathOperator{\ts}{ts}
\DeclareMathOperator{\tr}{tr}
\def\sparse{\cs}

\def\SDSOS{\hbox{\rm{SDSOS}}}

\DeclareMathOperator{\Conv}{conv}

\DeclareMathOperator{\diag}{Diag}
\newcommand{\relu}{{\rm ReLU}}

\def\var{\hbox{\rm{var}}}
\DeclareMathOperator{\New}{\mathcal{N}}

\DeclareMathOperator{\rank}{rank}

\DeclareMathOperator{\supp}{supp}

\DeclareMathOperator{\Sym}{Sym}

\DeclareMathOperator{\vol}{vol}
\newcommand{\vge}{\mathbin{\rotatebox[origin=c]{90}{$\ge$}}}
\newcommand{\lonenorm}[1]{\Vert #1 \Vert_1}

\newcommand{\psdp}{p_\text{sdp}}
\newcommand{\dsdp}{d_\text{sdp}}
\newcommand{\epsdp}{\epsilon_\text{sdp}}

\theoreminframepreskip{0.2cm}
\theoreminframepostskip{0.2cm}
\theoremframepreskip{0.4cm}
\theoremframepostskip{0.4cm}
\renewtheorem{theorem}{Theorem}[chapter]
\newframedtheorem{theoremf}[theorem]{Theorem}
\renewtheorem{lemma}[theorem]{Lemma}
\renewtheorem{proposition}[theorem]{Proposition}
\renewtheorem{corollary}[theorem]{Corollary}
\renewtheorem{example}[theorem]{Example}
\renewtheorem{remark}[theorem]{Remark}
\renewtheorem{definition}[theorem]{Definition}

\newcommand{\coq}{\text{\sc Coq}}
\newcommand{\hol}{\text{\sc Hol-light}}
\newcommand{\ncsostools}{\mathtt{NCSOStools}}

\newcommand{\sparsepop}{\mathtt{SparsePOP}}
\newcommand{\realtofloat}{\mathtt{Real2Float}}
\newcommand{\sthreefp}{\mathtt{s3fp}}
\newcommand{\fptaylor}{\mathtt{FPTaylor}}

\newcommand{\rosa}{\text{\sc Rosa}}
\newcommand{\tssos}{\text{\sc TSSOS}}
\newcommand{\spot}{\text{\sc SPOT}}
\newcommand{\sdpa}{\text{\sc SDPA}}

\newcommand{\ncpoltosdpa}{\textsc{Ncpol2sdpa}}
\newcommand{\gloptipoly}{\textsc{GloptiPoly}}
\newcommand{\yalmip}{\textsc{Yalmip}}
\newcommand{\mosek}{\textsc{Mosek}}

\newcommand{\sdpbound}{\mathtt{cs\_sdp}}
\newcommand{\sdpboundfun}[3]{\mathtt{cs\_sdp}(#1, #2, #3)}
\newcommand{\sdppoly}{\mathtt{cs\_sdp}}
\newcommand{\sdppolyfun}[3]{\mathtt{cs\_sdp}(#1, #2, #3)}
\newcommand{\iaboundfun}[2]{\mathtt{ia\_bound}(#1, #2)}
\newcommand{\iabound}{\mathtt{ia\_bound}}

\newcommand{\bop}{\mathtt{bop}}
\newcommand{\precrnd}{\text{prec}}
\newcommand{\sparsegns}{\texttt{SparseGNS}}
\newcommand{\sparseeiggns}{\texttt{SparseEigGNS}}

\def\NF{\hbox{\rm{NF}}}
\DeclareMathOperator*{\argmax}{arg\,max}

%% file: acronym.tex

\chapter*{List of Acronyms}


\begin{acronym}
\renewcommand{\\}{}
\acro{moment-SOS}{moment-sums of squares}
\acro{SDP}{semidefinite programming}
\acro{SOS}{sum of squares}
\acro{LP}{linear programming}
\acro{PSD}{positive semidefinite}
\acro{LMI}{linear matrix inequality}
\acro{RIP}{running intersection property}
\acro{POP}{polynomial optimization problem}
\acro{GMP}{generalized moment problem}
\acro{CS}{correlative sparsity}
\acro{csp}{correlative sparsity pattern}
\acro{CSSOS}{CS-adpated moment-SOS}
\acro{QCQP}{quadratically constrained quadratic program}
\acro{nc}{noncommutative}
\acro{SOHS}{sum of Hermitian squares}
\acro{GNS}{Gelfand-Naimark-Segal}
\acro{TS}{term sparsity}
\acro{tsp}{term sparsity pattern}
\acro{TSSOS}{TS-adpated moment-SOS}
\acro{CS-TSSOS}{CS-TS adpated moment-SOS}
\acro{CPOP}{complex polynomial optimization problem}
\acro{JSR}{joint spectral radius}
\acro{SONC}{sum of nonnegative circuits}
\acro{CTP}{constant trace property}

\end{acronym}

%% file: notation.tex
\chapter*{List of Symbols}

\addcontentsline{toc}{chapter}{List of Symbols}

\begin{longtable}{ll}
  $\N$ & $\{0,1,2,\ldots\}$\\
  $\N^*$ & $\{1,2,\ldots\}$\\
  $\Q$ & the field of rational numbers\\
  $\R$ & the field of real numbers\\
  $\C$ & the field of complex numbers\\
  $\mathbf{0}$&the zero vector\\
  $\x=(x_1,\ldots,x_n)$ & a tuple of real variables\\
  $\supp(f)$&the support of the polynomial $f$\\
  $[m]$&$\{1,2,\ldots,m\}$\\
  $|\cdot|$&the cardinality of a set or 1-norm of a vector\\
  $\R^{n\times m}$&the set of $n\times m$ real matrices\\
  $\Sbb_n$ & the set of real symmetric $n\times n$ matrices\\
  $\Sbb_n^+$ & the set of $n\times n$ PSD matrices\\
  $\langle \A,\,\B \rangle$ & the trace of $\A\B$ for $\A,\B\in\Sbb_n$\\
  $\I_n$ & the $n\times n$ identity matrix\\
  $\M \succcurlyeq 0$& $\M$ is a PSD matrix\\
  $F,G,H$&graphs\\
  $G(V,E)$&a graph with nodes $V$ and edges $E$\\
  $V(G)$ (resp. $E(G)$)&the node (resp. edge) set of the graph $G$\\
  $\B_G$ & the adjacency matrix of the graph $G$ with unit diagonal\\
  $G\subseteq H$& $G$ is a subgraph of $H$\\
  $G'$&a chordal extension of the graph $G$\\
  $\Sbb(G)$ & the set of real symmetric matrices with sparsity pattern $G$\\
  $\Pi_{G}$ & the projection from $\Sbb_{|V(G)|}$ to the subspace $\Sbb(G)$\\
  $\frakg = \{g_1,\dots,g_m\}$ & a set of polynomials defining the constraints\\
  $\R[\x]$ & the ring of real $n$-variate polynomials\\
  $\R[\x]_{2d}$ & the set of real $n$-variate polynomials of degree at most $2d$\\
  $\Sigma[\x]$ & the set of SOS polynomials\\
  $\Sigma[\x]_{d}$ & the set of SOS polynomials of degree at most $2d$\\
  $\cM(\frakg)$ & the quadratic module generated by $\frakg$\\
  $\cM(\frakg)_{r}$&the $r$-truncated quadratic module generated by $\frakg$\\
  $\X$ & a basic semialgebraic set\\
  $\mu, \nu$&measures\\
  $\N^{n}_r$ & $\{\a\in\N^{n} \mid \sum_{j=1}^{n} \alpha_j \leq r \}$\\
  $d_j$ & the ceil of half degree of $g_j \in \frakg$ \\
  $r$ & relaxation order \\
  $r_{\min}$ & minimum relaxation order \\
  $\y$ & a moment sequence\\
  $L_\y$& the linear functional associated to $\y$\\
  $\M_r(\y)$&the $r$-th order moment matrix associated to $\y$\\
  $\M_r(g\y)$ & the $r$-th order localizing matrix associated to $\y$ and $g$\\
  $\delta_{\mathbf{a}}$ & the Dirac measure centered at $\mathbf{a}$ \\
  $p$ & number of variable cliques \\
  $s$ & sparse order \\
  $\ux = (x_1,\dots, x_n)$&a tuple of noncommutating variables\\
  $\RX$ & the ring of real nc $n$-variate polynomials\\
  $\W_r$ & the vector of nc monomials of degree at most $r$\\
  $\SigmaX$ & the set of SOHS polynomials\\
  $\cD_{\frakg}$&the nc semialgebraic set associated to $\frakg$\\
\end{longtable}

%% file: preface.tex
\chapter*{Preface}

Consider the following list of problems arising from various distinct fields:

\begin{itemize}

\item Design certifiable algorithms for robust geometric perception in the presence of a large amount of outliers;
\item Minimizing a sum of rational fractions to estimate the fundamental matrix in epipolar geometry;
\item Computing the maximal roundoff error bound for the output of a numerical program;
\item Certifying the robustness of a deep neural network;
\item Computing the maximum violation level of Bell inequalities;
\item Verifying the stability of a networked system or a control system under deadline constraints;
\item Approximate stability regions of differential systems, such as reachable sets or positively invariant sets;
\item Finding a maximum cut in a graph;
\item Minimizing the generator fuel cost under alternative current power-flow constraints.
\end{itemize}

All these important applications related to computer vision, computer arithmetic, deep learning, entanglement in quantum information, graph theory and energy networks, can be successfully tackled within the
framework of polynomial optimization, an emerging field with growing research efforts in
the last two decades.
One key advantage of these techniques is their ability to model a wide range of problems
using optimization formulations.
Polynomial optimization heavily relies on the \ac{moment-SOS} approach proposed by Lasserre \citepreface{Las01sos}, which provides certificates for positive polynomials.
The problem of minimizing a polynomial over a set of polynomial (in)-equalities is an NP-hard non-convex problem.
It turns out that this problem can be cast as an infinite-dimensional linear problem over a set of probability measures.
Thanks to powerful results from real algebraic geometry \citepreface{Putinar1993positive}, one can convert this linear problem into a nested sequence of finite-dimensional convex problems.
At each step of the associated hierarchy, one needs to solve a fixed size semidefinite program (an optimization program with a linear cost and constraints over matrices with nonnegative eigenvalues), which can be in turn solved with efficient numerical tools.
On the practical side however, there is \emph{no-free lunch} and such optimization methods usually encompass severe scalability issues.
The underlying reason is that for optimization problems involving polynomials in $n$ variables of degree at most $2d$, the size of the matrices involved at step $r \geq d$ of Lasserre's hierarchy of \ac{SDP} relaxations is proportional to $\binom{n+r}{r}$.
Fortunately, for many applications, including the ones formerly mentioned, we can \emph{look at the problem in the eyes} and exploit the inherent data structure arising from the cost and constraints describing the problem, for instance sparsity or symmetries.

\begin{framed}
This book presents several research efforts to tackle this scientific challenge with important
computational implications, and provides the development of alternative optimization
schemes that scale well in terms of computational complexity, at least in some identified class
of problems.
\end{framed}
~\\
The presented algorithmic framework in this book mainly exploits the sparsity structure of the input data to solve large-scale polynomial optimization problems.
For unconstrained problems involving a few terms, a first remedy consists of
reducing the size of the relaxations by discarding the terms which never appear in the
support of the \ac{SOS} decompositions.
This technique, based on a result by Reznick \citepreface{Reznick78}, consists of computing the Newton polytope of the input polynomial (the convex hull of the support of this polynomial) and selecting only monomials with supports lying in half of this polytope.

We present sparsity-exploiting hierarchies of relaxations, for either unconstrained or constrained polynomial optimization problems.
By contrast with the dense hierarchies, they provide faster approximation of the solution in practice but also come with the same theoretical convergence guarantees.
Our framework is not restricted to \emph{static} polynomial optimization, and we expose hierarchies
of approximations for values of interest arising from the analysis of dynamical systems.
We also present various extensions to problems involving noncommuting variables, e.g., matrices of arbitrary size or quantum physic operators.

At this point, we would like to emphasize the existence of alternatives to the positivity certificates based on sparse \ac{SOS} decompositions.
Instead of computing \ac{SOS} decompositions with \ac{SDP}, one can compute other positivity certificates based on \ac{LP} for Bernstein decompositions or Krivine-Stengle certificates, geometric/second-order cone programming for nonnegative circuits and scaled diagonally dominant \ac{SOS}, relative entropy programming for arithmetic-geometric-exponentials.
This book also presents an overview of these various alternative decompositions. 

A second point to emphasize is that the concept of sparsity is inherent to many scientific fields, and we outline some similarities and differences with the algorithmic framework presented in this book.
In the context of machine learning, statistics, or signal processing, exploiting sparsity boils down to select variables or features, usually with $\ell_1$-norm regularization \citepreface{beck2009fast}.
It is commonly employed to make the model or the prediction more interpretable or less expensive to use.
In other words, even if the underlying problem does not admit sparse solutions, one still hopes to be able to find the best sparse approximation.
A similar situation occurs in the context of dynamical systems with sparse state constraints and dynamics, where the set of trajectories is not necessarily sparse.
In the context of algebraic geometry, people have considered sparse systems of polynomial equations, where \emph{sparse} means that the set of terms appearing in each equation is fixed.
Bernshtein's theorem \citepreface{bernshtein1975number} is a key ingredient as it provides an accurate bound for the expected number of complex roots, based on the mixed volume of the Newton polytopes of polynomials describing the system.
We similarly exploit support information given by Newton polytopes for our term-sparsity based hierarchies, presented in Part II.\\
~\\
This book is organized as follows:

~\\
\textbf{Chapter \ref{chap:sdpsparse}} recalls some preliminary background on semidefinite programming, sparse
matrix theory.

~\\
\textbf{Chapter \ref{chap:densehierarchy}} outlines the basic concepts of the \ac{moment-SOS} hierarchy in polynomial optimization.

~\\
\textbf{Part I}
~\\

The first part of the book focuses on the notion of "correlative sparsity", occurring when there
are few correlations between the variables of the input problem.
This research investigation was initially developed by \citepreface{Waki06} and \citepreface{Las06}.

~\\
\textbf{Chapter \ref{chap:cs}} is concerned with this first sparse variant of the \ac{moment-SOS} hierarchy, based on correlative sparsity.

~\\
\textbf{Chapter \ref{chap:roundoff}} explains how to apply the sparse \ac{moment-SOS} hierarchy to provide efficiently upper bounds on roundoff errors of floating-point nonlinear programs.

~\\
\textbf{Chapter \ref{chap:lip}} focuses on robustness certification of deep neural networks, in particular via Lipschitz constant estimation.

~\\
\textbf{Chapter \ref{chap:ncsparse}} describes a very distinct application for optimization of polynomials in noncommuting
variables.
We outline promising research perspectives in quantum information theory.

~\\
\textbf{Part II}
~\\

The second part of the book presents a complementary framework, where we show how to
exploit a distinct notion of sparsity, called "term sparsity", occurring when there are a small number of terms involved in the input problem by comparison with the fully dense case.

~\\
\textbf{Chapter \ref{chap:tssos}} focuses on this second sparse variant of the \ac{moment-SOS} hierarchy, based on term sparsity.

~\\
\textbf{Chapter \ref{chap:cstssos}} explains how to combine correlative and term sparsity.

~\\
\textbf{Chapter \ref{chap:opf}} extends this term sparsity framework to complex polynomial optimization and shows how the resulting scheme can handle optimal power flow problems with tens of thousands of variables and constraints.

~\\
\textbf{Chapter \ref{chap:ncts}} extends the framework of exploiting term sparsity to noncommutative polynomial optimization (namely, eigenvalue optimization).

~\\
\textbf{Chapter \ref{chap:sparsejsr}} is concerned with the application of this term sparsity framework to analyze the stability of various control
systems, either coming from the networked systems literature or systems under deadline
constraints.

~\\
\textbf{Chapter \ref{chap:misc}} presents alternative algorithms to improve the scalability of polynomial optimization methods. 
First, we present algorithms based on sums of nonnegative circuit
polynomials, recently introduced classes of
nonnegativity certificates for sparse polynomials, which are independent of well-known
methods based on sums of squares.
Then, we outline existing methods to speed-up the computation of the semidefinite relaxations. 
~\\
~\\
\textbf{Appendix}
~\\

At the end of the book, we describe how to use various solvers available either in MATLAB or
Julia.
This dedicated appendix aims at guiding practitioners to solve optimization problems involving sparse polynomials.

~\\
\textbf{Appendix \ref{chap:matlab}} explains how to implement \ac{moment-SOS} relaxations with software packages $\gloptipoly$ and Yalmip.

~\\
\textbf{Appendix \ref{chap:julia}} focuses on our sparsity exploiting algorithms, implemented in the {\tt TSSOS} library available at \href{https://github.com/wangjie212/TSSOS}{https://github.com/wangjie212/TSSOS}.

\begin{itemize}
\item For the sake of conciseness and clarity of exposition, most proofs are postponed to ease the reading.
When the proof is either short or simple, we sometimes include it right after its corresponding statement. Otherwise, we refer to this proof in the \emph{Notes and sources} section at the end of the corresponding chapter.
\item Some of the theorems are framed in the book, in order to emphasize their specific importance.
\end{itemize}


\input{preface.bbl}

%% file: sdp.tex
\part*{Preliminary background}

\chapter{Semidefinite programming and sparse matrices}\label{chap:sdpsparse}

In this chapter and the next one, we describe the foundations on which several parts of our work lie.
Semidefinite programming and sparse matrices are described in this chapter while Chapter \ref{chap:densehierarchy} is dedicated to the \ac{moment-SOS} hierarchy of \ac{SDP} relaxations, now widely used to certify lower bounds of polynomial optimization problems.

\section{SDP and interior-point methods}\label{sec:sdp}
Even though \ac{SDP} is not our main topic of interest, several encountered problems can be cast as such programs.

First, we introduce some useful notations. We consider the vector space $\Sbb_n$ of real symmetric $n\times n$ matrices, which is equipped with the usual inner product\index{inner product} $\langle \A,\,\B \rangle = \tr(\A\B)$ for $\A,\B\in\Sbb_n$\index{trace}.
Let $\I_n$\index{$\I_n$} be the $n \times n$ identity matrix. 
A matrix $\M\in\Sbb_n$ is called \emph{\ac{PSD}} (\resp \emph{positive definite}) if $\x^\intercal \M \x \geq 0$ (\resp $>0$), for all $\x\in\R^n$. In this case, we write $\M \succcurlyeq 0$ and
define a partial order by writing $\A \succcurlyeq \B$ (\resp $\A \succ B$) if and only if $\A - \B$ is positive semidefinite (\resp positive definite). The set of $n\times n$ \ac{PSD} matrices is denoted by $\Sbb_n^+$.

In semidefinite programming, one minimizes a linear objective function subject to a \ac{LMI}. The variable of the problem is the vector $\y \in \R^m$ and the input data of the problem are the vector $\cb \in \R^m$ and symmetric matrices $\F_0,\dots,\F_m \in \Sbb_n$. The primal semidefinite program is defined as follows:

\begin{equation}\label{eq:sdp_primal}
\begin{aligned}
\psdp \coloneqq  & \inf_{\y \in \R^m} & & \cb^\intercal \y\\
& \,\,\,\,\st & & \F(\y) \succeq 0
\end{aligned}
\end{equation}
where \[ \F(\y) \coloneqq  \F_0 + \sum_{i = 1}^m y_i \F_i.\]

The primal problem~\eqref{eq:sdp_primal} is convex since the linear objective function and the linear matrix inequality constraint are both convex. We say that $\y$ is primal feasible (\resp strictly feasible) if $\F(\y) \succeq 0$ (\resp $\F(\y) \succ 0$).
Furthermore, we associate the following dual problem with the primal problem~\eqref{eq:sdp_primal}:

\begin{equation}\label{eq:sdp_dual}
\begin{aligned}
\dsdp \coloneqq  & \sup_{\G \in \Sbb_n} & & - \langle \F_0, \G \rangle \\
& \,\,\, \st & & \langle \F_i, \G \rangle = c_i, \quad i \in [m]\\
&     & & \G \succeq 0
\end{aligned}
\end{equation}

The variable of the dual program~\eqref{eq:sdp_dual} is the real symmetric matrix $\G \in \Sbb_n$. We say that $\G$ is dual feasible (\resp strictly feasible) if $\langle \F_i, \G \rangle = c_i$, $i \in [m]$  and $\G \succeq 0$ (\resp $\G \succ 0$).

We will describe briefly the primal-dual interior-point method (used for instance by $\sdpa$ \citesdp{sdpa}, $\mosek$ \citesdp{mosek}), that solves the following primal-dual optimization problem:

\begin{equation}\label{eq:sdp_primal_dual}
\begin{cases}
\inf\limits_{\y \in \R^m, \G \in \Sbb_n} & \eta (\y, \G)  \\
\quad\,\,\,\,\st & \langle \F_i, \G \rangle = c_i, \quad i \in [m]\\
& \F(\y) \succeq 0,\G \succeq 0
\end{cases}
\end{equation}
where $\eta (\y, \G) \coloneqq  \cb^\intercal \y + \langle \F_0, \G \rangle$.

We notice that the objective function $\eta$ of the program~\eqref{eq:sdp_primal_dual} is the difference between the objective function of the primal program~\eqref{eq:sdp_primal} and its dual version~\eqref{eq:sdp_dual}. We call this function the duality gap. Let us suppose that $\y$ is primal feasible and $\G$ is dual feasible, then $\eta$ is nonnegative. Indeed, we have

\begin{equation}\label{eq:sdp_gap}
\eta (\y, \G) = \sum_{i = 1}^m \langle \F_i, \G \rangle y_i + \langle \F_0, \G \rangle = \langle \F(\y), \G \rangle \geq 0.
\end{equation}
The last inequality comes from the fact that the matrices $\F(\y)$ and $\G$ are both \ac{PSD}.

Then, one can easily prove that the nonnegativity of $\eta$ implies the following inequalities:

\begin{equation}\label{eq:sdp_ineqs}
\dsdp \leq - \langle \F_0, \G \rangle \leq \cb^\intercal \y \leq \psdp.
\end{equation}

Our problems that can be cast as \ac{SDP}s satisfy certain assumptions, so that there exists a (strictly feasible) primal-dual optimal solution (i.e., a primal strictly feasible $\y$ solving~\eqref{eq:sdp_primal} and a dual strictly feasible $\G$ solving~\eqref{eq:sdp_dual}). Then, all inequalities in~\eqref{eq:sdp_ineqs} become equalities and there is no duality gap ($\eta (\y, \G) = 0$):

\begin{equation}\label{eq:sdp_eqs}
\dsdp = - \langle \F_0, \G \rangle = \cb^\intercal \y = \psdp.
\end{equation}

Thus, we will assume that such a primal-dual optimal solution exists in the sequel.
We also introduce the barrier function\index{barrier function}

\begin{equation}\label{eq:barrier}
\Phi (\y) \coloneqq 
\left\{
\begin{array}{ll}
\log \det ( \F(\y)^{-1} ) & \text{if } \F(\y) \succ 0, \\
+\infty & \text{otherwise}.
\end{array}
\right.
\end{equation}

This barrier function exhibits several nice properties: $\Phi$ is strictly convex, analytic and self-concordant. The unique minimizer $\y^{\opt}$ of $\Phi$ is called the analytic center of the \ac{LMI} $\F(\y) \succeq 0$.
This self-concordant barrier function guarantees that the number of iterations of the interior-point method is bounded by a polynomial in the dimension ($n$ and $m$) and the number of accuracy digits of the solution.

\if{
We report the best known complexity bounds results for SDPs.
For this, let us consider an instance~$(e)$ of the primal \ac{SDP} program \eqref{eq:sdp_primal}.
An $\epsdp$-solution of~$(e)$ is a solution $\y \in \R^m$ of the following feasibility problem:

\begin{equation}\label{eq:sdp_eps}
\begin{aligned}
c^\intercal\y -\psdp&\leq\epsdp\\
F(\y) & \succeq- \epsdp I_n \enspace.
\end{aligned}
\end{equation}

We note $\text{Compl}(e, \epsdp)$ the number of real arithmetic operations needed to obtain an $\epsdp$-solution of~$(e)$ and   $\text{Digits}(e, \epsdp)$ the number of accuracy digits in an $\epsdp$-solution of~$(e)$.
The arithmetic complexity of an $\epsdp$-solution is:
\begin{equation}\label{eq:sdp_compl}
\text{Compl}(e, \epsdp) \coloneqq   O(1) (1 + \sqrt{n}) (m^3 + n^2 m^2  + m n ^3) \text{Digits}(e, \epsdp) \enspace,
\end{equation}
where the number of accuracy digits is given by the following:
\begin{equation}\label{eq:sdp_digits}
\text{Digits}(e, \epsdp) \coloneqq  \log \Bigl(\frac{\text{Size}(e) + \lonenorm{\text{Data}(e)} + \epsdp}{\epsdp^2} \Bigr) \enspace ,
\end{equation}
with Size$(e) = $  and Data$(e) = $.
}\fi

\section{Chordal graphs and sparse matrices}\label{sec:sparsemat}
We briefly recall some basic notions from graph theory. An {\em (undirected) graph} $G(V,E)$ or simply $G$ consists of a set of nodes $V$ and a set of edges $E\subseteq\{\{v_i,v_j\}\mid v_i\ne v_j,(v_i,v_j)\in V\times V\}$. For a graph $G$, we use $V(G)$ and $E(G)$ to indicate the node set of $G$ and the edge set of $G$, respectively. The {\em adjacency matrix} of a graph $G$ is denoted by $\B_G$ for which we put ones on its diagonal. For two graphs $G,H$, we say that $G$ is a {\em subgraph} of $H$ if $V(G)\subseteq V(H)$ and $E(G)\subseteq E(H)$, denoted by $G\subseteq H$. For a graph $G(V,E)$, a {\em cycle} of length $k$ is a set of nodes $\{v_1,v_2,\ldots,v_k\}\subseteq V$ with $\{v_k,v_1\}\in E$ and $\{v_i, v_{i+1}\}\in E$, for $i\in[k-1]$. A {\em chord} in a cycle $\{v_1,v_2,\ldots,v_k\}$ is an edge $\{v_i, v_j\}$ that joins two nonconsecutive nodes in the cycle. A {\em clique} $C\subseteq V$ of $G$ is a subset of nodes where $\{v_i,v_j\}\in E$ for any $v_i,v_j\in C$. If a clique is not a subset of any other clique, then it is called a {\em maximal clique}.
\begin{definition}[chordal graph]
A graph is called a {\em chordal graph} if all its cycles of length at least four have a chord. 
\end{definition}
The notion of chordal graphs plays an important role in sparse matrix theory. In particular, it is known that maximal cliques of a chordal graph can be enumerated efficiently in linear time in the number of nodes and edges of the graph. See e.g.\,\citesdp{gavril1972algorithms,va} for the details.

The maximal cliques $I_1,\dots,I_p$ of a chordal graph (possibly after some reordering) satisfy the so-called \emph{\ac{RIP}}, i.e., for every $k\in[p-1]$, it holds
\begin{align}\label{eq:RIP}
	\left(I_{k+1} \cap \bigcup_{j \leq k} I_j \right)\subseteq I_i \quad \text{for some } i \leq k.
\end{align}
The \ac{RIP} actually gives an equivalent characterization of chordal graphs.

\begin{theorem}\label{thm:rip}
A connected graph is chordal if and only if its maximal cliques after an appropriate ordering satisfy the \ac{RIP}.
\end{theorem}

Any non-chordal graph $G(V,E)$ can always be extended to a chordal graph $G'(V,E')$ by adding appropriate edges to $E$, which is called a {\em chordal extension} of $G(V,E)$. The chordal extension of $G$ is usually not unique. We use the symbol $G'$ to indicate a specific chordal extension of $G$. For graphs $G\subseteq H$, we assume that $G'\subseteq H'$ always holds for our purpose. 
For a graph $G$, among all chordal extensions of $G$, there is a particular one $G'$ which makes every connected component of $G$ to be a clique. Accordingly, a matrix with adjacency graph $G'$ is block diagonal (after an appropriate permutation on rows and columns): each block corresponds to a connected component of $G$. We call this chordal extension the {\em maximal} chordal extension.
Besides, we are also interested in smallest chordal extensions. By definition, a \emph{smallest chordal extension} is a chordal extension with the smallest clique number (i.e., the maximal size of maximal cliques). However, computing a smallest chordal extension is generally NP-complete \citesdp{arnborg1987complexity}. Therefore in practice we compute approximately smallest chordal extensions instead with efficient heuristic algorithms; see \citesdp{treewidth} for more detailed discussions.

\begin{example}\label{ex:graph}
	Let us consider the graph $G(V, E)$ represented in Figure \ref{fig:graph}, with the set of nodes $V = \{1,2,3,4,5,6 \}$ and
	\[
	E = \{ \{1,2\}, \{1,3\}, \{1,4\}, \{1,5\}, \{1,6\}, \{2,3\}, \{2,5\}, \{3,6\}, \{5,6\} \}.
	\]
	and the corresponding adjacency matrix 
	\[
	\B_G = 
	\begin{bmatrix}
		1 & 1 & 1 & 1 & 1 & 1 \\
		1 & 1 & 1 & 0 & 1 & 0 \\
		1 & 1 & 1 & 0 & 1 & 1 \\
		1 & 0 & 0 & 1 & 0 & 0 \\
		1 & 1 & 1 & 0 & 1 & 1 \\
		1 & 0 & 1 & 0 & 1 & 1 
	\end{bmatrix}.
	\]
	One example of cycle of length $3$ is $\{1,5,6\}$ and one example of cycle of length $4$ is $\{6,3,2,5\}$.
	Note that this graph is not chordal since there is no chord in this latter cycle.
	It is enough to add en edge between the nodes $2$ and $6$ (or alternatively between the nodes $3$ and $5$) to obtain a chordal extension of $G$.
\end{example}
\begin{figure}[!t]	
	\centering
	\includegraphics[scale=0.6]{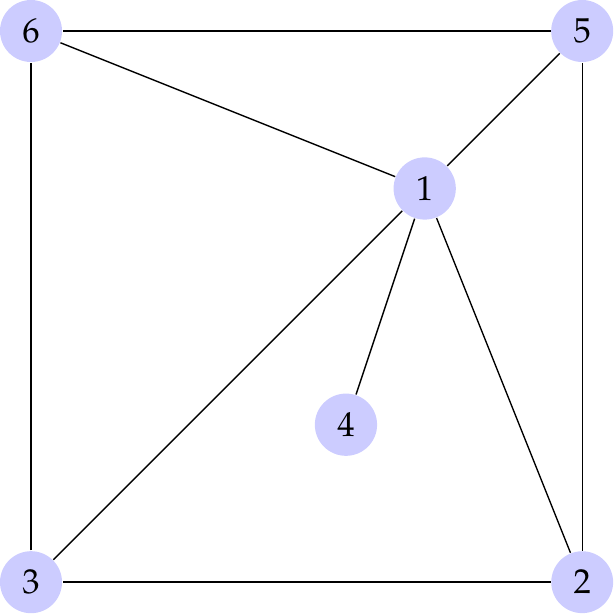}
	\caption{The graph from Example~\ref{ex:graph}.}
	\label{fig:graph}
\end{figure}

Let $n\in\N^*$.
Given a graph $G(V,E)$ with $V=[n]$, a symmetric matrix $\Qb$ with rows and columns indexed by $V$ is said to have sparsity pattern $G$ if $\Qb_{ij}=\Qb_{ij}=0$ whenever $i\ne j$ and $\{i,j\}\notin E$. Let $\Sbb(G)$ be the set of real symmetric matrices with sparsity pattern $G$. The \ac{PSD} matrices with sparsity pattern $G$ form a convex cone
\begin{equation}\label{sec2-eq5}
\Sbb^+_{|V|}\cap\Sbb(G)=\{\Qb\in\Sbb(G)\mid \Qb\succeq0\}.
\end{equation}

A matrix in $\Sbb(G)$ exhibits a block structure: each block corresponds to a maximal clique of $G$. Figure \ref{qbd} depicts an instance of such block structures.
Note that there might be overlaps between blocks because different maximal cliques may share nodes.

\begin{figure}[htbp]
\centering
\begin{tikzpicture}
\draw (0,0) rectangle (4,4);
\filldraw[fill=blue, fill opacity=0.3] (0,2.7) rectangle (1.3,4);
\filldraw[fill=blue, fill opacity=0.3] (0.7,1.8) rectangle (2.2,3.3);
\filldraw[fill=blue, fill opacity=0.3] (1.8,1.2) rectangle (2.8,2.2);
\fill (3.2,0.8) circle (0.3ex);
\fill (3.4,0.6) circle (0.3ex);
\fill (3.6,0.4) circle (0.3ex);
\end{tikzpicture}
\caption{An instance of block structures for matrices in $\Sbb(G)$. The blue area indicates the positions of possible nonzero entries.}\label{qbd}
\end{figure}
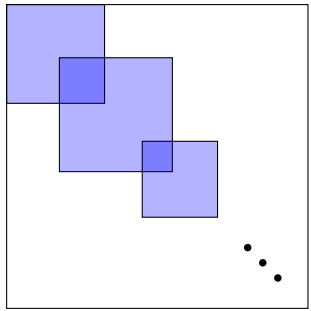

Given a maximal clique $C$ of $G(V,E)$, we define a matrix $\Rb_{C}\in\R^{|C|\times |V|}$ by
\begin{equation}\label{sec2-eq6}
[\Rb_{C}]_{ij}=\begin{cases}
1, &\textrm{if }C(i)=j,\\
0, &\textrm{otherwise},
\end{cases}
\end{equation}
where $C(i)$ denotes the $i$-th node in $C$, sorted with respect to the ordering compatible with $V$. Note that $\Qb_{C}=\Rb_{C}\Qb \Rb_{C}^\intercal\in \Sbb_{|C|}$ extracts a principal submatrix $\Qb_C$ indexed by the clique $C$ from a symmetry matrix $\Qb$, and $\Qb = \Rb_{C}^\intercal \Qb_{C} \Rb_{C}$ inflates a $|C|\times|C|$ matrix $\Qb_{C}$ into a sparse $|V|\times |V|$ matrix $\Qb$.

When the sparsity pattern graph $G$ is chordal, the cone $\Sbb^+_{|V|}\cap\Sbb(G)$ can be
decomposed as a sum of simple convex cones, as stated in the following theorem.
\begin{theoremf}\label{th:sparsesdpsum}
Let $G(V,E)$ be a chordal graph and assume that $C_1,\ldots,C_p$ are the list of maximal cliques of $G$.
Then a matrix $\Qb\in\Sbb^+_{|V|}\cap\Sbb(G)$ if and only if there exist $\Qb_{k}\in\Sbb^+_{|C_k|}$ for $k\in[p]$ such that $\Qb=\sum_{k=1}^p \Rb_{C_k}^\intercal \Qb_{k} \Rb_{C_k}$.
\end{theoremf}

Given a graph $G(V,E)$ with $V=[n]$, let $\Pi_{G}$ be the projection from $\Sbb_{|V|}$ to the subspace $\Sbb(G)$, i.e., for $\Qb\in\Sbb_{|V|}$,
\begin{equation}\label{sec2-eq7}
[\Pi_{G}(\Qb)]_{ij}=\begin{cases}
\Qb_{ij}, &\textrm{if }i=j\textrm{ or }\{i,j\}\in E,\\
0, &\textrm{otherwise}.
\end{cases}
\end{equation}
We denote by $\Pi_{G}(\Sbb^+_{|V|})$ the set of matrices in $\Sbb(G)$ that have a \ac{PSD} completion, i.e.,
\begin{equation}\label{sec2-eq8}
\Pi_{G}(\Sbb^+_{|V|})=\left\{\Pi_{G}(\Qb)\mid \Qb \in\Sbb^+_{|V|}\right\}.
\end{equation}
One can easily check that the \ac{PSD} completable cone $\Pi_{G}(\Sbb^+_{|V|})$ and the \ac{PSD} cone $\Sbb^+_{|V|}\cap\Sbb(G)$ form a pair of dual cones in $\Sbb(G)$.
Moreover, for a chordal graph $G$, the decomposition result for the cone $\Sbb^+_{|V|}\cap\Sbb(G)$ in Theorem \ref{th:sparsesdpsum} leads to the following characterization of the \ac{PSD} completable cone $\Pi_{G}(\Sbb^+_{|V|})$.
\begin{theoremf}
\label{th:sparsesdpproj}
Let $G(V,E)$ be a chordal graph and assume that $C_1,\ldots,C_p$ are the list of maximal cliques of $G$. Then a matrix $\Qb\in\Pi_{G}(\Sbb^+_{|V|})$ if and only if $\Qb_{k}=\Rb_{C_k} \Qb \Rb_{C_k}^\intercal\succeq0$ for all $k\in[p]$.
\end{theoremf}

Theorem \ref{th:sparsesdpsum} and Theorem \ref{th:sparsesdpproj} play an important role in sparse semidefinite programming since they admit us to decompose an \ac{SDP} with chordal sparsity pattern into an \ac{SDP} of smaller size, which yields significant computational improvement if the sizes of related maximal cliques are small.

\section{Notes and sources}
\label{sec:sourcesdp}
\ac{SDP} is relevant to a wide range of applications.
The interested reader can find more details on the connection between \ac{SDP} and combinatorial optimization in \citesdp{gvozdenovic2009block}, control theory in \citesdp{boyd1994linear}, positive semidefinite matrix completion in \citesdp{laurent2009matrix}.
A survey on semidefinite programming is available in the paper of Vandenberghe and  Boyd \citesdp{vandenberghe1996semidefinite}.
We emphasize the fact that \ac{SDP}s can be solved efficiently by software e.g., SeDuMi \citesdp{sedumi}, CSDP \citesdp{csdp}, SDPA \citesdp{sdpa}, $\mosek$ \citesdp{mosek}.

We refer to~\citesdp{nesterov1994interior} for more details on barrier functions.
Detailed complexity bounds related to \ac{SDP} solving with interior-point methods can be found in Section 4.6.3 from \citesdp{ben2001lectures}.
With prescribed accuracy, the time complexity of \ac{SDP} (in terms of arithmetic operations) is polynomial with resepct to the number of variables with an exponent greater than $3$; see~\citesdp[Chapter 4]{ben2001lectures} for more details.

For more details about sparse matrices and chordal graphs, the reader is referred to the survey \citesdp{Vandenberghe15}.
Theorem \ref{th:sparsesdpsum} and Theorem \ref{th:sparsesdpproj} are stated as Theorem 9.2 and Theorem  10.1 in \citesdp{Vandenberghe15}, respectively, and were derived much earlier in \citesdp{agler1988positive} and \citesdp{grone1984}, respectively.

The equivalence stated in Theorem \ref{thm:rip} could be read from Theorem 3.4 or Corollary 1 of \citesdp{blair1993introduction}.

\input{sdp.bbl}

%% file: sdp.bbl
\providecommand{\etalchar}[1]{$^{#1}$}

%% file: pop.tex
\chapter{Polynomial optimization and the moment-SOS hierarchy}
\label{chap:densehierarchy}

Polynomial optimization focuses on minimizing or maximizing a polynomial under a set of polynomial inequality constraints.
A polynomial is an expression involving addition, subtraction and multiplication of variables and coefficients.
An example of polynomial in two variables $x_1$ and $x_2$ with rational coefficients is $f(x_1, x_2) = 1/3 + x_1^2 + 2 x_1 x_2 + x_2^2$.
\emph{Semialgebraic} sets are defined with conjunctions and disjunctions of polynomial inequalities with real coefficients.
For instance the two-dimensional unit disk is a semialgebraic set defined as the set of all points $(x_1, x_2)$ satisfying the (single) inequality $1 - x_1^2 - x_2^2 \geq 0$.

In general, computing the \emph{exact} solution of a \ac{POP} over a semialgebraic set is an NP-hard problem. In practice, one can at least try to compute an \emph{approximation} of the solution by considering a \emph{relaxation} of the problem instead of the problem itself.
The approximated solution may not satisfy all the problem constraints but still gives useful information about the exact solution.
Let us illustrate this by considering the minimization of the above polynomial $f(x_1, x_2)$ on the unit disk.
One can replace this disk by a larger set, for instance the product of intervals $[-1, 1] \times [-1, 1]$.
Using basic interval arithmetic, one easily shows that the range of $f$ belongs to $[-4/3, 4/3]$.
Next, one can replace the monomials $x_1^2$, $x_1 x_2$ and $x_2^2$ by three new variables $y_1$, $y_2$ and $y_3$, respectively.
One can relax the initial problem by \ac{LP}, with a cost of $1/3 + y_1 + 2 y_2 + y_3$ and one single linear inequality constraint $1 - y_1 - y_3 \geq 0$.
By hand-solving or by using an \ac{LP} solver, one finds again a lower bound of $-4/3$.
Even if \ac{LP} gives more accurate bounds than interval arithmetic in general, this does not yield any improvement on this example.

One way to obtain more accurate lower bounds is to rely on more sophisticated techniques from the field of convex optimization, e.g.,~\ac{SDP}.
In the seminal paper~\citepop{Las01sos} published in 2001, Lasserre  introduced a hierarchy of relaxations allowing to obtain a converging sequence of lower bounds for the minimum of a polynomial over a semialgebraic set.
Each lower bound is computed by \ac{SDP}.

The idea behind Lasserre's hierarchy is to tackle the \emph{infinite-dimensional} initial problem by solving several \emph{finite-dimensional} primal-dual \ac{SDP} problems.
%
The primal is a \emph{moment} problem, that is an optimization problem where variables are the moments of a Borel measure. The first moment is related to means, the second moment is related to variances, etc. Lasserre showed in~\citepop{Las01sos} that \ac{POP} can be cast as a particular instance of the \ac{GMP}.
In a nutshell, the primal moment problem approximates Borel measures.
The dual is an \ac{SOS} problem, where the variables are the coefficients of \ac{SOS} polynomials (e.g., $(1/\sqrt{3})^2 + (x_1+x_2)^2$).
It is known that not all positive polynomials can be written with \ac{SOS} decompositions. However, when the set of constraints satisfies certain assumptions (slightly stronger than compactness) then one can represent positive polynomials with weighted \ac{SOS} decompositions.
In a nutshell, the dual \ac{SOS} problem approximates positive polynomials.
The \ac{moment-SOS} approach can be used on the above example with either three moment variables or \ac{SOS} of degree 2 to obtain a lower bound of $1/3$.
For this example,  the exact solution is obtained at the first step of the hierarchy. There is no need to go further, i.e.,~to consider primal with moments of greater order (e.g.~the integrals of $x_1^3$, $x_1^2 x_2$, $x_1^4$) or dual with \ac{SOS} polynomials of degree 4 or 6. The reason is that for convex quadratic problems, the first step of the hierarchy gives the exact solution!

In the sequel, we recall more formally some preliminary background material on the mathematical tools related to the \ac{moment-SOS} hierarchy.

\if{
For more general problems involving polynomials, there are several difficulties encountered while using the \ac{moment-SOS} hierarchy for concrete  applications.
When the initial problem involves $n$ variables, the  $r$-th step relaxation of the \ac{moment-SOS} hierarchy involves the rapidly prohibitive cost of $\binom{n+r}{r}$ \ac{SDP} variables.
We are interested in improving the scalability of the hierarchy by exploiting the specific sparsity structure of the polynomials involved in real-world problems.
Important applications arise from various fields, including computer arithmetic (roundoff error bounds), quantum information (noncommutative optimization), optimal power-flow and deep learning.
}\fi

Given $n,d\in\N$, let $\R[\x]$ (resp.~$\R[\x]_{2d}$) stand for the vector space of real $n$-variate polynomials (resp. of degree at most $2d$) in the variable $\x=(x_1,\ldots,x_n) \in \R^n$. A polynomial $f\in\R[\x]$ can be written as $f(\x)=\sum_{\a\in\A}f_{\a}\x^{\a}$ with $\A\subseteq\N^n$ and $f_{\a}\in\R, \x^{\a}=x_1^{\alpha_1}\cdots x_n^{\alpha_n}$. The \emph{support} of $f$ is defined by $\supp(f):=\{\a\in\A\mid f_{\a}\ne0\}$.
A basic compact semialgebraic set $\X$ is a finite conjunction of polynomial superlevel sets.
Namely, given $m \in \N^*$ and a set of polynomials $\frakg \coloneqq  \{g_1,\ldots,g_m \} \subset \R[\x]$, one has
\begin{align}\label{eq:defX}
 \X \coloneqq  & \{\x \in \R^n : g(\x)  \geq 0 \text{ for all } g \in \frakg \} \nonumber \\
  = & \{\x \in \R^n : g_1(\x)  \geq 0, \dots, g_m(\x) \geq 0 \}.
\end{align}
Many sets can be described as such basic closed semialgebraic sets, and the related description is not unique.
Consider for instance the 2-dimensional hypercube $\X = [0, 1]^2$.
A first possible description is given by
\begin{align}\label{ex:defX}
\X = \,& \{(x_1,x_2) \in \R^2 : x_1 - x_1^2  \geq 0 , x_2 - x_2^2  \geq 0 \} \\
 = \,& \{(x_1,x_2) \in \R^2 : g(x_1,x_2)  \geq 0 \text{ for all } g \in \frakg \}, \nonumber
\end{align}
with $\frakg = \{x_1 - x_1^2, x_2 - x_2^2 \}$.
A second one is given by taking $\frakg = \{x_1, 1 - x_1, x_2, 1-x_2 \}$.
\section{Sums of squares and quadratic modules}
\label{sec:sos}
Let $\Sigma[\x]$ stand for the cone of \ac{SOS} polynomials and let $\Sigma[\x]_{d}$ denote the cone of \ac{SOS} polynomials of degree at most $2d$, namely $\Sigma[\x]_{d} \coloneqq  \Sigma[\x] \cap \R[\x]_{2d}$.
Note that all \ac{SOS} polynomials with real coefficients are nonnegative on $\R^n$.
For instance, $\sigma_0 = \frac{1}{2} (x_1 + x_2 - \frac{1}{2})^2$ is a square in $n = 2$ variables of degree $2d = 2$, and is thus obviously  nonnegative on $\R^n$.

For the ease of further notation, we set $g_0(\x) \coloneqq  1$, and $d_j \coloneqq \lceil \deg (g_j) / 2 \rceil$, for all $j = 0,\dots,m$.
Given a basic compact semialgebraic set $\X$ as above and an integer $r\in\N^*$, let $\cM(\frakg)$ be the quadratic module generated by $g_1, \dots, g_m$:
\begin{align*}
	\cM(\frakg) \coloneqq  \left\{\sum_{j=0}^{m} \sigma_j(\x) {g_j} (\x) : \sigma_j \in \Sigma[\x], j = 0, \dots, m  \right\},
\end{align*}
and let $\cM(\frakg)_{r}$ be the $r$-truncated quadratic module:
\begin{align*}
\cM(\frakg)_{r}\coloneqq\left\{\sum_{j=0}^{m}\sigma_j(\x){g_j}(\x):\sigma_j\in\Sigma[\x]_{r-d_j}, j = 0,\dots,m\right\}.
\end{align*}
A first important remark is that all polynomials belonging to $\cM(\frakg)$ are positive on $\X$.
\begin{example}\label{ex:putinar}
To illustrate this later point, let us take the polynomial $f = x_1 x_2$ in two variables, and the 2-dimensional hypercube $X = [0, 1]^2$ described with the basic closed semialgebraic set given in \eqref{ex:defX}, with $\frakg = \{x_1 (1-x_1), x_2 (1 - x_2) \}$.
Let us consider the following decomposition of $f + \frac{1}{8}$:
\[
x_1 x_2 + \frac{1}{8} =  \frac{1}{2} (x_1 + x_2 - \frac{1}{2})^2 \cdot 1 + \frac{1}{2} \cdot x_1 (1 - x_1) + \frac{1}{2} \cdot x_2 (1 - x_2) .
\]
This later decomposition of degree 2 proves that $f = x_1 x_2 + \frac{1}{8}$ lies in the $1$-truncated quadratic module $\cM(\frakg)_1$, with $\sigma_0 = \frac{1}{2} (x_1 + x_2 - \frac{1}{2})^2$, $\sigma_1 = \sigma_2 = \frac{1}{2}$, $g_0 = 1$, $g_1 = x_1 (1 - x_1)$ and $g_2 = x_2 (1 - x_2)$.
Since $\sigma_0, \sigma_1, \sigma_2$ are nonnegative on $\R^2$ and $g_1,g_2$ are nonnegative on $\X$ (by definition),  the above decomposition \emph{certifies} that $f + \frac{1}{8} \geq 0$ on the hypercube $[0, 1]^2$.
This yields in particular that $-\frac{1}{8}$ is a lower bound on the minimum of $f$ on the hypercube.
Since the minimum of $f$ on the hypercube is obviously $0$, it is natural to ask if one can compute a lower bound greater than $-\frac{1}{8}$.
The answer is positive: for all arbitrary small $\varepsilon > 0$, there exists a decomposition of $f + \varepsilon$ in $\cM(\frakg)_r$, for some positive integer $r$ depending on $\varepsilon$.
\end{example}
A second important remark is that $\cM(\frakg)_r \subseteq \cM(\frakg)_{r+1}$, for all $r\in\N^*$, since all \ac{SOS} polynomials of degree $2 r$ can be viewed as \ac{SOS} polynomials of degree $2 r + 2$.

To guarantee the convergence behavior of the relaxations presented in the sequel, we rely on the fact that polynomials which are positive on $\X$ lie in $\cM(\frakg)_r$ for some $r \in \N^*$.
The existence of such \ac{SOS}-based representations is guaranteed when the following condition holds.
\begin{assumption}[Archimedean]
\label{hyp:archimedean}
There exists $N>0$ such that  $N - \| \x \|_2^2\in\cM(\frakg)$.
\end{assumption}
A quadratic module $\cM(\frakg)$ for which Assumption~\ref{hyp:archimedean} holds is said to be {\em Archimedean}.
\begin{theorem}[Putinar's Positivstellensatz]
\label{th:putinar}
Assume that the set $\X$ is defined in \eqref{eq:defX} and the quadratic module $\cM(\frakg)$ is Archimedean.
Then any polynomial positive on $\X$ belongs to $\cM(\frakg)$.
\end{theorem}
Assumption~\ref{hyp:archimedean} is slightly stronger than compactness.
Indeed, compactness of $\X$ already ensures that each variable has finite lower and upper bounds.
One (easy) way to ensure that Assumption~\ref{hyp:archimedean} holds is to add a redundant constraint involving a well-chosen $N$ depending on these bounds, in the definition of $\X$.

\section{Borel measures and moment matrices}
\label{sec:meas}
Given a compact set $\A \subseteq \R^n$, we denote by $\mathscr{M}(\A)$ the vector space of finite signed Borel measures supported on $\A$, namely real-valued functions from the Borel $\sigma$-algebra $\mathcal{B}(\A)$.
The support of a measure $\mu \in \mathscr{M}(\A)$ is defined as the closure of the set of all points $\x$ such that $\mu(\B) \neq 0$ for any open neighborhood $\B$ of $\x$.
We denote by $\mathscr{C}(\A)$ the Banach space of continuous functions on $\A$ equipped with the sup-norm.
Let $\mathscr{C}(\A)'$ stand for the topological dual of $\mathscr{C}(\A)$ (equipped with the sup-norm), i.e.,~the set of continuous linear functionals of $\mathscr{C}(\A)$.
By a Riesz identification theorem, $\mathscr{C}(\A)'$ is isomorphically identified with
$\mathscr{M}(\A)$ equipped with the total variation norm denoted by $\|
\cdot\|_{\text{TV}}$.
Let $\mathscr{C}_+(\A)$ (resp.~$\mathscr{M}_+(\A)$) stand for the cone of nonnegative elements of $\mathscr{C}(\A)$ (resp. $\mathscr{M}(\A)$).
The topology in $\mathscr{C}_+(\A)$ is the strong topology of uniform convergence in contrast with the weak-star topology in $\mathscr{M}_+(\A)$.

With $\X$ being a basic compact semialgebraic set, the restriction of the Lebesgue measure on a subset $\A \subseteq \X$ is
$\lambda_\A (\mathrm{d}\x) \coloneqq  \mathbf{1}_\A(\x)\mathrm{d}\x $,
where $\mathbf{1}_\A : \X \to \{0, 1\}$ stands for the indicator function of $\A$, namely $\mathbf{1}_\A(\x) = 1$ if $\x \in \A$ and
 $\mathbf{1}_\A(\x) = 0$ otherwise.
A sequence $\y\coloneqq (y_\a)_{\a \in \N^n}\subseteq\R$ is said to have a representing measure on $\X$ if there exists $\mu \in \mathscr{M}(\X)$ such that $y_\a = \int \x^\a \mu(\mathrm{d}\x)$ for all $\a \in \N^n$, where we use the multinomial notation $\x^{\a} \coloneqq  x^{\alpha_1}_1 x^{\alpha_2}_2 \cdots x^{\alpha_n}_n$.

Assume that $\mu, \nu \in \mathscr{M}_+(\X)$ have the same moments $\y$, namely $y_{\a} = \int_\X \x^\a\,\mathrm{d} \mu = \int_{\X} \x^\a \,\mathrm{d} \nu$, for all $\a \in \N^n$.
Let us fix $f \in \mathscr{C}(\X)$.
Since $\X$ is compact, the Stone-Weierstrass theorem implies that the polynomials are dense in $\mathscr{C}(\X)$, so $\int_\X f \,\mathrm{d}\mu = \int_{\X} f\, \mathrm{d}\nu$.
Since $f$ was arbitrary, the above equality holds for any $f \in \mathscr{C}(\X)$, which implies that $\mu = \nu$.
Therefore, any finite Borel measures supported on $\X$ is \emph{moment determinate}.

The moments of the Lebesgue measure on $\A$ are denoted by
\begin{equation}\label{momb}
y^\A_{\a} \coloneqq  \int \x^{\a} \lambda_\A\,\mathrm{d}\x \in \R , \quad \a \in \N^n.
\end{equation}
The Lebesgue volume of $\A$ is $\vol \A \coloneqq  y^\A_{\mathbf{0}} = \int \lambda_\A \,\mathrm{d}\x$.
%
%
For all $r \in \N$, let us define $\N^{n}_r \coloneqq  \{ \a \in \N^{n} \mid \sum_{j=1}^{n} \alpha_j \leq r \}$, whose cardinality is $\binom{n+r}{r}$.
Then a polynomial $f \in \R[\x]$ is written as follows:
\[\x \mapsto f(\x) =\sum_{\a \in\N^n} f_{\a} \x^\a, \]
and $f$ is identified with its vector of coefficients $\f=(f_{\a})_{\a\in\N^n}$ in the standard monomial basis $(\x^\a)_{\a\in\N^n}$.

Given a real sequence $\y =(y_{\a})_{\a \in \N^n}$, let us define the linear functional $L_\y : \R[\x] \to \R$ by $L_\y(f) \coloneqq  \sum_{\a} f_{\a} y_{\a}$, for every polynomial $f=\sum_{\a} f_{\a} \x^{\a}$.
Coming back to the previous 2-dimensional example from Chapter \ref{sec:sos}, with $f = x_1 x_2$, $g_1 = x_1 - x_1^2$ and $g_2 = x_2 - x_2^2$, we have $L_\y(f) = y_{1 1}$, $L_\y(g_1) = y_{1 0} - y_{2 0}$ and $L_\y(g_2) = y_{0 1} - y_{0 2}$.

Then, we associate to $\y$ the so-called {\it moment matrix} $\M_r(\y)$ of order $r$, that is the real symmetric matrix  with rows and columns indexed by $\N_r^n$ and the following entrywise definition:
\[
[\M_r(\y)]_{\b,\g} \coloneqq  L_\y(\x^{\b + \g}) , \quad
\forall \b, \g \in \N_r^n.
\]

Given $g \in \R[\x]$, we also associate to $\y$ and $g$
the so-called {\it localizing matrix} of order $r$, that is the real symmetric matrix $\M_r(g \,\y)$ with rows and columns indexed by $\N_r^{n}$ and the following entrywise definition:
\[
[\M_r(g \, \y)]_{\b, \g} \coloneqq  L_\y(g(\x) \, \x^{\b + \g}) , \quad
\forall \b, \g \in \N_r^n.
\]
Let $\X$ be a basic compact semialgebraic set as in \eqref{eq:defX}.
Then one can check that if $\y$ has a representing measure $\mu\in\mathscr{M}_+(\X)$ then
$\M_{r}(g_j \, \y) \succeq 0$, for all $j=0,\dots,m$.

Let us give a simple example to illustrate the construction of moment and localizing matrices.
\begin{example}
\label{ex:loc_matrix}
Let us take $n = 2$ and $r = 2$.
The moment matrix $\M_2 (\y)$ is indexed by $\N_2^2 = \{(0, 0), (1, 0), (0, 1), (2, 0), (1, 1), (0, 2) \}$ and can be written as follows:
\[
\M_2 (\y) =
\begin{bmatrix}
  1 &  \mid &   y_{1, 0} & y_{0, 1} & \mid &  y_{2, 0} & y_{1, 1} & y_{0, 2} \\
    &  -     &   -       &  -       & - &  - & - & -                                            \\
  y_{1, 0} & \mid & y_{2, 0} & y_{1, 1} & \mid &  y_{3, 0} & y_{2, 1} & y_{1, 2} \\
  y_{0, 1} & \mid & y_{1, 1} & y_{0, 2} & \mid &  y_{2, 1} & y_{1, 2} & y_{0, 3} \\
  &  -     &   -       &  -       & - &  - & - & -                                            \\
  y_{2, 0} & \mid & y_{3, 0} & y_{2, 1} & \mid &  y_{4, 0} & y_{3, 1} & y_{2, 2} \\
  y_{1, 1} & \mid & y_{2, 1} & y_{1, 2} & \mid &  y_{3, 1} & y_{2, 2} & y_{1, 3} \\
  y_{0, 2} & \mid & y_{1, 2} & y_{0, 3} & \mid &  y_{2, 2} & y_{1, 3} & y_{0, 4}
 \end{bmatrix}.
\]
Next, consider the  polynomial $g_1 (\xb) = x_1 - x_1^2$ of degree 2.
From the first-order moment matrix:
\[
\M_1(\y) =
\begin{bmatrix}
  1 &  \mid &   y_{1, 0} & y_{0, 1} \\
    &  -   &   -      &  -     \\
  y_{1, 0} & \mid & y_{2, 0} & y_{1, 1} \\
  y_{0, 1} & \mid & y_{1, 1} & y_{0, 2}
 \end{bmatrix},
\]
we obtain the following localizing matrix:
\[
\M_1(g_1 \y) =
\begin{bmatrix}
  y_{1, 0} - y_{2, 0}  &  y_{2, 0} - y_{3, 0}  & y_{1, 1} - y_{2, 1} \\
  y_{2, 0} - y_{3, 0}  & y_{3, 0} - y_{4, 0} & y_{2, 1} - y_{3, 1} \\
  y_{1, 1} - y_{2, 1}  & y_{2, 1} - y_{3, 1} & y_{1, 2} - y_{2, 2}
 \end{bmatrix}.
\]
For instance, the last entry $[\M_1(g_1\y)]_{3, 3}$ is equal to $L_{\y} (g_1 (\xb) \cdot x_2 \cdot x_2) = L_\y (x_1 x_2^2 -  x_1^2 x_2^2) = y_{1, 2} - y_{2, 2}$.

\end{example}

\section{The moment-SOS hierarchy}
\label{sec:hierarchy}
Let us consider the general \ac{POP}
\begin{align}\label{eq:pop}
\P: \quad f_{\min}=\inf_\x\,\{f(\x): \x\in\X\},
\end{align}
where $f$ is a polynomial and $\X$
is a basic closed semialgebraic set as in \eqref{eq:defX}.
It happens that this problem can be cast as an \ac{LP} over probability measures, namely,
\begin{align}\label{eq:gmppop}
f_{\text{meas}}\coloneqq  \inf_{\mu \in\mathscr{M}_+(\X)} \left\{ \int_{\X} f \,\mathrm{d}\mu : \int_\X \,\mathrm{d} \mu = 1 \right\}.
\end{align}
To see that $f_{\text{meas}} = f_{\min}$ holds, let us consider a global minimizer $\x^{\opt} \in \R^n$ of $f$ on $\X$ and consider the Dirac measure $\mu = \delta_{\x^{\opt}}$ supported on this point.
Note that this Dirac (probability) measure is feasible for the \ac{LP} stated in \eqref{eq:gmppop}, with value $\int_{\X} f \,\mathrm{d}\mu = f(\x^{\opt}) = f_{\min}$, which implies  that $\inf_{\mu \in\mathscr{M}_+(\X)} \{ \int_{\X} f \,\mathrm{d}\mu : \int_\X \mathrm{d} \mu = 1 \} \leq f_{\min} $.
For the other direction, let us consider a measure $\mu$ feasible for \ac{LP} \eqref{eq:gmppop}.
Then, simply observe that since $f(\x) \geq f_{\min}$, for all $\x \in \X$, the feasibility of $\mu$ implies that $\int_{\X} f \,\mathrm{d}\mu \geq \int_{\X} f_{\min} \,\mathrm{d}\mu = f_{\min} \int_{\X}  \,\mathrm{d}\mu = f_{\min}$.
Since it is true for any feasible solution, one has $\inf_{\mu \in\mathscr{M}_+(\X)} \{ \int_{\X} f \,\mathrm{d}\mu : \int_\X \,\mathrm{d} \mu = 1 \} \geq f_{\min} $.
Another way to state this equality is to write
\begin{align}
\label{eq:pospop}
f_{\min}=\sup_{b}\,\{b : f - b \geq 0 \text{ on } \X\},
\end{align}
which is an \ac{LP} over nonnegative polynomials, and to notice that the dual \ac{LP} of \eqref{eq:pospop} is \ac{LP}  \eqref{eq:gmppop}.
The equality then follows from the zero duality gap in infinite-dimensional \ac{LP}.

After reformulating $\P$ as \ac{LP} \eqref{eq:gmppop} over probability measures, one can then build a hierarchy of moment relaxations for the later problem.
This is done by using the fact that the condition $\mu\in\mathscr{M}_+(\X)$ can be relaxed as
$\M_{r-d_j}(g_j \, \y) \succeq 0$, for all $j=0,\dots,m$, and all $r \geq d_j = \lceil \deg (g_j) / 2 \rceil$.

Letting $r_{\min} \coloneqq  \max\,\{\lceil \deg(f) / 2 \rceil, d_1,\dots, d_m \}$, at step $r \geq r_{\min}$ of the hierarchy, one considers the following primal \ac{SDP} program:
\begin{align}\label{primalj}
\P^r: \quad
\begin{array}{rll}
f^r \coloneqq  &\displaystyle\inf_{\y}\displaystyle &L_\y(f)\\
&\rm{s.t.}&\M_{r}(\y)\succeq0\\
&&\M_{r-d_j}(g_j\,\y)\succeq0,\quad j\in[m]\\ 
&&y_{\mathbf{0}}=1
\end{array}
\end{align}
Before considering the corresponding dual \ac{SDP} program, let us remind that the moment and localizing matrices $\M_{r-d_j}(g_j\,\y)$ have entries which are linear in $\y$.
Namely, one has $\M_{r-d_j}(g_j\,\y) = \sum_{\a \in \N^n_{2r}} \Cb^j_\a \, y_\a$; the matrix $\Cb^j_\a$ has rows and columns indexed by $\N^n_{r- d_j}$ with $(\b,\g)$-entry equal to $\sum_{\b+\g+\d = \a} \, g_{j,\d}$.
In particular for $m=0$, one has $g_0 = 1$ and the matrix $\B_\a \coloneqq   \Cb^0_\a$ has $(\b,\g)$-entry equal to $1_{\b + \g = \a}$, where $1_{\a = \b}$ stands for the function which returns 1 if $\a = \b$ and 0 otherwise.
With $t_j\coloneqq \binom{n+r-d_j}{r-d_j}$, the dual of \ac{SDP} \eqref{primalj} is then the following \ac{SDP}:
\begin{equation}\label{dualj}
\begin{cases}
\sup\limits_{\G_j,b}&b\\
\rm{s.t.}&f_\a-b1_{\a=\mathbf{0}}=\displaystyle\sum_{j=0}^m\langle \Cb^j_\a,\G_j\rangle, \quad\a\in\N^n_{2r}\\
&\G_j \in \Sbb_{t_j}^+, \quad j=0,\ldots,m
\end{cases}
\end{equation}
We can rewrite the equality constraints from \ac{SDP} \eqref{dualj} in a more concise way, namely as $f - b \in \cM(\frakg)_r$.
To see this, let us first note that a sum of squares $\sigma$ of degree $2 r$ can be written as $\vb^\intercal \G  \vb$, with
\[
\vb\coloneqq  (1, x_1,\dots,x_n,x_1^2,x_1 x_2,\dots, x_1^r,\dots,x_n^r),
\]
being the vectors of all monomials of degree at most $r$, and $\G \succeq 0$.
The $\a$-coefficient of $\sigma = \vb^\intercal \G \vb$ is equal to $\langle \B_\a, \G \rangle$.
Similarly, for any $j \in [m]$ and \ac{SOS} $\sigma_j$ of degree at most $2 (r-d_j)$, one can write $\sigma_j = \vb_j^\intercal \G_j \vb_j$, with $\vb_j$ being the vector of all monomials of degree at most $r - d_j$, and $\G_j \succeq 0$.
One can also check that the $\a$-coefficient of $\sigma_j g_j$ is equal to  $\langle \Cb^j_\a, \G_j \rangle $.
Therefore, \ac{SDP} \eqref{dualj} is equivalent to the following optimization problem over \ac{SOS} polynomials:
\begin{equation}\label{dualjsos}
\begin{cases}
\sup\limits_{\sigma_j, b}&b\\
\rm{s.t.}&f-b=\displaystyle\sum_{j=0}^m \sigma_j g_j\\
&\sigma_j \in \Sigma[\x]_{r-d_j}, \quad j=0,\ldots,m
\end{cases}
\end{equation}
or more concisely as
\begin{equation}\label{dualjqm}
\sup_{b}\,\{b: \:f -\lowerb \in \cM(\frakg)_r \}.
\end{equation}
The dual \ac{SDP} \eqref{dualjqm} is obtained by replacing the nonnegativity condition $f - \lowerb \geq 0$ on $\X$ of the dual \ac{LP} \eqref{eq:pospop} by the more restrictive condition $f -\lowerb \in \cM(\frakg)_r$.
The sequences of \ac{SDP} programs \eqref{primalj} and \eqref{dualjqm} are called the \emph{moment} hierarchy and the \ac{SOS} hierarchy, respectively.
In the sequel, we refer to the sequence of primal-dual programs \eqref{primalj}--\eqref{dualjqm} as the \ac{moment-SOS} hierarchy.
\begin{theoremf}\label{th:cvgdense}
Under Assumption \ref{hyp:archimedean}, the hierarchy of primal-dual \ac{moment-SOS} relaxations \eqref{primalj}--\eqref{dualjqm} provides nondecreasing sequences of lower bounds converging to the global optimum $f_{\min}$ of $\P$ \eqref{eq:pop}.
\end{theoremf}
~\\
The above theorem provides the theoretical convergence guarantee of the \ac{moment-SOS} hierarchy.
\begin{remark}
\label{rk:equality}
Even though we only included inequality constraints in the definition of $\X$ for the sake of simplicity, equality constraints can be treated in a dedicated way.
For each equality constraint $h(\x) = 0$, with $h \in \R[\x]$, one adds the localizing constraint $\M_{r-d_h}(h \, \y) = 0$, with $d_h \coloneqq  \lceil \deg (h) / 2 \rceil $, in the primal moment program  \eqref{primalj}. 
Similarly, in the dual \ac{SOS} program \eqref{dualjsos}, one adds a term $\tau h$ to the sum $\sum_{j=0}^m \sigma_j g_j$, with $\tau \in \R[\x]_{2 r - 2 d_h}$.
\end{remark}
In practice, it is possible to obtain finite convergence of the hierarchy, which is the topic of the next section.
\section{Minimizer extraction}
\label{sec:extract}
Here we describe sufficient conditions to obtain finite convergence of the \ac{moment-SOS} hierarchy and extract the global minimizers of the polynomial $f$ on $\X$.
\begin{theoremf}
\label{th:extract}
Consider the sequence of primal moment relaxations defined in \eqref{primalj}.
If for some $r \geq r_{\min}$, \ac{SDP} \eqref{primalj} has an optimal solution $\y$ which satisfies
\begin{align}
\label{eq:rank}
\rank \M_{r'}(\y) = \rank \M_{r'-r_{\min}}(\y) \text{ for some }  r' \leq r,
\end{align}
then $f^r = f_{\min}$ and the infinite-dimensional \ac{LP} \eqref{eq:gmppop} has an optimal solution $\mu \in \mathscr{M}(\X)_+$, which is finitely supported on $t = \rank \M_{r'}(\y)$ points of $\X$, or equivalently $t$ global minimizers of $f$ on $\X$.
\end{theoremf}
If the rank stabilization (also called \emph{flatness}) condition \eqref{eq:rank} is satisfied, then finite convergence occurs, namely the \ac{SDP} relaxation \eqref{primalj} is exact with optimal value $f^r = f_{\min}$.
In addition, one can extract $\rank \M_{r'}(\y)$ global minimizers of $f$ on $\X$ with the following algorithm.

\begin{algorithm}\caption{${\tt Extract}$}\label{alg:extract}
\begin{algorithmic}[1]
\Require The moment matrix $\M_{r'}(\y)$ of rank $t$ satisfying the flatness condition \eqref{eq:rank}
\Ensure The $t$ points $\x(i) \in \X$, $i \in [t]$, global minimizers of Problem $\P$ \eqref{eq:pop}
\State Compute the Cholesky factorization $\Cb \Cb^\intercal = \M_{r'}(\y)$ \label{line:extractchol}
\State Reduce $\Cb$ to a column echelon form $\U$ \label{line:extractechelon}
\State Compute from $\U$ the multiplication matrices $\Nb_i$, $i\in [n]$ \label{line:extractmult}
\State Compute $\Nb \coloneqq  \sum_{i=1}^n \lambda_i \Nb_i$ with randomly generated coefficients $\lambda_i$ \label{line:extractrand}
\State Compute the Schur decomposition $\Nb = \Qb \Tb \Qb^\intercal$ \label{line:extractschur}
\State Compute the column vectors $\{\q_j\}_{1 \le j\le t}$ of $\Qb$ \label{line:extractcol}
\State Return $x_i(j) \coloneqq  \q_j^\intercal \Nb_i \q_j$, $i\in[n]$, $j\in[t]$ \label{line:extractout}
\end{algorithmic}
\end{algorithm}

\begin{proposition}\label{prop:extract}
The procedure $\mathtt{Extract}$ described in Algorithm~\ref{alg:extract} is sound and returns $t$ global optimizers of Problem $\P$ \eqref{eq:pop}.
\end{proposition}
\begin{proof}
Since the flatness condition \eqref{eq:rank} is satisfied, $\y$ is the moment sequence of a $t$-atomic Borel measure $\mu$ supported on $\X$. Namely, there are $t$ points $\x(1),\dots,\x(t) \in \X$ such that
\[
\mu = \sum_{j=1}^t \kappa_j \delta_{\x(j)}, \quad \kappa_j > 0 , \quad \sum_{j=1}^t \kappa_j = 1.
\]
By construction of the moment matrix $\M_{r'}(\y)$, one has
\[
\M_{r'}(\y) = \sum_{j=1}^r \kappa_j \vb_{r'} (\x(j)) \vb_{r'}^\intercal (\x(j)) = \V \D \V^\intercal,
\]
where the $j$-th column of $\V$ is $\vb_{r'} (\x(j))$ and $\D$ is a $t\times t$ diagonal matrix with diagonal $(\kappa_j)_{1 \le j \le t}$.
One can extract a Cholesky factor $\Cb$ as in Step \ref{line:extractchol}, for instance via singular value decomposition.
The following steps of the extraction algorithm consist of transforming $\Cb$ into $\V$ by suitable column operations.
The reduction of $\Cb$ to a column echelon form in Step \ref{line:extractechelon} is  done by Gaussian elimination with column pivoting.
By construction of the moment matrix, each row of $\U$ is indexed by a monomial $\x^{\a}$ involved in the vector $\vb_{r'}$.
Pivot elements in $\U$ correspond to monomials $\x^{\b(j)}$, $j\in[t]$ of the basis generating the $t$ solutions.
Namely, if $\w = (\x^{\b(1)}, \x^{\b(1)}, \dots, \x^{\b(t)})$ denotes this generating basis, then
\[
\vb_{r'}(\x(j)) = \U \w(\x(j)), \quad j \in [t].
\]
Overall, extracting the global minmizers boils down to solving the above systems of equations.
To solve this system, we compute at Step \ref{line:extractmult} each multiplication matrix
$\Nb_i$, $i\in[n]$, which contains the coefficients of the monomials $x_i \x^{\b(j)}$, $j\in[t]$, namely which satisfy
\[
\Nb_i \w (\x) = x_i \w (\x).
\]
The entries of the global minimizers are all eigenvalues of the multiplication matrices.
Since $\w(\x(j))$ is an eigenvector common to all multuplication matrices, one builds the random combination $\Nb$ of Step \ref{line:extractrand}, which ensures with probability 1 that its eigenvalues are all distinct with $1$-dimensional eigenspaces.
The Shur decomposition of Step \ref{line:extractschur} gives the decomposition $\Nb = \Qb \Tb \Qb^\intercal$ with an orthogonal matrix $\Qb$ and an upper triangular matrix $\Tb$ with eigenvalues of $\Nb$ sorted in increasing order along the diagonal. \qed
\end{proof}

\begin{example}
\label{ex:extract}
Consider the polynomial optimization problem $\P$ \eqref{eq:pop} with $f = -(x_1-1)^2 - (x_1-x_2)^2 - (x_2-3)^2$ and $\X = \{\x \in \R^2 : 1 - (x_1-1)^2 \geq 0, 1 - (x_1-x_2)^2 \geq 0, 1 - (x_1-3)^2 \geq 0 \}$.
The first \ac{SDP} relaxation outputs $f^1 = -3$ and $\rank M_1(\y) = 3$, while the second one outputs $f^2 = -2$ and the rank stabilizes with $\rank M_1(\y) = \rank M_2(\y) = 3$.
Therefore the flatness condition holds, which implies that $f_{\min} = f^2 = -2$.
The monomial basis is $\vb_2(\x) = (1, x_1, x_2, x_1^2, x_1 x_2, x_2^2)$.
The column echelon form $\U$ of the Cholesky factor of $\M_2(\y)$ is given by
\[\begin{bmatrix}
1 & \\
0 & 1 \\
0 & 0 & 1\\
-2 & 3 & 0\\
-4 & 2 & 2\\
-6 & 0 & 5
\end{bmatrix}.\]
Pivot entries correspond to the generating basis $\w(\x) = (1, x_1, x_2)$.
Therefore the entries of the 3 global minimizers satisfy the following system of polynomial equations:
\begin{align*}
x_1^2 &= -2 + 3 x_1\\
x_1 x_2 &= -4 + 2 x_1 + 2 x_2\\
x_2^2 &= -6 + 5 x_2.
\end{align*}
The multiplication matrices by $x_1$ and $x_2$ can be extracted from rows in $\U$ as follows:
\[
\Nb_1 = \begin{bmatrix}
0 & 1 & 0\\
-2 & 3 & 0\\
-4 & 2 & 2
\end{bmatrix}, \quad
\Nb_2 = \begin{bmatrix}
0 & 0 & 1\\
-4 & 2 & 2\\
-6 & 0 & 5
\end{bmatrix}.
\]
After selecting a random convex combination of $\Nb_1$ and $\Nb_2$ and computing the orthogonal matrix in the corresponding Schur decomposition, we obtain the 3 minimizers $\x(1) = (1, 2)$, $\x(2) = (2, 2)$ and $\x(3) = (2, 3)$.
\end{example}
\section{Notes and sources}

The representation of positive polynomials stated in Theorem \ref{th:putinar} is the well renowned Putinar's representation and is proved in \citepop{Putinar1993positive}.
Based on this representation, the convergence of the \ac{moment-SOS} hierarchy, stated in Theorem \ref{th:cvgdense}, has been proved in \citepop{Las01sos}.

The Riesz identification theorem can be found for instance in~\citepop{lieb2001analysis}.
We refer the interested reader to~\citepop[Section 21.7]{Royden} and \citepop[Chapter IV]{alexander2002course} or~\citepop[Section 5.10]{Luenberger97} for functional analysis, measure theory and applications in convex optimization.

The finite convergence of the \ac{moment-SOS} hierarchy has been proved to hold generically under Assumption~\ref{hyp:archimedean} in \citepop{nie2014optimality}.
The statement of Algorithm~\ref{alg:extract} extracting global minimizers and its correctness proof are available in \citepop{Henrion05} (combined with ideas from \citepop{lasserre2008semidefinite}).
The robustness of this algorithm has been studied in \citepop{klep2018minimizer}. It was proved that Algorithm~\ref{alg:extract} works under some generic conditions and Assumption~\ref{hyp:archimedean} in \citepop{nie2013certifying}.
An interpretation of some wrong results, due to numerical inaccuracies, already observed when solving \ac{SDP} relaxations for \ac{POP} on a double precision floating point \ac{SDP} solver is provided in \citepop{lasserre2019sdp}.


\input{pop.bbl}

%% file: cs.tex
\part{Correlative sparsity}
\chapter{The moment-SOS hierarchy based on correlative sparsity}\label{chap:cs}
In this chapter, we describe how to exploit sparsity arising in the data of \ac{POP}s from the perspective of variables, which leads to the notion of \emph{\ac{CS}}.
We start to explain how to build the \ac{csp} graph in Chapter \ref{sec:cs}.
Then, we provide in Chapter \ref{sec:csmom} an infinite-dimensional \ac{LP} formulation over probability measures for \ac{POP}s, based on \ac{CS}.
This \ac{LP} program is then handled with a \ac{CS}-adapted \ac{moment-SOS} hierarchy of \ac{SDP} relaxations, stated in Chapter \ref{sec:cshierarchy}.
An alternative approach based on bounded degree \ac{SOS} is given in Chapter \ref{sec:sbsos}.
Chapter \ref{sec:csextract} focuses on minimizer extraction.
Eventually, we explain in Chapter \ref{sec:csrat} how to extend this \ac{CS} exploitation scheme to optimization over rational functions.

\section{Correlative sparsity}\label{sec:cs}
Recall that a general \ac{POP} is formulized as
\begin{align}\label{eq:popeff}
\P:\quad f_{\min}=\inf_{\x}\,\{f(\x):\x\in\X\},
\end{align}
where $\X=\{\x\in\R^n:g_1(\x)\geq0,\dots,g_m(\x)\geq0\}$.
Roughly speaking, the exploitation of \ac{CS} in the \ac{moment-SOS} hierarchy for $\P$ consists of two steps:
\begin{enumerate}[(1)]
	\item decompose the variables $\x$ into a set of cliques according to the correlations between variables emerging in the input polynomial system;
	\item construct a sparse \ac{moment-SOS} hierarchy with respect to the former decomposition of variables.
\end{enumerate}

Let us proceed with more details. Recall $d_j\coloneqq\lceil\deg(g_j)/2\rceil$ for $j\in[m]$ and $r_{\min}\coloneqq\max\,\{\lceil\deg(f)/2\rceil,d_1,\dots,d_m\}$. Fix from now on a relaxation order $r\ge r_{\min}$. Let $J'\coloneqq\{j\in[m]\mid d_j=r\}$ which is possibly nonempty only when $r=r_{\min}$.
We define the \emph{\ac{csp}} graph $G^{\rm{csp}}(V, E)$ associated to \ac{POP} \eqref{eq:popeff} whose node set is $V=\{1,2,\ldots,n\}$ and whose edge set $E$ satisfies $\{i,j\}\in E$ if one of following conditions holds:
\begin{enumerate}[(i)]
    \item there exists $\a\in\supp(f)\cup\bigcup_{j\in J'}\supp(g_j)$ such that $\{i,j\}\subseteq\supp(\a)$;
    \item there exists $k\in[m]\setminus J'$ such that $\{i,j\}\subseteq\bigcup_{\a\in\supp(g_k)}\supp(\a)$,
\end{enumerate}
where $\supp(\a)\coloneqq \{k\in[n]\mid\alpha_k\ne0\}$ for $\a=(\alpha_1,\dots,\alpha_n)\in\N^n$.
Let $(G^{\rm{csp}})'$ be a chordal extension of $G^{\rm{csp}}$\footnote{If $G^{\rm{csp}}$ is already a chordal graph, then we do not need the chordal extension.} and $\{I_k\}_{k=1}^p$ be the list of maximal cliques of $(G^{\rm{csp}})'$ with $n_k\coloneqq |I_k|$ so that the \ac{RIP} \eqref{eq:RIP} holds.
Let $\R[\x,I_k]$ denote the ring of polynomials in the $n_k$ variables $\{x_i\}_{i\in I_k}$ for $k\in[p]$.
By construction, one can decompose the objective function $f$ as $f = f_1 + \dots + f_p$ with $f_k \in \R[\x,I_k]$ for all $k\in[p]$ (similarly for $g_j$ with $j\in J'$).
We then partition the constraint polynomials $g_j, j\in[m]\setminus J'$ into groups $\{g_j\mid j\in J_k\}, k\in[p]$ which satisfy
\begin{enumerate}[(i)]
    \item $J_1,\ldots,J_p\subseteq[m]\setminus J'$ are pairwise disjoint and $\bigcup_{k=1}^pJ_k=[m]\setminus J'$;
    \item for any $k\in[p]$ and any $j\in J_k$, $\bigcup_{\a\in\supp(g_j)}\supp(\a)\subseteq I_k$,
\end{enumerate}
so that $g_j \in \R[\x,I_k]$ for all $k\in[p]$ and $j \in J_k$.
%
In addition, suppose that Assumption~\ref{hyp:archimedean} holds. Then all variables involved in \ac{POP} \eqref{eq:popeff} are bounded.
To guarantee global convergence of the hierarchy that will be presented later, we need to add some redundant quadratic constraints to the description of the POP. We summarize all above in the following assumption.
\begin{assumption}\label{hyp:cs}
Consider \ac{POP} \eqref{eq:popeff}. The two index sets $[n]$ and $[m]$ are decomposed/partitioned into $\{I_1,\dots,I_p\}$ and $\{J', J_1,\dots,J_p\}$, respectively, such that
\begin{enumerate}[(i)]
\item The objective function $f$ can be decomposed as $f = f_1 + \dots + f_p$ with $f_k \in \R[\x,I_k]$ for $k\in[p]$ and the same goes for the constraint polynomial $g_j$ with $j\in J'$;
\item For all $k\in[p]$ and $j \in J_k$, $g_j \in \R[\x,I_k]$;
\item The \ac{RIP} \eqref{eq:RIP} holds for $I_1,\dots,I_p$ (possibly after some reordering);
\item For all $k\in[p]$, there exists $N_k>0$ such that one of the constraint polynomials is $N_k - \sum_{i \in I_k} x_i^2$.
\end{enumerate}
\end{assumption}

\begin{example}\label{ex:cs}
Consider an instance of \ac{POP} \eqref{eq:popeff} with
$f(\x) = x_2 x_5 + x_3 x_6 - x_2 x_3  - x_5 x_6
+ x_1 ( - x_1 +  x_2 +  x_3  - x_4 +  x_5 +  x_6)$ and
\[\X= \{ \x \in \R^n \, : \, g (\x) \geq 0, \text{ for all } g \in \frakg\},\]
with $\frakg = \{ (6.36 - x_1) (x_1 - 4), \dots, (6.36 - x_6) (x_6 - 4) \}$.
%
Here, there are $n = 6$ variables and the number of constraints is $m = 6$.
%
The related \ac{csp} graph $G^{\rm{csp}}$ is depicted in Figure~\ref{fig:csp_deltax}.
%
\begin{figure}[!t]
\centering
\includegraphics[width=4cm]{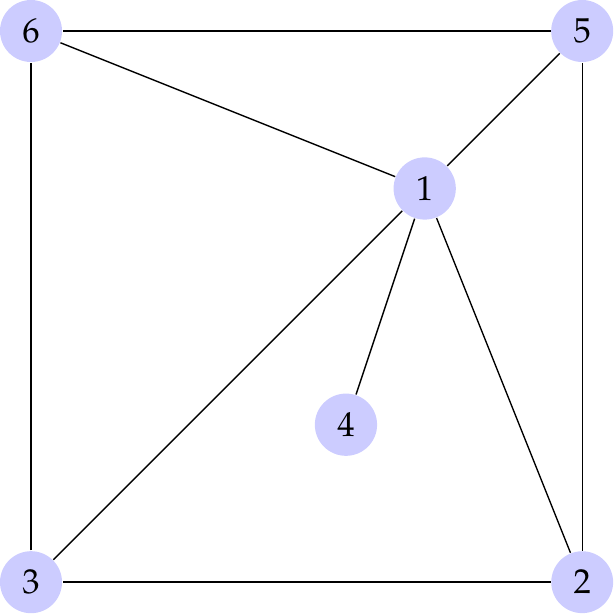}
\caption{The \ac{csp} graph for $f$ over $\X$ from Example~\ref{ex:sparsepop}.}
\label{fig:csp_deltax}
\end{figure}
After adding an edge between nodes 3 and 5, the resulting graph $(G^{\rm{csp}})'$ is chordal with maximal cliques $I_1 = \{1, 4\}$, $I_2 = \{1, 2, 3, 5\}$, $I_3 = \{1, 3, 5, 6\}$.
Here $p=3$ and one can write $f = f_1 + f_2 + f_3$ with
\begin{align*}
f_1 &= - x_1 x_4,\\
f_2 &= -x_1^2 + x_1 x_2 + x_1 x_3 - x_2 x_3 + x_2 x_5,\\
f_3 &= - x_5 x_6 + x_1 x_5 + x_1 x_6 + x_3 x_6.
\end{align*}
For the relaxation order $r=r_{\min}=1$, let $J'=[6]$ and $J_k=\emptyset$ for $k\in[3]$; for the relaxation order $r\ge2$, let $J'=\emptyset$ and $J_1=\{1,4\}$, $J_2=\{2,3,5\}$, $J_3=\{6\}$. Then Assumption \ref{hyp:cs}(i)-(ii) hold.
In addition, $I_1 \cap I_2 = \{1\} \subseteq I_3$, and so \ac{RIP} \eqref{eq:RIP}, or equivalently Assumption \ref{hyp:cs} (iii), holds.
For each $i \in [n]$, one has $6.36^2 - x_i^2 \geq 0$ for all $\x \in \X$, and so one can select $N_1 = 2\cdot 6.36^2$, $N_2 = N_3 = 4\cdot 6.36^2$ and add the redundant constraints $N_k - \sum_{i \in I_k} x_i^2 \geq 0$, $k \in [p]$ in the description of $\X$, so that Assumption \ref{hyp:cs} (iv) holds as well.
\end{example}

\section{A sparse infinite-dimensional LP formulation}
\label{sec:csmom}
In this section, we assume $J'=\emptyset$. We now introduce a \ac{CS} variant of the dense infinite-dimensional \ac{LP} \eqref{eq:gmppop} formulation over probability measures  stated in Chapter \ref{sec:hierarchy}.
The idea is to define a new measure for each subset $I_k$, $k \in [p]$, supported on a set $\X_k$ described by the constraints which only depend on the variables indexed by $I_k$, namely,
\[
\X_k \coloneqq  \{\x \in \R^{n_k} : g_j(\x) \geq 0, j \in J_k \},\text{ for } k\in[p].
\]
So $\X$ can be equivalently described as
\begin{align}
\label{eq:csX}
\X = \{\x \in \R^n : (x_i)_{i \in I_k} \in \X_k, k \in [p] \}.
\end{align}
Similarly, for all $j,k \in [p]$ such that $I_j \cap I_k \neq \emptyset$, define
\[
\X_{jk} = \X_{kj} \coloneqq  \{(x_i)_{i \in I_j \cap I_k} : (x_i)_{i \in I_j} \in \X_j , (x_i)_{i \in I_k} \in \X_k \}.
\]
Afterwards, for each $k\in[p]$ we define the projection $\pi_k: \mathscr{M}_+(\X) \to \mathscr{M}_+(\X_k)$ of the space of Borel measures supported on $\X$ on the space of Borel measures supported on $\X_k$, namely, for all $\mu \in \mathscr{M}_+(\X)$,
\[\pi_k \mu(\B) \coloneqq\mu(\{\x : \x \in \X, (x_i)_{i\in I_k} \in\B\}),\]
for each Borel set $\B \in \mathcal{B}(\X_k)$.
We define similarly the projections $\pi_{jk}$ for all $j,k \in [p]$ such that $I_j \cap I_k \neq \emptyset$.
For each $k\in[p-1]$, we also rely on the set
\[
U_k \coloneqq  \{j \in \{k+1,\dots,p \} : I_j \cap I_k \neq \emptyset\}.
\]
Then the \ac{CS} variant of \eqref{eq:gmppop} reads as follows:
\begin{align}\label{eq:csgmppop}
\begin{array}{rll}
f_{\cs}\coloneqq&\displaystyle\inf_{\mu_k}&  \,\displaystyle\sum_{k=1}^p \int_{\X_k} f_k(\x)\,\mathrm{d} \mu_k(\x) \\
&\rm{s.t.}& \pi_{jk} \mu_j = \pi_{kj} \mu_k, \quad j\in U_k, k\in [p-1]\\
&&\displaystyle\int_{\X_k}\,\mathrm{d}\mu_k(\x) = 1,\quad k \in[p]\\
&& \mu_k\in\mathscr{M}_+(\X_k), \quad k \in[p]
\end{array}
\end{align}
To prove $f_{\cs} = f_{\min}$ under Assumption \ref{hyp:cs}, we rely on the following auxiliary lemma,  illustrated in Figure \ref{diag:marginal} in the case $p = 2$.
This lemma uses the fact that one can disintegrate a probability measure on a product of Borel spaces into a marginal and a so-called \textit{stochastic} kernel.
Given two Borel
spaces $\X$, $\Zb$, a stochastic kernel $q(\mathrm{d} \x | \z)$ on $\X$ given $\Zb$ is defined by (1) $q(\mathrm{d} \x | \z) \in \mathscr{M}_+(\X)$ for all $\z \in \Zb$ and (2) the function $\z \mapsto q(\B | \z)$ is $\mathcal{B}(\Zb)$-measurable for all $\B \in \mathcal{B}(\Zb)$.
\begin{lemma}
\label{lemma:marginal}
Let $[n] = \cup_{k=1}^p I_k$ with $n_k = |I_k|$, $\X_k \subseteq \R^{n_k}$ be given compact sets, and let $\X \subseteq \R^n$ be defined as in \eqref{eq:csX}.
Let $\mu_1 \in \mathscr{M}_+(\X_1),\dots, \mu_p \in \mathscr{M}_+(\X_p)$ be measures satisfying the equality constraints of \ac{LP} \eqref{eq:csgmppop}.
If \ac{RIP} \eqref{eq:RIP} holds for $I_1,\ldots,I_p$, then there exists a probability measure $\mu \in \mathscr{M}_+(\X)$ such that
\begin{align}
\label{eq:marginal}
\pi_k \mu = \mu_k,
\end{align}
for all $k\in[p]$, that is, $\mu_k$ is the marginal of $\mu$ on $\R^{n_k}$, i.e., with respect to variables indexed by $I_k$.
\end{lemma}
\begin{figure}[!t]
\centering
\begin{tikzcd}[row sep=3ex,column sep=1ex]
&  \mathscr{M}_+(\X) \arrow[ddl,"\pi_1"']  \arrow[ddr,"\pi_2"]  \\
{} \\
 \mathscr{M}_+(\X_1) \arrow[dddr, "\pi_{21}"'] & & \mathscr{M}_+(\X_2) \arrow[dddl, "\pi_{12}"]   \\
{} \\
&&  \\
&  \mathscr{M}_+(\X_{12})
\end{tikzcd}
\caption{Illustration of Lemma \ref{lemma:marginal} in the case $p=2$.}
\label{diag:marginal}
\end{figure}
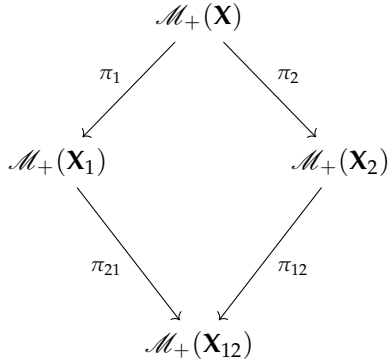
\begin{proof}
The proof boils down to constructing $\mu$ by induction on $p$.
If $p=1$ and $I_1 = [n]$, the configuration corresponds exactly to the dense \ac{LP}  \eqref{eq:gmppop} formulation from Chapter \ref{sec:hierarchy}, and one can simply take $\mu = \mu_1$.
For the sake of conciseness, we only provide a proof for the case $p=2$.
Let $I_{12} \coloneqq  I_1 \cap I_2$ with cardinality $n_{12}$.
If $I_{12} = \emptyset$, then one has $\X = \X_1 \times \X_2$ and we can simply define $\mu$ as the product measure of $\mu_1$ and $\mu_2$:
\[
\mu(\A \times \B)\coloneqq  \mu_1(\A) \times \mu_2(\B),
\]
for all $\A \in \mathcal{B}(\R^{n_1}), \B \in \mathcal{B}(\R^{n_2})$.
This product measure $\mu$ satisfies \eqref{eq:marginal}.

Next, let us focus on the hardest case where $I_{12} \neq \emptyset$.
Let $\overline \pi_k$ be the natural projection with respect to $I_k \backslash I_{12}$ and let us define the Borel set
$\Y_k \coloneqq  \{\overline \pi_k(\x) : \x \in \X_k \} \in \mathcal{B}(\R^{n_k - n_{12}})$.
It follows that $\mu_1, \mu_2$ can be seen as probability measures on the cartesian products $\Y_1 \times \X_{12} = \X_1$ and $\X_{12} \times \Y_2 = \X_2$, respectively.
Let $\nu_1$ and $\nu_2$ be the stochastic kernels of $\mu_1$ and $\mu_2$, respectively.
Since $\pi_{12} \mu_1 = \pi_{21} \mu_2 =: \nu$, one can disintegrate $\mu_1$ and $\mu_2$ as
\begin{align*}
\mu_1 (\A \times \B)& =  \int_\B \nu_1(\A | \x) \nu (\mathrm{d} \x) , \quad \forall \A \in \mathcal{B}(\Y_1), \B \in \mathcal{B}(\R^{n_{12}}),
\\
\mu_2 (\Cb \times \B)& = \int_\B \nu_2(\Cb | \x) \nu (\mathrm{d} \x) , \quad \forall \A \in \mathcal{B}(\Y_2), \B \in \mathcal{B}(\R^{n_{12}}).
\end{align*}
Then, one can define the measure $\mu \in \mathscr{M}_+(\Y_1 \times \R^{n_{12}} \times \Y_2)$ as follows:
\[
\mu(\A \times \B \times \Cb) = \int_\B \nu_1(\A | \x) \nu_2(\Cb | \x) \nu(\mathrm{d} \x),
\]
for every Borel rectangle $\A \times \B \times \Cb \in \mathcal{B}(\Y_1) \times \mathcal{B}(\X_{12}) \times \mathcal{B}(\Y_2)$.
In particular if $\A = \Y_1$, one has $\nu_1(\A | \x) = 1$ $\nu$-a.e., and $\mu(\Y_1 \times \B \times \Cb) = \int_\B  \nu_2(\Cb | \x) \nu(\mathrm{d} \x)  = \mu_2 (\B \times \Cb)$, implying that $\mu_2$ is the marginal of $\mu$ on $\X_{12} \times \Y_2 = \X_2$.
Similarly, $\mu_1$ is the marginal of $\mu$ on $\Y_1 \times \X_{12} = \X_1$, yielding the desired result.
\qed
\end{proof}
Now, we are ready to prove that \ac{LP} \eqref{eq:csgmppop} is not just a relaxation of the dense \ac{LP} \eqref{eq:gmppop}.
\begin{theoremf}\label{th:cspop}
Consider \ac{POP} \eqref{eq:popeff}.
If Assumption \ref{hyp:cs} holds, then $f_{\cs} = f_{\min}$.
\end{theoremf}
\begin{proof}
The first inequality $f_{\cs} \leq f_{\min}$ is easy to show: let $\ba$ be a global minimizer of $f$ on $\X$, assumed to exist thanks to the compactness hypothesis.
Let $\mu = \delta_{\ba}$ be the Dirac measure concentrated on $\ba$, and $\mu_k \coloneqq  \pi_k \mu$ be its projection on $\mathscr{M}_+(\X_k)$, for each $k\in[p]$.
Namely, $\mu_k$ is the Dirac measure concentrated on $(a_i)_{i \in I_k} \in \X_k$, and is in particular a probability measure supported on $\X_k$.
For each pair $j, k$ such that $I_j \cap I_k \neq \emptyset$, the measure $\pi_{jk} \mu_j$ is the Dirac measure concentrated on $(a_i)_{i \in I_j \cap I_k} \in \X_{jk}$, and so is $\pi_{kj} \mu_k$.
Therefore, each measure $\mu_k$ is a feasible solution of \eqref{eq:csgmppop}.
In addition, the objective value of \ac{LP} \eqref{eq:csgmppop} is equal to $\sum_{k=1}^p f_k(\ba) = f_{\min}$, which proves the first inequality.

To prove the other inequality $f_{\cs} \geq f_{\min}$, let us fix a feasible solution $(\mu_k)$ of \ac{LP} \eqref{eq:csgmppop}.
By Lemma \ref{lemma:marginal}, there exists a probability measure $\mu \in \mathscr{M}_+(\X)$ such that $\pi_k \mu = \mu_k$, for each $k \in [p]$.
Then, one has
\[
\sum_{k=1}^p \int_{\X_k} f_k\,\mathrm{d} \mu_k = \sum_{k=1}^p \int_{\X_k} f_k\,\mathrm{d} \mu = \int_{\X} \sum_{k=1}^p  f_k \,\mathrm{d} \mu = \int_{\X} f \mathrm{d} \mu \geq f_{\min}.
\]
\qed
\end{proof}
\section{The CS-adpated moment-SOS hierarchy}\label{sec:cshierarchy}
In this section, we continue assuming $J'=\emptyset$. For $k\in[p]$, a moment sequence $\y\subseteq\R$ and $g\in\R[\x,I_k]$, let $\M_r(\y, I_k)$ (resp. $\M_r(g \, \y, I_k)$)
be the moment (resp. localizing) submatrix obtained from $\M_r(\y)$ (resp. $\M_r(g \, \y)$) by retaining only those rows and columns indexed by $\b\in\N_r^n$ of $\M_r(\y)$ (resp. $\M_r(g\y)$) with
$\supp(\beta)\subseteq I_k$.
\begin{example}\label{ex:csmoment}
Consider again Example \ref{ex:cs}. The moment matrix $\M_1(\y, I_1)$ is indexed by the support vectors  $(0, 0, 0, 0, 0, 0), (1, 0, 0, 0, 0, 0), (0, 0, 0, 1, 0, 0)$ (corresponding to the monomials $1$, $x_1$ and $x_4$, respectively) and reads as follows:
\[
\M_1(\y) =
\begin{bmatrix}
  1 &  \mid &   y_{1, 0, 0, 0, 0, 0} & y_{0, 0, 0, 1, 0, 0} \\
    &  -   &   -      &  -     \\
  y_{1, 0, 0, 0, 0, 0} & \mid & y_{2, 0, 0, 0, 0, 0} & y_{1, 0, 0, 1, 0, 0} \\
  y_{0, 0, 0, 1, 0, 0} & \mid & y_{1, 0, 0, 1, 0, 0} & y_{0, 0, 0, 2, 0, 0}
 \end{bmatrix}.
\]
\end{example}

With $r \ge r_{\min}$, the moment hierarchy based on \ac{CS} for \ac{POP} \eqref{eq:csgmppop} is defined as
\begin{equation}\label{eq:csmultimom}
\begin{cases}
\inf\limits_{\y_k} &\sum_{k=1}^p L_{\y_k}(f_k)\\
\rm{ s.t.}&\M_{r}(\y_k, I_k) \succeq 0,\quad k\in[p]\\
&\M_{r-d_j}(g_j \y_k, I_k) \succeq 0,\quad j\in J_k, k\in[p]\\
&L_{\y_k}(\x^\a) = L_{\y_j}(\x^\a),\a\in\N^n_{2r},\supp(\a) \subseteq I_k \cap I_j,  j\in U_k, k \in[p]\\
&L_{\y_k}(1)=1, \quad  k\in[p]
\end{cases}
\end{equation}
Note that \ac{SDP} \eqref{eq:csmultimom} is equivalent to the following program:
\begin{equation}\label{eq:csmom}
\P_{\cs}^r:\quad
\begin{cases}
\inf\limits_{\y} &L_{\y}(f)\\
\rm{ s.t.}&\M_{r}(\y, I_k)\succeq0,\quad k\in[p]\\
&\M_{r-d_j}(g_j\y, I_k)\succeq0,\quad j\in J_k, k\in[p]\\
&y_{\mathbf{0}}=1
\end{cases}
\end{equation}
with optimal value denoted by $f_{\cs}^{r}$.
Indeed, for any sequence $\y = (y_\a)_{\a \in \N_{2r}^n}$, one can define $\y_k \coloneqq  \{ y_\a : \a \in \N_{2r}^n , \supp(\a) \subseteq I_k \}$, for all $k \in [p]$.
One obviously has $L_{\y_k}(1)=1$, and each moment matrix $\M_{r}(\y_k, I_k)$ is equal to $\M_{r}(\y, I_k)$ (and similarly for the localizing matrices).
In addition, if $I_k \cap I_j \neq \emptyset$ and $\supp(\a) \subseteq I_k \cap I_j$, then
\begin{align*}
L_{\y_k}(\x^\a)  = \{y_{\a} : \supp(\a) \subseteq I_k \cap I_j \}  = L_{\y_j}(\x^{\a}).
\end{align*}
%

Let $\Sigma[\x,I_k] \subseteq \R[\x,I_k]$ be the corresponding cone of \ac{SOS} polynomials. Then the dual of \eqref{eq:csmom} is
\begin{equation}\label{eq:cssos}
\begin{cases}
\sup\limits_{b, \sigma_{k,j}} & b  \\
\rm{ s.t.} & f - b = \sum_{k=1}^p \left(\sigma_{k,0} + \sum_{j \in J_k} \sigma_{k, j} g_j\right)\\
&\sigma_{k,0},\sigma_{k,j} \in \Sigma [\x,I_k], \quad j\in J_k, k\in[p]\\
& \deg (\sigma_{k,0} ), \deg (\sigma_{k,j} g_j) \leq 2 r, \quad j\in J_k, k\in[p]
\end{cases}
\end{equation}
In the following, we refer to \eqref{eq:csmom}--\eqref{eq:cssos} as the \ac{CSSOS} hierarchy.
To prove that the sequence $(f_{\cs}^r)_{r\ge r_{\min}}$ converges to the global optimum $f_{\min}$ of the original \ac{POP} \eqref{eq:popeff}, we rely on Lemma \ref{lemma:marginal}.
\begin{theoremf}\label{th:cscvg}
Consider \ac{POP} \eqref{eq:popeff}. If Assumption \ref{hyp:cs} holds, then the \ac{CSSOS} hierarchy \eqref{eq:csmom}--\eqref{eq:cssos} provides a nondecreasing sequence of lower bounds $(f_{\cs}^r)_{r\ge r_{\min}}$ converging to $f_{\min}$.
\end{theoremf}

\begin{remark}
Despite the convergence guarantee stated in Theorem~\ref{th:cscvg}, note that \ac{SDP}~\eqref{eq:csmom} is a relaxation of the dense \ac{SDP}~\eqref{primalj} in general, and one can have $f^r_{\cs} < f^r $ for some relaxation order $r$.
The underlying reason is that the situation here is different from the case of \ac{PSD} matrix completion (Theorem \ref{th:sparsesdpproj}).
Namely, there is no guarantee that one can obtain a \ac{PSD} matrix completion $\M_{r}(\y)$ from the submatrices $\M_{r}(\y, I_k)$, $k\in[p]$ because of the specific Hankel structure of $\M_{r}(\y)$.
At the end of Appendix \ref{sec:tssos_pop}, we provide a Julia script showing such conservatism behavior.
\end{remark}

As a corollary of Theorem \ref{th:cscvg}, we obtain the following representation result, which is a \ac{CS} version of Putinar's Positivstellensatz.
\begin{theoremf}\label{th:sparseputinar}
Let $f \in \R[\x]$ be positive on a basic compact semialgebraic set $\X$ as in \eqref{eq:defX}. Let Assumption \ref{hyp:cs} hold.
Then,
\begin{align}\label{eq:sparseputinar}
f = \sum_{k=1}^p \left(\sigma_{k,0} + \sum_{j \in J_k} \sigma_{k,j} g_j\right),
\end{align}
for some polynomials $\sigma_{k,0},\sigma_{k,j} \in \Sigma [\x,I_k]$, $j\in J_k$, $k\in[p]$.
\end{theoremf}

Let us compare the computational cost of the \ac{CSSOS} hierarchy \eqref{eq:cssos} with the dense hierarchy \eqref{dualjsos}.
For this, we define $\tau \coloneqq\max_{k\in[p]} |I_k| = \max_{k\in[p]} n_k$, that is, $\tau$ is the maximal size of the subsets $I_1,\dots,I_p$.
\begin{enumerate}[(1)]
\item The dense \ac{SOS} formulation \eqref{dualjsos} involves $m+1$ \ac{SOS} polynomials in $n$ variables of degree at most $2r$, yielding $m+1$ \ac{SDP} matrices of size at most $\binom{n + r}{r}$ and $\binom{n+2r}{2r}$ equality constraints.
\item The \ac{CSSOS} formulation \eqref{eq:cssos} involves $p+m$ \ac{SOS} polynomials in at most $\tau$ variables and of degree at most $2r$, yielding $p+m$ \ac{SDP} matrices of size at most $\binom{\tau + r}{r}$ and at most $p\binom{\tau + 2r}{2r}$ equality constraints.
\end{enumerate}
Overall, when $n$ is fixed and $r$ varies, the $r$-th step of the hierarchy involves $\bigo{(r^{2n})}$ equality constraints in the dense setting against $\bigo{(pr^{2\tau})}$ in the sparse setting.
This allows one to handle \ac{POP}s involving several hundred variables if the maximal subset size $\tau$ is small (say, $\tau\leq10$).
Furthermore, as shown in the following example, one can also benefit from the computational cost saving when $r$ increases for \ac{POP}s involving a small number of variables (say, $n\leq10$).
\begin{example}\label{ex:sparsepop}
Coming back to Example \ref{ex:cs}, let us compare  the hierarchy of dense relaxations given in Chapter \ref{sec:hierarchy} with the \ac{CS} variant.
For $r=1$, the dense \ac{SDP} relaxation \eqref{dualjsos} involves $\binom{n + 2 r}{2 r} = \binom{6 + 2}{2} = 28$ equality constraints and provides a lower bound of $f^1 = 20.755$ for $f_{\min}$.
The dense \ac{SDP} relaxation \eqref{dualjsos} with $r=2$ involves $\binom{6 + 4}{4} = 210$ equality constraints and provides a tighter lower bound of $f^2 = 20.8608$.
For $r=2$, the sparse \ac{SDP} relaxation \eqref{eq:cssos} involves $\binom{2 + 4}{4} + 2 \binom{4 + 4}{4} = 155$ equality constraints and provides the same bound $f^2_{\cs} = f^2 = 20.8608$.
In Appendix \ref{chap:matlab}, we provide the MATLAB script allowing one to obtain these results.
The dense \ac{SDP} relaxation with $r=3$ involves $924$ equality constraints against $448$ for the sparse variant.
This difference becomes significant while considering the polynomial time complexity of solving \ac{SDP}, already mentioned in Chapter \ref{sec:sourcesdp}.
\end{example}
\section{A variant with SOS of bounded degrees}\label{sec:sbsos}
In certain situations, alternative representations for positive polynomials can be more interesting as one can bound in advance the degree of the \ac{SOS} polynomials involved.
\begin{theoremf}\label{th:sbsos}
Let us fix an $r\in\N$ and let Assumption \ref{hyp:cs} hold with $J'=\emptyset$. Suppose that, possibly after scaling, $0 \leq g_j \leq 1$ on $\X$ for each $j\in[m]$.
If $f$ is positive on $\X$, then
\begin{align}\label{eq:sbsos}
f = \sum_{k=1}^p \left( \sigma_k + \sum_{\a,\b \in \N^{|J_k|}} c_{k,\a\b} \prod_{j\in J_k} g_j^{\alpha_j} (1-g_j)^{\beta_j} \right),
\end{align}
for some finitely many real scalars $c_{k,\a\b} \geq 0$ and \ac{SOS} polynomials $\sigma_k\in\Sigma[\x,I_k]$ with $\deg(\sigma_k)\le2r$.
\end{theoremf}
The representation from Theorem \ref{th:sbsos} is called the \emph{sparse bounded \ac{SOS} (SBSOS)} representation.
After replacing the right-hand side in the equality constraint of  \eqref{eq:cssos} by an SBSOS representation, we obtain the following SBSOS hierarchy, indexed by an integer $s\in\N^*$:
\begin{equation}\label{eq:sbsoshierarchy}
\begin{cases}
\sup\limits_{b, \sigma_{k},c_{k,\a \b}} & b  \\
\quad\rm{s.t.} & f - b = \sum_{k=1}^p \left( \sigma_k + \sum_{|\a+\b| \leq s} c_{k,\a \b} \prod_{j\in J_k} g_j^{\alpha_j} (1-g_j)^{\beta_j} \right)\\
&c_{k,\a\b}\geq0,|\a+\b|\leq s, \quad k\in[p],\a,\b\in\N^{|J_k|}\\
&\sigma_{k}\in\Sigma[\x,I_k],\deg(\sigma_k)\le2r,\quad j\in J_k, k\in[p]
\end{cases}
\end{equation}
While the degree of each $\sigma_{k}$ is a priori fixed and at most $2 r$, one can increase the degree $s$ of each term $\prod_{j\in J_k} g_j^{\alpha_j} (1-g_j)^{\beta_j}$, which boils down to multiplying the polynomials describing the constraint set $\X$.

Regarding the computational cost, \ac{SDP}~\eqref{eq:sbsoshierarchy} involves $p$ \ac{LMI} constraints (associated to the $\sigma_k$) of maximal size $ \binom{\tau+r}{r}$ (recall $\tau=\max_{k\in[p]}|I_k|$), which does not depend on $s$.
In addition, the number of coefficients $c_{k,\a\b}$ is equal to $\binom{|J_k| + s}{s}$.
Therefore, the SBSOS hierarchy offers a computational benefit when each $|J_k|$ is relatively small.

\section{Minimizer extraction}\label{sec:csextract}
As for the standard dense \ac{moment-SOS} hierarchy stated in Chapter \ref{sec:hierarchy}, one can also detect finite convergence of the \ac{CSSOS} hierarchy and extract global minimizers with a dedicated extraction algorithm --- the \ac{CS} variant of Algorithm \ref{alg:extract}.
\begin{theoremf}\label{th:csextract}
Consider \ac{POP} \eqref{eq:popeff}. Let Assumption \ref{hyp:cs} (i)--(ii) hold and let us consider the hierarchy of moment relaxations $(\P^r_{\cs})_{r \geq r_{\min}}$ defined in \eqref{eq:csmom}.
Let $a_k\coloneqq\max_{j \in J_k}\{d_j\}$ for all $k\in[p]$.
If for some $r \geq r_{\min}$, $\P^r_{\cs}$ has an optimal solution $\y$ which satisfies
\begin{align}\label{eq:csrank}
\rank \M_r(\y,I_k) = \rank \M_{r-a_k}(\y,I_k) \text{ for all } k\in[p],
\end{align}
and $\rank \M_r(\y,I_j \cap I_k) = 1$ for all pairs $(j,k)$ with $I_j \cap I_k \neq \emptyset$, then the \ac{SDP} relaxation \eqref{eq:csmom} is exact, i.e., $f_{\cs}^r = f_{\min}$.
In addition, for each $k\in[p]$, let $\Delta_k \coloneqq  \{\x(k)\} \subseteq \R^{n_k}$ be a set of solutions obtained from the extraction procedure $\mathtt{Extract}$, stated in Algorithm \ref{alg:extract} and applied to the moment matrix $\M_r(\y,I_k)$.
Then every $\x \in \R^n$ obtained by $(x_i)_{i\in I_k} = \x(k)$ for some $\x(k) \in \Delta_k$ is a global minimizer of \ac{POP} \eqref{eq:popeff}.
\end{theoremf}

Note that Assumption \ref{hyp:cs} (iii)--(iv) are not required in Theorem \ref{th:csextract}, as the rank conditions are strong enough to ensure finite convergence and extraction of a subset of global minimizers.

\section{From polynomial to rational functions}\label{sec:csrat}
Here, we consider the following optimization problem:
\begin{equation}\label{eq:sumrat}
f_{\min} \coloneqq  \inf_{\x \in \X} \sum_{i=1}^t \frac{p_i(\x)}{q_i(\x)},
\end{equation}
where $\X = \{\x \in \R^n : g_1(\x)  \geq 0, \dots, g_m(\x) \geq 0 \} $ is a compact set, all numerators and denominators are polynomials, and all denominators are positive on $\X$.

Problem \eqref{eq:sumrat} is a \emph{fractional programming} problem of a rather general form.
Here, we assume that the degree of each numerator/denominator is rather small ($\le10$), but that the number of terms $t$ can be much larger (10 to 100).
For dense data, the number of variables should also be small ($\le10$).
However, this number can be also relatively large (10 to 100) provided that the problem data exhibits \ac{CS}, as in Section \ref{sec:cs}.
One naive strategy is to reduce all fractions to the same denominator and obtain a single rational fraction to minimize.
However, this approach is not adequate since the  degree  of the common denominator will be potentially large and even if $n$ is small, one might not be able to solve the first-order relaxation of the related \ac{moment-SOS} hierarchy.
A more suitable strategy for solving \eqref{eq:sumrat} is to cast it as a particular instance of the \ac{GMP}, namely the following infinite-dimensional \ac{LP} over measures:
\begin{align}\label{eq:sumratmom}
\begin{array}{rll}
f_{\text{meas}} \coloneqq  &\displaystyle\inf_{\mu_i}&  \,\displaystyle\sum_{i=1}^t \int_\X p_i(\x)\,\mathrm{d}\mu_i(\x) \\
&\,\rm{ s.t.}&\displaystyle\int_\X\x^\a q_i(\x)\,\mathrm{d} \mu_i(\x)=\displaystyle\int_\X \x^\a q_1(\x)\,\mathrm{d}\mu_1(\x),\a\in\N^n,i=2,\ldots,t\\
&&\displaystyle\int_\X q_1(\x)\,\mathrm{d}\mu_1(\x) \, = \, 1\\
&& \mu_1,\dots,\mu_t\in\mathscr{M}_+(\X)
\end{array}
\end{align}
\begin{theoremf}\label{th:sumrat}
Consider \eqref{eq:sumrat}. Let $\X$ be a compact set and assume that each $q_i$ is positive on $\X$ for all $i \in [t]$.
Then $f_{\text{meas}} = f_{\min}$.
\end{theoremf}
\begin{proof}
First, we prove that $f_{\min} \geq f_{\text{meas}}$.
Let $\ba$ be a global minimizer of $f = \sum_{i=1}^t \frac{p_i}{q_i}$ on $\X$, assumed to exist thanks to the compactness hypothesis.
For each $i \in [t]$, let $\mu_i = \frac{1}{q_i(\ba)} \delta_{\ba}$ be the Dirac measure centered at $\ba$ weighted by $\frac{1}{q_i(\ba)}$.
We then have
\[
\int_{\X} q_1(\x)\,\mathrm{d}\mu_1(\x)
= \frac{1}{q_1(\ba)} \int_{\X} q_1(\x)  \delta_{\ba}\,\mathrm{d}\x = 1,
\]
and for each $i < t$ and all $\a \in \N^n$,
\[
\int_{\X} q_i(\x) \x^\a\,\mathrm{d}\mu_i(\x) = \ba^\a,
\]
ensuring that the measures $(\mu_i)_{i\in[t]}$ are feasible for \eqref{eq:sumratmom}.
The associated optimal value is
\[
\sum_{i=1}^t \int_\X p_i(\x)\,\mathrm{d}\mu_i(\x) = \sum_{i=1}^t \frac{p_i(\ba)}{q_i(\ba)} = f(\ba) = f_{\min}.
\]
To prove the other direction, let $(\mu_i)_{i\in[t]}$ be a feasible solution of \eqref{eq:sumratmom}.
For each $i\in[t]$, let us define the measure $\nu_i$ as follows:
\[
\nu_i(\B) \coloneqq  \int_{\X\cap \B} q_i(\x)\,\mathrm{d}\mu_i(\x),
\]
for each Borel set $\B$ in the Borel $\sigma$-algebra of $\R^n$.
The support of $\nu_i$ is $\X$.
Since measures on compact sets are moment determinate (see Chapter \ref{sec:meas}), the moment equality constraints of \eqref{eq:sumratmom} imply that $\nu_i = \nu$, for each $i \in \{2,\dots, t\}$.
The constraint $\int_\X q_1\,\mathrm{d}\mu_1 = 1$ implies that $\nu_1$ is a probability measure on $\X$ (since its mass is 1).
Then one can rewrite the objective value as
\begin{align*}
	\sum_{i=1}^t \int_\X p_i\,\mathrm{d}\mu_i &= \sum_{i=1}^t \int_\X \frac{p_i}{q_i} q_i\,\mathrm{d}\mu_i
	= \sum_{i=1}^t \int_{\X} \frac{p_i}{q_i}\,\mathrm{d}\nu_1\\
	&= \int_\X f\,\mathrm{d}\nu_1 \geq \int_\X f_{\min}\,\mathrm{d}\nu_1 = f_{\min},
\end{align*}
where the inequality follows from the fact that $f \geq f_{\min}$ on $\X$.
\qed
\end{proof}
This first \ac{GMP} formulation \eqref{eq:sumratmom} allows one to handle general (possibly dense) rational programs.
When the numerators and denominators satisfy a \ac{csp} similar to the one stated in Chapter \ref{sec:cs}, one can rely on a dedicated \ac{CS} formulation.
\begin{assumption}\label{hyp:csrat}
There is a decomposition of $[n] = \cup_{i=1}^t I_i$ and a partition of $[m] = \cup_{i=1}^t J_i$ such that for each $i\in[t]$, $p_i,q_i \in \R[\x,I_i]$ (as in Assumption \ref{hyp:cs} (i) in the case of polynomials), and Assumption \ref{hyp:cs} (ii)--(iv) hold.
\end{assumption}
Then, as in Chapter \ref{sec:csmom} for each $i\in[t]$, with $n_i = |I_i|$, let us define
\[
\X_i \coloneqq  \{\x \in \R^{n_i} : g_j(\x) \geq 0, j \in J_i \},
\]
so that $\X$ can be equivalently described as
\[
\X = \{\x \in \R^n : (x_k)_{k \in I_i} \in \X_i, i \in [t] \}.
\]
Similarly, for all pairs $i,j \in [t]$ such that $I_i \cap I_j \neq \emptyset$, we define $\X_{ij}$, as well as the projection $\pi_i: \mathscr{M}_+(\X) \to \mathscr{M}_+(\X_i)$, for each $i\in[t]$, and $\pi_{ij}$ for all pairs $i,j \in [t]$ such that $I_i \cap I_j \neq \emptyset$.
Then the \ac{CS} variant of the infinite-dimensional \ac{LP} \eqref{eq:sumratmom} is given by
\begin{align}\label{eq:sparsesumratmom}
\begin{array}{rll}
f_{\cs} \coloneqq &\displaystyle\inf_{\mu_i}&  \,\displaystyle\sum_{i=1}^t \int_{\X_i} p_i(\x)\,\mathrm{d}\mu_i(\x) \\
&\,\rm{ s.t.}& \pi_{ij} (q_i\,\mathrm{d}\mu_i) = \pi_{ji} (q_j\,\mathrm{d} \mu_j), \quad j \in U_i,i \in [t-1]\\
&&\displaystyle\int_{\X_i} q_i(\x)\,\mathrm{d}\mu_i(\x) = 1,\mu_i\in\mathscr{M}_+(\X_i), \quad i \in[t]
\end{array}
\end{align}
\begin{theoremf}\label{th:csrat}
Consider \eqref{eq:sumrat}. Let $\X$ be a compact set, and assume that $q_i$ is positive on $\X$ for all $i \in [t]$.
If Assumption \ref{hyp:csrat} holds, then $f_{\cs} = f_{\min}$.
\end{theoremf}
One can then derive a \ac{CS}-adpated hierarchy of \ac{SDP} relaxations for \ac{LP} \eqref{eq:sparsesumratmom}:
\begin{equation}\label{eq:csratmom}
\P_{\cs}^r:
\begin{cases}
\displaystyle\inf_{\y_i} & \displaystyle\sum_{k=1}^t L_{\y_i}(p_i)\\
\rm{s.t.}&\M_{r}(\y_i, I_i)\succeq0 , \quad  i\in[t]\\
&\M_{r-d_j}(g_j\y_i, I_i)\succeq0,\quad j\in J_i, i\in[t]\\
& L_{\y_i}(\x^\a q_i) = L_{\y_j}(\x^\a q_j),\quad |\a| + \max\,\{\deg(q_i), \deg(q_j)\}\leq 2r\\
& \quad \quad \quad \quad \quad \quad \quad \quad \quad \quad \quad  \text{ and } \supp(\a) \subseteq I_i \cap I_j, j\in U_i, i\in[t]\\
& L_{\y_i }(q_i)=1, \quad  i\in[t]
\end{cases}
\end{equation}
where $|\a|\coloneqq \sum_{i=1}^n\alpha_i$ for $\a\in\N^n$.
\begin{example}\label{ex:rat}
From the Rosenbrock problem
\[
\inf_{\x\in\R^n} \sum_{i=1}^{n-1} \left( 100 (x_{i+1}-x_i^2)^2 + (x_i-1)^2 \right),
\]
we define the following rational optimization problem
\begin{align}\label{eq:rat_rosenbrock}
f_{\max} \coloneqq \sup_{\x\in\R^n} \sum_{i=1}^{n-1} \frac{1}{100 (x_{i+1}-x_i^2)^2 + (x_i-1)^2 }.
\end{align}
Note that Problem \eqref{eq:rat_rosenbrock} has the same critical points as the Rosenbrock problem, which yields numerical issues for local optimization solvers embedded in optimization software such as BARON or NEOS.
The global optimum $f_{\max} = n - 1$ of Problem \eqref{eq:rat_rosenbrock} is attained at $x_i = 1$, $i \in [n]$.
Each summand depends on 2 variables, so that we can define $I_i = \{i, i+1 \}$ for $i \in [n-1]$.
To bound the problem, we let $g_i=16 - x_i^2$ for $i \in [n]$.
When calling multiple times local optimization solvers with random initial guesses, we obtain local optima far away from the global optimum.
This is in deep contrast with our \ac{CS}-adpated hierarchy of \ac{SDP} \eqref{eq:csratmom}.
After selecting the minimal relaxation order $r = 1$, we obtain the global optimum for Problem \eqref{eq:rat_rosenbrock} with up to $1000$ variables.
The typical CPU time ranges from a few seconds to a few minutes on a modern standard laptop.
The BARON software (on the NEOS server) can find the global optimum in most cases when $n < 640$.
For $n \geq 640$, BARON returns wrong values (the first corresponding coordinate is equal to $-0.995$ instead of $1$).
Overall this rational problem yields numerical instabilities for such local solvers, which can be handled with the \ac{CS}-adpated \ac{SDP} relaxations.
\end{example}

\section{Notes and sources}

The \ac{CSSOS} hierarchy for \ac{POP}s was first studied in \citecs{Waki06}.
Its global convergence was proved in \citecs{Las06} soon after by introducing $p$ additional quadratic constraints.
%
%
%
Lemma \ref{lemma:marginal} is proved in \citecs[\S~6]{Las06}.
Theorem \ref{th:sparseputinar} is stated and proved in \citecs[Corollary 3.9]{Las06}.
An alternative proof is provided in \citecs{grimm2007note}.


The alternative SBSOS representation stated in Theorem \ref{th:sbsos} is provided in \citecs{weisser2018sparse}.
We refer to this former paper for more details on properties of the related \ac{SOS} hierarchy.

Theorem \ref{th:csextract} is stated in \citecs[Theorem 3.7]{Las06} and proved in \citecs[\S~4.2]{Las06}.

The results from Chapter \ref{sec:csrat} are mostly stated in \citecs{bugarin2016minimizing}.
The proof of Theorem \ref{th:csrat}, very similar to the one of Theorem \ref{th:cspop}, can be found in \citecs[\S~3.1]{bugarin2016minimizing}.
Example \ref{ex:rat} is provided in \citecs[\S~4.4.2]{bugarin2016minimizing}.
The interested reader can find detailed illustrations of the \ac{CSSOS} hierarchy for rational programming in \citecs[\S~4]{bugarin2016minimizing}, together with comparison with local optimization solvers such as the BARON software \citecs{tawarmalani2005polyhedral}, publicly available on the NEOS server \citecs{czyzyk1998neos}.


For the more advanced problem of set approximation, in particular in the context of dynamical systems, exploiting \ac{CS} is more delicate as the set of trajectories of a system with sparse dynamics is not necessarily sparse.
Recent research efforts have been pursued for volume approximation \citecs{tacchi2021exploiting} and region of attraction \citecs{tacchi2020approximating,schlosser2020sparse}.


\input{cs.bbl}

%% file: roundoff.tex
\chapter{Application in computer arithmetic}\label{chap:roundoff}

In this chapter, we describe an optimization framework to provide upper bounds on absolute roundoff errors of floating-point nonlinear programs, involving polynomials. 
The efficiency of this framework is based on the \ac{CSSOS} hierarchy which exploits \ac{CS} of the input polynomials, as described in Chapter \ref{chap:cs}.

%
Constructing numerical programs which perform accurate computation turns out to be difficult, due to finite numerical precision of implementations such as floating-point or fixed-point representations. Finite-precision numbers induce roundoff errors,  and knowledge of the range of these roundoff errors is required to fulfill safety criteria of critical programs, as typically arising in modern embedded systems such as aircraft controllers. Such knowledge can be used in general for developing accurate numerical software, but is also particularly relevant when considering migration of algorithms onto hardware (e.g., FPGAs). The advantage of architectures based on FPGAs is that they allow more flexible choices in number representations, rather than limiting the choice between  IEEE standard single or double precision. Indeed, in this case, we benefit from a more flexible number representation while still ensuring guaranteed bounds on the program output.

Our method to bound the error is a decision procedure based on a specialized variant of the Lasserre hierarchy~\citeroundoff{Las06}, outlined in Chapter \ref{chap:cs}. 
The procedure relies on \ac{SDP} to provide sparse \ac{SOS} decompositions of nonnegative polynomials. 
Our framework handles polynomial program analysis (involving the operations $+,\times,-$) as well as extensions to the more general class of semialgebraic and transcendental programs (involving $\sqrt{\,\,\,}, /, \min, \max, \arctan, \exp$), following the approximation scheme described in~\citeroundoff{Magron15sdp}.
For the sake of conciseness, we focus in this book on polynomial programs only. 
The interested reader can find more details in the related publication \citeroundoff{toms17}.

\section{Polynomial programs}
Here we consider a given program that implements a polynomial expression $f$ with input variables $\x$ taking values in a region $\X$. 
We assume that  $\X$ is included in a box (i.e., a product of closed intervals) and that $\X$ is encoded as in \eqref{eq:defX}: 
\[ 
\X \coloneqq  \{\, \x \in \R^n \, : \, g_1 (\x) \geq 0, \dots, g_{l} (\x) \geq 0 \,\} ,
\]
for polynomial functions $g_1, \dots, g_l$. 

The type of numerical constants is denoted by $\mathtt{C}$. 
In our current implementation, the user can choose either 64 bit floating-point or arbitrary-size rational numbers.
The inductive type of polynomial expressions $f,g_1,\dots,g_l$ with coefficients in $\mathtt{C}$ is $\mathtt{pExprC}$ defined as follows:\\
~\\
$\mathtt{type \, pexprC = }$\\
$\mathtt{\, \, \, Pc \, of \, C}$\\
$\mathtt{| \, Px \, of \, positive}$\\
$\mathtt{| \, Psub \, of \, pexprC *pexprC \, \, 
| \, \, Pneg \, of \, pexprC}$\\
$\mathtt{| \, Padd \, of \, pexprC * pexprC}$\\
$\mathtt{| \, Pmul \, of \, pexprC * pexprC}$\\
~\\
The constructor $\mathtt{Px}$ takes a positive integer as argument to represent either an input or local variable.
One obtains rounded expressions using a recursive procedure $\mathtt{round}$.
We adopt the standard practice~\citeroundoff{higham2002accuracy} to approximate a real number $x$ with its closest floating-point representation $\hat{x} = x (1 + e)$, with $|e|$ is less than the machine precision $\epsilon$. 
In the sequel, we neglect both overflow and denormal range values.
The operator $\hat{\cdot}$ is called the rounding operator and can be selected among rounding to nearest, rounding toward zero (resp.~$\pm\infty$). In the sequel, we assume rounding to nearest.
The scientific notation of a binary (resp.~decimal) floating-point number $\hat{x}$ is a triple $(\text{s}, \text{sig}, \exp)$ consisting of a sign bit $\text{s}$, a {\em significand} $\text{sig} \in [1, 2)$ (resp.~$[1, 10)$) and an {\em exponent} $\exp$, yielding numerical evaluation $(-1)^{\text{s}} \, \text{sig} \, 2^{\exp}$ (resp.~$(-1)^{\text{s}} \, \text{sig} \, 10^{\exp}$). 

The upper bound on the relative floating-point error is given by  $\epsilon = 2^{-\precrnd}$, where $\precrnd$ is called the {\em precision}, referring to the number of significand bits used. For single precision floating-point, one has $\precrnd = 24$. For double (resp.~quadruple) precision, one has $\precrnd = 53$ (resp.~$\precrnd=113$). Let $\F$ be the set of binary floating-point numbers.

For each real-valued operation $\bop_\R \in \{+, -, \times \}$, the result of the corresponding floating-point operation $\bop_\F \in \{\oplus, \ominus, \otimes\}$ satisfies the following when complying with IEEE 754 standard arithmetic~\citeroundoff{IEEE} (without overflow, underflow and denormal occurrences):
\begin{equation}\label{eq:roundbop}
\bop_\F(\hat{x}, \hat{y}) = \bop_\R (\hat{x}, \hat{y})(1 + e), \quad |e| \leq \epsilon = 2^{-\precrnd}.
\end{equation}

Then, we denote by $\hat{f}(\x,\e)$ the rounded expression of $f$ after applying the $\mathtt{round}$ procedure, introducing additional error variables $\e$. 

\section{Upper bounds on roundoff errors}

The algorithm $\mathtt{roundoff\_bound}$, depicted in Algorithm \ref{alg:bound}, takes as input $\x$, $\X$, $f$, $\hat{f}$, $\e$ as well as the set $\E$ of bound constraints over $\e$. 
For a given machine $\epsilon$, one has $\E \coloneqq  [-\epsilon, \epsilon]^m$, with $m$ being the number of error variables. 
This algorithm actually relies on the \ac{CSSOS} hierarchy from Chapter \ref{chap:cs}, thus $\mathtt{roundoff\_bound}$ also takes as input a relaxation order $r \in \N^*$. The algorithm provides as output an interval enclosure of the error $ \hat{f}(\x,\e) - f(\x)$ over $\K\coloneqq  \X \times \E$. 
From this interval $[f_{\min}^r, f_{\max}^r]$, one can compute $|f|_{\max}^r \coloneqq  \max\,\{- f_{\min}^r, f_{\max}^r \}$, which is a sound upper bound of the maximal absolute error $|\Delta|_{\max} \coloneqq  \max_{(\x,\e)\in \K} \mid \hat{f}(\x,\e) - f(\x) \mid $. 

\begin{algorithm}\caption{$\mathtt{roundoff\_bound}$}\label{alg:bound}
\begin{algorithmic}[1]
\Require 
input variables $\x$, input constraints $\X$, nonlinear expression $f$, rounded expression $\hat{f}$, error variables $\e$, error constraints $\E$, relaxation order $r$
\Ensure 
interval enclosure of the error $\hat{f} - f$ over $\K \coloneqq  \X \times \E$
\State Define the absolute error $\Delta(\x, \e) \coloneqq  \hat{f}(\x,\e) - f(\x)$ \label{line:r}

\State Compute $\ell(\x,\e) \coloneqq  \Delta(\x, \mathbf{0}) + \sum_{j=1}^m \frac{\partial \Delta(\x,\e)} {\partial e_j} (\x,\mathbf{0}) \, e_j$ \label{line:l}

\State Define $h \coloneqq  \Delta - \ell$ \label{line:h}

\State  $[\underline h, \overline h] \coloneqq  \iaboundfun{h}{\K}$ \label{line:iabound} \Comment{Compute bounds for $h$}
\State $[\ell_{\min}^r, \ell_{\max}^r] \coloneqq  \sdpboundfun{\ell}{\K}{r}$  \label{line:sdpbound} \Comment{Compute bounds for $\ell$}
\State \Return $[\ell_{\min}^r + \underline h,  \ell_{\max}^r + \overline h]$ 
\end{algorithmic}
\end{algorithm}

After defining the absolute roundoff error $\Delta \coloneqq  \hat{f} - f$ (Step~\ref{line:r}), one decomposes $\Delta$ as the sum of an expression $\ell$ which is affine with resepct to the error variable $\e$ and a remainder $h$. One way to obtain $\ell$ is to compute the vector of partial derivatives of $\Delta$ with resepct to $\e$ evaluated at $(\x, \mathbf{0})$ and finally to take the inner product of this vector and $\e$ (Step~\ref{line:l}). 
Then, the idea is to compute a precise bound of $\ell$ and a coarse bound of $h$. 
The underlying reason is that $h$ involves error term products of degree greater than 2 (e.g.~$e_1 e_2$), yielding an interval enclosure of \textit{a priori} much smaller width, compared to the interval enclosure of $\ell$. 
One obtains the interval enclosure of $h$ using the procedure $\iabound$ implementing basic interval arithmetic (Step~\ref{line:iabound}) to bound the remainder of the multivariate Taylor expansion of $\Delta$ with resepct to $\e$, expressed as a combination of the second-error derivatives (similar as in~\citeroundoff{fptaylor15}).
The algorithm $\mathtt{roundoff\_bound}$ is very similar to the algorithm of $\fptaylor$~\citeroundoff{fptaylor15}, except that \ac{SDP} based techniques are used instead of the global optimization procedure from~\citeroundoff{fptaylor15}. Note that overflow and denormal are neglected here but one could handle them, as in~\citeroundoff{fptaylor15}, by adding additional error variables and discarding the related terms using naive interval arithmetic.

The bounds of $\ell$ are provided through the $\sdpbound$ procedure, which solves two instances of \eqref{eq:cssos}, at relaxation order $r$. We now give more explanation about this procedure.
We can map each input variable $x_i$ to the integer $i$, for all $i\in [n]$, as well as each error variable $e_j$ to $n+j$, for all $j\in [m]$. Then, define the sets $I_1 \coloneqq  [n] \cup \{n+1\}, \dots, I_m \coloneqq  [n] \cup \{n+m\}$.
Here, we take advantage of the \ac{csp} of $\ell$ by using $m$ distinct sets of cardinality $n+1$ rather than a single one of cardinality $n+m$, i.e., the total number of variables. 
Note that these subsets satisfy \eqref{eq:RIP} and one can write $\ell(\x,\e) = \Delta(\x, \mathbf{0}) + \sum_{j=1}^m \frac{\partial \Delta(\x,\e)} {\partial e_j}(\x, \mathbf{0}) \, e_j$. 
After noticing that $\Delta(\x,\mathbf{0}) = \hat{f}(\x,\mathbf{0}) - f(\x) = 0$, one can scale the optimization problems by writing 
\begin{align}\label{eq:lscale}
\ell(\x,\e) = \sum_{j=1}^m s_j (\x) e_j = \epsilon \sum_{j=1}^m s_j (\x) \frac{e_j}{\epsilon},
\end{align}
with $s_j(\x) \coloneqq  \frac{\partial \Delta(\x,\e)} {\partial e_j} (\x,\mathbf{0})$, for all $j\in[m]$. Replacing $\e$ by $\e/\epsilon$ leads to computing an interval enclosure of $\ell/\epsilon$ over $\K' \coloneqq  \X \times [-1, 1]^m$.
As usual from Assumption \ref{hyp:archimedean}, there exists an integer $N > 0$ such that $N - \sum_{i=1}^n x_i^2 \geq 0$, as the input variables satisfy box constraints.
Moreover, to fulfill Assumption~\ref{hyp:cs},  one encodes $\K'$ as follows: 
\begin{align*}
\K' \coloneqq  \{(\x,\e) \in \R^{n+m} \, : \, g_1 (\x) \geq 0, \dots, g_l(\x) \geq 0, \\
g_{l+1}(\x,e_1) \geq 0, \dots, g_{l+m} (\x, e_m) \geq 0\},
\end{align*}
with $g_{l+j}(\x, e_j) \coloneqq  N + 1 -  \sum_{i=1}^n x_i^2 - e_j^2$, for all $j\in[m]$. 
The index set of variables involved in $g_j$ is $[n]$ for all $j \in [l]$. 
The index set of variables involved in $g_{l+j}$ is $I_j$ for all $j\in[m]$. 

Then, one can compute a lower bound of the minimum of $\ell'(\x,\e) \coloneqq  \ell(\x, \e) / \epsilon = \sum_{j=1}^m s_j (\x) e_j$ over $\K'$ by solving the following \ac{CSSOS} problem:
\begin{align}\label{eq:lscalesdp1}	
\begin{array}{rll}
\ell'^{r}_{\min} \coloneqq  &\sup\limits_{b, \sigma_j} & b\\	 
&\,\rm{s.t.} & \ell' - b = \sigma_0 + \sum_{j = 1}^{l+m} \sigma_j g_j\\ 
&& \sigma_0 \in \sum_{j = 1}^m \Sigma [(\x, \e), I_j]\\
&& \sigma_j \in \Sigma[(\x,\e), J_j] , \quad j \in [l+m]\\
&& \deg (\sigma_j g_j) \leq 2 r, \quad j = 0,\dots,l+m
\end{array} 
\end{align}
A feasible solution of Problem~\eqref{eq:lscalesdp1} ensures the existence of $\sigma^1 \in \Sigma[(\x,e_1)]$, $\dots$, $\sigma^m \in \Sigma[(\x,e_m)]$ such that $\sigma_0 = \sum_{j=0}^m \sigma^j$, allowing the following reformulation:
\begin{align}
\begin{split}
\label{eq:lscalesdp2}			
\begin{array}{rll}
\ell'^{r}_{\min}= &\sup\limits_{b, \sigma^j, \sigma_j} & b \\	
&\,\,\rm{s.t.} & \ell' - b = \sum_{j=1}^m \sigma^j + \sum_{j = 1}^{l+m} \sigma_j g_j\\
&& \sigma_j \in \Sigma[\x] , \quad j \in[m]\\
&& \sigma^j  \in \Sigma [(\x, e_j)],  \deg (\sigma^j) \leq 2 r, \quad j \in [m]\\
&&\deg (\sigma_j g_j) \leq 2 r  , \quad j \in[l+m]
\end{array} 
\end{split}
\end{align}
An upper bound $\ell'^{r}_{\max}$ can be obtained by replacing $\sup$ with $\inf$ and $\ell' - b$ by $b - \ell'$ in Problem~\eqref{eq:lscalesdp2}.
Our optimization procedure $\sdppoly$ computes the lower bound $\ell'^{r}_{\min}$ as well as an upper bound $\ell'^{r}_{\max}$ of $\ell'$ over $\K'$, and then returns the interval $[\epsilon \, \ell'^{r}_{\min}, \epsilon \, \ell'^{r}_{\max}] $, which is a sound enclosure of the values of $\ell$ over $\K$.
%

We emphasize two advantages of the decomposition $\Delta = \ell + h$ and more precisely of the linear dependency of $\ell$ with resepct to $\e$: scalability and robustness to \ac{SDP} numerical issues.
First, no computation is required to determine the \ac{csp} of $\ell$, by comparison to the general case. Thus, it becomes much easier to handle the optimization of $\ell$ with the sparse \ac{SDP} \eqref{eq:lscalesdp2} rather than with the corresponding instance of the dense relaxation~$(\P^r)$, given in \eqref{primalj}. While the latter involves $\binom{n+m+2r}{2r}$ \ac{SDP} variables, the former involves only $m \, \binom{n+1+2r}{2r}$ \ac{SDP} variables, ensuring the scalability of our framework.
In addition, the linear dependency of $\ell$ with resepct to $\e$ allows us to scale the error variables and optimize over a set of variables lying in $\K' \coloneqq  \X \times [-1, 1]^m$. It ensures that the range of input variables does not significantly differ from the range of error variables. This condition is mandatory in considering \ac{SDP} relaxations because most \ac{SDP} solvers (e.g., {\sc Mosek}~\citeroundoff{mosek}) are implemented using double precision floating-point. It is impossible to optimize $\ell$ over $\K$ (rather than $\ell'$ over $\K'$) when the maximal value $\epsilon$ of error variables is less than $2^{-53}$, due to the fact that \ac{SDP} solvers would treat each error variable term as 0, and consequently $\ell$ as the zero polynomial. Thus, this decomposition insures our framework against numerical issues related to finite-precision implementation of \ac{SDP} solvers.
%

Let us consider the interval $ [\ell_{\min}, \ell_{\max}]$, with $\ell_{\min} \coloneqq  \inf_{(\x,\e) \in \K} \ell(\x,\e)$ and $\ell_{\max} \coloneqq  \sup_{(\x,\e) \in \K} \ell(\x,\e)$.
The next lemma states that one can approximate this interval as closely as desired using the $\sdppoly$ procedure.
\begin{lemma}[Convergence of the $\sdppoly$ procedure]
\label{th:cvg_sdppoly}
Let $[\ell_{\min}^r, \ell_{\max}^r]$ be the interval enclosure returned by the procedure $\sdppolyfun{\ell}{\K}{r}$. Then the sequence $([\ell_{\min}^r, \ell_{\max}^r])_{r \in \N}$ converges to $[\ell_{\min}, \ell_{\max}]$ when $r$ goes to infinity.
\end{lemma}
The proof of Lemma~\ref{th:cvg_sdppoly} is based on the fact that the assumptions of Theorem \ref{th:sparseputinar} are fulfilled for our specific roundoff error problem.
This result guarantees asymptotic convergence to the exact enclosure of $\ell$ when the relaxation order $r$ tends to infinity.
However, it is more reasonable in practice to keep this order as small as possible to obtain tractable \ac{SDP} relaxations. Hence, we generically solve each instance of Problem~\eqref{eq:lscalesdp2} at the minimal relaxation order, that is $r_{\min} = \max\,\{\lceil\deg\ell/2\rceil, \lceil\deg(g_j)/2\rceil,j=1,\ldots,l+m\}$. 
Afterwards, we can rely on the $\coq$ computer assistant to obtain formally certified upper bounds for the roundoff error; see \citeroundoff[\S~2.3]{toms17} for more details.

\section{Overview of numerical experiments}
We present an overview of our method and of the capabilities of related techniques, using an example.
Consider a program implementing the following polynomial expression $f$:
\begin{align*}
f(\x) \coloneqq \,&x_2 \times x_5 + x_3 \times x_6 - x_2 \times x_3  - x_5 \times x_6 \\
&+ x_1 \times ( - x_1 +  x_2 +  x_3  - x_4 +  x_5 +  x_6),
\end{align*}
where the six-variable vector $\x \coloneqq   (x_1, x_2, x_3, x_4, x_5, x_6)$ is the input of the program. Here the set $\X$ of possible input values is a product of closed intervals: $\X = [4.00, 6.36]^6$.
This function $f$ together with the set $\X$ appear in many inequalities arising from the the proof of the Kepler Conjecture~\citeroundoff{Flyspeck06}, yielding challenging global optimization problems.

The polynomial expression $f$ is obtained by performing 15 basic operations (1 negation, 3 subtractions, 6 additions and 5 multiplications). 
When executing this program with a set of floating-point numbers $\hat{\x} \coloneqq   (\hat{x}_1, \hat{x}_2, \hat{x}_3, \hat{x}_4, \hat{x}_5, \hat{x}_6) \in \X$, one actually computes a floating-point result $\hat{f}$, where all operations $+, -, \times$ are replaced by the respectively associated floating-point operations $\oplus, \ominus, \otimes$.
The results of these operations comply with IEEE 754 standard arithmetic~\citeroundoff{IEEE}. 
Here, for the sake of clarity, we do not consider real input variables.
For instance, (in the absence of underflow) one can write $\hat{x}_2 \otimes \hat{x}_5 =  (x_2 \times x_5) (1 + e_1)$, by introducing an error variable $e_1$ such that $-\epsilon \leq e_1 \leq \epsilon$, where the bound $\epsilon$ is the machine precision (e.g., $\epsilon = 2^{-24}$ for single precision). One would like to bound the absolute roundoff error $|\Delta(\x, \e)| \coloneqq  | \hat{f}(\x, \e) - f (\x) |$ over  all possible input variables $\x \in \X$ and error variable  $e_1, \dots, e_{15} \in [-\epsilon, \epsilon]$. Let us define $\E \coloneqq  [-\epsilon, \epsilon]^{15}$ and $\K \coloneqq  \X \times \E$. Then our bound problem can be cast as finding the maximum $|\Delta|_{\max}$ of $|\Delta |$ over $\K$, yielding the following nonlinear optimization problem:
\begin{align}
\begin{split}
\label{eq:roptim}
|\Delta|_{\max} \coloneqq  & \max_{(\x, \e) \in \K} | \Delta(\x, \e) | \\
 = & \ \ \max\,\{-\min_{(\x, \e) \in \K} \Delta(\x, \e), \max_{(\x, \e) \in \K} \Delta(\x,\e)\}.
\end{split}
\end{align}
One can directly try to solve these two \ac{POP}s using classical \ac{SDP} relaxations~\citeroundoff{Las01sos}.
As in~\citeroundoff{fptaylor15}, one can also decompose the error term $\Delta$ as the sum of a term $\ell(\x,\e)$, which is affine with resepct to $\e$, and a nonlinear term $h(\x,\e) \coloneqq  \Delta(\x,\e) - \ell(\x,\e)$. Then the triangular inequality yields:
\begin{equation}
\label{eq:lhoptim} 
|\Delta|_{\max}  \leq \max_{(\x, \e) \in \K} |\ell(\x, \e)| + \max_{(\x, \e) \in \K} |h(\x, \e)|. 
\end{equation}
It follows for this example that $\ell(\x,\e) = x_2 x_5 e_1 + x_3 x_6 e_2 +  (x_2 x_5 + x_3 x_6) e_3 + \dots + f(\x) e_{15} = \sum_{i=1}^{15} s_i(\x) e_i$, with $s_1(\x) \coloneqq  x_2 x_5, s_2(\x) \coloneqq  x_3 x_6, \dots, s_{15}(\x) \coloneqq  f(\x)$. The {\em Symbolic Taylor Expansions} method~\citeroundoff{fptaylor15} consists of using a simple branch and bound algorithm based on interval arithmetic to compute a rigorous interval enclosure of each polynomial $s_i$, $i \in[15]$, over $\X$ and finally obtain an upper bound of $|\ell| + |h|$ over $\K$. 
In contrast, our method uses sparse \ac{SDP} relaxations for polynomial optimization (derived from \citeroundoff{Las06}) to bound $|\ell|$ and basic interval arithmetic as in~\citeroundoff{fptaylor15} to bound $|h|$ (i.e., we use interval arithmetic to bound second-order error terms in the multivariate Taylor expansion of $\Delta$ with resepct to $\e$).
\begin{itemize}[noitemsep,nolistsep]
\item A direct attempt to solve the two polynomial problems occurring in Equation~\eqref{eq:roptim} fails as the \ac{SDP} solver (in our case $\sdpa$~\citeroundoff{sdpa}) runs out of memory. 
\item Using our method implemented in the $\realtofloat$ tool\footnote{\url{https://forge.ocamlcore.org/projects/nl-certify/}}, one obtains an upper bound of $760 \epsilon$ for $|\ell| + |h|$ over $\K$ in $0.15$ seconds. This bound is provided together with a certificate which can be formally checked inside the $\coq$ proof assistant in $0.20$ seconds.
\item After normalizing the polynomial expression and using basic interval arithmetic, one obtains 8 times more quickly a coarser bound of $922 \epsilon$. 
\item Symbolic Taylor expansions implemented in $\fptaylor$ \citeroundoff{fptaylor15} provide a more precise bound of $721 \epsilon$, but the analysis time is 28 times slower than with our implementation. Formal verification of this bound inside the $\hol$ proof assistant takes $27.7$ seconds, which is 139 times slower than proof checking with $\realtofloat$ inside $\coq$. One can obtain an even more precise bound of $528 \epsilon$ (but 37 times slower than with our implementation) by turning on the improved rounding model of $\fptaylor$ and limiting the number of branch and bound iterations to 10000. The drawback of this bound is that it cannot be formally verified.
\item Finally, a slightly coarser bound of $762 \epsilon$ is obtained with the $\rosa$ real compiler~\citeroundoff{Darulova14Popl}, but the analysis is 19 times slower than with our implementation and we cannot get formal verification of this bound.
\end{itemize}

\section{Notes and sources}
To obtain lower bounds on roundoff errors, one can rely on testing approaches, such as meta-heuristic search~\citeroundoff{Borges12Test} or under-approximation tools (e.g., $\sthreefp$~\citeroundoff{Chiang14s3fp}). Here, we are interested in efficiently handling  the complementary over-approximation problem, namely to obtain precise upper bounds on the error. This problem boils down to finding tight abstractions of linearities or non-linearities while being able to bound the resulting approximations in an efficient way.  
For computer programs consisting of linear operations, automatic error analysis can be obtained with well-studied optimization techniques based on SAT/SMT solvers~\citeroundoff{hgbk2012fmcad} and affine arithmetic~\citeroundoff{fluctuat}. However, non-linear operations are key to many interesting computational problems arising in physics, biology, controller implementations and global optimization. 
Two promising frameworks have been designed to provide upper bounds for roundoff errors of nonlinear programs. The corresponding algorithms rely on Taylor-interval methods~\citeroundoff{fptaylor15}, implemented in the $\fptaylor$ tool, and on combining SMT with interval arithmetic~\citeroundoff{Darulova14Popl}, implemented in the $\rosa$ real compiler. 
We refer the interested reader to \citeroundoff[\S~4]{toms17} for more details on the extensive experimental evaluation that we performed.

%

\input{roundoff.bbl}

%% file: roundoff.bbl
\providecommand{\etalchar}[1]{$^{#1}$}

%% file: lip.tex
\chapter{Application in deep networks}\label{chap:lip}

The Lipschitz constant of a network plays an important role in many applications of deep learning, such as robustness certification and Wasserstein Generative Adversarial Network. We introduce an \ac{SDP} hierarchy to estimate the global and local Lipschitz constant of a multiple layer deep neural network. The novelty is to combine a polynomial lifting for ReLU function derivatives with a weak generalization of Putinar's positivity certificate.
We empirically demonstrate that our method provides a trade-off with respect to the state-of-the-art \ac{LP} approach, and in some cases we obtain better bounds in less time.

\section{Multiple layer networks}\label{sec:nn}
We focus on multiple layer networks with ReLU activations.
Recall that a function $f$, defined on a convex set $\mathcal{X} \subseteq \mathbb{R}^n$, is $L$-Lipschitz with respect to the norm $\|\cdot\|$ if for all $\mathbf{x}, \mathbf{z} \in \mathcal{X}$, we have $|f(\mathbf{x}) - f(\mathbf{z})| \le L \|\mathbf{x}-\mathbf{z}\|$.
The Lipschitz constant of $f$ with respect to norm $\|\cdot\|$, denoted by $L_f^{\|\cdot\|}$, is the infimum of all those valid $L$'s:
\begin{align}
L_f^{\|\cdot\|} \coloneqq  \inf\,\{L: \forall \mathbf{x}, \mathbf{z} \in \mathcal{X}, |f(\mathbf{x}) - f(\mathbf{z})| \le L \|\mathbf{x}-\mathbf{z}\|\}.
\end{align}

We denote by $F$ the multiple layer neural network, $m$ the number of hidden layers, $p_0, p_1, \ldots, p_m$ the number of nodes in the input layer and each hidden layer.
For simplicity, $(p_0, p_1, \ldots, p_m)$ will denote the layer structure of network $F$.
Let $\mathbf{x}_0$ be the initial input, and $\mathbf{x}_1, \ldots, \mathbf{x}_m$ be the activation vectors in each hidden layer. Each $\mathbf{x}_i$, $i \in [m]$, is obtained by a weight $\mathbf{A}_i$, a bias $\mathbf{b}_i$, and an activation function $a$, i.e., $\mathbf{x}_i = a (\z_i )$ with $\z_i = \mathbf{A}_i \mathbf{x}_{i-1} + \mathbf{b}_i$; see Figure~\ref{fig:nn}.

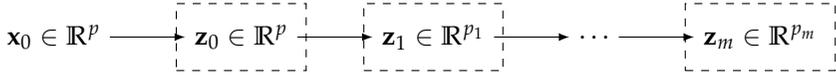
\begin{figure}[H]
\begin{tikzpicture}[scale=0.2,
init/.style={
draw,
circle,
inner sep=2pt,
font=\Huge,
join = by -latex
},
squa/.style={
draw,
inner sep=2pt,
font=\Large,
join = by -latex
},
start chain=2,node distance=10mm
]
\node[on chain=2] (top) {$\mathbf{x}_0 \in \R^p$};
\node[on chain=2,join=by -latex ] (x2) {$\mathbf{z}_0 \in \R^p$};
\node[on chain=2,join=by -latex] (w2) {$\mathbf{z}_1 \in \R^{p_1}$};
\node[on chain=2,join=by -latex] (sigma) {$\ldots$};
\node[on chain=2,join=by -latex] (output) {$\mathbf{z}_m \in \R^{p_m}$};
\node at ($(x2) + (0,1)$) (x1) {};
\node at ($(x2) + (0,-1)$) (x3) {};
\node at (w2|-x1) (w1) {};
\node at (w2|-x3) (w3) {};
\node at (output|-x1) (o1) {};
\node at (output|-x3) (o3) {};
\node[draw,dashed,fit=(w1) (w2) (w3)] {};
\node[draw,dashed,fit=(x1) (x2) (x3)] {};
\node[draw,dashed,fit=(o1) (output) (o3)] {};
\end{tikzpicture}
\caption{Description of a multiple layer neural network.}\label{fig:nn}
\end{figure}

We only consider coordinatewise application of the $\relu$ activation function, defined as $\relu (x) = \max\,\{0, x\}$ for $x \in \mathbb{R}$. The $\relu$ function is non-smooth, and we define its generalized derivative as the set-valued function $G (x)$ such that $G(x) = 1$ for $x > 0$, $G(x) = 0$ for $x < 0$ and $G(x) = \{0,1\}$ for $x = 0$.
The key reason why neural networks with $\relu$ activation function can be tackled using polynomial optimization techniques is semialgebraicity of the $\relu$ function, i.e., it can be expressed with a system of polynomial (in)equalities. For $x, u \in \mathbb{R}$, we have $u = \relu (x) = \max\,\{0,x\}$ if and only if $u(u-x) = 0, u \ge x, u \ge 0$.
Similarly, one can exploit the semialgebraicity of its derivative $\relu'$; see Figure \ref{fig:relu} and Figure \ref{fig:relu_deriv}.

\begin{figure}[H]
    \centering
    \begin{subfigure}[t]{0.45\textwidth}
        \centering
        \includegraphics[width=\textwidth]{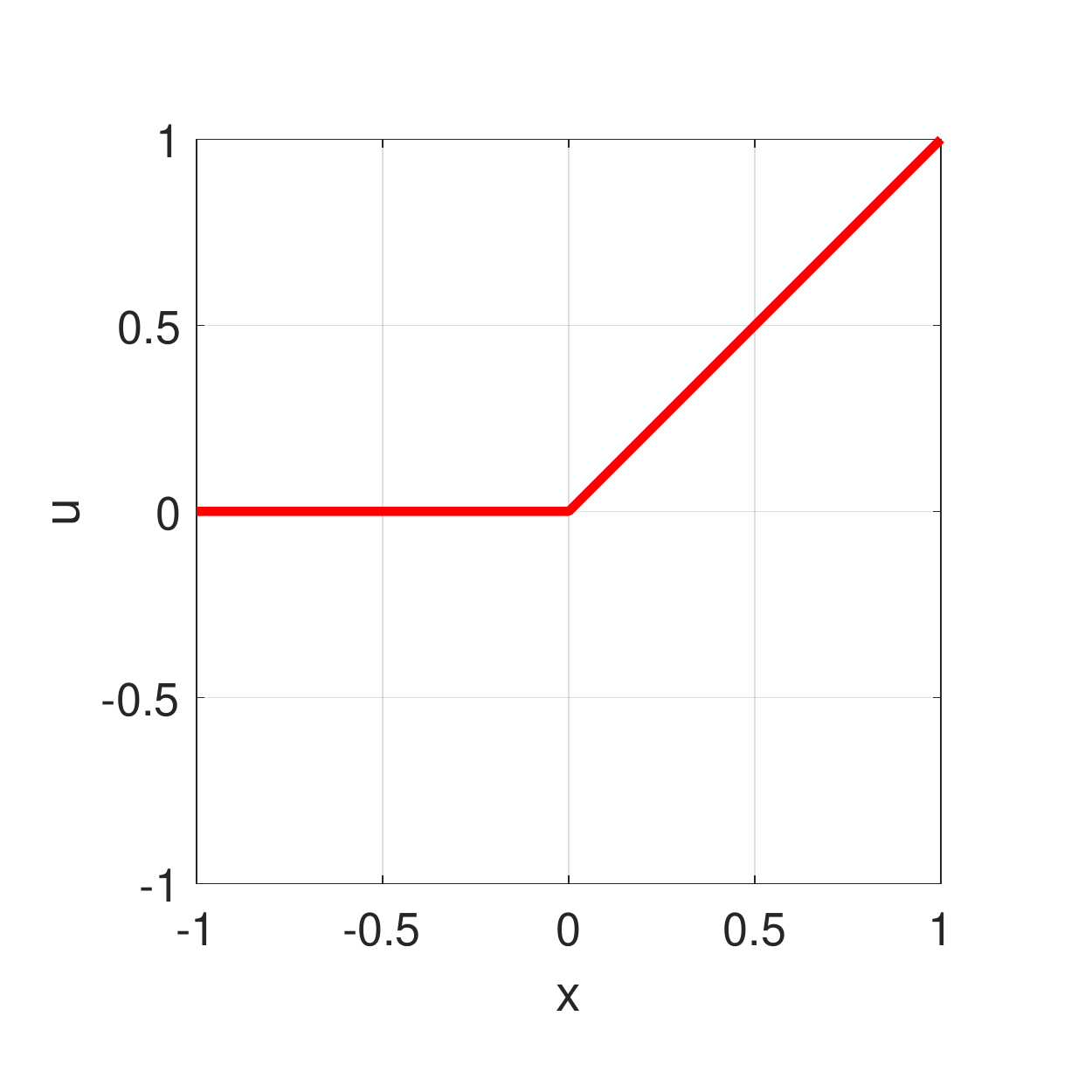}
        \caption{$u = \max\,\{x,0\}$}
    \end{subfigure}
    \hfill
    \begin{subfigure}[t]{0.45\textwidth}
        \centering
        \includegraphics[width=\textwidth]{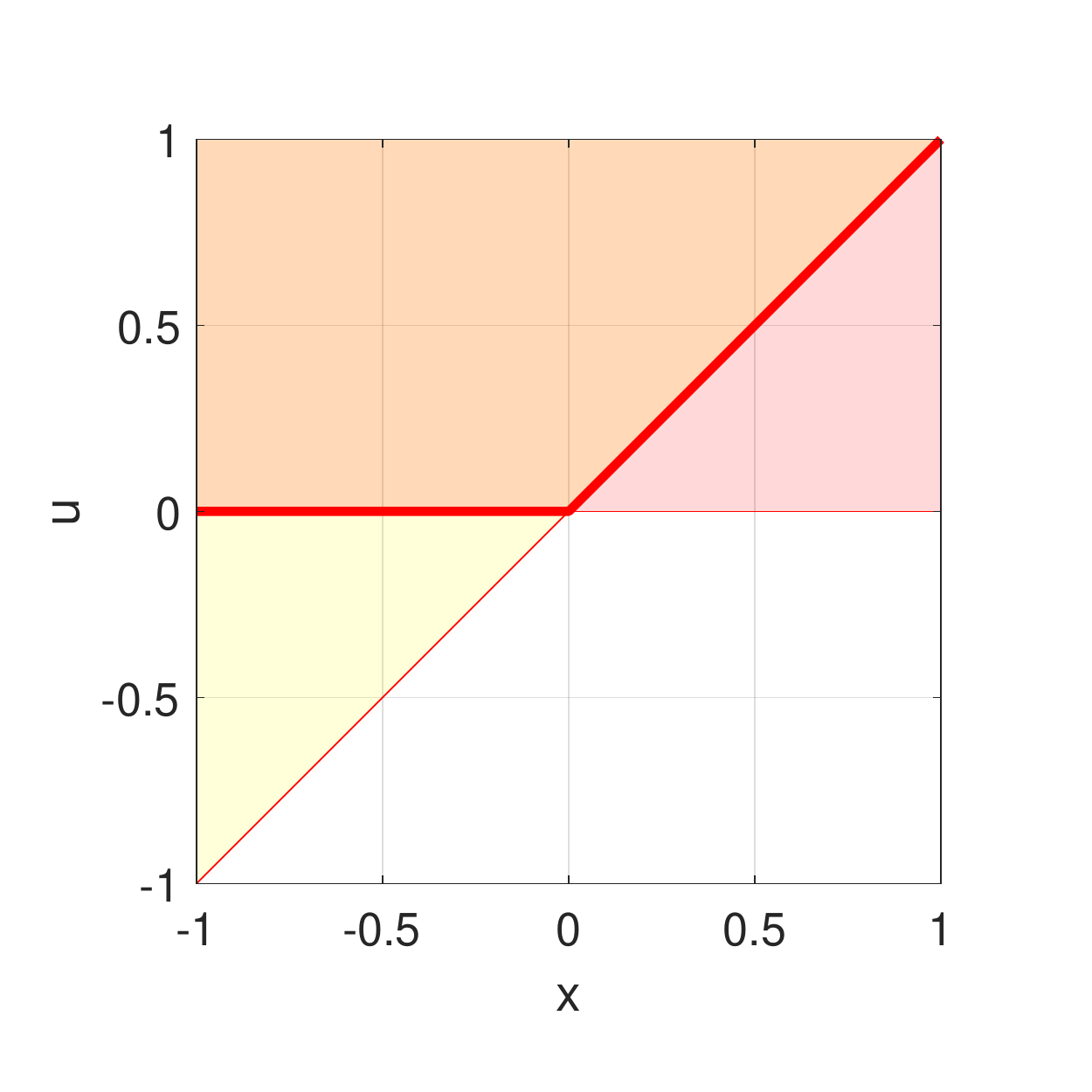}
        \caption{$u(u-x)=0, u \ge x, u \ge 0$}
    \end{subfigure}
\caption{$\relu$ (left) and its semialgebraicity (right).}
\label{fig:relu}
\end{figure}
\begin{figure}[H]
    \centering
    \begin{subfigure}[t]{0.45\textwidth}
        \centering
        \includegraphics[width=\textwidth]{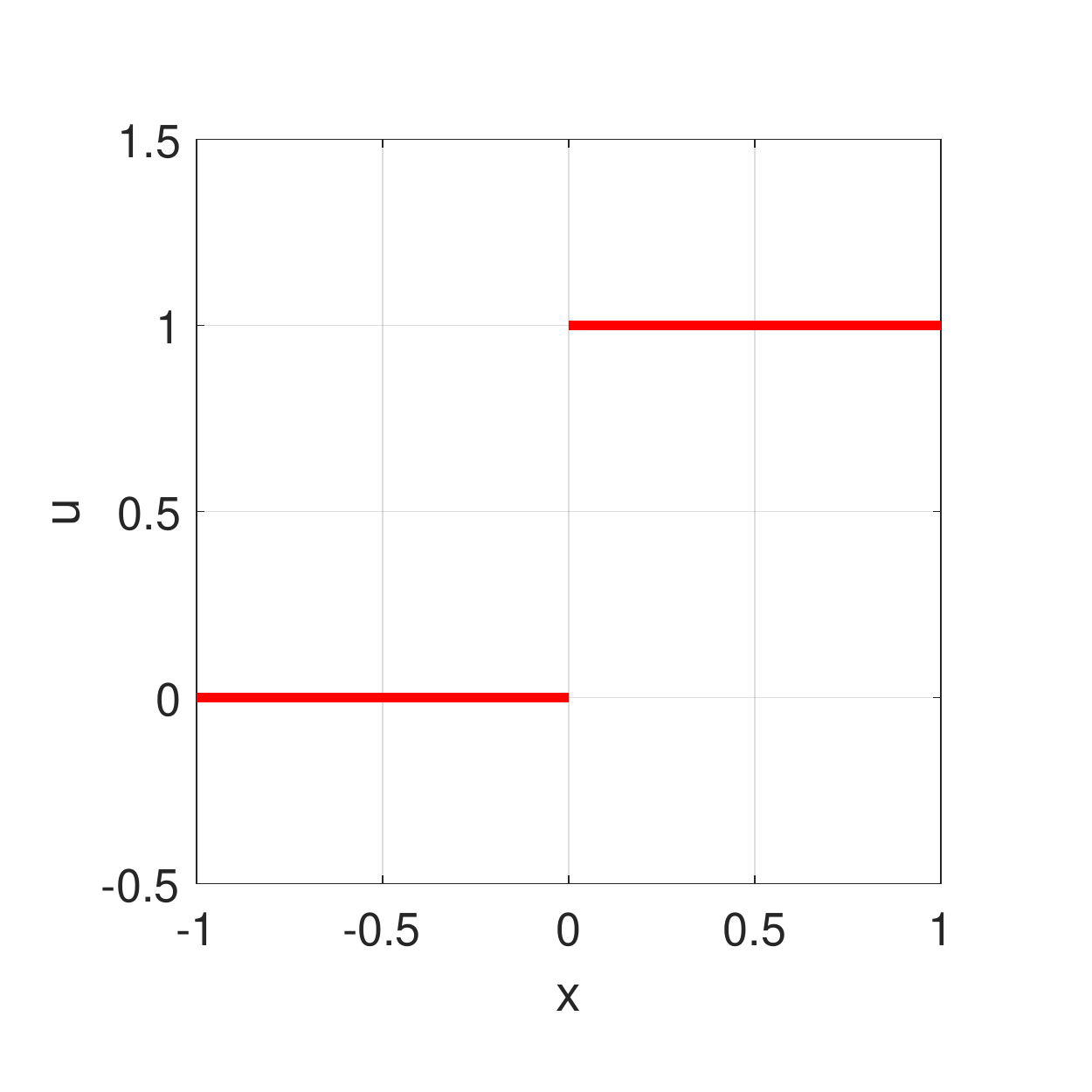}
        \caption{$u = \mathbf{1}_{\{x \ge 0\}}$}
    \end{subfigure}
    \hfill
    \begin{subfigure}[t]{0.45\textwidth}
        \centering
        \includegraphics[width=\textwidth]{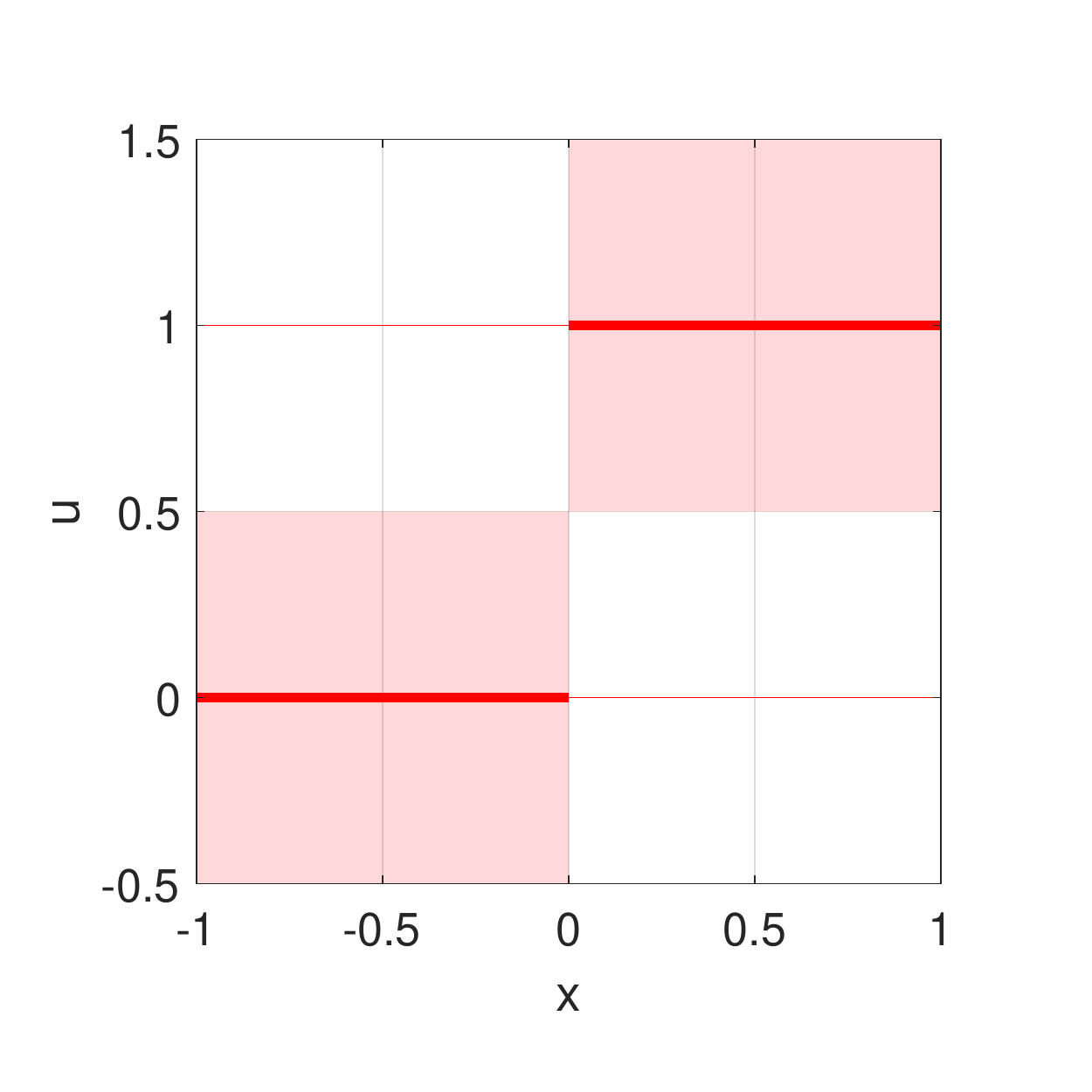}
        \caption{$u(u-1)=0, (u-\frac{1}{2})x \ge 0$}
    \end{subfigure}
\caption{$\relu'$ (left) and its semialgebraicity (right).}
\label{fig:relu_deriv}
\end{figure}

We assume that the last layer in our neural network is a softmax layer with $K$ entries, that is, the network is a classifier for $K$ labels. For each label $k \in \{1, \ldots, K\}$, the score of label $k$ is obtained by an affine product with the last activation vector, i.e., $\mathbf{c}_k^{\intercal} \mathbf{x}_m$ for some $\mathbf{c}_k \in \mathbb{R}^{p_m}$. The final output is the label with the highest score, i.e., $z = \arg \max_k \mathbf{c}_k^{\intercal} \mathbf{x}_m$. The product $\mathbf{x} \mathbf{z}$ of two vectors $\mathbf{x}$ and $\mathbf{z}$ is considered as the coordinate-wise product.

\section{Lipschitz constants}
\label{sec:lip}
Suppose we train a neural network $F$ for $K$-classifications and denote by $\mathbf{A}_i, \mathbf{b}_i, \mathbf{c}_k$ its parameters.
Thus for an input $\mathbf{x}_0 \in \mathbb{R}^{p_0}$, the targeted score of label $k$ can be expressed as $F_k (\mathbf{x}_0) = \mathbf{c}_k^{\intercal} \mathbf{x}_m$, where $\mathbf{x}_i = \relu (\mathbf{A}_i \mathbf{x}_{i-1} + \mathbf{b}_i)$, for $i \in [m]$. Let $\mathbf{z}_i = \mathbf{A}_i \mathbf{x}_{i-1} + \mathbf{b}_i$ for $i \in [m]$. By applying the chain rule on the non-smooth function $F_k$, we obtain a set valued map for $F_k$ at point any $\mathbf{x}_0$ as $G_{F_k} (\mathbf{x}_0) = (\prod_{i = 1}^{m} \mathbf{A}_i^\intercal \diag (G (\mathbf{z}_{i}))) \mathbf{c}_k$.

We fix a targeted label (label 1 for example) and omit the symbol $k$ for simplicity. We define $L_F^{\|\cdot\|}$
as the supremum of the gradient's dual norm:
\begin{align} \label{lip}
	L_F^{\|\cdot\|} = \sup_{\mathbf{x}_0 \in \Omega,\, \mathbf{v} \in G_{F_k}(\mathbf{x}_0)} \|\mathbf{v} \|_* = \sup_{\mathbf{x}_0 \in \Omega} \bigg\|\bigg(\prod_{i = 1}^{m} \mathbf{A}_i^{\intercal} \diag (G (\mathbf{z}_{i}))\bigg) \mathbf{c}\bigg\|_*,
\end{align}
where $\Omega$ is the convex input space, and $\|\cdot \|_*$ is the dual norm of $\|\cdot \|$, which is defined by $\|\mathbf{x}\|_* \coloneqq  \sup_{\|\mathbf{t}\| \le 1} \vert \langle \mathbf{t}, \mathbf{x}\rangle \vert $ for all $\mathbf{x} \in \mathbb{R}^n$. 	In general, the chain rule cannot be applied to composition of non-smooth functions \citelip{kakade2018provably, bolte2019conservative}. Hence the formulation of $G_{F_k}$ and \eqref{lip} may lead to incorrect gradients and bounds on the Lipschitz constant of the networks.
The following ensures that this is not the case and that the approach is sound.
\begin{theoremf}\label{th:lip}
If $\Omega$ is convex, then $L_F^{\|\cdot\|}$ is a Lipschitz constant for $F_k$ on $\Omega$.
\end{theoremf}
When $\Omega = \mathbb{R}^n$, $L_F^{\|\cdot\|}$ is the \emph{global} Lipschitz constant of $F$ with respect to norm $\|\cdot\|$.
In many cases we are also interested in the \emph{local} Lipschitz constant of a neural network constrained in a small neighborhood of a fixed input $\bar{\mathbf{x}}_0$.
In this situation the input space $\Omega$ is often the ball around $\bar{\mathbf{x}}_0 \in \mathbb{R}^n$ with radius $\varepsilon$: $\Omega = \{\mathbf{x}: \|\mathbf{x}-\bar{\mathbf{x}}_0\| \le \varepsilon\}$.
In particular, with the $L_{\infty}$-norm  (and using $l \le x \le u
	\Leftrightarrow (x-l) (x-u) \le 0$), the input space $\Omega$ is the basic closed semialgebraic set:
\begin{align} \label{infnorm}
\Omega & = \{\mathbf{x} \in \R^n: (\mathbf{x}- \bar{\mathbf{x}}_0 + \varepsilon) (\mathbf{x} - \bar{\mathbf{x}}_0 - \varepsilon) \le 0 \}.
\end{align}
In view of Theorem \ref{th:lip} and \eqref{lip}, the \emph{Lipschitz constant estimation problem (LCEP)} for neural networks with respect to the norm $\|\cdot\|$ is the following \ac{POP}:
\begin{equation} \label{lcep}
\begin{cases}
\sup\limits_{\mathbf{x}_i,\mathbf{u}_i,\mathbf{t}}&\mathbf{t}^{\intercal} \bigg(\prod_{i = 1}^{m} \mathbf{A}_i^{\intercal} \diag (\mathbf{u}_i)\bigg) \mathbf{c} \\
\,\rm{ s.t.}&\mathbf{u}_i (\mathbf{u}_i - 1) = 0, (\mathbf{u}_i - 1/2) (\mathbf{A}_{i} \mathbf{x}_{i-1} + \mathbf{b}_i) \ge 0, i \in [m]\\
&\mathbf{x}_{i-1} (\mathbf{x}_{i-1} - \mathbf{A}_{i-1} \mathbf{x}_{i-2} - \mathbf{b}_{i-1}) = 0\\
&\mathbf{x}_{i-1} \ge 0, \mathbf{x}_{i-1} \ge \mathbf{A}_{i-1} \mathbf{x}_{i-2} + \mathbf{b}_{i-1}, i=2,\ldots,m\\
&\mathbf{t}^2 \le 1, (\mathbf{x}_0- \bar{\mathbf{x}}_0 + \varepsilon) (\mathbf{x}_0 - \bar{\mathbf{x}}_0 - \varepsilon) \le 0
\end{cases}
\end{equation}
In \citelip{latorre2020lipschitz} the authors only use the constraint $0 \le \mathbf{u}_i \le 1$ on the variables $\mathbf{u}_i$, only capturing the Lipschitz character of the considered activation function. We could use the same constraints, and this would allow to use activations which do not have semi-algebraic representations such as the Exponential Linear Unit (ELU). However, such a relaxation, despite very general, is a lot coarser than the one we propose. Indeed, \eqref{lcep} treats an \emph{exact formulation} of the generalized derivative of the ReLU function by exploiting its semialgebraic character.

\section{Nearly sparse problems}
\label{sec:hr}
For illustration purpose, consider 1-hidden layer networks. Then in \eqref{lcep} we can define natural subsets $I_i = \{u_1^{(i)}, \mathbf{x}_0\}$, $i \in [p_1]$ (with resepct to constraints $\mathbf{u}_1 (\mathbf{u}_1 - 1) = 0$, $(\mathbf{u}_1 - 1/2) (\mathbf{A}_1 \mathbf{x}_0 + \mathbf{b}_1) \ge 0$, and $(\mathbf{x}_0 - \bar{\mathbf{x}}_0 + \varepsilon) (\mathbf{x}_0 - \bar{\mathbf{x}}_0 - \varepsilon) \le 0$); and $J_j = \{t^{(j)}\}$, $j \in [p_0]$ (with resepct to constraints $\mathbf{t}^2 \le 1$). Clearly, $I_i, J_j$ satisfy the \ac{RIP} condition \eqref{eq:RIP} and are subsets with smallest possible size. Recall that $\mathbf{x}_0 \in \mathbb{R}^{p_{0}}$.
Hence $|I_i| = 1+p_{0}$ and the maximum size of the \ac{PSD} matrices involved in the sparse Lasserre's hierarchy $\P_{\cs}^r$, given in \eqref{eq:csmom}, is $\binom{1+p_{0} + r}{r}$.
Therefore, as in real deep neural networks $p_0$ can be as large as 1000, the second-order sparse Lasserre's hierarchy $\P_{\cs}^2$,  cannot be implemented in practice.

In fact \eqref{lcep} can be considered as a ``nearly sparse'' \ac{POP}, i.e., a sparse \ac{POP} with some additional ``bad" constraints that violate the sparsity assumptions.
More precisely, suppose that $f, g_i$ and subsets $I_k$ satisfy Assumption \ref{hyp:cs}.
Let $g$ be a polynomial that violates Assumption \ref{hyp:cs} (iii), i.e., \ac{RIP} \eqref{eq:RIP}.
Then we call the \ac{POP}
	\begin{align} \label{nearly_sparse}
	& \inf\limits_{\mathbf{x} \in \mathbb{R}^n} \{f(\mathbf{x}): g(\mathbf{x}) \ge 0, g_i(\mathbf{x}) \ge 0, i\in [m]\}, \tag{Nearly}
	\end{align}
	a \emph{nearly sparse} \ac{POP} because only one constraint, namely $g\geq0$, does not satisfy the sparsity pattern corresponding to the \ac{RIP} \eqref{eq:RIP}.
This single ``bad" constraint $g\geq0$ precludes us from applying the sparse Lasserre hierarchy  \eqref{eq:csmom}.

In this situation, we propose a heuristic method which can be applied to problems with arbitrary many constraints that possibly destroy the sparsity pattern.
The key idea of our algorithm is: \textbf{(i)} Keep the ``nice" sparsity pattern defined without the bad constraints; \textbf{(ii)} Associate only low-order localizing matrix constraints to the ``bad'' constraints. In brief, the $r$-th order \emph{heuristic hierarchy} (\emph{HR-$r$}) reads as
\begin{equation}\label{mom_nearly_sparse}
\begin{cases}
\inf\limits_{\mathbf{y}}&L_{\mathbf{y}} (f)\\
\rm{ s.t.}&\mathbf{M}_1 (\mathbf{y}) \succeq 0\\
&\mathbf{M}_r (\mathbf{y}, I_k) \succeq 0, \quad k \in [l]\\ &\mathbf{M}_{r-d_i} (g_i \,\mathbf{y}, I_{k(i)}) \succeq 0,\quad i \in [m]\\
&L_{\mathbf{y}} (g) \ge 0, \quad y_{\mathbf{0}} = 1
\end{cases}\tag{HR-$r$}
\end{equation}
%
We already have a sparsity pattern with subsets $I_k$ and an additional ``bad'' constraint $g\geq0$ (assumed to be quadratic). Then
	we consider the sparse moment relaxations \eqref{eq:csmom} applied to \eqref{nearly_sparse}
	\emph{without} the bad constraint $g\geq0$ and simply add two constraints: (i) the moment constraint $\mathbf{M}_1(\mathbf{y})\succeq0$ (with full dense first-order moment matrix $\mathbf{M}_1(\mathbf{y})$),
	and (ii) the linear moment inequality constraint $L_{\mathbf{y}}(g)\geq0$ (which is the lowest-order localizing matrix constraint  $\mathbf{M}_0(g\,\mathbf{y})\succeq0$).

	To see why the full moment constraint $\mathbf{M}_1 (\mathbf{y})\succeq0$ is needed, consider the following toy problem:
\begin{align} \label{eg_sparse_opt}
	\inf_{\mathbf{x} \in \mathbb{R}^3} \{x_1 x_2 + x_2 x_3: x_1^2 + x_2^2 \le 1, x_2^2 + x_3^2 \le 1\}.
\end{align}
	Define the subsets $I_1 = \{1,2\}$, $I_2 = \{2,3\}$. It is easy to check that Assumption \ref{hyp:cs} holds.
	Define $$\mathbf{y} = \{y_{000}, y_{100}, y_{010}, y_{001}, y_{200}, y_{110}, y_{101}, y_{020}, y_{011}, y_{002}\} \in \mathbb{R}^{10}.$$ For $r = 1$, the first-order dense moment matrix reads as
	\begin{align*}
	& \mathbf{M}_1 (\mathbf{y}) = \begin{bmatrix}
	\textcolor[rgb]{1,0,1}{y_{000}} & \textcolor[rgb]{1,0,0}{y_{100}} & \textcolor[rgb]{1,0,1}{y_{010}} & \textcolor[rgb]{0,0,1}{y_{001}} \\
	\textcolor[rgb]{1,0,0}{y_{100}} & \textcolor[rgb]{1,0,0}{y_{200}} & \textcolor[rgb]{1,0,0}{y_{110}} & y_{101} \\
	\textcolor[rgb]{1,0,1}{y_{010}} & \textcolor[rgb]{1,0,0}{y_{110}} & \textcolor[rgb]{1,0,1}{y_{020}} & \textcolor[rgb]{0,0,1}{y_{011}} \\
	\textcolor[rgb]{0,0,1}{y_{001}} & y_{101} & \textcolor[rgb]{0,0,1}{y_{011}} & \textcolor[rgb]{0,0,1}{y_{002}}
	\end{bmatrix},
	\end{align*}
	whereas the sparse moment matrix $\mathbf{M}_1(\mathbf{y}, I_1)$ (resp. $\mathbf{M}_1(\mathbf{y}, I_2)$) is the submatrix of $\mathbf{M}_1 (\mathbf{y})$ taking red and pink (resp. blue and pink) entries. That is, $\mathbf{M}_1(\mathbf{y}, I_1)$ and $\mathbf{M}_1(\mathbf{y},I_2)$ are submatrices of $\mathbf{M}_1(\mathbf{y})$, obtained by restricting to rows and  columns concerned with subsets $I_1$ and $I_2$ only.

Now suppose that we need to  consider an additional ``bad'' constraint $(1-x_1 - x_2 -x_3)^2 = 0$. After developing $L_{\mathbf{y}}(g)$, one needs to consider the moment variable $y_{101}$ corresponding to the monomial $x_1x_3$ in the expansion of $g=(1-x_1-x_2-x_3)^2$, and $y_{101}$ does \emph{not} appear in the moment matrices $\mathbf{M}_1(\mathbf{y,}\,I_1)$ and $\mathbf{M}_1(\mathbf{y},\,I_2)$
	because $x_1$ and $x_3$ are not in the same subset. However $y_{101}$ appears in
	$\mathbf{M}_1(\mathbf{y})$, which is of size $n+1$.

	Now let us see how this works for problem \eqref{lcep}. First introduce new variables $\mathbf{z}_i$ with associated constraints $\mathbf{z}_i-\mathbf{A}_i \mathbf{x}_{i-1} - \mathbf{b}_i=0$, so that all ``bad'' constraints are affine. Equivalently, we may and will consider the single ``bad" constraint $g\geq 0$ with $g(\mathbf{z}_1,\ldots,\mathbf{x}_0,\mathbf{x}_1,\ldots)=-\sum_{i}\|\mathbf{z}_i-\mathbf{A}\mathbf{x}_{i-1}-\mathbf{b}_i\|^2$
	and solve \eqref{mom_nearly_sparse}.
	We briefly sketch the rationale behind this reformulation.
	Let $(\mathbf{y}^r)_{r\in\mathbb{N}}$
	be a sequence of optimal solutions of \eqref{mom_nearly_sparse}. If $r\to\infty$, then $\mathbf{y}^r\to\mathbf{y}$ (possibly for a subsequence $(r_k)_{k\in\mathbb{N}}$), and $\mathbf{y}$ corresponds to the moment sequence of a measure
	$\mu$, supported on $\{(\mathbf{x},\mathbf{z}): g_i(\mathbf{x},\mathbf{z})\geq0,\:i\in [p];\:\int g\,\mathrm{d}\mu\geq0\}$. But as $-g$ is a square, $\int g\,\mathrm{d}\mu\geq0$ implies $g=0$, $\mu$-a.e., and therefore
	$\mathbf{z}_i=\mathbf{A}\mathbf{x}_{i-1}+\mathbf{b}_i$, $\mu$-a.e.. This is why we do not need to consider the higher-order  constraints $\mathbf{M}_r(g\,\mathbf{y})\succeq0$ for $r>0$; only $\mathbf{M}_0(g\,\mathbf{y})\succeq0$ ($\Leftrightarrow L_{\mathbf{y}}(g)\geq0$) suffices. In fact, we impose the stronger linear constraints $L_{\mathbf{y}}(g)=0$ and $L_{\mathbf{y}}(\mathbf{z}_i-\mathbf{A}\mathbf{x}_{i-1}-\mathbf{b}_i)=0$ for all $i\in [p]$.

	For simplicity, assume that the neural networks have only one single hidden layer, i.e., $m = 1$. Denote by $\mathbf{A}, \mathbf{b}$ the weight and bias respectively. As in \eqref{infnorm}, we use the fact that $l \le x \le u$ is equivalent to $(x-l) (x-u) \le 0$. Then the local Lipschitz constant estimation problem
	with respect to $L_{\infty}$-norm can be written as
	\begin{equation}\label{rlcep}
		\begin{cases}
			\sup\limits_{\mathbf{x},\mathbf{u},\mathbf{z},\mathbf{t}} &\mathbf{t}^{\intercal} \mathbf{A}^{\intercal} \diag (\mathbf{u}) \mathbf{c}\\
			\,\,\rm{ s.t.} &(\mathbf{z} - \mathbf{A} \mathbf{x} - \mathbf{b})^2 = 0, \quad\mathbf{t}^2 \le 1\\
			&(\mathbf{x}- \bar{\mathbf{x}}_0 + \varepsilon) (\mathbf{x} - \bar{\mathbf{x}}_0 - \varepsilon) \le 0\\
			&\mathbf{u} (\mathbf{u} - 1) = 0, \quad (\mathbf{u} - 1/2) \mathbf{z} \ge 0
		\end{cases}\tag{LCEP$_1$}
	\end{equation}

Define the subsets of \eqref{rlcep} to be $I^i = \{x^i, t^i\}$, $J^j = \{u^j, z^j\}$ for $i \in [p_0]$, $j \in [p_1]$, where $p_0$, $p_1$ are the number of nodes in the input layer and the hidden layer, respectively.
Then the second-order ($r=2$) heuristic relaxation of \eqref{rlcep} is the following \ac{SDP}:
\begin{equation}\label{mom_rlcep}
\begin{cases}
\inf\limits_{\mathbf{y}}&L_{\mathbf{y}} (\mathbf{t}^{\intercal} \mathbf{A}^{\intercal} \diag (\mathbf{u}) \mathbf{c})\\
\rm{ s.t.}& y_{\mathbf{0}} = 1, \quad \mathbf{M}_1 (\mathbf{y}) \succeq 0\\
&L_{\mathbf{y}} (\mathbf{z} - \mathbf{A} \mathbf{x} - \mathbf{b}) = 0, \quad L_{\mathbf{y}} ((\mathbf{z} - \mathbf{A} \mathbf{x} - \mathbf{b})^2) = 0\\
&\mathbf{M}_{1} (-(x^{(i)}- \bar{x}^{(i)}_0 + \varepsilon) (x^{(i)} - \bar{x}^{(i)}_0 - \varepsilon) \mathbf{y}, I^i) \succeq 0,\quad i \in [p_0]\\
&\mathbf{M}_{1} ((1-t_i^2) \mathbf{y}, I^i) \succeq 0, \quad\mathbf{M}_2 (\mathbf{y}, I^i) \succeq 0,\quad i \in [p_0]\\
&\mathbf{M}_2 (\mathbf{y}, J^j) \succeq 0, \quad \mathbf{M}_{1} (u_j (u_j - 1) \mathbf{y}, J^j) = 0, \quad j \in [p_1]\\
&\mathbf{M}_{1} ((u_j - 1/2) z_j \mathbf{y}, J^j) \succeq 0, \quad j \in [p_1]
\end{cases}\tag{HR-2}
\end{equation}
	The $r$-th order heuristic relaxation \eqref{mom_nearly_sparse} also applies to multiple layer neural networks. However,
	if the neural network has $m$ hidden layers, then the criterion in \eqref{lcep} is of degree $m+1$.
	If $m \ge 2$, then the first-order moment matrix $\mathbf{M}_1 (\mathbf{y})$ is no longer sufficient, as
	moments of degree $>2$ are \emph{not} encoded in $\mathbf{M}_1(\mathbf{y})$ and some may not be encoded in the moment matrices $\mathbf{M}_2(\mathbf{y},\,I^i)$, if they include variables of different subsets.
See \citelip[Appendix E]{LipPOP20}  for more information to deal with higher-degree polynomial objective.

\section{Overview of numerical experiments}\label{experiments}

In this section, we provide results for the \emph{global} and \emph{local} Lipschitz constants of \emph{random} networks of fixed size $(80,80)$ and with various sparsities. We also compute bounds of a \emph{real} trained 1-hidden layer network. For all experiments we focus on the $L_{\infty}$-norm, the most interesting case
for robustness certification. Moreover, we use the Lipschitz constants computed by various methods to certify robustness of a trained network, and compare the ratio of certified inputs with different methods. Let us first provide an overview of the methods with which we compare our results:
\begin{itemize}
\item	\textbf{SHOR}: First-order dense moment relaxation (also called Shor's relaxation) applied to \eqref{lcep}.
\item	\textbf{HR-1/2}: first/second-order heuristic relaxation applied to \eqref{lcep}.
\item	\textbf{LP-3/4}: \ac{LP}-based method, called \textbf{LipOpt}, by \citelip{latorre2020lipschitz} with degree 3/4, which corresponds to \eqref{eq:sbsos} without \ac{SOS} multipliers.
\item	\textbf{LBS}: lower bound obtained by sampling 50000 random points and evaluating the dual norm of the gradient.
\end{itemize}
The reason why we list \textbf{LBS} here is because \textbf{LBS} is a valid lower bound on the Lipschitz constant.
Therefore all methods should provide a result not lower than \textbf{LBS}, a basic necessary condition of consistency.

As discussed earlier, if we want to estimate the global Lipschitz constant, we need the input space $\Omega$ to be the whole space. In consideration of numerical issues, we set $\Omega$ to be the ball of radius $10$ around the origin. For the local Lipschitz constant, we set by default the radius of the input ball as $\varepsilon = 0.1$. In both cases, we compute the Lipschitz constant with respect to the first label.  We use the (Python) code provided by \citelip{latorre2020lipschitz}\footnote{\url{https://openreview.net/forum?id=rJe4_xSFDB}.} to execute the experiments for \textbf{LipOpt} with the Gurobi solver. For \textbf{HR-2} and \textbf{SHOR}, we use the YALMIP toolbox (MATLAB) \citelip{YALMIP} with $\mosek$ as a backend to calculate the Lipschitz constants for \emph{random} networks. For \emph{trained} network, we implement our algorithm in Julia with the $\mosek$ optimizer to accelerate the computation.
In the following, running time is referred to the time taken by the \ac{LP}/\ac{SDP} solver (Gurobi/$\mosek$) and ``-'' means running out of memory during solving the \ac{LP}/\ac{SDP} model.

In \citelip{latorre2020lipschitz} a certain sparsity structure arising from a neural network was exploited. Consider a neural network $F$ with one single hidden layer, and 4 nodes in each layer. The network $F$ is said to have a sparsity $\omega=4$ if its weight matrix $\mathbf{A}$ is symmetric with diagonal blocks of size at most $2 \times 2$:
	\begin{align}
	\label{form}
	\begin{bmatrix}
	* & * & 0 & 0 \\
	* & * & * & 0 \\
	0 & * & * & * \\
	0 & 0 & * & *
	\end{bmatrix}.
	\end{align}
	Larger sparsity values refer to symmetric matrices with band structure of a given size.
	This sparsity structure \eqref{form} of the networks greatly influences the number of variables involved in the \ac{LP} program to solve in \citelip{latorre2020lipschitz}. This is in deep contrast with our method which does not require the weight matrix to be as in \eqref{form}. Hence
	when the network is fully-connected, our method is more efficient and provides tighter upper bounds.

\subsection{Lipschitz Constant Estimation}
\label{lip_exp}
\paragraph{Random networks.} Table \ref{table_bound_time} gives a brief comparison outlook of the results obtained by our method and the method in \citelip{latorre2020lipschitz}. For $(80,80)$ networks, apart from $\omega$ = 20, which is not significative, \textbf{HR-2} obtains much better bounds and is also much more efficient than \textbf{LP-3}. \textbf{LP-4} provides tighter bounds than \textbf{HR-2} but suffers more computational time, and run out of memory when the sparsity increases. For $(40,40,10)$ networks, \textbf{HR-1} is a trade-off between \textbf{LP-3} and \textbf{LP-4}, it provides tighter (resp. looser) bounds than \textbf{LP-3} (resp. \textbf{LP-4}), but takes more (resp. less) computational time.
\begin{table}[t]
\caption{Bounds of global Lipschitz constants (opt) and solver running time (in seconds) of networks of size $(80,80)$ and $(40,40,10)$.}\label{table_bound_time}
\renewcommand\arraystretch{1.2}
\centering
\resizebox{\linewidth}{!}{
\begin{tabular}{ccccccccccc}
						\toprule
						&& \multicolumn{4}{c}{$(80,80)$} && \multicolumn{4}{c}{$(40,40,10)$} \\
						\cline{3-6}\cline{8-11}
						&& $\omega$ = 20 & $\omega$ = 40 & $\omega$ = 60 & $\omega$ = 80 && $\omega$ = 20 & $\omega$ = 40 & $\omega$ = 60 & $\omega$ = 80\\
						\midrule
						\multirow{2}{*}{\textbf{HR-2}} &opt & 1.45 & 2.05 & 2.41 & 2.68 & \multirow{2}{*}{\textbf{HR-1}} & 0.50 & 1.16 & 1.82 & 2.05 \\
						\cline{2-6}\cline{8-11}
						& time & 3.14 & 7.78 & 8.61 & 9.82 && 271.34 & 165.68 & 174.86 & 174.02 \\
						\midrule
						\multirow{2}{*}{\textbf{LP-3}} & opt & 1.55 & 2.86 & 3.85 & 4.68 &\multirow{2}{*}{\textbf{LP-3}}& 0.56 & 1.68 & 3.01 & 3.57 \\
						\cline{2-6}\cline{8-11}
						& time & 2.44 & 10.36 & 20.99 & 71.49 && 3.84 & 4.83 & 7.91 & 6.33 \\
						\midrule
						\multirow{2}{*}{\textbf{LP-4}} & opt & 1.43 & - & - & - &\multirow{2}{*}{\textbf{LP-4}}& 0.29 & 0.85 & - & - \\
						\cline{2-6}\cline{8-11}
						& time & 127.99 & - & - & - && 321.89 & 28034.27 & - & - \\
						\midrule
						\textbf{LBS} & opt & 1.05 & 1.56 & 1.65 & 1.86 & \textbf{LBS} & 0.20 & 0.48 & 0.61 & 0.62 \\
						\bottomrule
\end{tabular}}
\end{table}

\paragraph{Trained Network}
	Here we use the MNIST classifier SDP-NN described in \citelip{raghunathan2018certified}\footnote{\url{https://worksheets.codalab.org/worksheets/0xa21e794020bb474d8804ec7bc0543f52/}}. The network is of size $(784, 500)$. In Table \ref{table_real}, we see that the \textbf{LP-3} algorithm runs out of memory when applied to the real network SDP-NN to compute the global Lipschitz bound. In contrast, \textbf{SHOR} and \textbf{HR-2} still work and moreover,
	\textbf{HR-2} provides tighter upper bounds than \textbf{SHOR} in both global and local cases. As a trade-off, the running time of \textbf{HR-2} is around 5 times longer than that of \textbf{SHOR}.

	\begin{table}[t]
		\caption{Comparison of bounds of global Lipschitz constants and solver running time on trained network SDP-NN obtained by \textbf{HR-2}, \textbf{SHOR}, \textbf{LP-3} and \textbf{LBS}. The network is a fully connected neural network with one hidden layer, with 784 nodes in the input layer and 500 nodes in the hidden layer. The network is for 10-classification, and we calculate the upper bound with respect to label 2.}
		\label{table_real}
		\centering

					\resizebox{\linewidth}{!}{
					\begin{tabular}{cccccccccc}
						\toprule
						& \multicolumn{4}{c}{Global} && \multicolumn{4}{c}{Local} \\
						\midrule
						& \textbf{HR-2} & \textbf{SHOR} & \textbf{LP-3} & \textbf{LBS} && \textbf{HR-2} & \textbf{SHOR} & \textbf{LP-3} & \textbf{LBS} \\
						\midrule
						opt & 14.56 & 17.85 & - & 9.69 && 12.70 & 16.07 & - & 8.20 \\ 
						\hline
						time & 12246 & 2869 & - &  && 20596 & 4217 & - &  \\
						\bottomrule
					\end{tabular}}

	\end{table}

\subsection{Robustness certification}\label{cert}

\paragraph{Multi-Classifier} The above SDP-NN network is a well-trained $(784,500)$ network to classify the digit images from 0 to 9. Denote the parameters of this network by $\mathbf{A} \in \mathbb{R}^{500 \times 784}, \mathbf{b}_1 \in \mathbb{R}^{500}, \mathbf{C} \in \mathbb{R}^{10 \times 500}, \mathbf{b}_2 \in \mathbb{R}^{10}$.
The score of an input $\mathbf{x}$ is denoted by $\mathbf{y}^{\mathbf{x}}$, i.e., $\mathbf{y}^{\mathbf{x}} = \mathbf{C} \cdot \text{ReLU} (\mathbf{A} \mathbf{x}_0 + \mathbf{b}_1) + \mathbf{b}_2$.
The label of $\mathbf{x}$, denoted by $r^{\mathbf{x}}$, is the index with the largest score, i.e., $r^{\mathbf{x}} = \text{arg}\max \mathbf{y}^{\mathbf{x}}$.
Suppose an input $\mathbf{x}_0$ has label $r$. For $\epsilon$ and $\mathbf{x}$ such that $\|\mathbf{x} - \mathbf{x}_0\|_{\infty} \le \epsilon$, if for all $i \ne r$, $y_i^{\mathbf{x}} - y_r^{\mathbf{x}} < 0$, then $\mathbf{x}_0$ is $\epsilon$-robust.
Alternatively, denote by $L_{i,r}^{\mathbf{x}_0, \epsilon}$ the local Lipschitz constant of function $f_{i,r} (\mathbf{x}) = y_i^{\mathbf{x}} - y_r^{\mathbf{x}}$ with respect to $L_{\infty}$-norm in the ball $\{\mathbf{x}: \|\mathbf{x} - \mathbf{x}_0\|_{\infty} \le \epsilon\}$.
Then the point $\mathbf{x}_0$ is $\epsilon$-robust if for all $i \ne r$, $f_{i,r} (\mathbf{x}_0) + \epsilon L_{i,r}^{\mathbf{x}_0, \epsilon} < 0$. Since the $28\times 28$ MNIST images are flattened and normalized into vectors taking value in $[0,1]$, we compute the local Lipschitz constant (by \textbf{HR-2}) with respect to $\mathbf{x}_0 = \mathbf{0}$ and $\epsilon = 2$, the complete value is referred to the following matrix:
$$
	\mathbf{L} = \begin{bmatrix}
	* & 7.94 & 7.89 & 8.28 & 8.64 & 8.10 & 7.66 & 8.04 & 7.46 & 8.14 \\
	7.94 & * & 7.74 & 7.36 & 7.68 & 8.81 & 8.06 & 7.55 & 7.36 & 8.66\\

	7.89 & 7.74 & * & 7.63 & 8.81 & 10.23 & 8.18 & 8.13 & 7.74 & 9.08\\

	8.28 & 7.36 & 7.63 & * & 8.52 & 7.74 & 9.47 & 8.01 & 7.37 & 7.96\\

	8.64 & 7.68 & 8.81 & 8.52 & * & 9.44 & 7.98 & 8.65 & 8.49 & 7.47\\

	8.10 & 8.81 & 10.23 & 7.74 & 9.44 & * & 8.26 & 9.26 & 8.17 & 8.55 \\

	7.66 & 8.06 & 8.18 & 9.47 & 7.98 & 8.26 & * & 10.18 & 8.00 & 9.83\\

	8.04 & 7.55 & 8.13 & 8.01 & 8.65 & 9.26 & 10.18 & * & 8.28 & 7.65\\

	7.46 & 7.36 & 7.74 & 7.37 & 8.49 & 8.17 & 8.00 & 8.28 & * & 7.87 \\

	8.14 & 8.66 & 9.08 & 7.96 & 7.47 & 8.55 & 9.83 & 7.65 & 7.87 & *
	\end{bmatrix}$$
	where $\mathbf{L} = (L_{ij})_{i\ne j}$. Note that if we replace the vector $\mathbf{c}$ in \eqref{lcep} by $-\mathbf{c}$, the problem is equivalent to the original one. Therefore, the matrix $\mathbf{L}$ is symmetric, and we only need to compute 45 Lipschitz constants (the upper triangle of $\mathbf{L}$).

	Figure \ref{cert_eg} shows several certified and non-certified examples taken from the MNIST test dataset.
	\begin{figure}[H]
		\centering
		\captionsetup[subfigure]{labelformat=empty}
		\begin{subfigure}[t]{\textwidth}
			\centering
			\includegraphics[width=0.09\textwidth]{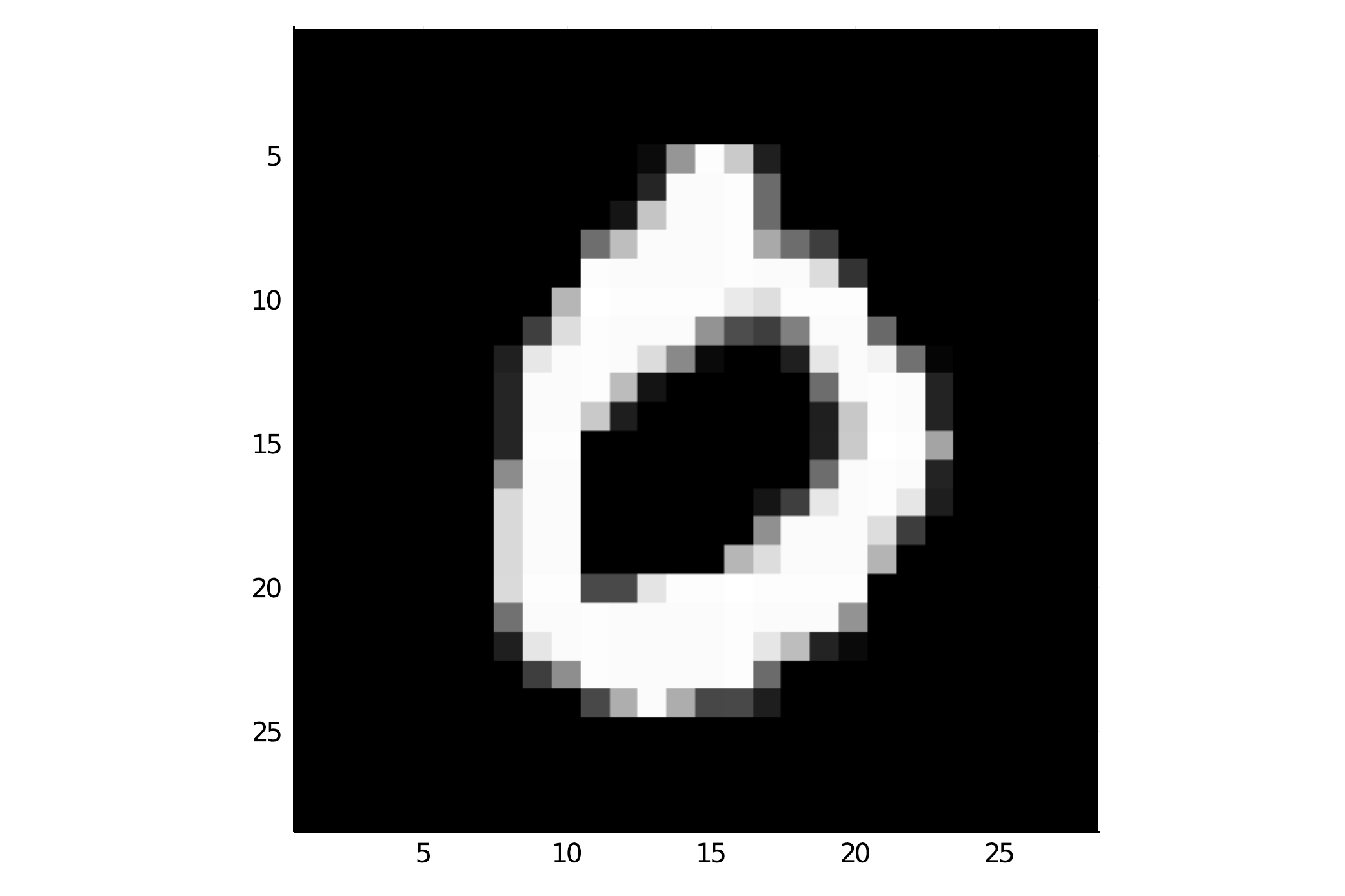}
			\includegraphics[width=0.09\textwidth]{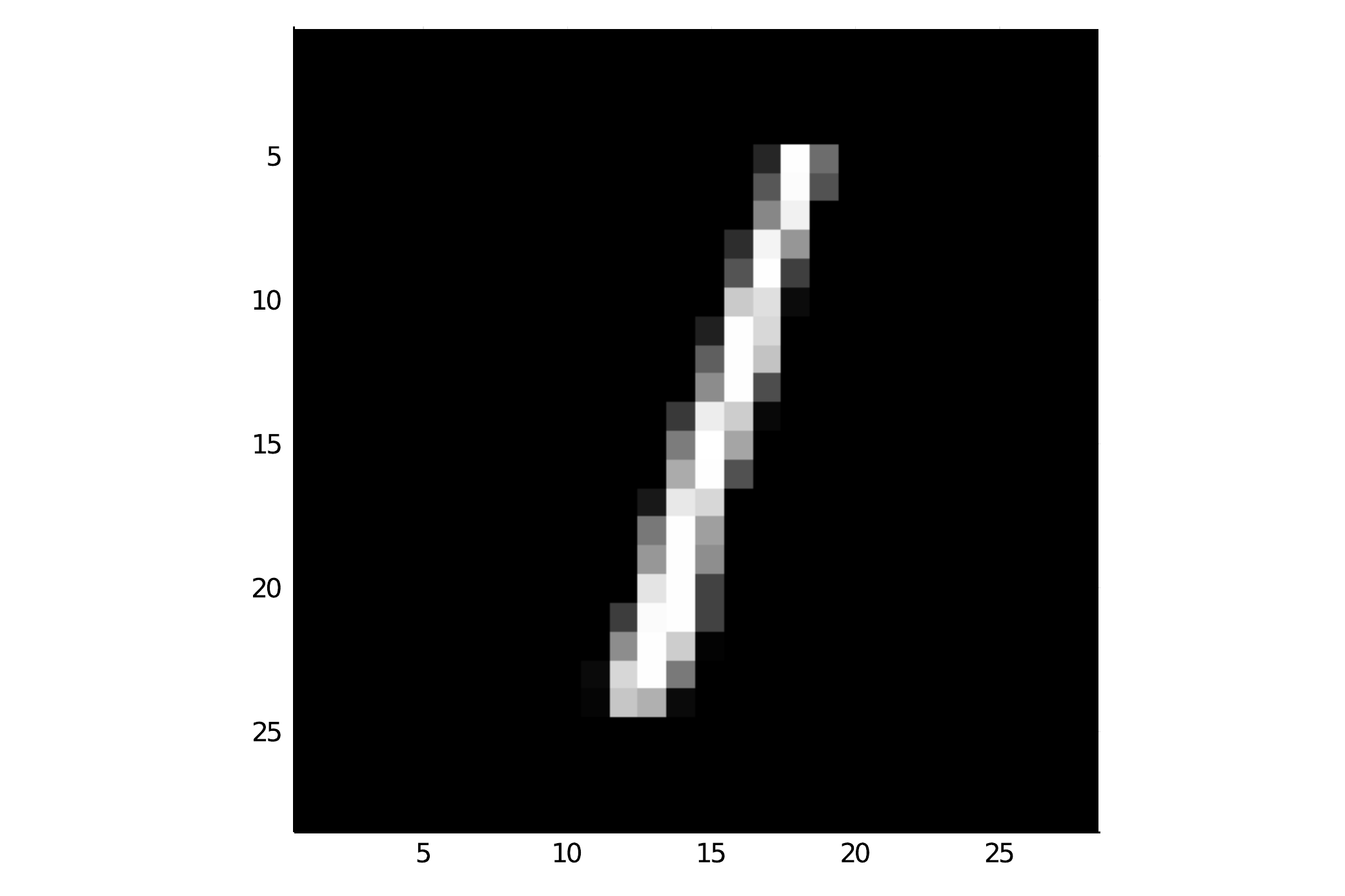}
			\includegraphics[width=0.09\textwidth]{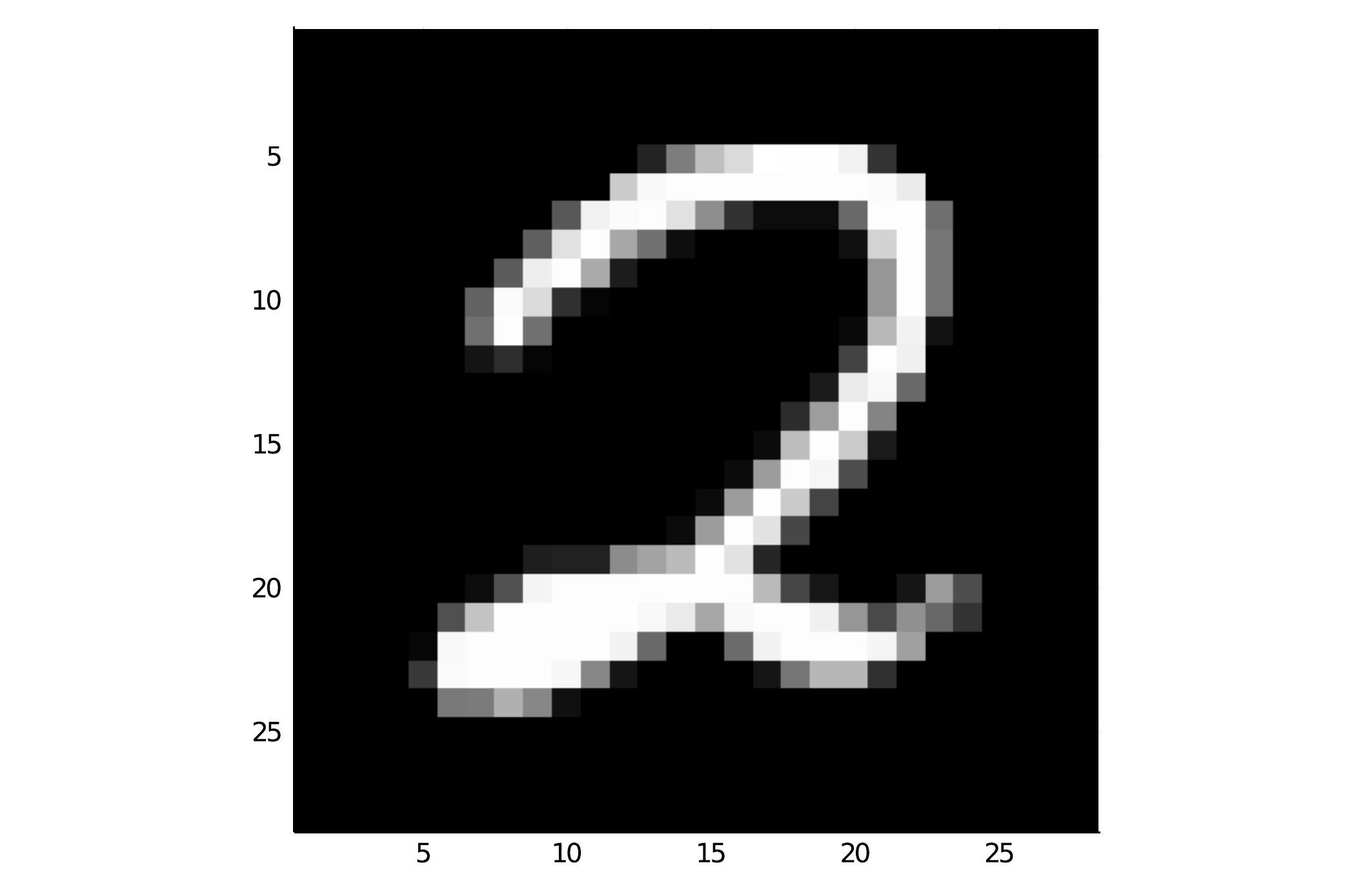}
			\includegraphics[width=0.09\textwidth]{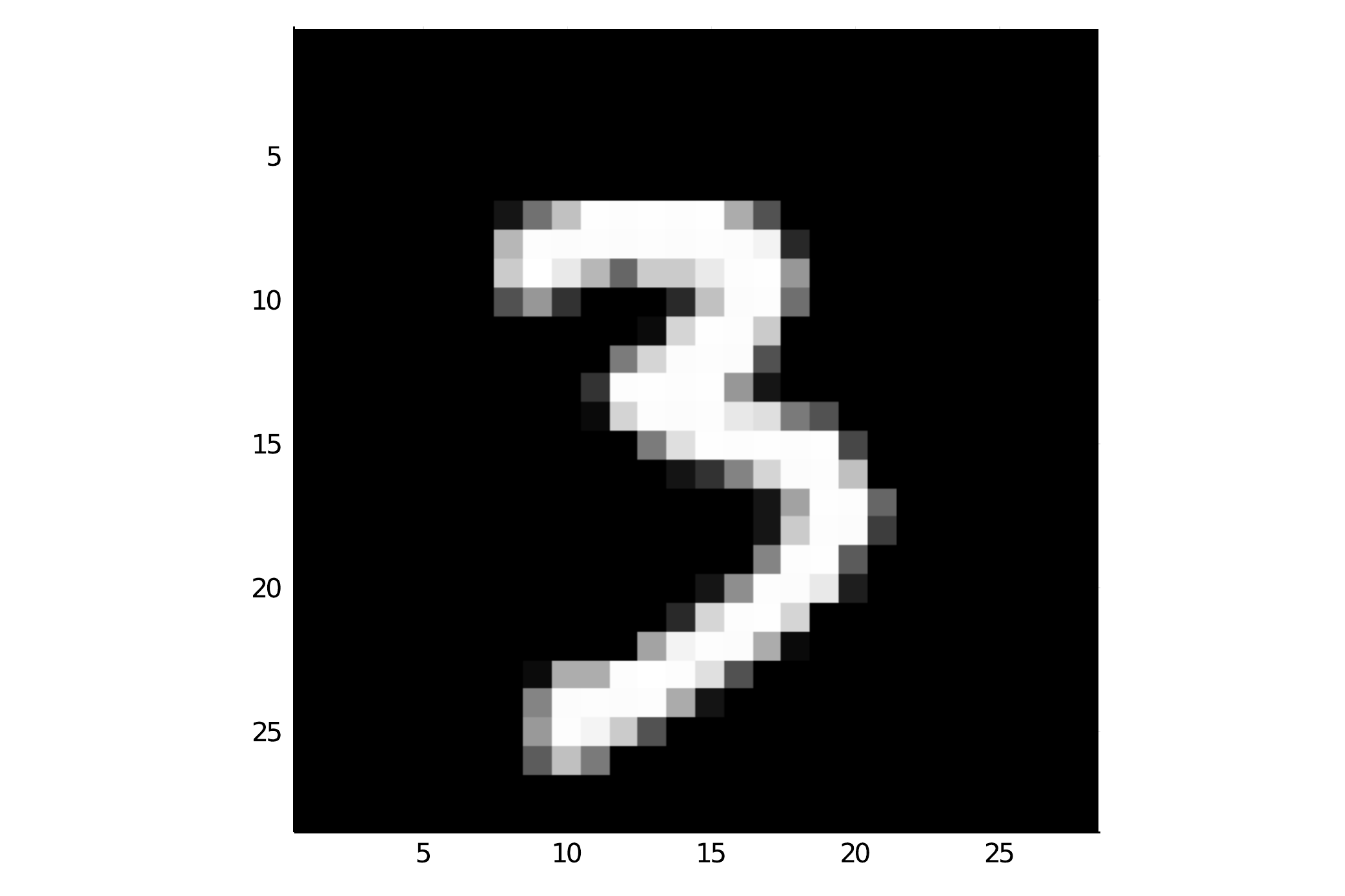}
			\includegraphics[width=0.09\textwidth]{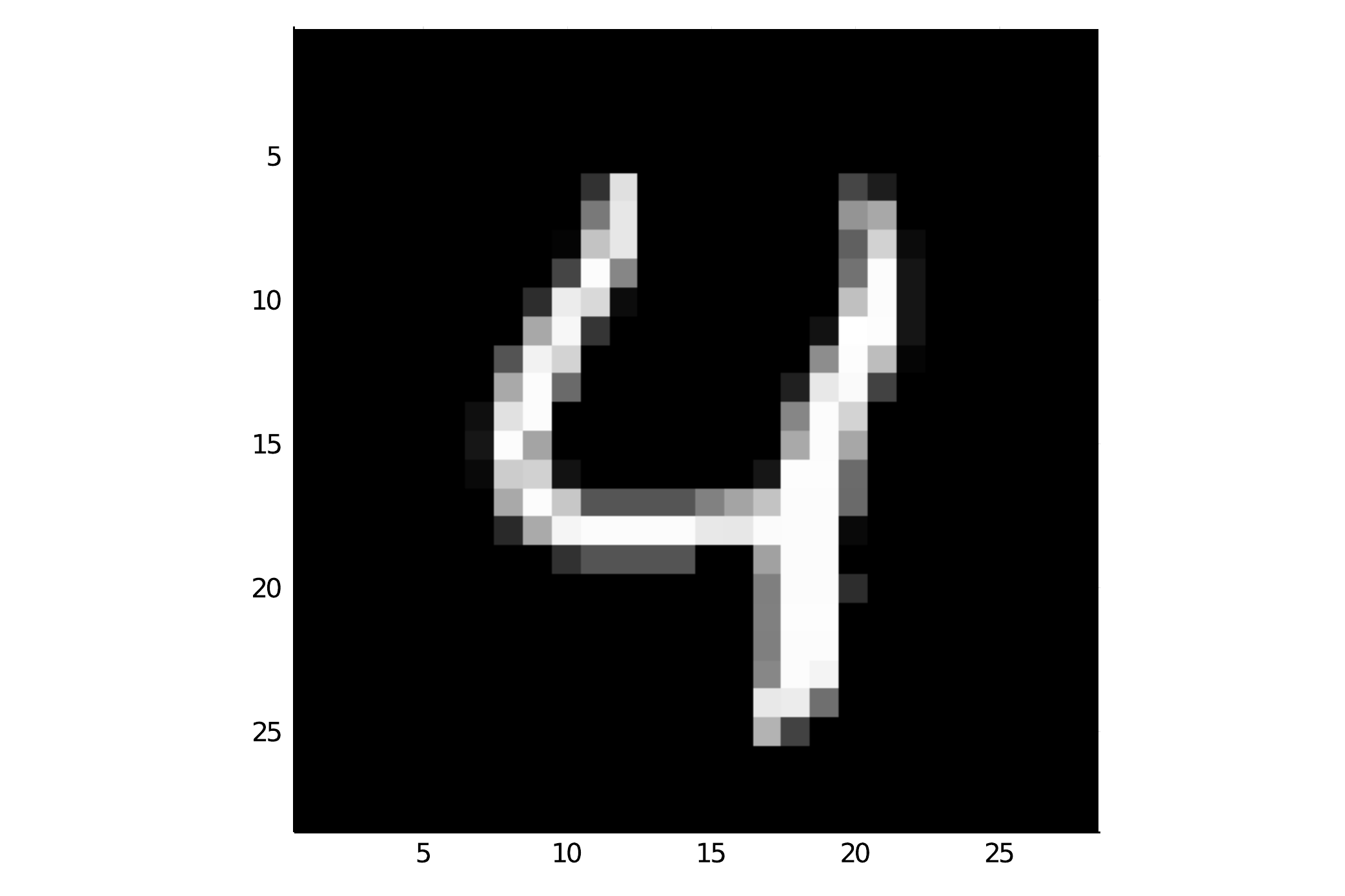}
			\includegraphics[width=0.09\textwidth]{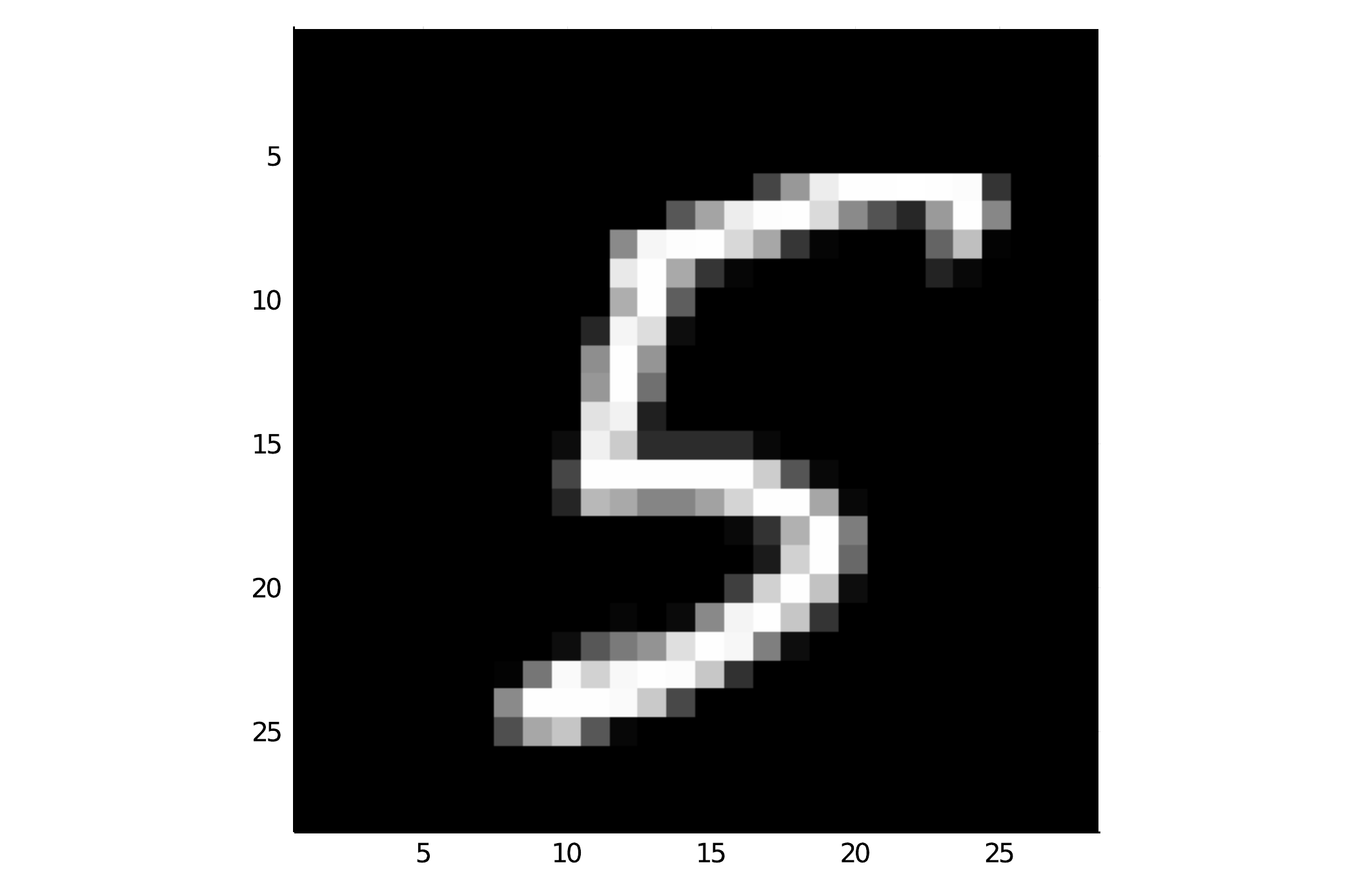}
			\includegraphics[width=0.09\textwidth]{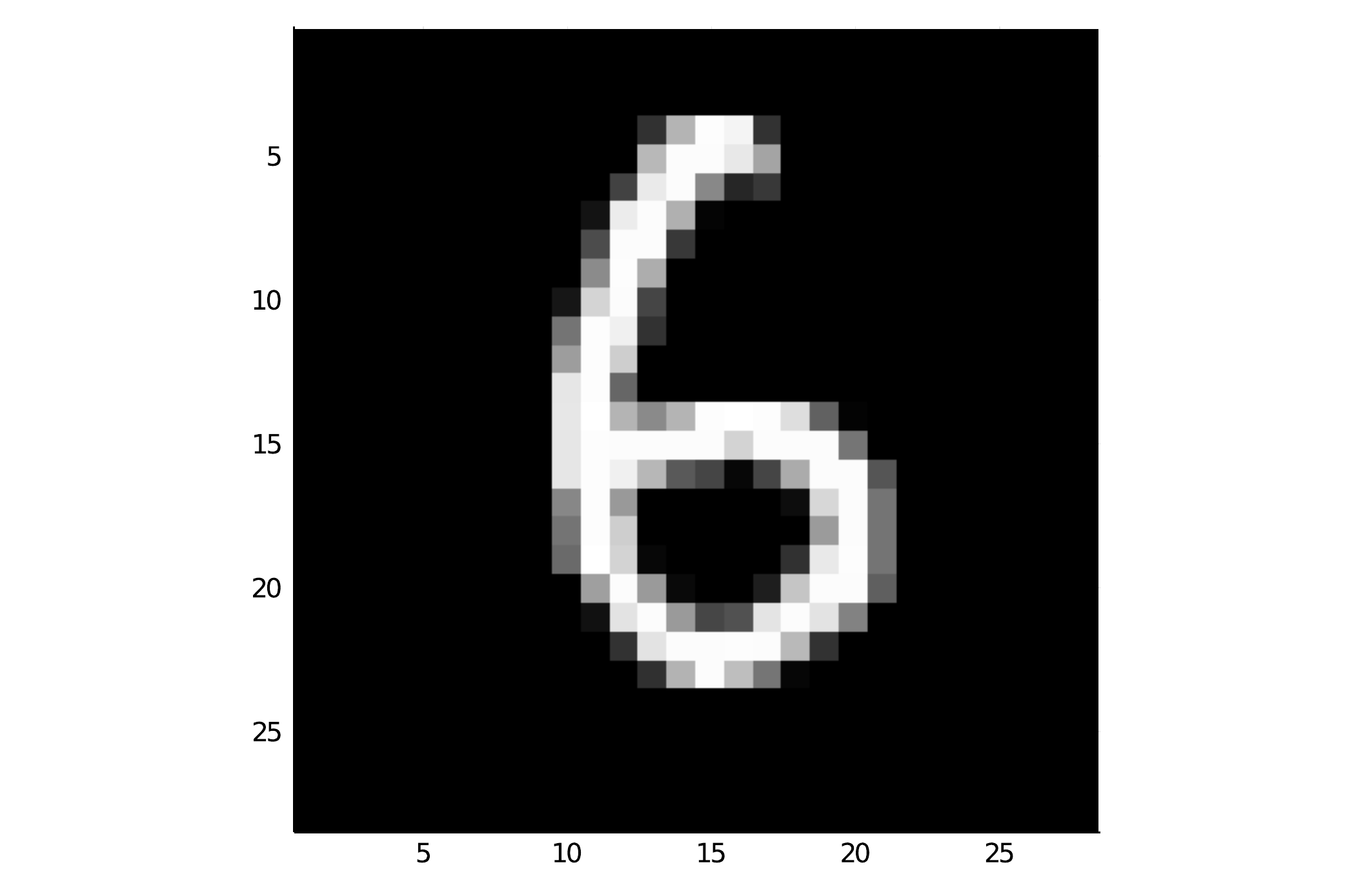}
			\includegraphics[width=0.09\textwidth]{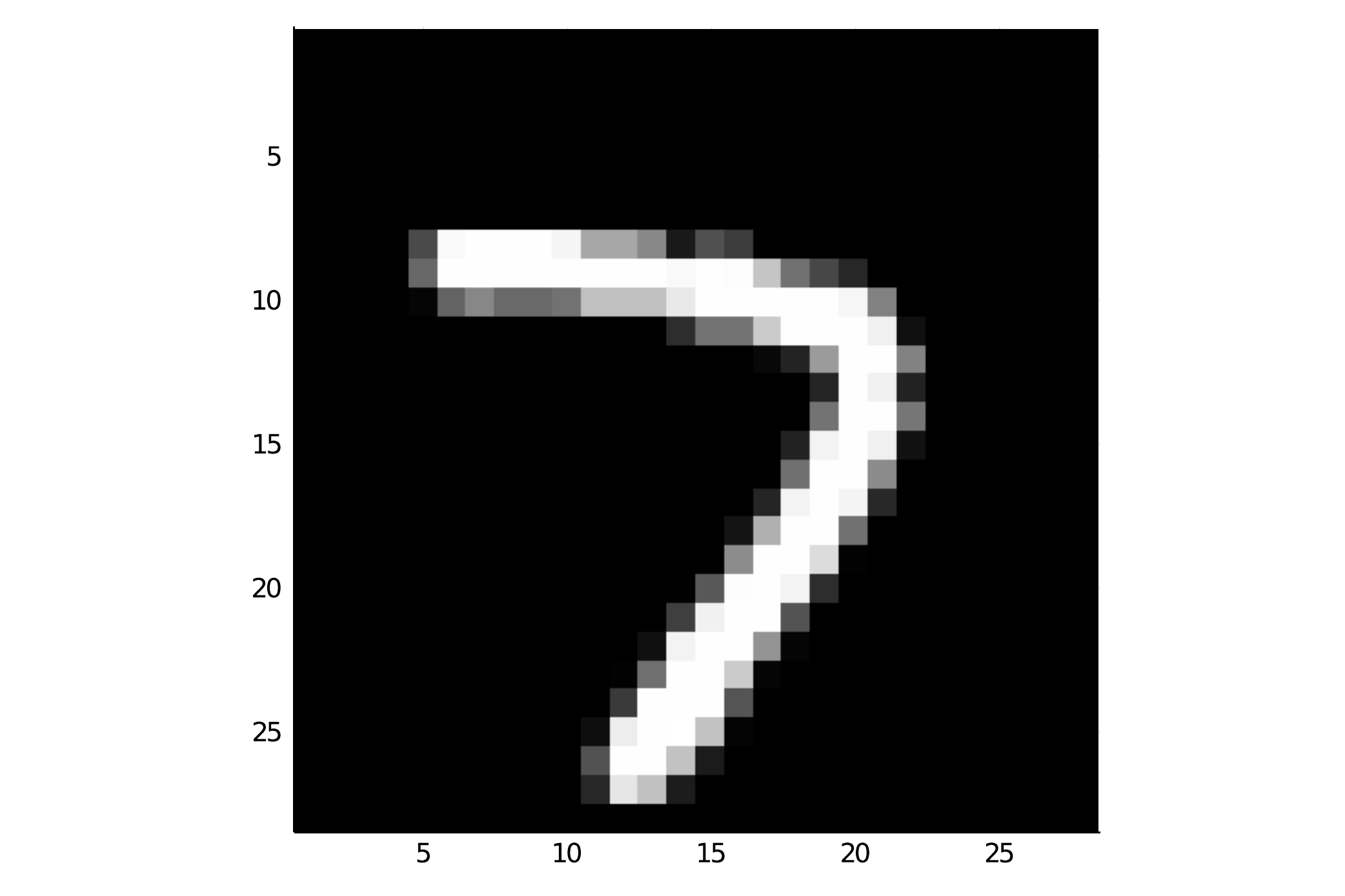}
			\includegraphics[width=0.09\textwidth]{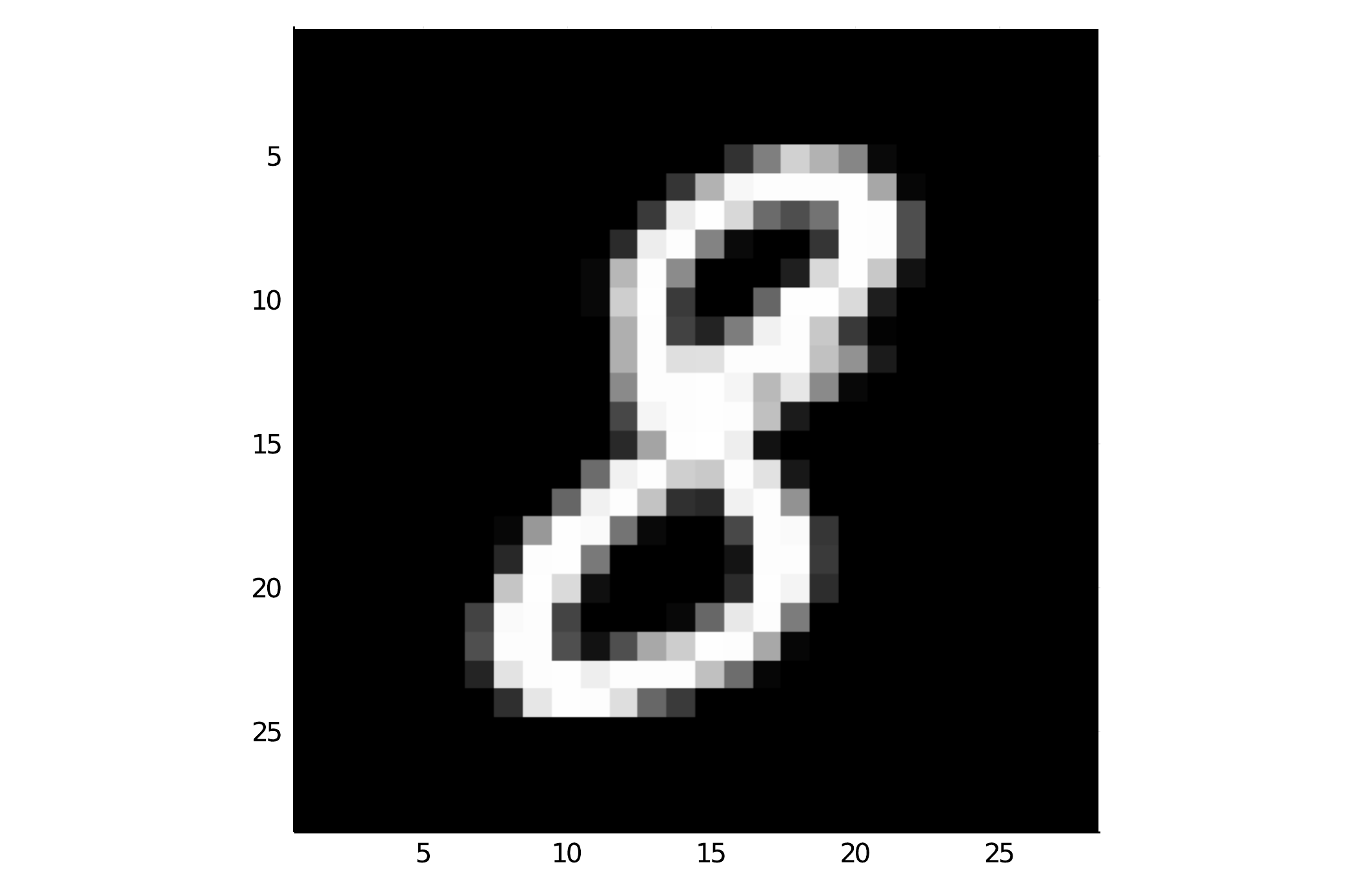}
			\includegraphics[width=0.09\textwidth]{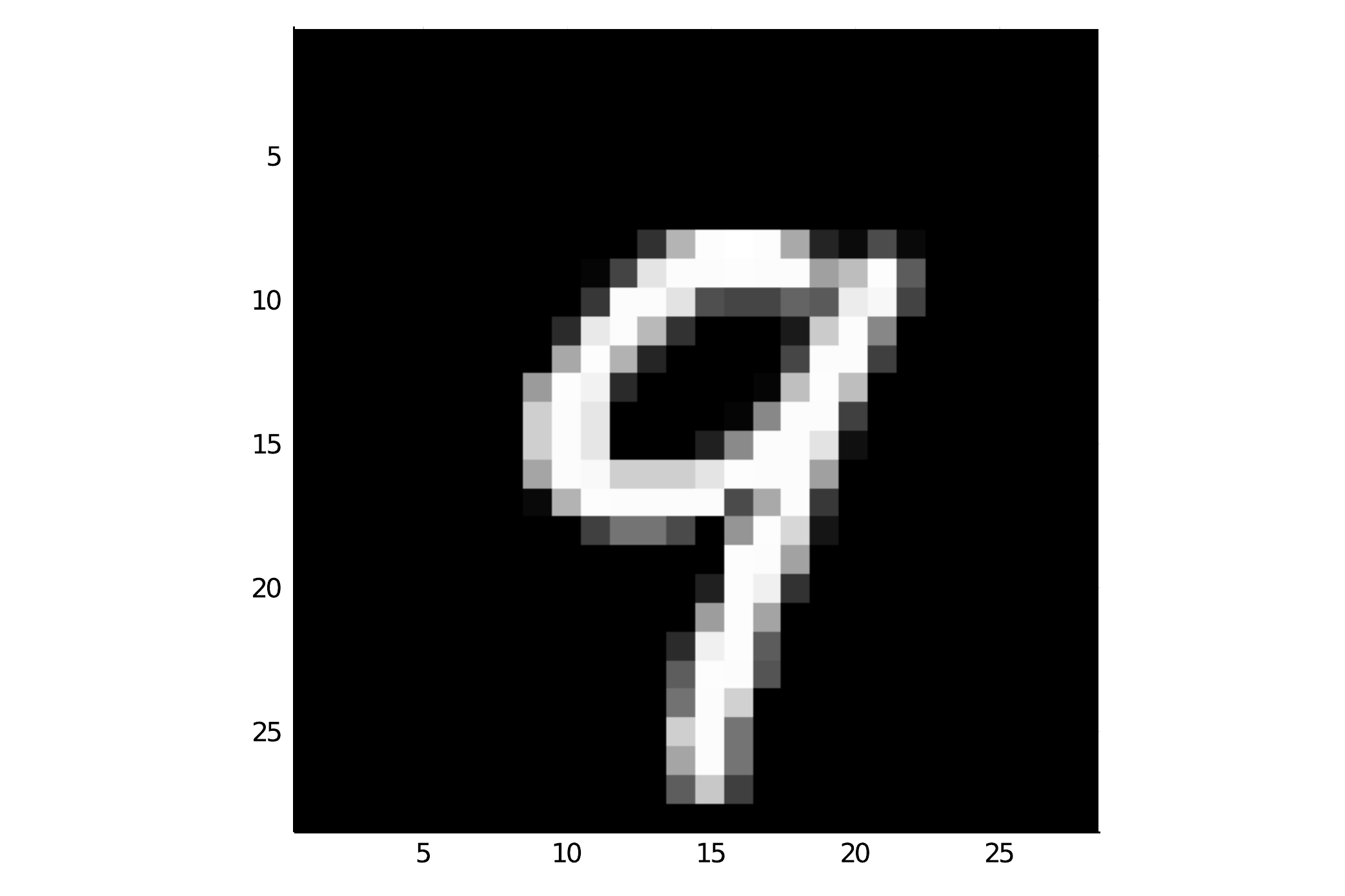}
			\caption{}
			\label{certified}
		\end{subfigure}
		\hfill
		\begin{subfigure}[t]{\textwidth}
			\centering
			\includegraphics[width=0.09\textwidth]{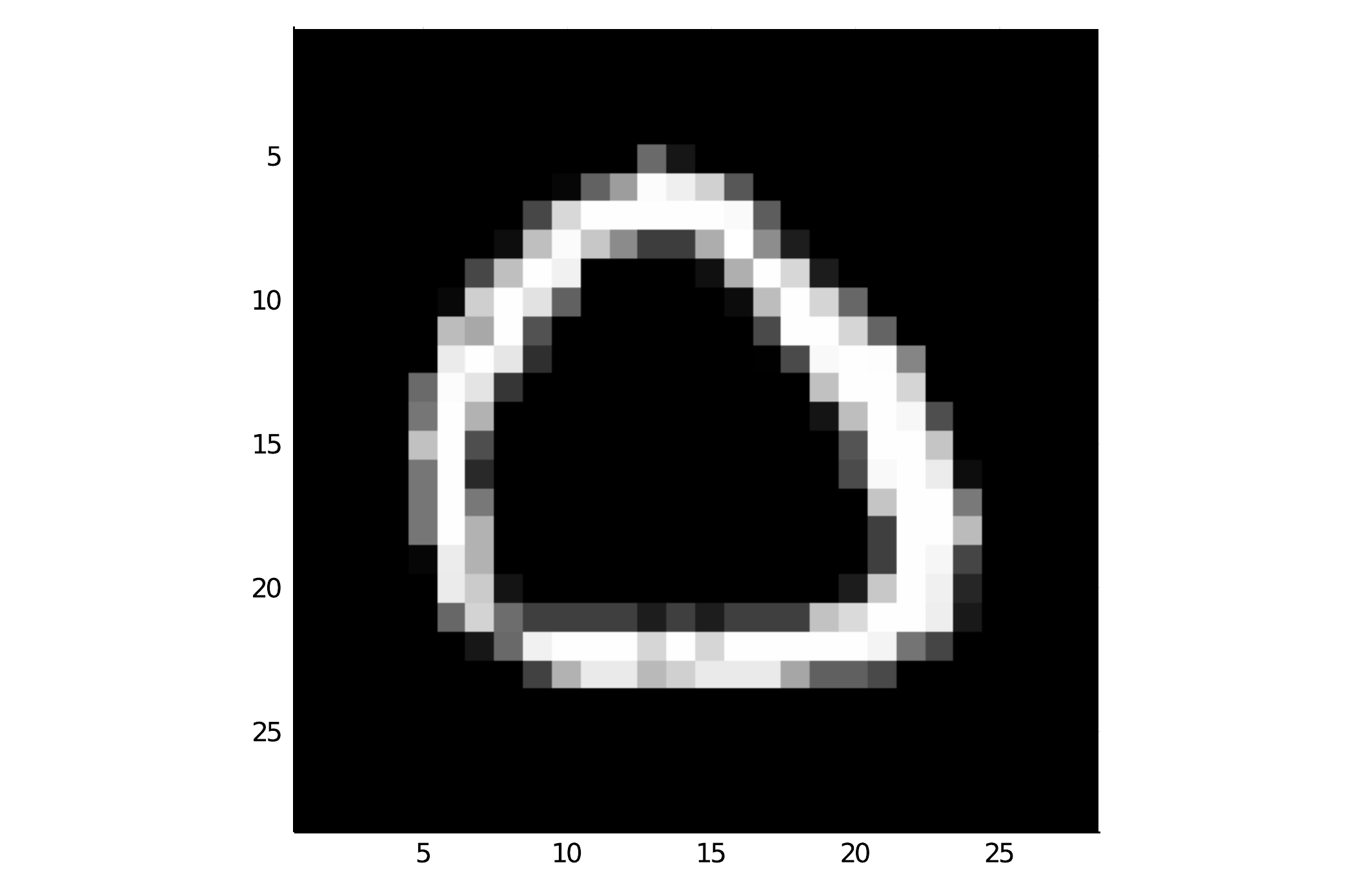}
			\includegraphics[width=0.09\textwidth]{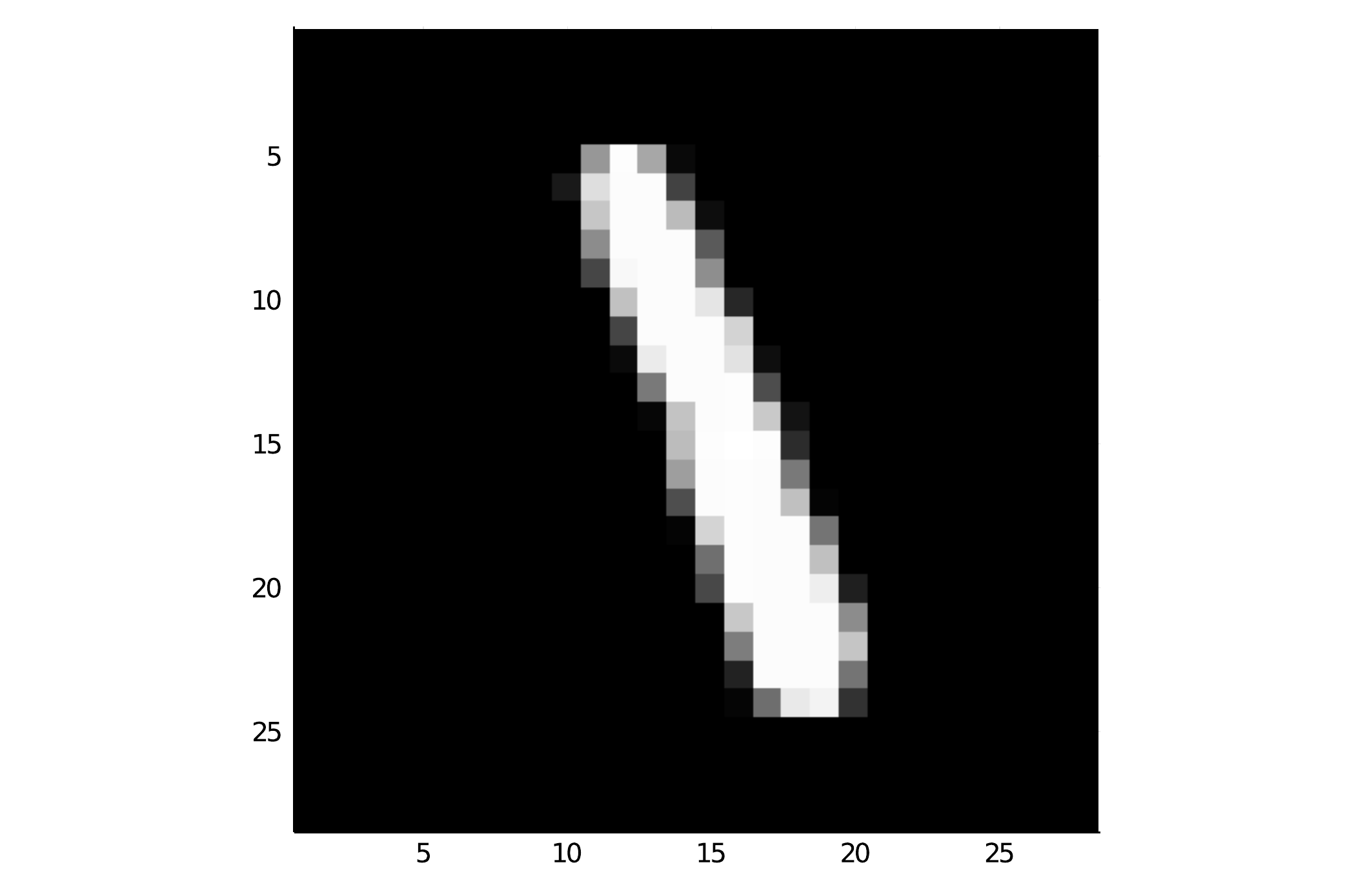}
			\includegraphics[width=0.09\textwidth]{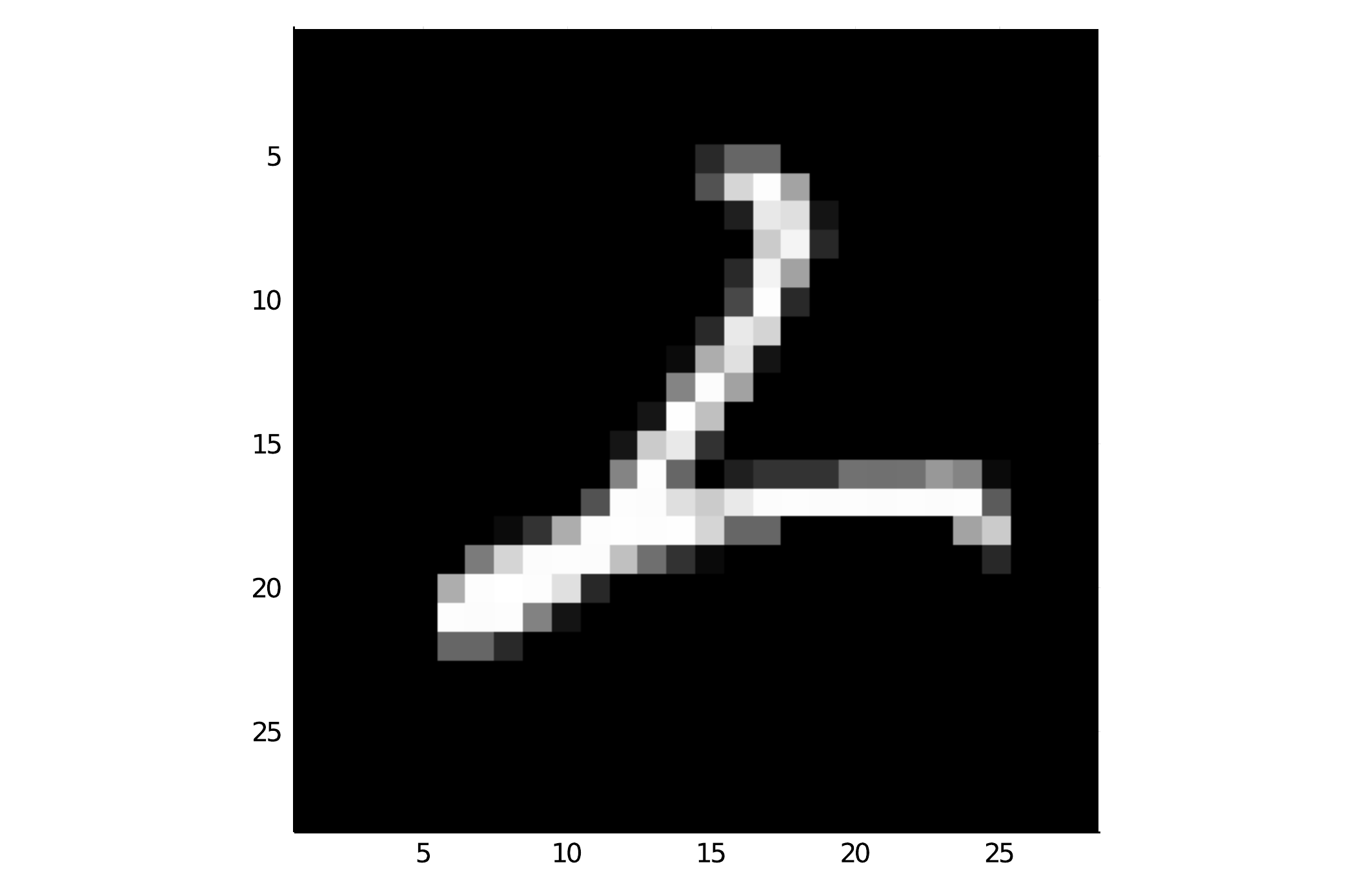}
			\includegraphics[width=0.09\textwidth]{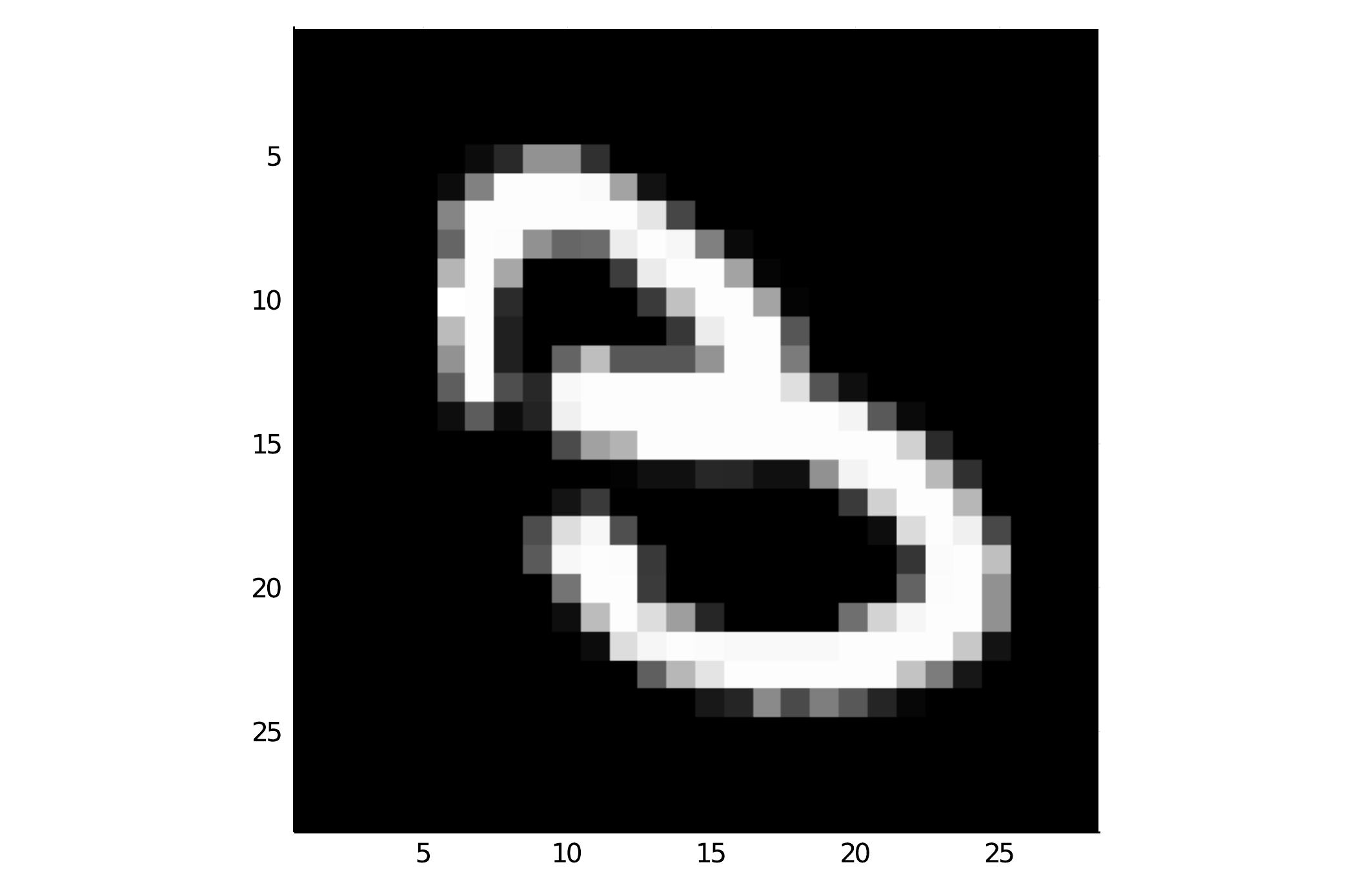}
			\includegraphics[width=0.09\textwidth]{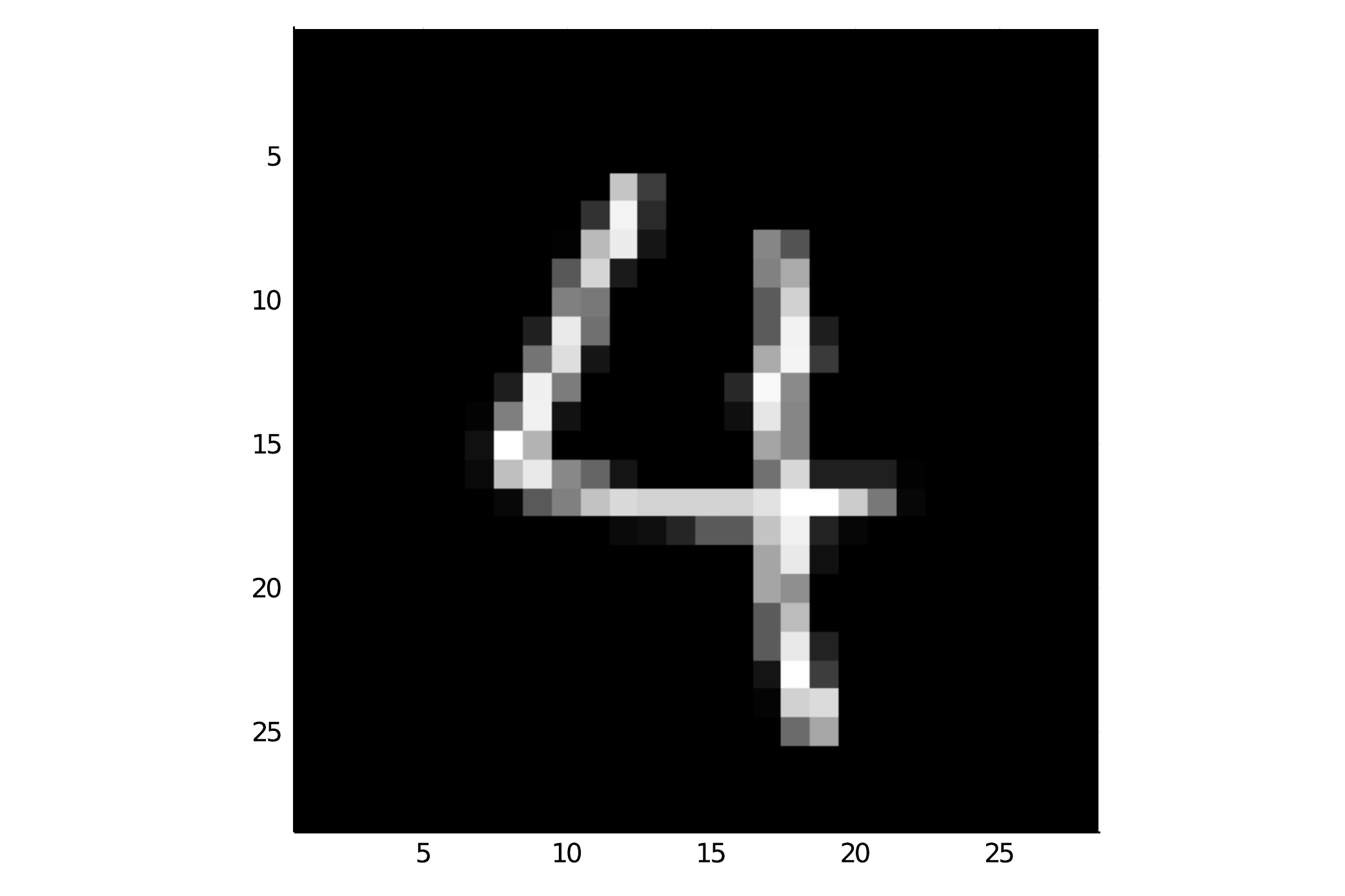}
			\includegraphics[width=0.09\textwidth]{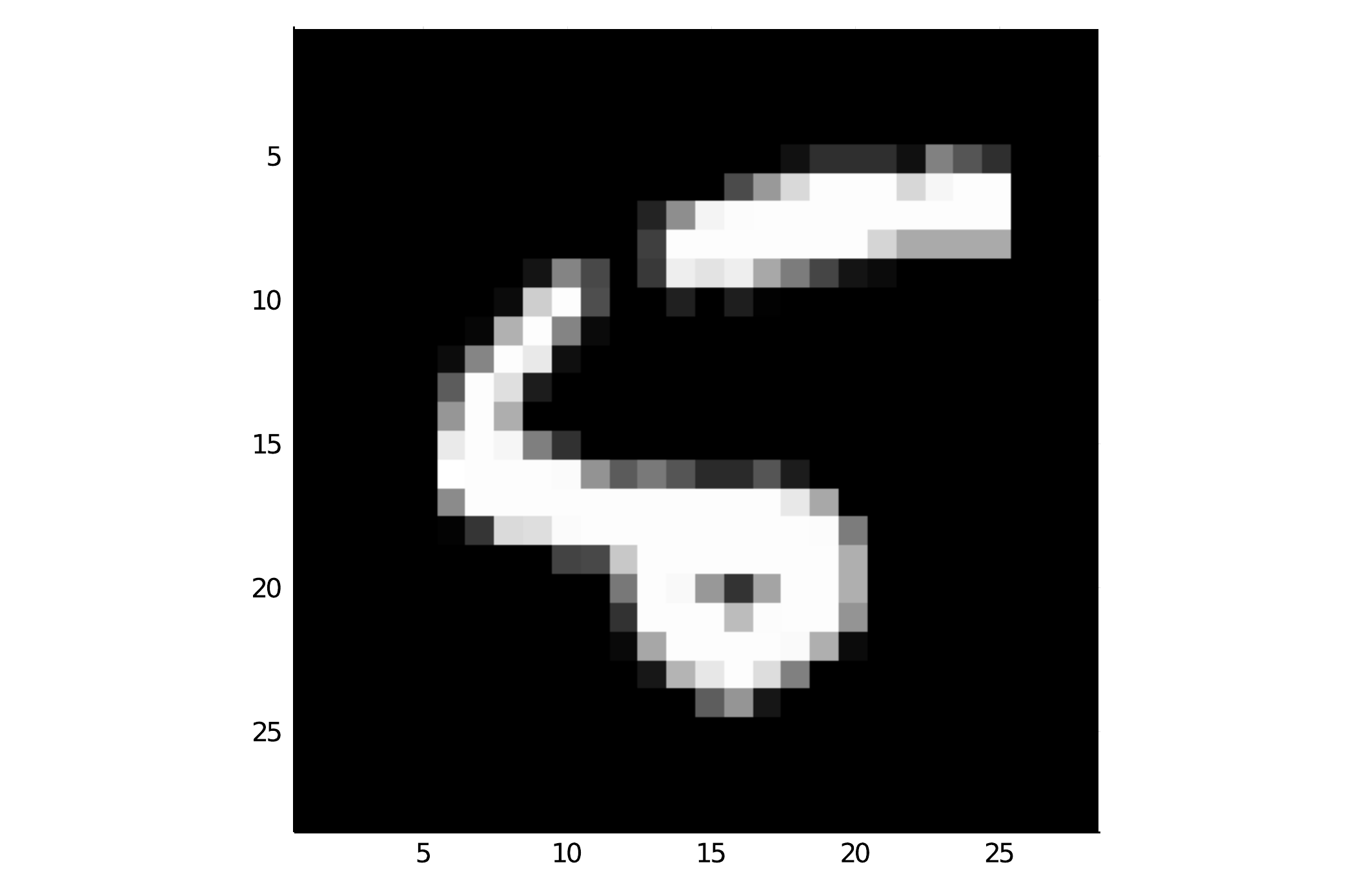}
			\includegraphics[width=0.09\textwidth]{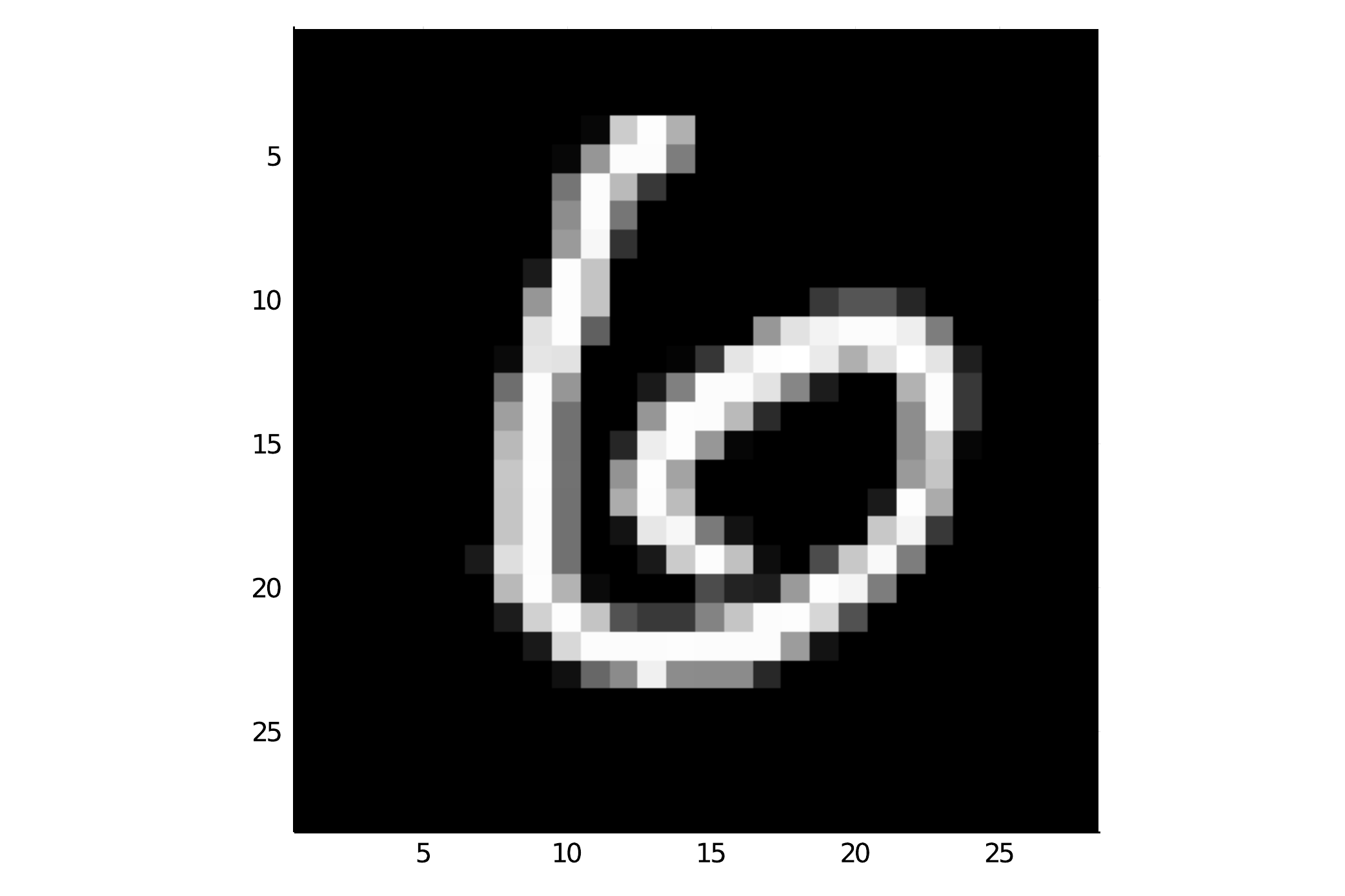}
			\includegraphics[width=0.09\textwidth]{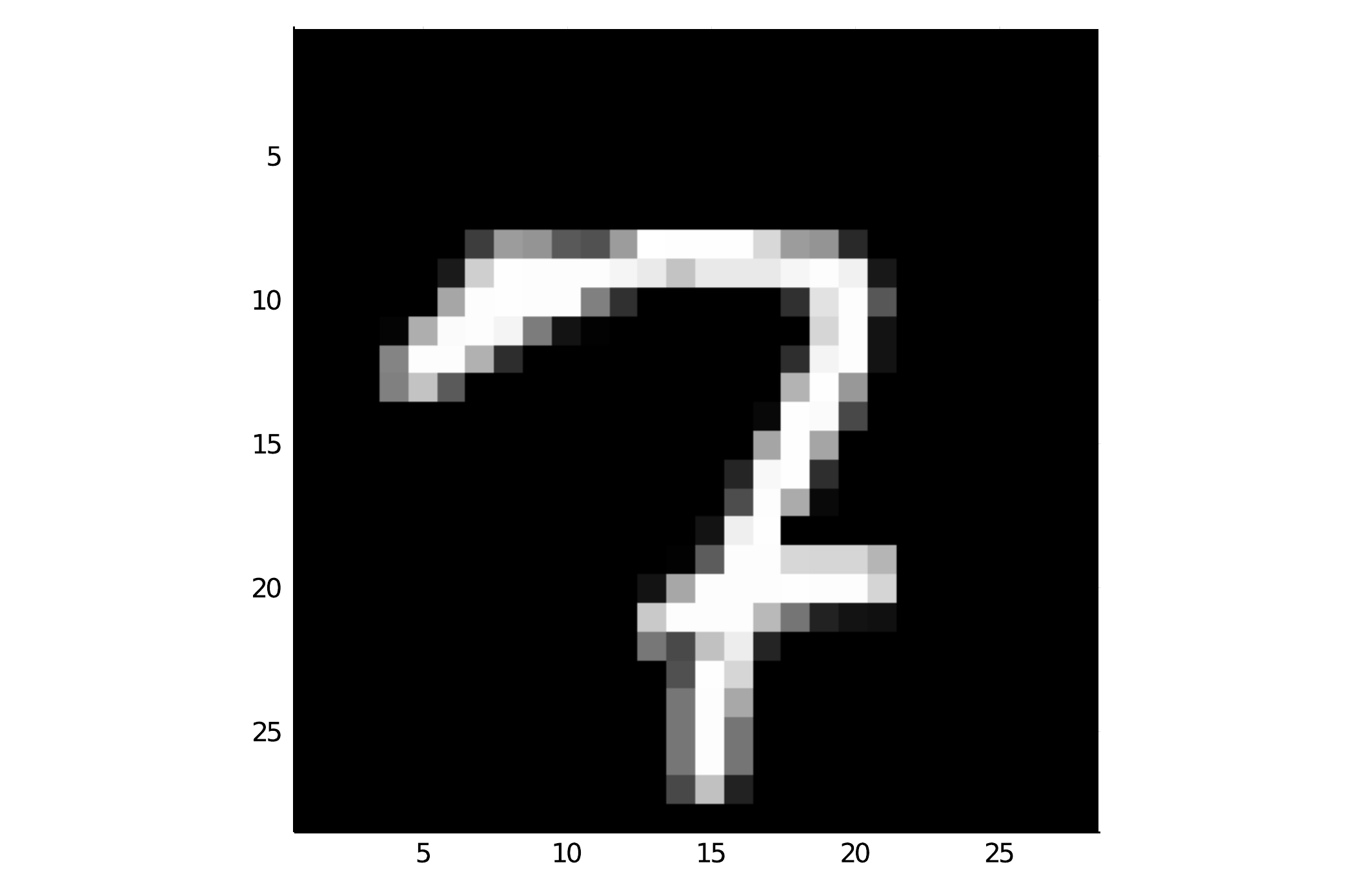}
			\includegraphics[width=0.09\textwidth]{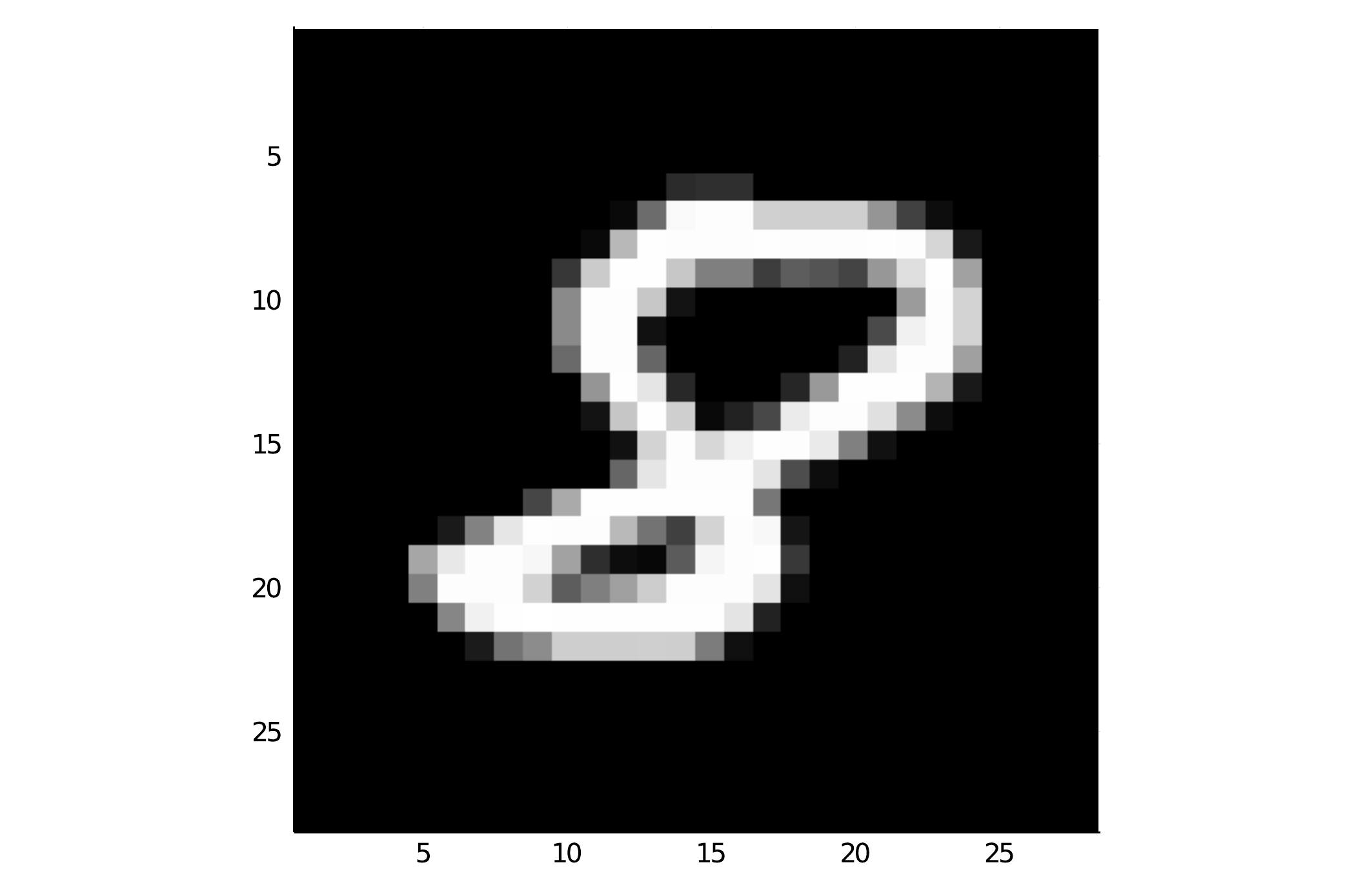}
			\includegraphics[width=0.09\textwidth]{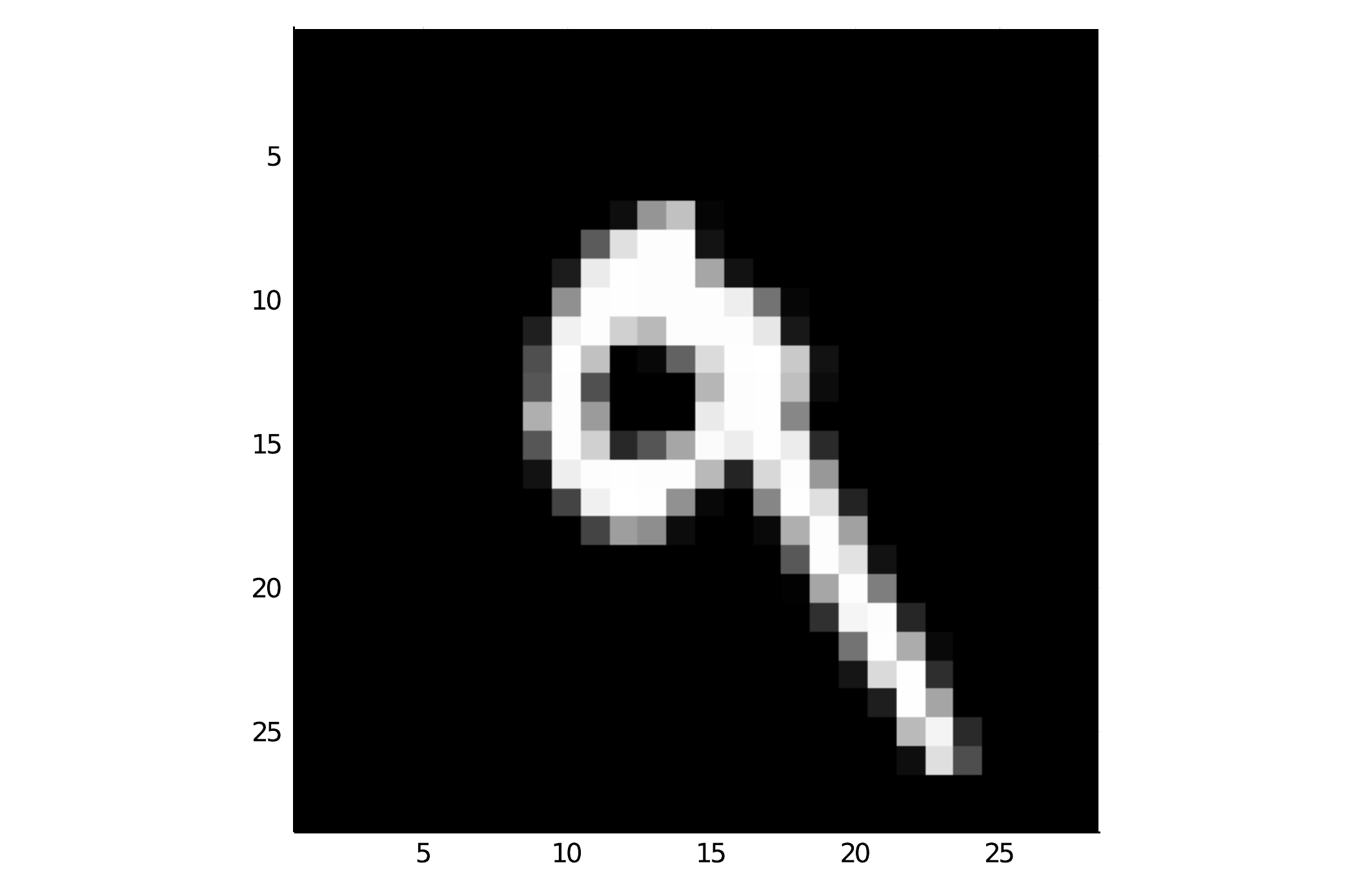}
			\caption{}
			\label{non-certified}
		\end{subfigure}
		\caption{Examples of certified points (above) and non-certified points (bellow).} \label{cert_eg}
	\end{figure}
We take different values of $\epsilon$ from 0.02 to 0.1, and compute the ratio of certified examples among the 10000 MNIST test data by the Lipschitz constants we obtain, as shown in Table \ref{mnist_cert}.
Note that for $\epsilon = 0.1$, we improve a little bit by 67\% compared to \textbf{Grad-cert} (65\%) described in \citelip{raghunathan2018certified}, as we use an exact formulation of the derivative of $\relu$ function.
	\begin{table}[t]
		\caption{Ratios of certified test examples for SDP-NN network by \textbf{HR-2}.}
		\label{mnist_cert}
		\centering
					\begin{tabular}{cccccc}
						\toprule
						$\epsilon$ & 0.02 & 0.04  & 0.06 & 0.08 & 0.1 \\
						\midrule
						Ratios  & 97.24\% & 92.84\%  & 87.10\%  & 78.34\% & 67.63\% \\
						\bottomrule
					\end{tabular}
	\end{table}

\section{Notes and sources}

Various optimization frameworks have been recently used to certify the robustness of deep neural networks, including \ac{SDP} \citelip{Raghuathan18, fazlyab2019safety}, \ac{LP} \citelip{dvijotham2018dual, wong2017provable}, mixed integer programming (MIP) \citelip{tjeng2017evaluating}, outer polytope approximation \citelip{boopathy2019cnn,zhang2018efficient,weng2018evaluating,weng2018towards}, and averaged activation operators \citelip{combettes2019lipschitz}.

Upper bounds on Lipschitz constants of deep networks can be obtained by a product of the layer-wise Lipschitz constants \citelip{huster2018limitations}.
This is however extremely loose and has many limitations. Note that \citelip{virmaux2018lipschitz} proposed an improvement via a finer product.
Departing from this approach, \citelip{latorre2020lipschitz} proposed a \ac{QCQP} formulation to estimate the Lipschitz constant of neural networks.
Shor's relaxation allows to obtain a valid upper bound. Alternatively, using the \ac{LP} hierarchy, \citelip{latorre2020lipschitz} obtains tighter upper bounds. By another \ac{SDP}-based method, \citelip{fazlyab2019efficient} provides an upper bound of the Lipschitz constant.
However this method is restricted to the $L_2$-norm whereas most robustness certification problems in deep learning are rather concerned with the $L_{\infty}$-norm.
The heuristic approach presented in this chapter can also apply to nearly sparse polynomial optimization problems, coming from the general optimization literature \citelip{chen2022sublevel}, or for other types of networks, such as monotone deep equilibrium networks \citelip{mondeq21}.

The proof of Theorem \ref{th:lip} is available in \citelip[Appendix 1]{LipPOP20}.
The interested reader can find more complete results for global/local Lipschitz constants of both 1-hidden layer and 2-hidden layer networks with various sizes and sparsities in \citelip[Appendix F and G]{LipPOP20}.

\input{lip.bbl}

%% file: lip.bbl
\providecommand{\etalchar}[1]{$^{#1}$}

%% file: ncsparse.tex
\chapter{Noncommutative optimization and quantum information}\label{chap:ncsparse}
In this chapter, we handle a specific class of sparse \ac{POP}s with noncommuting variables, and adapt the concept of \ac{CS} from Section \ref{chap:cs} in this setting.
A converging hierarchy of semidefinite relaxations for eigenvalue optimization is provided.
The Gelfand-Naimark-Segal (GNS) construction is applied to extract optimizers if flatness and irreducibility conditions are satisfied.
Among the main techniques used are amalgamation results from operator algebra.
The theoretical results are utilized to compute lower bounds on minimal eigenvalue of noncommutative polynomials from the literature, in particular arising from quantum information theory.

\section{Noncommutative polynomials}
We consider a finite alphabet $x_1,\dots,x_n$ (called noncommutating variables) and generate all possible words (monomials) of finite length in these letters.
The empty word is denoted by 1.
The resulting set of words is  $\langle \ux \rangle$, with $\ux = (x_1,\dots, x_n)$.
We denote by $\RX$ the ring of real polynomials in the noncommutating variables $\ux$. An element in $\RX$ is called a \emph{\ac{nc} polynomial}.
The \emph{support} of an nc polynomial $f=\sum_{w\in\langle \ux \rangle}a_ww$ is defined by $\supp(f)\coloneqq\{w\in\langle \ux \rangle\mid a_w\ne0\}$ and the \textit{degree} of $f$, denoted by $\deg(f)$, is the length of the longest word in $\supp(f)$. The set of nc polynomials of degree at most $r$ is denoted by $\RX_r$.
Let us denote by $\W_r$ the vector of all words of degree at most $r$ with resepct to the lexicographic order.
Note that $\W_r$ serves as a monomial basis of $\RX_r$ and the length of $\W_r$ is equal to $\bsigma(n,r) \coloneqq \sum_{i=0}^r n^i = \frac{n^{r+1}-1}{n-1}$.
The ring $\RX$ is equipped with the involution $\star$ that fixes $\R \cup \{x_1,\dots,x_n\}$ point-wise and reverses words, so that $\RX$ is the $\star$-algebra freely generated by $n$ symmetric letters $x_1,\dots,x_n$.
For instance $(x_1 x_2 + x_2^2 + 1)^\star = x_2 x_1 + x_2^2 + 1$.
The set of all \textit{symmetric elements} is defined as $\SymRX \coloneqq \{f \in \RX \mid f = f^\star  \}$.
A simple example of element of $\SymRX$ is $x_1 x_2 + x_2 x_1 + x_2^2 + 1$.
An nc polynomial of the form $g^\star g$ is called an {\em hermitian square}.
A given $f \in \SymRX$ is a \ac{SOHS} if there exist nc polynomials $h_1,\dots,h_t \in \RX$ such that $f = h_1^\star h_1 + \dots + h_t^\star h_t$.
Let $\SigmaX$ stand for the set of \ac{SOHS}. We denote by $\SigmaX_{r} \subseteq \SigmaX$ the set of \ac{SOHS} polynomials of degree at most $2 r$.
We now recall how to check whether a given $f \in \SymRX$ is an \ac{SOHS}.
The existing procedure, known as the \emph{Gram matrix method}, relies on the following proposition.
\begin{proposition}\label{prop:ncGram}
Assume that $f\in\SymRX$ is of degree at most $2d$.
Then $f \in \SigmaX$ if and only if there exists $\G_f \succeq 0$ satisfying
\begin{align}\label{eq:ncGram}
f = \W_d^\star \, \G_f \, \W_d.
\end{align}
Conversely, given such $\G_f \succeq 0$ of rank $t$, one can construct $g_1,\dots,g_t \in \RX$ of degree at most $d$ such that $f = \sum_{i=1}^t g_i^\star g_i$.
\end{proposition}
Any symmetric matrix $\G_f$ (not necessarily positive semidefinite) satisfying \eqref{eq:ncGram} is called a \emph{Gram matrix} of $f$.

Given a set of nc polynomials $\frakg = \{g_1,\dots,g_m \} \subseteq \SymRX$, the nc semialgebraic set ${\cD_{\frakg}}$ associated to $\frakg$ is defined as follows:
\begin{align}\label{eq:DS}
\cD_{\frakg}\coloneqq\bigcup_{k\in\N^*}\{\underline{A}=(A_1,\dots,A_n)\in(\Sbb_k)^n\mid g_j(\underline{A}) \succeq 0, j\in[m]\}.
\end{align}
When considering only tuples of $k\times k$ symmetric matrices, we use the notation ${\cD_{\frakg}^k} \coloneqq {\cD_{\frakg}} \cap (\Sbb_k)^n$.
The operator semialgebraic set ${\cD_{\frakg}^\infty}$ is the set of all bounded self-adjoint operators $\underline{A}$ {on a Hilbert space $\mathcal{H}$ endowed with a scalar product $\langle \cdot \mid \cdot \rangle$, making $g(\underline{A})$ a positive semidefinite operator for all $g \in \frakg$, i.e., $\langle g(\underline{A}) v \mid v \rangle \geq 0$, for all $v \in \mathcal{H}$.
We say that an nc polynomial $f$ is positive (denoted by $f \succ 0$) on ${\cD_{\frakg}^\infty}$  if for all $\underline{A} \in {\cD_{\frakg}^\infty}$ the operator $f(\underline{A})$  is positive definite, i.e., $\langle f(\underline{A}) v \mid v \rangle > 0$, for all nonzero $v \in \mathcal{H}$.
}
%
The quadratic module ${\cM(\frakg)}$, generated by $\frakg$, is defined by
\begin{align}\label{eq:MS}
{\cM(\frakg)}\coloneqq \left\{ \sum_{i=1}^t  a_i^\star g_i a_i \mid t \in \N^* , a_i \in \RX , g_i \in \frakg \cup \{1\}  \right\}.
\end{align}
Given $r \in \N^*$, the truncated quadratic module ${\cM(\frakg)_r}$ of order $r$, generated by $\frakg$, is
\begin{align}\label{eq:MS2d}
{\cM(\frakg)_r} \coloneqq \left\{  \sum_{i=1}^t  a_{i}^\star g_i a_i \mid t \in \N^* , a_i \in \RX , g_j \in \frakg \cup \{1\} , \deg (a_i^\star g_i a_i) \leq 2 r  \right\}.
\end{align}
A quadratic module ${\cM}$ is said to be {\em Archimedean} if for each $ a \in \RX$, there exists $N>0$ such that $N - a^\star a \in {\cM}$.
{One can show that
this is equivalent to the existence of an $N>0$ such that $N - \sum_{i=1}^n x_i^2 \in \cM$}.

We also recall the nc analog of Putinar's Positivstellensatz (Theorem \ref{th:putinar})
describing nc polynomials positive on ${\cD_{\frakg}^\infty}$ with Archimedean ${\cM(\frakg)}$.
\begin{theorem}[Helton-McCullough]\label{th:densePsatz}
Let $\{f\}\cup\frakg \subseteq \SymRX$ and assume that ${\cM(\frakg)}$ is Archimedean.
If $f(\underline{A}) \succ 0$ for all $\underline{A} \in {\cD_{\frakg}^\infty}$, then $f \in {\cM(\frakg)}$.
\end{theorem}

\section{Correlative sparsity patterns}\label{chap6:sec2}

We rely on the same \ac{CS} framework as in Chapter \ref{chap:cs}. More concretely, assuming $f=\sum_{w}a_{w}w\in\SymRX$ and $\frakg=\{g_1,\ldots,g_m\}\subseteq\SymRX$, we define the \ac{csp} graph associated with $f$ and $\frakg$ to be the graph $G^{\text{csp}}$ with nodes $V=[n]$ and with edges $E$ satisfying $\{i,j\}\in E$ if one of following conditions holds:
\begin{enumerate}[(i)]
	\item there exists $w\in\supp(f)\text{ s.t. }x_i,x_j\in\var(w)$;
	\item there exists $k\in[m] \text{ s.t. }x_i,x_j\in\var(g_k)$,
\end{enumerate}
where we use $\var(g)$ to denote the set of variables effectively involved in $g\in\RX$. Let $(G^{\text{csp}})'$ be a chordal extension of $G^{\text{csp}}$ and $I_k,k\in[p]$ be the maximal cliques of $(G^{\text{csp}})'$ with cardinality being denoted by $n_k,k\in[p]$.
We denote by $\langle \ux(I_k) \rangle$ (resp. $\RXI{k}$) the set of words (resp. nc polynomials) in the $n_k$ variables $\ux(I_k) = \{x_i : i \in I_k \}$.
We also define $\SymRXI{k} \coloneqq \SymRX \cap \RXI{k}$. Let $\SigmaXI{k}$ stand for the set of \ac{SOHS} in $\RXI{k}$ and we denote by $\SigmaXI{k}_{r}$ the restriction of $\SigmaXI{k}$ to nc polynomials of degree at most $2 r$.
In the sequel, we will rely on two specific assumptions. The first one is as follows.
\begin{assumption}[Boundedness]\label{hyp:sparsity}
Let ${\cD_{\frakg}}$ be as in~\eqref{eq:DS}.
There exists $N>0$ such that $\sum_{i=1}^n x_i^2 \preceq N$, for all $\ux \in {\cD_{\frakg}^\infty}$.
\end{assumption}
Then, Assumption~\ref{hyp:sparsity} implies that $\sum_{j \in I_k} x_j^2 \preceq N$, for all $k\in[p]$.
Thus we define
\begin{align}\label{eq:additional}
g_{m+k} \coloneqq N - \sum_{j \in I_k} x_j^2, \quad k\in[p],
\end{align}
and set $m' = m + p$ in order to describe the same set ${\cD_{\frakg}}$ again as
\begin{align}\label{eq:newDS}
\cD_{\frakg}\coloneqq\bigcup_{k\in\N^*}\{\underline{A} \in (\Sbb_k)^n \mid g_j(\underline{A}) \succeq 0, j\in[m']\},
\end{align}
as well as the operator semialgebraic set ${\cD_{\frakg}^\infty}$.

The second assumption, which is the strict nc analog of Assumption \ref{hyp:cs} (i)--(iii), is as follows.
\begin{assumption}\label{hyp:sparsityRIP}
Let ${\cD_{\frakg}}$ be as in~\eqref{eq:newDS} and let $f \in \SymRX$.
The index set $J \coloneqq \{1, \dots, m' \}$ is partitioned into $p$ disjoint sets $J_1,\dots,J_p$ and the two collections $\{I_1,\dots,I_p\}$ and $\{J_1,\dots,J_p \}$ satisfy
\begin{enumerate}[(i)]
\item The objective function $f$ can be decomposed as $f = f_1 + \dots + f_p$, with $f_k \in \SymRXI{k}$ for all $k\in[p]$;
\item For all $k\in[p]$ and $j \in J_k$, $g_j \in \SymRXI{k}$;
\item The \ac{RIP} \eqref{eq:RIP} holds for $I_1,\dots,I_p$ (possibly after some reordering).
\end{enumerate}
\end{assumption}
\section{Noncommutative moment and localizing matrices}
Given a sequence $\y=(y_w)_{w\in \W_{2r}}\in\R^{\bsigma(n,2r)}$ (here we allow $r=\infty$), let us define the linear functional $L_\y : \RX_{2r} \to \R$ by $L_\y(f) \coloneqq \sum_{w} a_{w} y_{w}$, for every polynomial $f=\sum_{w} a_{w}w$ of degree at most $2r$. The sequence $\y$ is said to be {\em unital} if $y_{1} = 1$ and is said to be {\em symmetric} if $y_{w^\star} = y_{w}$ for all $w\in \W_{2r}$. Suppose $g\in\SymRX$ with $\deg(g)\le 2r$. We further associate to $\y$ the following two matrices:
\begin{enumerate}[(1)]
\item the {\em (noncommutative) moment matrix} $\M_r(\y)$ is the matrix indexed by words $u,v \in \W_r$, with $[\M_r(\y)]_{u,v} = L_{\y}(u^\star v)=y_{u^\star v}$;
\item the {\em localizing matrix} $\M_{r-\lceil\deg(g)/2\rceil}(g\y)$ is the matrix indexed by words $u,v\in\W_{r-\lceil \deg(g)/2\rceil}$, with $[\M_{r-\lceil\deg(g)/2\rceil}(g\y)]_{u,v} = L_{\y}(u^\star g v)$.
\end{enumerate}
We recall the following useful facts.
\begin{lemma}\label{lemma:moment}
Let $g\in\SymRX$ with $\deg(g)\le 2r$ and let $L$ be the linear functional associated to a symmetric sequence $\y\coloneqq (y_w)_{w\in \W_{2r}}\in\R^{\bsigma(n,2r)}$. Then,
\begin{enumerate}[\rm (1)]
\item $L_{\y}(h^\star h) \geq 0$ for all $h\in\RX_{r}$ if and only if the moment matrix $\M_r(\y) \succeq 0$;
\item $L_{\y}(h^\star g h) \geq 0$ for all $h\in\RX_{r-\lceil\deg(g)/2\rceil}$ if and only if the localizing matrix $\M_{r-\lceil\deg(g)/2\rceil}(g\y) \succeq 0$.
\end{enumerate}
\end{lemma}
\begin{definition}\label{def:flatextension}
Let $\y=(y_w)_{w\in\W_{2r+2\delta}}\in\R^{\bsigma(n,2r+2\delta)}$ and $\tilde{\y}=(y_w)_{w\in\W_{2r}}$ be its truncation. We can write the moment matrix $\M_{r+\delta}(\y)$ in block form:
\[\M_{r+\delta}(\y) = \begin{bmatrix}
\M_r(\tilde{\y})& B\\
B^\intercal & C
\end{bmatrix}.\]
We say that $\y$ is \emph{$\delta$-flat} or that $\y$ is a \emph{flat extension} of $\tilde{\y}$, if $\M_{r+\delta}(\y)$ is flat over $\M_r(\tilde{L})$, i.e., if $\rank\M_{r+\delta}(\y) = \rank\M_r(\tilde{\y})$.
\end{definition}
For a subset $I\subseteq[n]$, let us define $\M_r(\y,I)$ to be the moment submatrix obtained from $\M_r(\y)$ after retaining only those rows and columns indexed by $w\in\langle \underline{x}(I) \rangle_r$.
For $g\in\R\langle\underline{x}, I\rangle$ with $\deg(g)\le 2r$, we also define the localizing submatrix $\M_{r-\lceil\deg(g)/2\rceil}(g\y,I)$ in a similar fashion.
%

\section{Sparse representations}
Here, we state our main theoretical result, which is a sparse version of
the Helton-McCullough Positivstellensatz (Theorem~\ref{th:densePsatz}).
For this, we rely on amalgamation theory for $C^\star$-algebras.

Given a Hilbert space $\mathcal{H}$, we denote by $\mathcal{B}(\mathcal{H})$ the set of bounded operators on $\mathcal{H}$.
A $C^\star$-algebra is a complex Banach algebra $\mathcal{A}$ (thus also a Banach space), endowed with a norm $\|\cdot\|$,
and with an involution $\star$ satisfying $\|xx^\star\|=\|x\|^2$
for all $x\in\mathcal A$. Equivalently, it is a
norm
closed subalgebra with involution of $\mathcal B(\mathcal H)$ for some Hilbert space $\mathcal H$.
{
Given a $C^\star$-algebra $\cA$, {a {\em state $\varphi$} is defined to be a positive linear functional of unit norm on $\cA$}, and we write often $(\cA,\varphi)$ when $\cA$ comes together with the state $\varphi$.
Given two $C^\star$-algebras $(\cA_1,\varphi_1)$ and $(\cA_2,\varphi_2)$, a homomorphism $\iota : \cA_1 \to \cA_2$ is called  {\em state-preserving}  if $\varphi_2 \circ \iota  = \varphi_1$.
Given a $C^\star$-algebra $\cA$,  a {\em unitary representation} of $\cA$ in $\cH$ is a $*$-homomorphism $\pi : \cA \to \mathcal{B} (\cH)$ which is {\em strongly continuous}, i.e., the mapping $\cA \to \cH$, $g \mapsto \pi(g) \xi$ is continuous for every $\xi \in \cH$.
}
\begin{theorem}\label{th:amalgamation}
Let
$(\mathcal{A},\varphi_0)$ and $\{(\mathcal B_k,\varphi_k) : k \in I\}$ be $C^\star$-algebras with states, and let $\iota_k$ be a state-preserving embedding of $\mathcal{A}$ into $\mathcal B_k$, for each $k\in I$.
Then there exists a $C^\star$-algebra ${\mathcal C}$ amalgamating the $(\mathcal B_k,\varphi_k)$
over $(\mathcal{A},\varphi_0)$. That is,
there is a state $\varphi$ on ${\mathcal C}$, and state-preserving homomorphisms
$j_k : \mathcal B_k \to {\mathcal C}$, such that {$j_k \circ \iota_k  = j_i \circ \iota_i$}, for all $k,i \in I$, and such that $\bigcup_{k \in I} j_k (\mathcal B_k)$ generates $\mathcal C$.
\end{theorem}
Theorem \ref{th:amalgamation} is illustrated in Figure \ref{diag:amalgamation} in the case $I=\{1,2\}$.
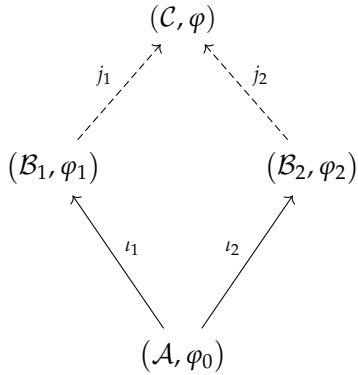
\begin{figure}[!t]
\centering
\begin{tikzcd}[row sep=3ex,column sep=1ex]
&  ({\mathcal C},\varphi) \\
{} \\
 \big(\mathcal{B}_1,\varphi_1\big) \arrow[uur,"j_1",dashed] &&  \big(\mathcal{B}_2,\varphi_2\big) \arrow[uul,"j_2"',dashed]  \\
{} \\
&& \\
&  \big(\mathcal{A},\varphi_0\big) \arrow[uuur, "\iota_2"] \arrow[uuul, "\iota_1"']
\end{tikzcd}
\caption{Illustration of Theorem~\ref{th:amalgamation} in the case $I=\{1,2\}$.}
\label{diag:amalgamation}
\end{figure}
We also recall the \ac{GNS} construction establishing a correspondence between $\star$-representations of a  $C^\star$-algebra and positive linear functionals on it.
In our context, the next result restricts to linear functionals on $\RX$ which are positive on an Archimedean quadratic module.
\begin{theorem}\label{th:gns}
Let $\frakg \subseteq \SymRX$ be given such that its quadratic module ${\cM(\frakg)}$ is Archimedean.
Let $L : \RX \to \R$ be a nontrivial linear functional with $L({\cM(\frakg)}) \subseteq \R_{\geq 0}$.
Then there exists a tuple $\underline{A} = (A_1,\dots,A_n) \in {\cD_{\frakg}^\infty}$ and a vector $\vb$ such that $L(f) = \langle f(\underline{A}) \vb , \vb \rangle$, for all $f \in \RX$.
\end{theorem}

Let $I_k,k\in[p]$ and $J_k,k\in[p]$ be given as in Chapter \ref{chap6:sec2}.
For $k\in[p]$, let us define
\begin{align*}
\cM(\frakg)^{k}\coloneqq\left\lbrace a_0^\star a_0+\sum_{i\in J_k} a_i^\star g_i a_i \mid a_{i}\in\RXI{k}, i\in J_k\cup\{0\}\right\rbrace
\end{align*}
and
\begin{align}\label{eq:sparseMS}
\cM(\frakg)^{\sparse} \coloneqq \cM(\frakg)^{1} + \dots + \cM(\frakg)^{p}.
\end{align}
Next, we state the main foundational result of this section.
\begin{theoremf}\label{th:sparsePsatz}
Let $\{f\}\cup\frakg\subseteq\SymRX$ and let ${\cD_{\frakg}}$ be as in~\eqref{eq:newDS} with the additional quadratic constraints~\eqref{eq:additional}.
Suppose Assumption~\ref{hyp:sparsityRIP} holds.
If $f(\underline{A}) \succ 0$ for all $\underline{A} \in {\cD_{\frakg}^\infty}$, then $f \in {\cM(\frakg)^{\sparse}}$.
\end{theoremf}
We provide an example demonstrating that sparsity without an \ac{RIP}-type condition is not sufficient to deduce sparsity in \ac{SOHS} decompositions.
\begin{example}
Consider the case of three variables $\underline x=(x_1,x_2,x_3)$
and the polynomial
\begin{align*}
f & =(x_1+x_2+x_3)^2  \\
& = x_1^2+x_2^2+x_3^2+x_1x_2+x_2x_1+x_1x_3+x_3x_1+x_2x_3+x_3x_2 \in\Sigma\langle\underline x\rangle.
\end{align*}
Then $f=f_1+f_2+f_3$, with
\[
\begin{split}
f_1 & =\frac12 x_1^2+\frac12 x_2^2+x_1x_2+x_2x_1\in\R\langle x_1,x_2\rangle, \\
f_2 & = \frac12x_2^2+\frac12x_3^2+x_2x_3+x_3x_2\in\R\langle x_2,x_3\rangle,\\
f_3 & = \frac12x_1^2+ {\frac12} x_3^2+x_1x_3+x_3x_1\in\R\langle x_1,x_3\rangle.
\end{split}
\]
However,
the sets $I_1=\{1,2\}$, $I_2=\{2,3\}$ and  $I_3=\{1,3\}$ do not satisfy the \ac{RIP} condition~\eqref{eq:RIP} and
 $f\not\in\SigmaX^{\sparse}\coloneqq\Sigma\langle x_1,x_2\rangle+\Sigma\langle x_2,x_3\rangle+\Sigma\langle x_1,x_3\rangle$ since it has a unique Gram matrix by homogeneity.

Now consider $\frakg=\{1-x_1^2,\, 1-x_2^2,\, 1-x_3^2\}$.
Then ${\cD_{\frakg}}$ is as in~\eqref{eq:newDS}, ${\cM(\frakg)^{\sparse}}$ is as in~\eqref{eq:sparseMS} and $f|_{{\cD_{\frakg}^\infty}}\succeq0$.
However, we claim that
$f-b\in {\cM(\frakg)^{\sparse}}$ if and only if $b\leq -3$.
Clearly,
\begin{align*}
	f+3=\,&(x_1+x_2)^2+(x_1+x_3)^2+(x_2+x_3)^2\\
	&+(1-x_1^2)+(1-x_2^2)+(1-x_3^2)\in {\cM(\frakg)^{\sparse}}.
\end{align*}
So one has $-3\leq\sup\,\{b : f - b \in \cM(\frakg)^{\sparse} \}$, and the dual of this latter problem is given by
\begin{equation}\label{eq:norip}
\begin{cases}
\inf\limits_{\y_k}&\sum_{k=1}^3 L_{\y_k}(f_k)\\
\rm{ s.t.}
&L_{\y_k}(1) = 1, \quad k=1,2,3\\
&L_{\y_k}(h^\star h) \succeq 0, \quad \forall h \in \RXI{k},  \quad k=1,2,3\\
&L_{\y_k}(h^\star (1 - x_i^2) h)  \succeq 0, \quad \forall h \in \RXI{k},i\in I_k,k=1,2,3\\
&L_{\y_j}|_{\R\langle \underline X(I_j\cap I_k)\rangle}=L_{\y_k}|_{\R\langle \underline X(I_j\cap I_k)\rangle}, \quad j,k=1,2,3\\
\end{cases}
\end{equation}
Hence, by weak duality, it suffices to show that there exist linear functionals
$L_{\y_k}:\RXI k\to\R$ satisfying the constraints of problem \eqref{eq:norip} and such that $\sum_k L_{\y_k}(f_k)=-3$.
Define
\[
A=\begin{bmatrix}0&1\\ 1&0\end{bmatrix}, \quad B=-A
\]
and let
\[L_{\y_k}(g)=\Trace (g(A,B)) \quad\text{ for }g\in\RXI k.\]
Since $L_{\y_k}(f_k)=-1$, the three first constraints of problem \eqref{eq:norip} are easily verified and $\sum_k L_{\y_k}(f_k)=-3$.
 For the last one, given, say
$h\in\RXI 1\cap\RXI 2=\R\langle x_2\rangle$, we have
\[
\begin{split}
{L_{\y_1}(h)}&=\Trace (h(B)), \\
L_{\y_2}(h)&=\Trace (h(A)),
\end{split}
\]
since $L_{\y_1}$ (resp.~$L_{\y_2}$) is defined on  $\R\langle x_1, x_2\rangle$ (resp.~$\R\langle x_2, x_3\rangle$) and $h$ depends only on the second (resp.~first) variable $x_2$ corresponding to $B$ (resp.~$A$).

But matrices $A$ and $B$ are orthogonally equivalent as
$UAU^\intercal=B$ for
\[
U=\left[
\begin{array}{rr}
 0 & 1 \\
 -1 & 0 \\
\end{array}
\right],
\]
whence $h(B)=h(UAU^\intercal)=Uh(A)U^\intercal$ and $h(A)$ have the same trace.
 \end{example}

\section{Sparse GNS construction}
Next, we provide the main theoretical tools to extract solutions of nc optimization problems with \ac{CS}.
To this end, we first present sparse nc versions of theorems by Curto and Fialkow.
As recalled in Section \ref{sec:extract} for the commutative case, Curto and Fialkow provided sufficient conditions for linear functionals on the set of degree $2 r$ polynomials to be represented  by integration with respect to a nonnegative measure.
The main sufficient condition to guarantee such a representation is flatness (see  Definition~\ref{def:flatextension}) of the corresponding moment matrix.
We recall this result, which relies on a finite-dimensional \ac{GNS} construction.
\begin{theorem}\label{th:dense_flat}
Let $\frakg \subseteq \SymRX$ and set $\delta \coloneqq \max\,\{ \lceil \deg (g)/2 \rceil : g \in \frakg\}$.
For $r \in \N^*$, let $L_{\y} : \RX_{2 r + 2 \delta} \to \R$ be the linear functional associated to a unital sequence  $\y=(y_w)_{w\in\W_{2r+2\delta}}\in\R^{\bsigma(n,2r+2\delta)}$ satisfying $L_{\y}({\cM(\frakg)_{r +  \delta}}) \subseteq \R_{\geq 0}$.
If $\y$ is $\delta$-flat, then there exists $\hat{A} \in {\cD_{\frakg}^r}$ for some $t \leq \bsigma(n,r)$ and a unit vector $\vb$ such that
\begin{align}\label{eq:flat_representation}
L_{\y}(g) = \langle g(\hat{A}) \vb , \vb \rangle,
\end{align}
for all $g \in \SymRX_{2r}$.
\end{theorem}
We now give the sparse version of Theorem~\ref{th:dense_flat}.
\begin{theoremf}\label{th:sparse_flat}
Suppose $r\in\N^*$. Let $\frakg \subseteq \SymRX_{2r}$, and assume ${\cD_{\frakg}}$ is as in~\eqref{eq:newDS} with the additional quadratic constraints~\eqref{eq:additional}.
Suppose Assumption~\ref{hyp:sparsityRIP}(i) holds.
Set $\delta \coloneqq \max\,\{\lceil\deg(g)/2\rceil:g\in\frakg\}$.
Let $L_{\y} : \RX_{2 r + 2 \delta} \to \R$ be the linear functional associated to a unital sequence  $\y=(y_w)_{w\in\W_{2r+2\delta}}\in\R^{\bsigma(n,2r+2\delta)}$ satisfying $L_{\y}({\cM(\frakg)_{r +  \delta}}) \subseteq \R_{\geq 0}$.
Assume that the following holds:
\begin{enumerate}
\item[(H1)] ${{\M_{r+\delta}}(\y,I_k)}$ and ${{\M_{r+\delta}}(\y,I_k\cap I_j)}$ are $\delta$-flat, for all $j,k\in[p]$.
\end{enumerate}
Then, there exist finite-dimensional Hilbert spaces $\cH(I_k)$ with dimension $t_k$, for all $k\in[p]$,
Hilbert spaces
$\cH(I_j \cap I_k)\subseteq\cH(I_j),\cH(I_k)$ for all pairs $(j,k)$ with $I_j \cap I_k \neq 0$, and operators $\hat A^k$, $\hat A^{j k}$,  acting on them, respectively.
Further, there are unit vectors $\vb^j\in\cH(I_j)$ and $\vb^{jk}\in\cH(I_j\cap I_k)$ such that
\begin{equation}\label{eq:desired}
\begin{split}
L_{\y}(f) & = \langle f(\hat{A^j}) \vb^j , \vb^j \rangle \quad \text{for all } f\in\RXI j_{2r},\\
L_{\y}(g) & = \langle g(\hat{A^{jk}}) \vb^{jk} , \vb^{jk} \rangle  \quad \text{for all } g\in\R \langle \underline{X}( I_j \cap I_k) \rangle_{2r}.
\end{split}
\end{equation}

Assuming that for all pairs $(j,k)$ with $I_j \cap I_k \neq {\emptyset}$, one has
\begin{enumerate}
\item[(H2)] the matrices $(\hat A_i^{j k})_{i \in I_j \cap I_k}$ have no common complex invariant subspaces,
\end{enumerate}
 then there exist $\underline{A} \in {\cD_{\frakg}^t}$, with $t \coloneqq t_1 \cdots t_p$, and a unit vector $\vb$ such that
\begin{align}\label{eq:sparse_flat_representation}
L_{\y}(f) = \langle f(\underline{A}) \vb , \vb \rangle,
\end{align}
for all $f\in\sum_k \RXI k_{2r}$.
\end{theoremf}
\begin{figure}[ht]
\centering
\begin{tikzcd}[row sep=3ex,column sep=1ex]
& \mathcal A \\
{} \\
\mathcal A(I_1) \arrow[uur,"j_1",dashed] && \mathcal A(I_2) \arrow[uul,"j_2"',dashed]  \\
{} \\
& \mathcal A(I_1\cap I_2) \arrow[uur, "\iota_2"'] \arrow[uul, "\iota_1"]
\end{tikzcd}
\caption{Amalgamation of finite-dimensional $C^\star$-algebras.}\label{diag:amalgamSmall}
\label{fig:2}
\end{figure}
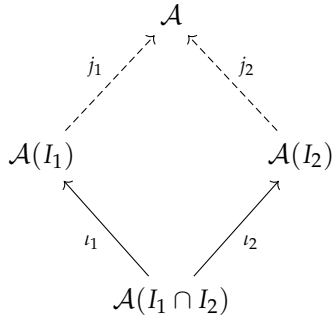
\begin{example}[Non-amalgamation in finite-dimensional algebras]
\label{ex:no_finite_amalgam}
Given $I_1$ and $I_2$, suppose $\cA(I_1\cap I_2)$ is generated by the $2\times2$ diagonal matrix
\[A^{1 2} =\begin{bmatrix}
1 \\ &2
\end{bmatrix},
\]
and assume $\cA(I_1)=\cA(I_2)= \Mbb_3(\R)$.
(Observe that $\cA(I_1\cap I_2)$ is the algebra of all diagonal matrices.)
For each $k \in \{1,2\}$, let us define $\iota_k(A)\coloneqq A \oplus k$, for all $A \in \cA(I_1\cap I_2)$.
We claim that there is no finite-dimensional
$C^\star$-algebra
 $\cA$ amalgamating the above Figure \ref{fig:2}.
Indeed, by the Skolem-Noether theorem, every homomorphism $\Mbb_n(\R)\to \Mbb_m(\R)$ is of the form
$x\mapsto P^{-1}(x\otimes \I_{m/n})P$ for some invertible $P$; in particular, $n$ divides $m$.
If a desired $\cA$ existed, then the matrices
$(A^{12}\oplus 1)\otimes \I_k$ and $(A^{12}\oplus 2)\otimes \I_k$ would be similar.
But they are not as is easily seen from eigenvalue multiplicities.
\end{example}
\if{
\begin{remark}
\label{rk:gns_sparse}

\end{remark}
}\fi
As in the dense case, we can summarize the sparse \ac{GNS} construction procedure described in the proof of Theorem~\ref{th:sparse_flat} into an algorithm, called $\sparsegns$ (see \citencsparse[Algorithm 4.6]{ncsparse}).

\section{Eigenvalue optimization}\label{sec:eig}
We provide \ac{SDP} relaxations allowing one to under-approximate the smallest eigenvalue that a given nc polynomial can attain on a tuple of symmetric matrices from a given semialgebraic set.
%
We first recall the celebrated Helton-McCullough theorem stating the equivalence between \ac{SOHS} and positive semidefinite nc polynomials.
\begin{theorem}[Helton-McCullough]\label{th:Helton}
Given $f \in \SymRX$, $f(\underline{A}) \succeq 0$, for all $\underline{A} \in (\Sbb_k)^n$, $k\in\N^*$, if and only if $f \in \SigmaX$.
\end{theorem}
In contrast with the constrained case where we obtain the analog of Putinar's Positivstellensatz in Theorem~\ref{th:sparsePsatz}, there is no sparse analog of Theorem \ref{th:Helton}, as shown in the following example.
{
\begin{lemma}
\label{lemma:nosparseHelton}
There exist polynomials which are sparse sums of hermitian squares but are not sums of sparse hermitian squares.
\end{lemma}
\begin{proof}
Let $v=\begin{bmatrix}x_1&x_1x_2&x_2&x_3&x_3x_2 \end{bmatrix}$,
\begin{equation}\label{eq:gram}
\G_f=\left[\begin{array}{rrrrr}
1&-1&-1&0&\alpha\\
-1&2&0&-\alpha&0\\
-1&0&3&-1&9\\
0&-\alpha&-1&6&-27\\
\alpha&0&9&-27&142
 \end{array}\right],\qquad \alpha\in\R,
\end{equation}
and consider
\begin{equation}\label{eq:polySohs}
\begin{split}
f&=v\G_fv^\star \\
&= x_1^2-x_1 x_2-x_2 x_1+3
   x_2^2-2
   x_1 x_2 x_1+2 x_1 x_2^2 x_1 \\
   &\phantom{=\ }
-x_2 x_3-x_3 x_2+6 x_3^2
   +9 x_2^2 x_3+9 x_3  x_2^2-54
   x_3 x_2 x_3 +142
   x_3 x_2^2 x_3.
\end{split}
\end{equation}
The polynomial $f$ is clearly sparse with resepct to $I_1=\{x_1,x_2\}$ and $I_2=\{x_2,x_3\}$. Note that the matrix $\G_f$ is positive semidefinite if and only if  $0.270615 \lesssim \alpha \lesssim 1.1075$, whence $f$  is a sparse polynomial that is an \ac{SOHS}.

We claim that $f\not\in\SigmaXI1+\SigmaXI2$, i.e., $f$ is not a sum of sparse hermitian squares. By the Newton chip method  only monomials in $v$ can appear in an \ac{SOHS} decomposition of $f$. Further, every Gram matrix of $f$ in the monomial basis $v$ is of the form \eqref{eq:gram}. However, the matrix $\G_f$ with $\alpha=0$ is not positive semidefinite, and hence $f\not\in\SigmaXI1+\SigmaXI2$.
\end{proof}
}
\subsection{Unconstrained eigenvalue optimization}\label{sec:unconstr_eig}
Let $\I_k$ stand for the $k\times k$ identity matrix.
Given $f \in \SymRX$ of degree $2 d$, the smallest eigenvalue of $f$ is obtained by solving the following  optimization problem:
\begin{align}\label{eq:eigmin}
\lambda_{\min}(f) \coloneqq \inf\,\{\langle f(\underline{A}) \vb,  \vb \rangle : \underline{A}\in(\Sbb_k)^n, k\in\N^*,\| \vb \|_2 = 1\}.
\end{align}
The optimal value $\lambda_{\min}(f)$ of Problem~\eqref{eq:eigmin} is the greatest lower bound on the eigenvalues of $f(\underline{A})$ over all $n$-tuples $\underline{A}$ of real symmetric matrices.
Problem~\eqref{eq:eigmin} can be rewritten as follows:
\begin{equation}\label{eq:eigmin2}
\begin{array}{rll}
\lambda_{\min}(f) = &\sup\limits_{b} & b \\
&\,\rm{ s.t.} & f(\underline{A})-b\I_k \succeq 0, \quad \forall \underline{A} \in (\Sbb_k)^n, k\in\N^*
\end{array}
\end{equation}
which is in turn equivalent to
\begin{equation}\label{eq:eigmin_dual}
\begin{array}{rll}
\lambda^d(f) \coloneqq &\sup\limits_{b}  & b \\
&\,\rm{ s.t.}& f(\ux) - b \in \SigmaX_{d}
\end{array}
\end{equation}
as a consequence of Theorem~\ref{th:Helton}.

The dual of \ac{SDP}~\eqref{eq:eigmin_dual} is
\begin{equation}
\label{eq:eigmin_primal}
\begin{array}{rll}
\eta^d(f) \coloneqq &\inf\limits_{{\y}}  & L_{\y}(f)  \\
&\,\rm{ s.t.} & {y_{1} = 1} , \quad {{\M_d}(\y)} \succeq 0\\
\end{array}
\end{equation}
%
One can compute $\lambda_{\min}(f)$ by solving a single \ac{SDP},  either \ac{SDP}~\eqref{eq:eigmin_primal} or
\ac{SDP}~\eqref{eq:eigmin_dual}, since there is no duality gap between these two programs, that is, one has $\eta^d(f) = \lambda^d(f) = \lambda_{\min}(f)$.

Now, we address eigenvalue optimization for a given  sparse nc polynomial $f = f_1 + \dots + f_p$ of degree $2 d$, with $f_k \in \SymRXI{k}_{2d}$, for all $k\in[p]$.
For all $k\in[p]$, let $\G_{f_k}$ be a Gram matrix associated to $f_k$.
The sparse variant of \ac{SDP}~\eqref{eq:eigmin_primal} is
\begin{equation}
\label{eq:sparse_eigmin_primal}
\begin{array}{rll}
\eta^{d}_{\cs}(f) \coloneqq &\inf\limits_{{\y}}  & L_{\y}(f)\\
&\,\rm{ s.t.} & y_{1} = 1, \quad {{\M_d}(\y,I_k)}  \succeq 0, k\in[p]
\end{array}
\end{equation}
whose dual is the sparse variant of \ac{SDP}~\eqref{eq:eigmin_dual}:
\begin{equation}
\label{eq:sparse_eigmin_dual}
\begin{array}{rll}
\lambda^{d}_{\cs}(f) \coloneqq &\sup\limits_{b}&b\\
&\,\rm{ s.t.} & f - b \in \SigmaXI{1}_{d} + \dots + \SigmaXI{p}_{d}
\end{array}
\end{equation}
To prove that there is no duality gap between \ac{SDP}~\eqref{eq:sparse_eigmin_primal} and \ac{SDP}~\eqref{eq:sparse_eigmin_dual}, we need the following result.
\begin{proposition}
\label{prop:sparseclosed}
The set $\SigmaX^{\sparse}_d$ is a closed convex subset of $\RXI{1}_{2 d} + \dots + \RXI{p}_{2 d}$.
\end{proposition}
From Proposition~\ref{prop:sparseclosed}, we obtain the following theorem which does not require Assumption~\ref{hyp:sparsityRIP}.
\begin{theoremf}
\label{th:sparse_eig_nogap}
Let $f \in \SymRX$ of degree $2 d$, with $f = f_1 + \dots + f_p$, $f_k \in \SymRXI{k}_{2d}$, for all $k\in[p]$. 
Then, one has $\lambda^{d}_{\cs}(f) = \eta^{d}_{\cs}(f)$, i.e., there is no duality gap between \ac{SDP}~\eqref{eq:sparse_eigmin_primal} and \ac{SDP}~\eqref{eq:sparse_eigmin_dual}.
\end{theoremf}
\begin{remark}
\label{rk:sparsevsdense}
By contrast with the dense case, it is not enough to compute the solution of \ac{SDP}~\eqref{eq:sparse_eigmin_primal} to obtain the optimal value $\lambda_{\min}(f)$ of the unconstrained optimization problem~\eqref{eq:eigmin}.
However, one can still compute a certified lower bound  $\lambda^{d}_{\cs}(f)$ by solving a single \ac{SDP}, either in the primal form~\eqref{eq:sparse_eigmin_primal} or in the dual form~\eqref{eq:sparse_eigmin_dual}.
Note that the related computational cost is potentially much less expensive.
Indeed, \ac{SDP}~\eqref{eq:sparse_eigmin_dual} involves \ac{PSD} matrices of size at most $\max\,\{\bsigma(n_k,d)\}_{k=1}^p$ and $\sum_{k=1}^p \bsigma(n_k,2 d)$ equality constraints.
This is in contrast with the dense version~\eqref{eq:eigmin_dual}, which involves \ac{PSD} matrices of size $\bsigma(n,d)$ and $\bsigma(n,2 d)$ equality constraints.
\end{remark}
\subsection{Constrained eigenvalue optimization}
\label{sec:constr_eig}
Here, we focus on providing lower bounds for the constrained eigenvalue optimization of nc polynomials.
Given $f \in \SymRX$ and $\frakg = \{g_1,\dots,g_{m} \}$ $\subseteq \SymRX$ as in~\eqref{eq:DS}, let us define  $\lambda_{\min} (f, \frakg)$ as follows:
\begin{align}
\label{eq:constr_eigmin}
\lambda_{\min}(f,\frakg) \coloneqq \inf\,\{\langle f(\underline{A}) \vb , \vb \rangle : \underline{A} \in {\cD_{\frakg}^\infty}, \| \vb \| = 1 \},
\end{align}
which is, as for the unconstrained case, equivalent to
\begin{equation}\label{eq:constr_eigmin2}
\begin{array}{rll}
\lambda_{\min}(f,\frakg) = &\sup\limits_{b}  & b \\
&\,\rm{ s.t.} & f(\underline{A}) - b \I_k \succeq 0, \quad \forall \underline{A} \in {\cD_{\frakg}^\infty}.
\end{array}
\end{equation}
As usual, let $d_j \coloneqq \lceil \deg(g_j) / 2 \rceil$ for each $j\in[m]$, and let $$r_{\min} \coloneqq \max\,\{\lceil \deg(f) / 2 \rceil, d_1, \dots, d_m \}.$$
One can approximate $\lambda_{\min} (f, \frakg)$ from below via the following hierarchy of \ac{SDP} programs, indexed by $r \geq r_{\min}$:
\begin{equation}
\label{eq:constr_eigmin_dual}
\begin{array}{rll}
\lambda^r (f,\frakg) \coloneqq &\sup\limits_{b}  & b \\
&\,\rm{ s.t.}& f - b \in {\cM(\frakg)_r}
\end{array}
\end{equation}
The dual of \ac{SDP}~\eqref{eq:constr_eigmin_dual} is
\begin{equation}
\label{eq:constr_eigmin_primal}
\begin{array}{rll}
\eta^{r}(f,\frakg) \coloneqq &\inf\limits_{{\y}}  & L_{\y}(f)  \\
&\,\rm{ s.t.} & {y_{1} = 1},\quad{{\M_r}(\y)} \succeq 0\\
&& {{\M_{r-d_j}}(g_j \y)} \succeq 0, \quad j\in[m]\\
\end{array}
\end{equation}
%
%
Under additional assumptions, this hierarchy of primal-dual \ac{SDP}~\eqref{eq:constr_eigmin_dual}-\eqref{eq:constr_eigmin_primal} converges to the optimal value of the constrained eigenvalue problem.
\begin{theorem}
\label{th:constr_eig}
{Assume that ${\cD_{\frakg}}$ is as in~\eqref{eq:newDS} with the additional quadratic constraints~\eqref{eq:additional}} and that the quadratic module $\cM(\frakg)$ is Archimedean.
Then the following holds for $f \in \SymRX$:
\begin{align}
\label{eq:cvg_eig}
\lim_{r \to \infty} \eta^{r}(f,\frakg) = \lim_{r \to \infty}  \lambda^r (f,\frakg) = \lambda_{\min} (f, \frakg).
\end{align}
\end{theorem}
The main ingredient of the proof is the nc analog of Putinar's Positivstellensatz, stated in Theorem~\ref{th:densePsatz}.\\

Let ${\cD_{\frakg}}$ be as in~\eqref{eq:newDS} with the additional quadratic constraints~\eqref{eq:additional}.
Let ${\cM(\frakg)^{\sparse}}$ be as in~\eqref{eq:sparseMS} and let us define ${\cM(\frakg)_r^{\sparse}}$ in the same way as the truncated quadratic module ${\cM(\frakg)_r}$ in~\eqref{eq:MS2d}.
Now, let us state the sparse variant of the primal-dual hierarchy~\eqref{eq:constr_eigmin_dual}-\eqref{eq:constr_eigmin_primal} of lower bounds for $\lambda_{\min} (f, \frakg)$.

For $r \geq r_{\min}$, the sparse variant of \ac{SDP}~\eqref{eq:constr_eigmin_primal} is
\begin{equation}
\label{eq:sparse_constr_eigmin_primal}
\begin{array}{rll}
\eta^{r}_{\cs}(f,\frakg) \coloneqq &\inf\limits_{{\y}}  & L_{\y}(f)  \\
&\,\rm{ s.t.}& {y_{1} = 1},\quad{{\M_r}(\y,I_k)} \succeq 0, k\in[p]\\
&&  {{\M_{r-d_j}}(g_j\y,I_k)} \succeq 0 ,  \quad j \in J_k,k\in[p]
\end{array}
\end{equation}
whose dual is the sparse variant of \ac{SDP}~\eqref{eq:constr_eigmin_dual}:
\begin{equation}
\label{eq:sparse_constr_eigmin_dual}
\begin{array}{rll}
\lambda_{\cs}^{r} (f,\frakg) \coloneqq &\sup\limits_{b} & b \\
&\,\rm{ s.t.}& f - b \in \cM(\frakg)^{\sparse}_r.
\end{array}
\end{equation}
An $\varepsilon$-neighborhood of 0 is the set $\mathcal{N}_{\varepsilon}$ defined for a given $\varepsilon > 0$ by
\[
\mathcal{N}_{\varepsilon} \coloneqq \bigcup_{k \in \N^*} \left\lbrace (A_1,\dots,A_n) \in (\Sbb_k)^n : \varepsilon^2 - \sum_{i=1}^n A_i^2 \succeq 0  \right\rbrace.
\]
%
%
\begin{proposition}
\label{prop:constr_eigmin_slater}
Let $\{f\}\cup\frakg\subseteq\SymRX$. Assume that ${\cD_{\frakg}}$ contains an $\varepsilon$-neighborhood of 0 and that ${\cD_{\frakg}}$ is as in~\eqref{eq:newDS} with the additional quadratic constraints~\eqref{eq:additional}.
Then \ac{SDP}~\eqref{eq:sparse_constr_eigmin_primal} admits strictly feasible solutions. As a result, there is no duality gap between \ac{SDP}~\eqref{eq:sparse_constr_eigmin_primal} and its dual \eqref{eq:sparse_constr_eigmin_dual}.
\end{proposition}

Moreover, we have the following convergence result implied by Theorem \ref{th:sparsePsatz}.
\begin{theorem}\label{th:sparse_constr_eig}
Let $\{f\}\cup\frakg \subseteq \SymRX$. Assume that  ${\cD_{\frakg}}$ is as in~\eqref{eq:newDS} with the additional quadratic constraints~\eqref{eq:additional}.
Let Assumption~\ref{hyp:sparsityRIP} hold.
Then, one has
\begin{align}
\label{eq:sparse_cvg_eig}
\lim_{r \to \infty} \eta_{\cs}^{r}(f,\frakg) = \lim_{r \to \infty}  \lambda_{\cs}^{r} (f,\frakg) = \lambda_{\min} (f, \frakg).
\end{align}
\end{theorem}
As for the unconstrained case, there is no sparse variant of the ``perfect'' Positivstellensatz, for constrained eigenvalue optimization over convex nc semialgebraic sets \citencsparse[Chapter 4.4]{burgdorf16}, such as those associated either to the sparse nc ball $\Bbb^{\sparse} \coloneqq \{1 - \sum_{i \in I_1} x_i^2, \dots, 1 - \sum_{i \in I_p} x_i^2 \}$ or the nc polydisc $\Dbb \coloneqq \{1 - x_1^2,\dots,$ $1-x_n^2 \}$.
Namely, for an nc polynomial $f$ of degree $2 d + 1$,
 computing only \ac{SDP}~\eqref{eq:sparse_eigmin_primal} with optimal value $\lambda_{\sparse}^{d+1}(f, \frakg)$ when $\frakg = \Bbb^{\sparse}$ or $\frakg = \Dbb$
 does not suffice
 to obtain the value of $\lambda_{\min}(f,\frakg)$.
This is explained in Example~\ref{ex:nosparseEigBall} below.
\begin{example}
\label{ex:nosparseEigBall}
Let us consider a randomly generated cubic polynomial $f = f_1 + f_2$ with
\begin{align*}
f_1 = \,& 4 - x_1 + 3x_2 - 3x_3 - 3x_1^2 - 7x_1x_2 +   6x_1x_3 - x_2x_1 -5x_3x_1 + 5x_3x_2 \\
& -  5x_1^3 - 3x_1^2 x_3 + 4x_1x_2x_1 -  6x_1x_2x_3 + 7x_1x_3x_1
 + 2x_1x_3x_2 -   x_1x_3^2 \\
 & - x_2x_1^2 + 3x_2x_1x_2 -   x_2x_1x_3 - 2 x_2^3 - 5 x_2^2 x_3
 -   4x_2x_3^2 - 5x_3x_1^2 \\
 & + 7x_3x_1x_2 +   6x_3x_2x_1 - 4x_3x_2x_2 - x_3^2 x_1 -   2x_3^2 x_2 + 7x_3^3 , \\
f_2  = \,& -1 + 6x_2 + 5x_3 + 3x_4 - 5x_2^2 + 2x_2x_3 +   4x_2x_4 - 4x_3x_2 + x_3^2 - x_3x_4 \\
& +   x_4x_2 - x_4x_3 + 2x_4^2 - 7x_2^3 +   4x_2x_3^2 + 5x_2x_3x_4 - 7x_2x_4x_3 -   7x_2x_4^2 \\
& + x_3x_2^2 + 6x_3x_2x_3 -   6x_3x_2x_4 - 3x_3^2 x_2 - 7x_3^2x_4 +   6x_3x_4x_2 \\
& - 3x_3x_4x_3 - 7x_3x_4^2 +   3x_4x_2^2 - 7x_4x_2x_3 - x_4x_2x_4 -   5x_4x_3^2  \\
& + 7x_4x_3x_4 + 6x_4^2 x_2 -   4 x_4^3,
\end{align*}
and the nc polyball $\frakg = \Bbb^{\sparse} = \{1-x_1^2-x_2^2-x_3^2, 1-x_2^2-x_3^2 - x_4^2 \}$ corresponding to $I_1 = \{1,2,3\}$ and $I_2 = \{2,3,4\}$.
Then, one has $
\lambda^2_{\sparse} (f, \frakg ) \simeq
-27.536 < \lambda^3_{\sparse} (f,\frakg) \simeq -27.467
\simeq \lambda_{\min}^2 (f,\frakg) = \lambda_{\min} (f,\frakg)$.
In Appendix \ref{sec:tssos_nc}, we provide a Julia script to compute these bounds.
\end{example}
\subsection{Extracting optimizers}
Here, we explain how to extract a pair
of optimizers
$(\underline{A},\vb)$ for the eigenvalue optimization problems
when the flatness and irreducibility conditions of Theorem~\ref{th:sparse_flat} hold.
We apply the $\sparsegns$ procedure on the optimal solution of \ac{SDP}~\eqref{eq:sparse_eigmin_primal} in the unconstrained case or \ac{SDP}~\eqref{eq:sparse_constr_eigmin_primal} in the constrained case.
\begin{proposition}
\label{prop:sparse_eig_flat}
Given $f$ as in Theorem~\ref{th:sparse_eig_nogap}, let us assume that \ac{SDP}~\eqref{eq:sparse_eigmin_primal} (with $d$ being replaced by $d+1$) yields an optimal solution $\y$ associated to $\eta_{\cs}^{d+1}(f)$.
If the sequence $\y$ satisfies the  flatness (H1) and irreducibility (H2) conditions stated in Theorem~\ref{th:sparse_flat}, then one has
\[\lambda_{\min} (f) = \eta_{\cs}^{d+1}(f) = L_{\y}(f).\]
\end{proposition}
We can extract optimizers for the unconstrained minimal eigenvalue problem~\eqref{eq:eigmin} thanks to the following algorithm.

\begin{algorithm}\caption{$\sparseeiggns$}\label{algorithm:sparse_eig}
\begin{algorithmic}[1]
\Require $f\in\SymRX_{2d}$ satisfying Assumption~\ref{hyp:sparsityRIP}
\Ensure $\underline{A}$ and $\vb$
\State Compute $\eta_{\cs}^{d+1}(f)$ by solving \ac{SDP}~\eqref{eq:sparse_eigmin_primal}
\If{\ac{SDP}~\eqref{eq:sparse_eigmin_primal} is unbounded or its optimum is not attained}
\State Stop
\EndIf
\State Let ${\M_{d+1}}(\y)$ be an optimizer of \ac{SDP}~\eqref{eq:sparse_eigmin_primal}
\State Compute $\underline{A},\vb \coloneqq \sparsegns (\M_{d+1}(\y))$
%
\end{algorithmic}
\end{algorithm}
%
In the constrained case, the next result is a direct corollary of Theorem~\ref{th:sparse_flat}.
\begin{corollary}\label{th:sparse_cons_eig}
Let $\{f\}\cup\frakg\subseteq\SymRX$, and assume that ${\cD_{\frakg}}$ is as in~\eqref{eq:newDS} with the additional quadratic constraints~\eqref{eq:additional}.
Suppose Assumptions~\ref{hyp:sparsityRIP}(i)-(ii) hold.
Let $\y$ be an optimal solution of \ac{SDP}~\eqref{eq:sparse_constr_eigmin_primal} with optimal value $\eta_{\cs}^{r}(f,\frakg)$ for $r \geq r_{\min} + \delta$, such that $\y$ satisfies the assumptions of Theorem~\ref{th:sparse_flat}.
Then, there exist $t \in \N^*$, $\underline{A} \in {\cD_{\frakg}^t}$ and a unit vector $\vb$ such that
\[
\lambda_{\min}(f,\frakg) =
\langle f(\underline{A}) \vb , \vb \rangle =
\eta_{\cs}^{r}(f,\frakg).
\]
\end{corollary}
\if{
\begin{remark} \rm
\end{remark}
}\fi

\begin{example}
Consider the sparse polynomial $f=f_1+f_2$ from Example~\ref{ex:nosparseEigBall}. The moment matrix ${{\M_3}(\y)}$ obtained
by solving \eqref{eq:sparse_constr_eigmin_primal} with $r=3$
satisfies the  flatness (H1) and irreducibility (H2) conditions of Theorem~\ref{th:sparse_flat}.
 We can thus apply the $\sparsegns$ algorithm yielding
 \begin{align*}
A_1&=\left[\begin{array}{rrrr}
   0.0059 &  0.0481 &   0.1638&    0.4570\\
    0.0481&  -0.2583 &   0.5629&   -0.2624\\
    0.1638&   0.5629  &  0.3265 &  -0.3734\\
    0.4570&  -0.2624   &-0.3734  & -0.2337
\end{array}\right] \\
A_2&=\left[\begin{array}{rrrr}
   -0.3502 &   0.0080&    0.1411&    0.0865\\
    0.0080  & -0.4053 &   0.2404 &  -0.1649\\
    0.1411   & 0.2404  & -0.0959  &  0.3652\\
    0.0865   &-0.1649   & 0.3652   & 0.4117
\end{array}\right] \\
A_3&=\left[\begin{array}{rrrr}
 -0.7669 &  -0.0074 &  -0.1313  & -0.0805\\
   -0.0074&   -0.4715&   -0.2238 &   0.1535\\
   -0.1313 &  -0.2238 &   0.0848  & -0.3400\\
   -0.0805  &  0.1535  & -0.3400   &-0.2126
\end{array}\right] \\
A_4&=\left[\begin{array}{rrrr}
    0.3302&   -0.1839&    0.1811&   -0.0404\\
   -0.1839 &  -0.1069 &   0.5114 &  -0.0570\\
    0.1811  &  0.5114  &  0.1311  & -0.3664\\
   -0.0404   &-0.0570   &-0.3664   & 0.4440
\end{array}\right]
\end{align*}
 where
 \[
f(\underline A)
=\left[\begin{array}{rrrr}
   -10.3144 &   3.9233 &  -5.0836 &   -7.7828\\
    3.9233  &  1.8363  &  4.5078 &  -7.5905\\
   -5.0836   & 4.5078   &-19.5827  &  13.9157\\
   -7.7828   &-7.5905    &13.9157   & 8.3381
\end{array}\right]
 \]
 has  minimal eigenvalue $-27.4665$ with unit eigenvector
 \[\vb=\begin{bmatrix} 0.1546 &  -0.2507 &   0.8840  & -0.3631\end{bmatrix}^\intercal.\]
 In this case all the ranks involved are equal to four. So
$A_2$ and $A_3$ are computed from ${{\M_3}(\y,I_1\cap I_2)}$, after
an appropriate basis change $A_1$ (and the same $A_2,A_3$) is obtained
from ${{\M_3}(\y,I_1)}$, and finally $A_4$ is computed from ${{\M_3}(\y,I_2)}$.
\end{example}


\if{
\section{Trace optimization}
\label{sec:trace}

The aim of this section is to provide SDP relaxations allowing one to under-approximate the smallest trace of an nc polynomial on a semialgebraic set.
In Section~\ref{sec:sparse_tracial_gns}, we provide a sparse tracial representation for tracial linear functionals.
In Section~\ref{sec:unconstr_trace}, we address the unconstrained trace minimization problem.
As in Section~\ref{sec:unconstr_eig}, we compute a lower bound on the smallest trace via SDP.
The constrained case is handled in Section~\ref{sec:constr_trace}, where we derive a hierarchy of lower bounds converging to the minimal trace, assuming that the quadratic module is Archimedean and that \ac{RIP} holds (Assumption~\ref{hyp:sparsityRIP}).
Most proofs are similar to the ones of eigenvalue problems addressed in Section~\ref{sec:eig}.

We start this section by introducing useful notations about commutators and trace zero polynomials.
Given $g,h \in \RX$, the nc polynomial $[g,h] \coloneqq g h - h g$ is called a \emph{commutator}.
Two nc polynomials $g, h \in \RX$ are called \emph{cyclically equivalent} ($g  \cyc h$) if $g - h$ is a sum of commutators.
Given $\frakg \subseteq \SymRX$ with corresponding quadratic module $\cM(\frakg)$ and truncated variant ${\cM(\frakg)_d}$, one defines $\Theta(\frakg)_d \coloneqq \{g \in \SymRX_{2d} : g \cyc h \text{ for some } h \in {\cM(\frakg)_d} \}$ and $\Theta(\frakg) \coloneqq \bigcup_{d \in \N} \Theta(\frakg)_d$.
In this case, $\Theta(\frakg)$ stands for the \emph{cyclic quadratic module} generated by $\frakg$ and $\Theta(\frakg)_d$ stands for the \emph{truncated cyclic quadratic module} generated by $\frakg$.\\
For $\frakg \subseteq \SymRX$  and ${\cD_{\frakg}}$ as in~\eqref{eq:newDS} with the additional quadratic constraints~\eqref{eq:additional}, let us define $\Theta(\frakg)^k_d \coloneqq \{g \in \SymRX_{2 d} : g \cyc h \text{ for some } h \in \cM(\frakg)^k_d \}$, $\Theta(\frakg)^k \coloneqq \bigcup_{d \in \N} \Theta(\frakg)^k_d$, for all $k\in[p]$ and the sum
\begin{align}
\label{eq:sparse_cyclic_module}
\Theta(\frakg)^{\sparse}_d \coloneqq \Theta(\frakg)^1_d + \dots + \Theta(\frakg)^p_d ,
\end{align}
as well as $\Theta(\frakg)^{\sparse} \coloneqq \bigcup_{d \in \N} \Theta(\frakg)^{\sparse}_d$.
If $\frakg$ is empty, we drop the $\frakg$ in the above notations.\\
The normalized trace of a matrix $\A \in \Sbb_t$ is given by $\Trace \A = \frac{1}{t} \trace \A$.
An nc polynomial $g \in \SymRX$ is called a \emph{trace zero} nc polynomial if $\Trace (g(\underline{A})) = 0$, for all $\underline{A} \in \Sbb^n$.
This is equivalent to $g \cyc 0$ (see e.g.~\citencsparse[Proposition~2.3]{klep2008sums}).\\
{For a given nc polynomial $g$, the cyclic degree of $g$, denoted by $\cdeg(g)$, is
the smallest degree of a polynomial cyclically equivalent to $g$.}
\subsection{Sparse tracial representations}
\label{sec:sparse_tracial_gns}
~\\
The next theorem  allows one to obtain a sparse tracial representation of a tracial linear functional, under the same  flatness and irreducibility conditions stated in Theorem~\ref{th:sparse_flat}.
This is a sparse variant of~\citencsparse[Theorem~1.71]{burgdorf16}.
\begin{theorem}
\label{th:sparse_flat_tracial}
Let $\frakg  \subseteq \SymRX_{2d}$,  and assume that the semialgebraic set ${\cD_{\frakg}}$ is as in~\eqref{eq:newDS} with the additional quadratic constraints~\eqref{eq:additional}.
Let Assumption~\ref{hyp:sparsityRIP}(i) hold.
Set $\delta \coloneqq \max \{ \lceil \deg (g)/2 \rceil : g \in \frakg\cup {\{1\}} \}$.
Let $L : \RX_{2 d + 2 \delta} \to \R$ be a unital tracial linear functional satisfying $L(\Theta(\frakg)^{\sparse}_d) \subseteq \R^{\geq 0}$.
Assume that the flatness (H1) and irreducibility (H2) conditions of Theorem~\ref{th:sparse_flat} hold.
Then there are finitely many $n$-tuples $\underline{A}^{(j)}$ of symmetric matrices in ${\cD_{\frakg}^r}$ for some $r  \in \N$,  and positive scalars $\lambda_j$ with $\sum_j \lambda_j = 1$, such that for all $f \in \RXI{1}_{2 d} + \dots + \RXI{p}_{2 d}$, one has:
\begin{align}
\label{eq:tracial}
L(f) =  \sum_j \lambda_j \Trace {f ( \underline{A}^{(j)} ) } .
\end{align}
\end{theorem}
\subsection{Unconstrained trace optimization}
\label{sec:unconstr_trace}
~\\
Given $f \in \SymRX$, the \emph{trace-minimum} of $f$ is obtained by solving the following  optimization problem
\begin{align}
\label{eq:tracemin}
\Trace_{\min}(f) \coloneqq \inf \{\Trace f(\underline{A}) : \underline{A} \in \Sbb^n \} ,
\end{align}
which is equivalent to
\begin{align}
\label{eq:tracemin2}
\Trace_{\min}(f) = \sup \{ b : \Trace (f - b) (\underline{A}) \geq 0 , \forall \underline{A} \in \Sbb^n \} ,
\end{align}
If the cyclic degree of $f$ is odd, then $\Trace_{\min}(f) = -\infty$, thus let us assume that $2d = \cdeg(f)$.
To approximate $\Trace_{\min}(f)$ from below, one considers the following relaxation:
\begin{align}
\label{eq:trace_dual}
\Trace^d(f) = \sup \{ b : f - b \in \Theta_d \} ,
\end{align}
whose dual is
\begin{equation}
\label{eq:trace_primal}
\begin{aligned}
L_{\Theta}^d(f) \coloneqq \inf\limits_{{\y}} \quad  & \langle {{\M_d}(\y)}, \G_f \rangle  \\
\rm{ s.t.}
\quad & ({{\M_d}(\y)})_{u,v} = ({{\M_d}(\y)})_{w,z}  , \quad \text{for all } u^\star v \cyc w^\star z , \\
\quad & {y_{1}} = 1 , \quad  {{\M_d}(\y)} \succeq 0 , \\
\quad & {L : \RX_{2 d}  \to \R \,
\text{ linear}},
\end{aligned}
\end{equation}
One has $ \Trace^d(f) = L_{\Theta}^d(f) \leq \Trace_{\min}(f)$, where the inequality comes from~\citencsparse[Lemma~5.2]{burgdorf16} and the equality results from the strong duality between SDP~\eqref{eq:trace_primal} and SDP~\eqref{eq:trace_dual}, see e.g.~\citencsparse[Theorem~5.3]{burgdorf16} for a proof.
In addition, if the optimizer ${{\M_d}(\y)}^{\opt}$ of SDP~\eqref{eq:trace_primal} satisfies the flatness condition, i.e., the linear functional underlying ${{\M_d}(\y)}^{\opt}$ is $1$-flat (see Definition~\ref{def:flatextension}), then the above relaxations are exact and one has $ \Trace^d(f) = L_{\Theta}^d(f) = \Trace_{\min}(f)$. This exactness result is stated in~\citencsparse[Theorem~5.4]{burgdorf16}.

For a given nc polynomial $f = f_1 + \dots + f_p$, with $f_k \in \SymRXI{k}_{2d}$, for all $k\in[p]$, we consider the following sparse variant of SDP~\eqref{eq:trace_primal}:
\begin{equation}
\label{eq:sparse_trace_primal}
\begin{aligned}
L^{d}_{\Theta,\sparse}(f) = \inf\limits_{{\y}} \quad  & \sum_{k=1}^p \langle {{\M_d}(\y,I_k)}, \G_{f_k} \rangle \\
\rm{ s.t.}
\quad & ({{\M_d}(\y,I_k)})_{u,v} = ({{\M_d}(\y,I_k)})_{w,z}  , \quad \text{for all } u^\star v \cyc w^\star z ,  \\
\quad & y_{1} = 1, \quad {{\M_d}(\y,I_k)}  \succeq 0 ,   \quad k\in[p],\\
\quad & {L : \RXI{1}_{2 d} + \dots + \RXI{p}_{2 d} \to \R \,
\text{ linear}},
\end{aligned}
\end{equation}
whose dual is the sparse variant of SDP~\eqref{eq:trace_dual}:
\begin{equation}
\label{eq:sparse_trace_dual}
\begin{aligned}
\Trace_{\sparse}^{d}(f) = \sup\limits_{\lambda} \quad  & \lambda  \\
\rm{ s.t.}
\quad & f  - \lambda \in \Theta^{\sparse}_d .
\end{aligned}
\end{equation}
Now, we are ready to state the sparse variant of~\citencsparse[Theorem~5.3]{burgdorf16}.
\begin{theorem}
\label{th:sparse_trace_nogap}
Let $f \in \SymRX$ of degree $2 d$, with $f = f_1 + \dots + f_p$, $f_k \in \SymRXI{k}_{2d}$, for all $k\in[p]$. 
There is no duality gap between SDP~\eqref{eq:sparse_trace_primal} and SDP~\eqref{eq:sparse_trace_dual}, namely $\Trace_{\sparse}^{d}(f) = L^{d}_{\Theta,\sparse}(f)$.
\end{theorem}
As for unconstrained eigenvalue optimization, one can retrieve the solution of the initial trace minimization problem under the same assumptions as Theorem~\ref{th:sparse_flat}.
This is stated in the next proposition, which is the sparse variant of~\citencsparse[Theorem~5.4]{burgdorf16}.
\begin{proposition}
\label{prop:sparse_trace_flat}
Let $f$ be as in Theorem~\ref{th:sparse_trace_nogap}, and  assume that SDP~\eqref{eq:sparse_trace_primal} admits an optimal  solution ${{\M_d}(\y)}$.
If the linear functional $L$ underlying ${{\M_d}(\y)}$ satisfies the flatness (H1) and irreducibility (H2) conditions stated in Theorem~\ref{th:sparse_flat}, then
\[ \Trace_{\sparse}^{d}(f) = L^{d}_{\Theta,\sparse}(f) = \Trace_{\min}(f) . \]
\end{proposition}
In practice, Proposition~\ref{prop:sparse_trace_flat} allows one to derive an algorithm similar to the $\sparseeiggns$ procedure (described in Algorithm~\ref{algorithm:sparse_eig}) to find flat optimal solutions for the unconstrained trace problem.
\subsection{Constrained trace optimization}
\label{sec:constr_trace}
~\\
In this subsection, we provide the sparse tracial version of Lasserre's hierarchy to minimize the trace of a noncommutative polynomial on a semialgebraic set.
Given $f \in \SymRX$ and $\frakg = \{g_1,\dots,g_{m} \}$ $\subset \SymRX$ as in~\eqref{eq:DS}, let us define  $\Trace_{\min} (f, S)$ as follows:
\begin{align}
\label{eq:constr_trace}
\Trace_{\min}(f,\frakg) \coloneqq \inf \{\Trace f(\underline{A}) : \underline{A} \in {\cD_{\frakg}} \} .
\end{align}
%
Since an infinite-dimensional Hilbert space does not admit a trace, we obtain lower bounds on the minimal trace by considering a particular subset of ${\cD_{\frakg}^\infty}$.
This subset is obtained by restricting from the algebra of all bounded operators $\mathcal{B}(\mathcal{H})$ on a Hilbert space $\mathcal{H}$ to finite von Neumann algebras~\citencsparse{Takesaki03} of type I and type II.
We introduce $\Trace_{\min}(f,\frakg)^{\II_1}$ as the trace-minimum of $f$ on ${\cD_{\frakg}^{\II_1}}$.
This latter set is defined as follows (see~\citencsparse[Definition~1.59]{burgdorf16}):
\begin{definition}
\label{def:DSII}
Let $\mathcal{F}$ be a type-$\II_1$-von Neumann algebra~\citencsparse[Chapter~5]{Takesaki03}.
Let us define $\cD_{\frakg}^{\mathcal{F}}$ as the set of all tuples $\underline{A} = (A_1,\dots,A_n) \in \mathcal{F}^n$ making $g(\underline{A})$ a positive semidefinite operator for every $g \in \frakg$.
The von Neumann semialgebraic set ${\cD_{\frakg}^{\II_1}}$ generated by $\frakg$ is defined as
\[
{\cD_{\frakg}^{\II_1}} \coloneqq \bigcup_{\mathcal{F}} \mathcal{D}_{\frakg}^{\mathcal{F}} ,
\]
where the union is over all type-$\II_1$-von Neumann algebras with separable predual.
\end{definition}
By~\citencsparse[Proposition~1.62]{burgdorf16}, if $f\in \Theta(\frakg)$, then $\Trace f(\underline{A}) \geq 0$, for all $A \in {\cD_{\frakg}}$ and $A \in {\cD_{\frakg}^{\II_1}}$.
Since ${\cD_{\frakg}}$ can be modeled by ${\cD_{\frakg}^{\II_1}} $, one has $\Trace_{\min}(f,\frakg)^{\II_1} \leq \Trace_{\min}(f,\frakg)$.
With $r_{\min}$ being defined as in Section~\ref{sec:constr_eig},
one can approximate $\Trace_{\min}(f,\frakg)^{\II_1}$ from below via the following hierarchy of SDP programs, indexed by $r \geq r_{\min}$:
\begin{align}
\label{eq:constr_trace_dual}
\Trace^{r}(f,\frakg) = \sup \{ b : f - b \in \Theta(\frakg)_r \} ,
\end{align}
whose dual is
\begin{equation}
\label{eq:constr_trace_primal}
\begin{aligned}
L_{\Theta}^r(f,\frakg) \coloneqq \inf\limits_{{\y}} \quad  & \langle {{\M_r}(\y)}, \G_f \rangle  \\
\rm{ s.t.}
\quad & ({{\M_r}(\y)})_{u,v} = ({{\M_r}(\y)})_{w,z}  , \quad \text{for all } u^\star v \cyc w^\star z , \\
\quad & y_{1} = 1, \\
\quad  & {{\M_r}(\y)} \succeq 0 ,
\quad  {{\M_{r-d_j}}(g_j \y)} \succeq 0 ,    \quad j \in [m] ,\\
\quad & {L : \RX_{2 d}  \to \R \,
 \text{ linear}}.
\end{aligned}
\end{equation}
If the quadratic module $\cM(\frakg)$ is Archimedean, the resulting hierarchy of SDP programs provides a sequence of lower bounds $\Trace^r (f,\frakg)$ monotonically converging to $\Trace_{\min}(f,\frakg)^{\II_1}$, see e.g. ~\citencsparse[Corollary~3.5]{burgdorf16}.

Next, we present a sparse variant hierarchy of SDP programs providing a sequence of lower bounds $\Trace_{\sparse}^{r} (f,\frakg)$ monotonically converging to $\Trace_{\min}(f,\frakg)^{\II_1}$.
Let $\frakg \cup \{f\} \subseteq \SymRX$  and let ${\cD_{\frakg}}$ be as in~\eqref{eq:newDS} with the additional quadratic constraints~\eqref{eq:additional}.
Let us define the sparse variant of SDP~\eqref{eq:constr_trace_primal}, indexed by $r \geq r_{\min}$:
\begin{equation}
\label{eq:sparse_constr_trace_primal}
\begin{aligned}
L^{r}_{\Theta,\sparse}(f,\frakg) = \inf\limits_{{\y}} \quad  & \sum_{k=1}^p \langle {{\M_r}(\y,I_k)}, \G_{f_k} \rangle \\
\rm{ s.t.}
\quad & ({{\M_r}(\y,I_k)})_{u,v} = ({{\M_r}(\y,I_k)})_{w,z}  , \quad \text{for all } u^\star v \cyc w^\star z ,  \\
\quad & y_{1} = 1, \\
\quad & {{\M_r}(\y,I_k)} \succeq 0 , \quad k\in[p] , \\
\quad &  {{\M_{r-d_j}}(g_jL,I_k)} \succeq 0 ,  \quad j\in J_k , \quad k\in[p] , \\
\quad & {L : \RXI{1}_{2 d} + \dots + \RXI{p}_{2 d} \to \R \,
\text{ linear}}.
\end{aligned}
\end{equation}
whose dual is the sparse variant of SDP~\eqref{eq:constr_trace_dual}:
\begin{align}
\label{eq:sparse_constr_trace_dual}
\Trace_{\sparse}^{r}(f,\frakg) = \sup \{ b : f - b \in \Theta(\frakg)_d^{\sparse} \} ,
\end{align}
With the same conditions as the ones assumed in Proposition~\ref{prop:constr_eigmin_slater} for constrained eigenvalue optimization, SDP~\eqref{eq:sparse_constr_trace_primal} admits strictly feasible solutions, so there is no duality gap between SDP~\eqref{eq:sparse_constr_trace_primal} and SDP~\eqref{eq:sparse_constr_trace_dual}.
The proof is the same since the constructed linear functional in Proposition~\ref{prop:constr_eigmin_slater} is tracial.
In order to prove convergence of the hierarchy of bounds given by the SDP~\eqref{eq:sparse_constr_trace_primal}-\eqref{eq:sparse_constr_trace_dual},
we need the following proposition, which is the sparse variant of~\citencsparse[Proposition ~1.63]{burgdorf16}.
\begin{proposition}
\label{prop:sparse_tracial_psatz}
Let $\frakg \cup \{f\} \subseteq \SymRX$  and let ${\cD_{\frakg}}$ be as in~\eqref{eq:newDS} with the additional quadratic constraints~\eqref{eq:additional}.
Let Assumption~\ref{hyp:sparsityRIP} hold.
Then the following are equivalent:
\begin{itemize}
\item[(i)] $\Trace f(\underline{A}) \geq 0$ for all $\underline{A} \in {\cD_{\frakg}^{\II_1}}$;
\item[(ii)] for all $\varepsilon > 0$, there exists $g \in {\cM(\frakg)^{\sparse}}$ with $f + \varepsilon \cyc g$.
\end{itemize}
\end{proposition}
Proposition~\ref{prop:sparse_tracial_psatz} implies the following convergence property.
\begin{corollary}
\label{coro:sparse_tracial_cvg}
Let $\frakg \cup \{f\} \subseteq \SymRX$  and let ${\cD_{\frakg}}$ be as in~\eqref{eq:newDS} with the additional quadratic constraints~\eqref{eq:additional}.
Let Assumption~\ref{hyp:sparsityRIP} hold.
Then
\[
\lim_{r \to \infty} \Trace^{r}_{\sparse}(f,\frakg) = \lim_{r \to \infty} L_{\Theta,\sparse}^{r}(f,\frakg) = \Trace_{\min}(f,\frakg)^{\II_1}
 .\]
\end{corollary}
To extract solutions of constrained trace minimization problems, we rely on the following variant of Theorem~\ref{th:sparse_flat_tracial}. It is, in turn, the tracial analog of Theorem~\ref{th:sparse_flat}.
\begin{proposition}
\label{prop:sparse_constr_trace_flat}
Let $\frakg  \subseteq \SymRX_{2d}$,  and assume that the semialgebraic set ${\cD_{\frakg}}$ is as in~\eqref{eq:newDS} with the additional quadratic constraints~\eqref{eq:additional}.
Let Assumption~\ref{hyp:sparsityRIP}(i) hold.
Set $\delta \coloneqq \max \{ \lceil \deg (g)/2 \rceil : g \in \frakg \cup {1} \}$.
Let ${{\M_r}(\y)}$ be an optimal solution of SDP~\eqref{eq:sparse_trace_primal} with value $L_{\Theta,\sparse}^{r}(f,\frakg)$, for $r \geq d + \delta$, such that $L$ satisfies the flatness (H1) and irreducibility (H2) conditions of Theorem~\ref{th:sparse_flat}.
Then there are finitely many $n$-tuples $\underline{A}^{(j)}$ of symmetric matrices in ${\cD_{\frakg}^t}$ for some $t \in \N$, and positive scalars $\lambda_j$ with $\sum_j \lambda_j = 1$ such that
\[L(f) =  \sum_j \lambda_j \Trace {f ( \underline{A}^{(j)})}.\]
In particular, one has $\Trace_{\min}(f, \frakg) = \Trace_{\min}(f,\frakg)^{\II_1} = L_{\Theta, \sparse}^{r}(f, \frakg)$.
\end{proposition}
As in the dense case~\citencsparse[Algorithm~5.1]{burgdorf16}, one can rely on Proposition~\ref{prop:sparse_constr_trace_flat} to provide a randomized algorithm to look for flat optimal solutions for the constrained trace problem~\eqref{eq:constr_trace}.
}\fi

\section{Overview of numerical experiments}

The aim of this section is to provide experimental comparison between the bounds given by the dense hierarchy
and the ones produced by our \ac{CS} variant.
%
The numerical results were obtained with the Julia package \texttt{NCTSSOS} employing $\mosek$ as an \ac{SDP} solver. The computation was carried out on Intel(R) Core(TM) i9-10900 CPU@2.80GHz with 64G RAM.

\subsection{An unconstrained problem}
In Table~\ref{table:unc}, we report results obtained for minimizing the eigenvalue of the nc variant of the chained singular function~\citencsparse{conn1988testing}:
\begin{align*}
	f =\sum_{i \in J} \bigl( (x_i + 10 x_{i+1})^2 &+ 5 (x_{i+2} - x_{i+3})^2 \\
	&+ (x_{i+1} - 2 x_{i+2})^4 + 10 (x_{i} - 10 x_{i+3})^4\bigr),
\end{align*}
where $J=[n-3]$ and $n$ is a multiple of $4$. We compute lower bounds on the minimal eigenvalue of $f$ for $n=40,80,120,160,200,240$.
For each value of $n$, ``mb'' stands for maximal sizes of PSD blocks involved either in the sparse relaxation \eqref{eq:sparse_eigmin_primal} or the dense relaxation \eqref{eq:eigmin_primal}.
As one can see, the size of the \ac{SDP} programs is significantly reduced after exploiting \ac{CS}, which is consistent with Remark~\ref{rk:sparsevsdense}.
In addition, the sparse approach turns out to be much more efficient and scalable than the dense approach.

\begin{table}[htbp]
\caption{Sparse versus dense approaches for minimizing eigenvalues of the chained singular function. mb: maximal size of PSD blocks, opt: optimum, time: running time in seconds. ``-'' indicates an out of memory error.}\label{table:unc}
\centering
\begin{tabular}{c|ccc|ccc}
\hline
\multirow{2}{*}{$n$}  & \multicolumn{3}{c|}{{\bf sparse}} & \multicolumn{3}{c}{{\bf dense}} \\
\cline{2-7}
& mb & opt & time & mb & opt & time \\
\hline
40 & 13 & 0 &  0.17 & 157 &0 &  142 \\
80  & 13 & 0 &  0.43 & - & -  & - \\
120 &13 &0 &  0.65 & - & -  & - \\
160 &13 &0   &0.89  & - & -  & - \\
200 &13 &$-0.0014$   &1.02  & - & -  & - \\
240 &13 &$-0.0016$  &1.28 & - & -  & - \\
\hline
\end{tabular}
\end{table}

\subsection{Bell inequalities}\label{sec:bell}
Upper bounds on quantum violations of Bell inequalities can be computed using eigenvalue maximization of nc polynomials.
The classical (also most concise) Bell inequality states that $v^\star (A_1 \otimes B_1 + A_1 \otimes B_2 + A_2 \otimes B_1 - A_2 \otimes B_2) v$ is at most $2$ for all separable states $v \in \C^k \otimes \C^k$ and self-adjoint $A_j, B_j \in \C^{k\times k}$ with $A_j^2 = B_j^2 = \I_k$.
The so-called Tsirelson's bound implies that the above quantity is at most $2 \sqrt{2}$ when one allows arbitrary states.
This bound on the maximum violation level can be obtained by eigenvalue-maximizing $a_1 b_1 + a_1 b_2 + a_2 b_1 - a_2 b_2$ under the constraints $a_j^2 = b_j^2 = 1$ and $a_i b_j =  b_j a_i$.
To show the potential benefits of our approach based on \ac{CS}, we consider the Bell inequality, called $I_{3322}$, and compute upper bounds of its maximum violation level.
The associated objective function is $f = a_1 (b_1+b_2+b_3) + a_2 (b_1+b_2-b_3) + a_3 (b_1-b_2) - a_1 - 2 b_1 - b_2$.
The set of constraints is $a_j^2=a_j$, $b_j^2=b_j$, and  $a_i b_j =  b_j a_i$.
Table \ref{table:bell} compares the efficiency and accuracy of the sparse approach with the dense one, for different relaxation orders. It can be seen that the sparse approach spends much less time while providing almost the same bounds.

\begin{table}[htbp]
\caption{Sparse versus dense approaches for maximizing the violation level of the Bell inequality $I_{3322}$. $r$ denotes the relaxation order.}\label{table:bell}
\centering
\begin{tabular}{c|ccc|ccc}
	\hline
	\multirow{2}{*}{$r$}  & \multicolumn{3}{c|}{{\bf sparse}} & \multicolumn{3}{c}{{\bf dense}} \\
	\cline{2-7}
	& mb & opt & time & mb & opt & time\\
	\hline
	2 & 28 & 0.2509398 &  0.01 & 13 &0.2590718 &  0.01 \\
	3  & 88 & 0.2508758 &  0.22 & 25 & 0.2512781  & 0.02 \\
	4 &244 &0.2508754 &  8.40 & 41 & 0.2509057  & 0.02 \\
	5 &628 &0.2508752   &456  & 61 & 0.2508774  & 0.04 \\
	6 &- &-  &- &85  & 0.2508754 & 0.09 \\
	\hline
\end{tabular}
%
\end{table}

\section{Notes and sources}
The main results presented in this chapter have been published in \citencsparse{ncsparse}.
Applications of interest connected with noncommutative optimization arise from quantum theory and quantum information science~\citencsparse{navascues2008convergent,pozas2019bounding} as well as control theory~\citencsparse{skelton1997unified,engineeringFRAG}.
Further motivation relates to the generalized Lax conjecture~\citencsparse{lax1957differential}, where the goal is to obtain computer-assisted proofs based on \ac{SOHS} in Clifford algebras~\citencsparse{netzer2014hyperbolic}.
The verification of noncommutative polynomial trace inequalities has also been motivated by a conjecture formulated by Bessis, Moussa and Villani (BMV) in 1975~\citencsparse{bessis1975monotonic}, which has been recently proved by Stahl \citencsparse{stahl2013proof} (see also the Lieb and Seiringer reformulation~\citencsparse{lieb2004equivalent}).
Further efforts focused on applications arising from bipartite quantum correlations~\citencsparse{Gribling18}, and matrix factorization ranks in~\citencsparse{Gribling19}.
In a related analytic direction, there has been recent progress on multivariate generalizations of the Golden-Thompson inequality and the Araki-Lieb-Thirring inequality~\citencsparse{GTineq1,GTineq2}.

There is a plethora of prior research in quantum information theory
involving reformulating problems as optimization of noncommutative polynomials.
One famous application is to characterize the set of quantum correlations.
Bell inequalities \citencsparse{bell1964einstein} provide a method to investigate entanglement, which allows two or more parties to be correlated in a non-classical way, and is often studied through the set of bipartite quantum correlations.
Such correlations consist of the conditional probabilities that two physically separated parties can generate by performing measurements on a shared entangled state.
These conditional probabilities satisfy some  inequalities  classically, but violate them in the quantum realm \citencsparse{clauser1969proposed}.

In this context, a given noncommutative polynomial in $n$ variables and of degree $2d$ is \ac{PSD} if and only if it decomposes as a \ac{SOHS}~\citencsparse{Helton02,McCullSOS}.
In practice, an \ac{SOHS} decomposition can be computed by solving an \ac{SDP} involving \ac{PSD} matrices of size $O(n^d)$, which is even larger than the size of the matrices involved in the commutative case.
\ac{SOHS} decompositions are also used for constrained optimization, either to minimize eigenvalues or traces of noncommutative polynomial objective functions, under noncommutative polynomial (in)equality constraints.
The optimal value of such constrained problems can be approximated, as closely as desired, while relying on the noncommutative analogue of Lasserre's hierarchy~\citencsparse{pironio2010convergent,cafuta2012constrained,nctrace}.
The $\ncsostools$~\citencsparse{cafuta2011ncsostools,burgdorf16} library can compute such approximations for optimization problems involving polynomials in noncommuting variables.
By comparison with the commutative case, the size $O(n^r)$ of the \ac{SDP} matrices at a given step $r$ of the noncommutative hierarchy becomes intractable even faster.

A remedy for unconstrained problems is to rely on the adequate noncommutative analogue of the standard Newton polytope method, which is called the \emph{Newton chip method} (see e.g., \citencsparse[\S2.3]{burgdorf16}) and can be further improved with the \emph{augmented Newton chip method} (see e.g.,~\citencsparse[\S2.4]{burgdorf16}), by removing certain terms which can never appear in an \ac{SOHS} decomposition of a given input.
As in the commutative case, the Newton polytope method cannot be applied for constrained problems.
When one cannot go from step $r$ to step $r+1$ in the hierarchy because of the computational burden, one can always consider matrices indexed by all terms of degree $r$ plus a fixed percentage of terms of degree $r+1$.
This is used for instance to compute tighter upper bounds for maximum violation levels of Bell inequalities~\citencsparse{pal2009quantum}.
Another trick, implemented in the $\ncpoltosdpa$ library~\citencsparse{wittek2015algorithm}, consists of exploiting simple equality constraints, such as ``$x^2 = 1$'', to derive substitution rules for variables involved in the \ac{SDP} relaxations.
Similar substitutions are performed in the commutative case by $\gloptipoly$~\citencsparse{gloptipoly}.

Proposition \ref{prop:ncGram} can be found, e.g., in~\citencsparse[\S2.2]{Helton02}.
The noncommutative analog of Putinar's Positivstellensatz is due to Helton and McCullough \citencsparse[Theorem~1.2]{Helton04}.
Lemma \ref{lemma:moment} is proved in \citencsparse[Lemma~1.44]{burgdorf16}.
Theorem \ref{th:gns} is proved in   \citencsparse[Theorem~1.27]{burgdorf16}.
For more details on amalgamation theory for $C^\star$-algebras, see, e.g., \citencsparse{Amalgam78,Voi83}.
Theorem \ref{th:amalgamation} can be found in \citencsparse{Amalgam78} or \citencsparse[Section 5]{Voi83}.
This amalgamation theory serves to prove our sparse representation result \ref{th:sparsePsatz}, originally stated in \citencsparse[Theorem 3.3]{ncsparse}.

The notion of flatness was exploited in a noncommutative  setting for the first time by McCullough \citencsparse{McCullSOS} in his proof of the Helton-McCullough theorem, cf.~\citencsparse[Lemma 2.2]{McCullSOS}.
In the dense case \citencsparse{pironio2010convergent} (see also ~\citencsparse[Chapter 21]{anjos2011handbook} and~\citencsparse[Theorem~1.69]{burgdorf16}) provides a first noncommutative variant for the eigenvalue problem.  See~\citencsparse{nctrace} for a similar construction for the trace problem.
The sparse version of the flat extension theorem is stated in \citencsparse[Theorem 4.2]{ncsparse} and
the {\tt SparseGNS} algorithm is explicitly given in \citencsparse[Algorithm 4.6]{ncsparse} for the case of two subsets of variables (the general case is similar).
Theorem~\ref{th:sparse_flat} can be seen as a noncommutative variant of the result by Lasserre stated in~\citencsparse[Theorem~3.7]{Las06}, related to the minimizer extraction in the context of sparse polynomial optimization.
In the sparse commutative case, Lasserre assumes flatness of each moment matrix indexed by the canonical basis of $\R[\x,I_k]_r$, for each $k\in [p]$, which is similar to our flatness condition~(H1).
{
The difference is that this technical flatness condition on each $I_k$ adapts to the degree of the constraint polynomials in variables in $I_k$, resulting in an adapted parameter $\delta_k$ instead of global $\delta$.
We could assume the same in Theorem~\ref{th:sparse_flat} but for the sake of simplicity, we assume that these parameters are all equal.
}
In addition, Lasserre assumes that each moment matrix indexed by the canonical basis of $\R[\ux, I_j \cap I_k)]_r$ is of rank one, for all pairs $(j,k)$ with $I_j \cap I_k \neq \emptyset$, which is the commutative analog of our irreducibility condition~(H2).

The absence of duality gap for unconstrained eigenvalue minimization is derived in the dense case, e.g., in~\citencsparse[Theorem~4.1]{burgdorf16} and relies on ~\citencsparse[Proposition~3.4]{mccullough2005noncommutative}, which says that the set of \ac{SOHS} polynomials is closed in the set of nc polynomials.
Proposition \ref{prop:sparseclosed} is a sparse version of this latter result, leading to Theorem \ref{th:sparse_eig_nogap}, originally proved in \citencsparse[Theorem 5.4]{ncsparse}.

The ``perfect'' Positivstellensatz for constrained eigenvalue optimization over convex nc semialgebraic set (e.g., the nc ball or the nc polydisc) is stated in~\citencsparse[\S4.4]{burgdorf16} or \citencsparse{hkmConvex}.
Proposition~\ref{prop:sparse_eig_flat} and Corollary~\ref{th:sparse_cons_eig} are the sparse variants of~\citencsparse[Proposition~4.4]{burgdorf16} and~\citencsparse[Theorem~4.12]{burgdorf16}, in the unconstrained and constrained settings, respectively.

As in the dense case~\citencsparse[Algorithm~4.2]{burgdorf16}, one can provide a randomized algorithm to look for flat optimal solutions for the constrained eigenvalue problem~\eqref{eq:constr_eigmin}.
The underlying reason which motivates this randomized approach is  work by Nie, who derives in~\citencsparse{NieRand14} a hierarchy of \ac{SDP} programs, with a random objective function, that converges to a flat solution (under mild assumptions).

The interested reader can find more about trace minimization of sparse polynomials in~\citencsparse[\S~6]{ncsparse} and more detailed numerical experiments in~\citencsparse[\S~7]{ncsparse}.
For a detailed account about maximal violation levels of  Bell inequalities (in particular the one mentioned in Table \ref{table:bell}), we refer to~\citencsparse{pal2009quantum}.
\ac{CS} can be exploited in a similar way to solve trace polynomial optimization problems \citencsparse{klep2022optimization}.

\input{ncsparse.bbl}

%% file: ncsparse.bbl
\providecommand{\etalchar}[1]{$^{#1}$}

%% file: ts.tex
\part{Term sparsity}\label{part:ts}
\chapter{The moment-SOS hierarchy based on term sparsity}\label{chap:tssos}
As emphasized earlier for distinct applications, exploiting \ac{CS} arising from a \ac{POP} may allow to significantly reduce the computational cost of the related hierarchy of \ac{SDP} relaxations assuming that the \ac{csp} is sufficiently sparse.
Nevertheless a \ac{POP} can be fairly sparse (namely, involving only a small number of terms) whereas its \ac{csp} is (nearly) dense. For instance, if some constraint (e.g., $1-\Vert\x\Vert_2^2\geq0$) involves all variables, then the \ac{csp} is dense.
On the other hand, instead of exploiting sparsity from the perspective of \emph{variables}, one can also exploit sparsity from the perspective of \emph{terms}, which leads to the notions of \emph{\ac{TS}} and \emph{\ac{tsp}}.

Roughly speaking, the \ac{tsp} can be also represented by a graph, which is called the \ac{tsp} graph.
But unlike the \ac{csp} graph, the nodes of the \ac{tsp} graph are monomials (coming from a monomial basis) and the edges of the graph grasp the links between monomials emerging from the related \ac{SOS} decomposition.
We are able to design an iterative procedure to enlarge the \ac{tsp} graph in order to iteratively exploit \ac{TS} for the given \ac{POP}.
Each iteration consists of two successive operations: (i) a support extension followed by (ii) a chordal extension.
In doing so we obtain a finite ascending chain of graphs:
\begin{equation*}
    G^{(1)}\subseteq G^{(2)}\subseteq \cdots\subseteq G^{(s)}=G^{(s+1)}.
\end{equation*}
Then combining this iterative procedure with the standard \ac{moment-SOS} hierarchy results in a two-level \ac{moment-SOS} hierarchy involving PSD blocks. When the sizes of the blocks are small, the associated \ac{SDP} relaxations can be drastically much cheaper to solve.

To some extent, \ac{TS} (working on the monomial level) is finer than \ac{CS} (working on the variable level) in describing sparsity for a \ac{POP}.
A natural idea for solving large-scale \ac{POP}s then is: first exploiting \ac{CS} to decompose variables into a set of cliques, and second exploiting \ac{TS} for each subsystem involving only one variable clique to further reduce the size of \ac{SDP}s.
This idea has been successfully carried out in \citets{cstssos} and will be presented later on, in Chapter \ref{chap:cstssos}.

Chapter \ref{sec:tssos-uncons} focuses on exploiting \ac{TS} for unconstrained \ac{POP}s.
We prove in Chapter \ref{sec:sdsos} that the resulting scheme is always more accurate than the framework based on {\em scaled diagonally dominant} sum of squares (SDSOS).
Chapter \ref{sec:tssos-cons} presents a \ac{TS} variant of the \ac{moment-SOS} hierarchy of \ac{SDP} relaxations for general \ac{POP}s with compact constraints.
Next, we provide in Chapter \ref{sec:tssos-basis} more sophisticated algorithms to reduce even further the size of the resulting \ac{SDP} relaxations.
Chapter \ref{sec:signtssos} explains the relation between the block structures arising from the \ac{TS}-adapted relaxations and the one provided by sign symmetries.
Eventually, we illustrate in Chapter \ref{sec:benchtssos} the efficiency and the accuracy of the \ac{TS}-adapted relaxations on benchmarks coming from the global optimization literature and networked systems stability.


\section{The TSSOS hierarchy for unconstrained POPs}\label{sec:tssos-uncons}
In this section, we describe an iterative procedure to exploit \ac{TS} for the primal-dual \ac{moment-SOS} relaxations of unconstrained \ac{POP}s. Recall the formulation of an unconstrained \ac{POP}:
\begin{equation}\label{upop}
\P:\quad f_{\min}\coloneqq\inf\,\{f(\x):\x\in\R^n\}=\sup\,\{b:f-b\ge0\},
\end{equation}
where $f=\sum_{\a}f_{\a}\x^{\a}\in\R[\x]$. Suppose that $f$ is of degree $2d$ with $\supp(f)=\sA$ (w.l.o.g. assuming $\mathbf{0}\in\sA$) and $\x^{\Basis}\coloneqq(\x^{\b})_{\b\in\Basis}$ is a monomial basis arranged with resepct to any fixed ordering. For convenience, we slightly abuse notation in the sequel and denote by $\Basis$ (resp.~$\b$) instead of $\x^{\Basis}$ (resp.~$\x^{\b}$) a monomial basis (resp.~a monomial).
One may choose $\Basis$ to be the standard monomial basis $\N^n_{d}$. But when $f$ is sparse, the following theorem due to Reznick \citets{reznick} allows us to use a (possibly) smaller monomial basis by considering Newton polytopes. Recall that for a polynomial $f\in\R[\x]$, the \emph{Newton polytope} of $f$, denoted by $\New(f)$, is the convex hull generated by its support.
\begin{theorem}\label{reznick}
If $f,p_i\in\R[\x]$ and $f=\sum_{i}p_i^2$, then $\New(p_i)\subseteq\frac{1}{2}\New(f)$.
\end{theorem}
As an immediate corollary, we can take the integer points in half of the Newton polytope of $f$ to form a monomial basis, i.e.,
\begin{equation}\label{sec2-eq3}
\Basis=\frac{1}{2}\New(f)\cap\N^n\subseteq\N^n_{d}.
\end{equation}

\begin{example}
Let $f=4x_1^4x_2^6+x_1^2-x_1x_2^2+x_2^2$. Then $\supp(f)=\{(4,6),(2,0),(1,2),$ $(0,2)\}$ and $\Basis=\frac{1}{2}\New(f)\cap\N^n=\{(1,0),(2,3),(0,1),(1,2),(1,1)\}$ (Figure \ref{newton}).

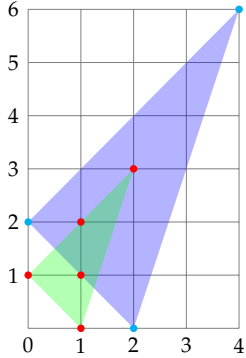
\begin{figure}[htbp]
\centering
\begin{tikzpicture}[scale=0.7]
\draw[help lines] (0,0) grid (4,6);
\fill[fill=blue,fill opacity=0.3] (2,0)--(4,6)--(0,2);
\fill[fill=cyan] (2,0) circle (2pt);
\fill[fill=cyan] (4,6) circle (2pt);
\fill[fill=cyan] (1,2) circle (2pt);
\fill[fill=cyan] (0,2) circle (2pt);
\fill[fill=green,fill opacity=0.3] (1,0)--(2,3)--(0,1);
\fill[fill=red] (1,0) circle (2pt);
\fill[fill=red] (2,3) circle (2pt);
\fill[fill=red] (0,1) circle (2pt);
\fill[fill=red] (1,2) circle (2pt);
\fill[fill=red] (1,1) circle (2pt);
\node[below] (1) at (1,0) {{\footnotesize $1$}};
\node[below] (11) at (0,0) {{\footnotesize $0$}};
\node[below] (2) at (2,0) {{\footnotesize $2$}};
\node[below] (3) at (3,0) {{\footnotesize $3$}};
\node[below] (4) at (4,0) {{\footnotesize $4$}};
\node[left] (5) at (0,1) {{\footnotesize $1$}};
\node[left] (6) at (0,2) {{\footnotesize $2$}};
\node[left] (7) at (0,3) {{\footnotesize $3$}};
\node[left] (8) at (0,4) {{\footnotesize $4$}};
\node[left] (9) at (0,5) {{\footnotesize $5$}};
\node[left] (10) at (0,6) {{\footnotesize $6$}};
\end{tikzpicture}
\caption{The monomial basis given by half of the Newton polytope (marked by red).}\label{newton}
\end{figure}
\end{example}

Given a monomial basis $\Basis$ and a sequence $\y\subseteq\R$, the moment matrix $\M_{\Basis}(\y)$ associated with $\Basis$ and $\y$ is the block of the moment matrix $\M_{d}(\y)$ indexed by $\Basis$. Then the moment relaxation of $\P$ in the monomial basis $\Basis$ is given by
\begin{equation}\label{sec4-eq3}
\P_{\text{mom}}:\quad\begin{array}{rll}
f_{\text{mom}} \coloneqq&\inf\limits_{\y}&L_{\y}(f)\\
&\rm{ s.t.}&\M_{\Basis}(\y)\succeq0\\
&&y_{\mathbf{0}}=1
\end{array}
\end{equation}



%

For a graph $G(V,E)$ with $V\subseteq\N^n$, let the support of $G$ be given by
\begin{equation}
\supp(G) \coloneqq \left\{\b+\g\mid\b=\g\text{ or }\{\b,\g\}\in E\right\}.
\end{equation}
We define $G^{\text{tsp}}$ to be the graph with nodes $V=\Basis$ and with edges
\begin{equation}\label{e0}
    E(G^{\text{tsp}})=\left\{\{\b,\g\}\mid\b\ne\g\in V,\,\b+\g\in\sA\cup(2\Basis)\right\},
\end{equation}
which is called the \ac{tsp} graph associated with $f$.

Starting with the initial graph $G^{(0)}=G^{\text{tsp}}$, we now define a sequence of graphs $(G^{(s)})_{s\ge1}$ by iteratively performing two successive operations:\\
(1) {\bf support extension}. Let $F^{(s)}$ be the graph with nodes $V$ and with edges
\begin{equation}\label{es}
    E(F^{(s)})=\left\{\{\b,\g\}\mid\b\ne\g\in V,\,\b+\g\in\supp(G^{(s-1)})\right\}.
\end{equation}
(2) {\bf chordal extension}. Let $G^{(s)}=\bigl(F^{(s)}\bigr)'$.\\

\begin{example}\label{ex:support}
Let us consider the graph $G(V,E)$ with $$V=\{1,x_1,x_2,x_3,x_2x_3,x_1x_3,x_1x_2\} \text{ and } E=\left\{\{1,x_2x_3\},\{x_2,x_1x_3\}\right\}.$$
Figure \ref{fg:support} illustrates the support extension of $G$.
\begin{figure}[!t]
\centering
{\tiny
\begin{tikzpicture}[every node/.style={circle, draw=blue!50, thick, minimum size=6mm}]
\node (n1) at (0,0) {$1$};
\node (n2) at (2,0) {$x_1$};
\node (n3) at (4,0) {$x_2$};
\node (n4) at (6,0) {$x_3$};
\node (n5) at (1,-2) {$x_2x_3$};
\node (n6) at (4,-2) {$x_1x_3$};
\node (n7) at (6,-2) {$x_1x_2$};
\draw (n1)--(n5);
\draw (n3)--(n6);
\draw[dashed] (n2)--(n5);
\draw[dashed] (n3)--(n4);
\draw[dashed] (n4)--(n7);
\end{tikzpicture}}
\caption{The support extension of $G$ in Example \ref{ex:support}. The dashed edges are added after support extension.}\label{fg:support}
\end{figure}
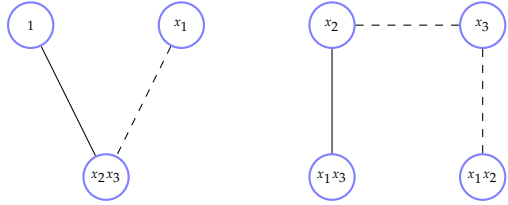
\end{example}

By construction, one has $G^{(s)}\subseteq G^{(s+1)}$ for $s\ge1$ and therefore the sequence of graphs $(G^{(s)})_{s\ge1}$ stabilizes after a finite number of steps.
Following what we introduced in Chapter \ref{sec:sparsemat}, we denote by $\Pi_{G^{(s)}}(\Sbb^+_{|\Basis|})$ the set of matrices in $\Sbb(G^{(s)})$ that have a \ac{PSD} completion, and denote by $\B_{G^{(s)}}$ the adjacency matrix of $G^{(s)}$.
If $f$ is sparse, by replacing $\M_{\Basis}(\y)\succeq0$ with the weaker condition $\M_{\Basis}(\y)\in\Pi_{G^{(s)}}(\Sbb^+_{|\Basis|})$ in \eqref{sec4-eq3}, we then obtain a sparse moment relaxation of \eqref{upop} for each $s\ge1$:
\begin{equation}\label{sec4-ukmom}
\P_{\ts}^s:\quad
\begin{cases}
\inf\limits_{\y} &L_{\y}(f)\\
\rm{ s.t.}&\B_{G^{(s)}}\circ \M_{\Basis}(\y)\in\Pi_{G^{(s)}}(\Sbb_{|\Basis|}^+)\\
&y_{\mathbf{0}}=1
\end{cases}
\end{equation}
with optimum denoted by $f_{\ts}^{s}$.
We call $(\P_{\ts}^s)_{s\ge1}$ the \ac{TSSOS} hierarchy for $\P$ and call $s$ the \emph{sparse order}.


\begin{remark}
The intuition behind the support extension operation is that once one position related to $y_{\a}$ in the moment matrix $\M_{\Basis}(\y)$ is ``activated'' in the sparsity pattern, then all positions related to $y_{\a}$ in $\M_{\Basis}(\y)$ should be ``activated''.
In addition, Theorems \ref{th:sparsesdpsum} and \ref{th:sparsesdpproj} provide the rationale behind the mechanism of the chordal extension operation.
\end{remark}

\begin{theorem}\label{sec3-thm}
The sequence $(f_{\ts}^{s})_{s\ge 1}$ is monotonically nondecreasing and $f_{\ts}^{s}\le f_{\text{mom}}$ for all $s\ge1$.
\end{theorem}
\begin{proof}
The inclusion $G^{(k)}\subseteq G^{(k+1)}$ implies that each maximal clique of $G^{(k)}$ is a subset of some maximal clique of $G^{(k+1)}$. Thus by Theorem \ref{th:sparsesdpproj}, we see that $\P_{\ts}^s$ is a relaxation of $\P_{\ts}^{s+1}$ (and also a relaxation of $\P_{\text{mom}}$). This yields the desired conclusions. \qed
\end{proof}

As a consequence of Theorem \ref{sec3-thm}, we obtain the following hierarchy of lower bounds for the optimum of $\P$:
\begin{equation}\label{cliquehier}
f_{\ts}^1 \le f_{\ts}^2 \le\cdots\le\ f_{\text{mom}} \le f_{\min}.
\end{equation}


If the maximal chordal extension is chosen for the chordal extension operation, then we can show (see \citets{tssos} for more details) that the sequence $(f_{\ts}^{s})_{s\ge1}$ converges to the global optimum
$f_{\min}$. Otherwise, there is no guarantee of such convergence as illustrated by the following example.
\begin{example}\label{ex1}
Consider the commutative version of the polynomial from \eqref{eq:polySohs}:
\begin{align*}
	f=\,&x_1^2-2x_1x_2+3x_2^2-2x_1^2x_2+
	2x_1^2x_2^2-2x_2x_3+6x_3^2\\
	&+18x_2^2x_3
	-54x_2x_3^2+142x_2^2x_3^2.
\end{align*}
The monomial basis computed from the Newton polytope is $\{1,x_1,x_2,x_3,x_1x_2,$ $x_2x_3\}$.
Figure \ref{tsp} shows the \ac{tsp} graph $G^{\text{tsp}}$ (without dashed edges) and its smallest chordal extension $G^{(1)}$ (with dashed edges) for $f$. The graph sequence $(G^{(s)})_{s\ge1}$ stabilizes at $s=1$. Solving $\P_{\ts}^1$, we obtain $f_{\ts}^1 \approx-0.00355$ while $f_{\text{mom}} = f_{\min} = 0$. On the other hand, note that $G^{\text{tsp}}$ has only one connected component. So with the maximal chordal extension, we immediately get the complete graph and it follows $f_{\ts}^1=f_{\text{mom}}=0$ in this case.

\begin{figure}[htbp]
\centering
{\tiny
\begin{minipage}{0.45\linewidth}
\begin{tikzpicture}[every node/.style={circle, draw=blue!50, thick, minimum size=6mm}]
\node (n2) at (0,0) {$x_1$};
\node (n3) at (2,0) {$x_2$};
\node (n4) at (4,0) {$x_3$};
\node (n5) at (0,2) {$x_1x_2$};
\node (n1) at (2,2) {$1$};
\node (n6) at (4,2) {$x_2x_3$};
\draw (n2)--(n3);
\draw (n3)--(n4);
\draw (n6)--(n4);
\draw (n1)--(n6);
\draw (n1)--(n5);
\draw (n2)--(n5);
\draw (n3)--(n6);
\draw[dashed] (n2)--(n1);
\draw[dashed] (n3)--(n1);
\end{tikzpicture}
\end{minipage}
\begin{minipage}{0.45\linewidth}
\begin{tikzpicture}[every node/.style={circle, draw=blue!50, thick, minimum size=6mm}]
\node (n2) at (0,0) {$x_1$};
\node (n3) at (2,0) {$x_2$};
\node (n4) at (4,0) {$x_3$};
\node (n5) at (0,2) {$x_1x_2$};
\node (n1) at (2,2) {$1$};
\node (n6) at (4,2) {$x_2x_3$};
\draw (n2)--(n3);
\draw (n3)--(n4);
\draw (n6)--(n4);
\draw (n1)--(n6);
\draw (n1)--(n5);
\draw (n2)--(n5);
\draw (n3)--(n6);
\draw[dashed] (n2)--(n1);
\draw[dashed] (n3)--(n1);
\draw[dashed] (n3)--(n5);
\draw[dashed] (n4)--(n5);
\draw[dashed] (n2)--(n6);
\draw[dashed] (n1)--(n4);
\draw[dashed] (0,1.72) arc (180:360:2 and 0.4);
\draw[dashed] (4,0.28) arc (0:180:2 and 0.4);
\end{tikzpicture}
\end{minipage}}
\caption{The \ac{tsp} graph $G^{\text{tsp}}$ and a smallest chordal extension (left) as well as the maximal chordal extension (right) for Example \ref{ex1}.}\label{tsp}
\end{figure}
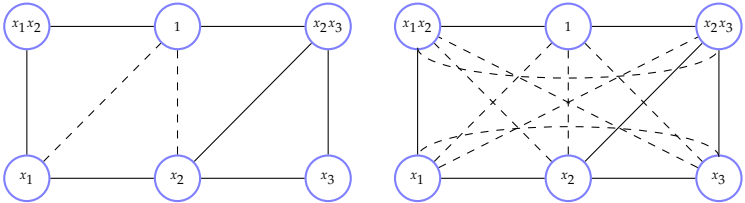
\end{example}


For each $s\ge1$, the dual \ac{SDP} of \eqref{sec4-ukmom} is
\begin{equation}\label{sec4-uksos}
\begin{cases}{}
\sup\limits_{\G,b}& b\\
\text{ s.t.}&\langle\G,\B_{\a}\rangle=f_{\a}-b 1_{\a=\mathbf{0}},\quad\forall\a\in\supp(G^{(s)})\\
&\G\in\Sbb^+_{|\Basis|}\cap\Sbb(G^{(s)})
\end{cases}
\end{equation}
where $\B_{\a}$ has been defined after \eqref{dualj}.

\begin{proposition}\label{chap7:prop1}
For each $s\ge1$, there is no duality gap between $\P_{\ts}^s$ and its dual \eqref{sec4-uksos}.
\end{proposition}
\begin{proof}
This easily follows from the fact that $\P_{\ts}^s$ satisfies Slater's condition by Proposition 3.1 of \citets{Las01sos} and Theorem \ref{th:sparsesdpproj}. \qed
\end{proof}

\section{Comparison with SDSOS}\label{sec:sdsos}
SDSOS polynomials were introduced and studied in \citets{ahmadi2014dsos} as cheaper alternatives to \ac{SOS} polynomials in the context of polynomial optimization. 
More concretely, a symmetric matrix $\G\in\Sbb_t$ is \emph{diagonally dominant} if $\G_{ii}\ge\sum_{j\ne i}|\G_{ij}|$ for $i \in [t]$, and is \emph{scaled diagonally dominant} if there exists a positive definite $t\times t$ diagonal matrix $\D$ such that $\D \G \D$ is diagonally dominant.
We say that a polynomial $f\in\R[\x]$ is a \emph{scaled diagonally dominant sum of squares} (SDSOS) polynomial if it admits a Gram matrix representation as in \eqref{sec4-uksos} (with $b = 0$) with a scaled diagonally dominant Gram matrix $\G$. We denote the set of SDSOS polynomials by $\SDSOS$.

By replacing the nonnegativity condition in $\P$ with the SDSOS condition, one obtains the SDSOS relaxation of $\P$:
\begin{equation*}
(\text{SDSOS}):\quad f_{\text{sdsos}}\coloneqq\sup\,\{b : f-b\in\SDSOS\}.
\end{equation*}
It turns out that the first value of the \ac{TSSOS} hierarchy is already better than or equal to the bound provided by the SDSOS relaxation.
\begin{theorem}\label{sec4-thm1}
With the above notation, one has $f_{\ts}^1 \ge f_{\rm{sdsos}}$.
\end{theorem}
\begin{proof}
Let $\sA=\supp(f)$ and $\Basis$ be a monomial basis. Assume $f\in\SDSOS$, i.e., $f$ admits a scaled diagonally dominant Gram matrix $\G\in\Sbb^+_{|\Basis|}$ indexed by $\Basis$. We then construct a Gram matrix $\tilde{\G}$ for $f$ by
\begin{equation*}
\tilde{\G}_{\b\g}=
\begin{cases}
\G_{\b\g},\quad&\text{if } \b+\g\in\sA\cup(2\Basis),\\
0,\quad&\text{otherwise}.
\end{cases}
\end{equation*}
It is easy to see that we still have $f=(\x^{\Basis})^\intercal\tilde{\G}\x^{\Basis}$. Note that we only replace off-diagonal entries by zeros in $\G$ to obtain $\tilde{\G}$ and replacing off-diagonal entries by zeros does not affect the scaled diagonal dominance of a matrix. Hence $\tilde{\G}$ is also a scaled diagonally dominant matrix. Moreover, we have $\tilde{\G}\in\Sbb^+_{|\Basis|}\cap\Sbb(G^{(1)})$ by construction. It follows that $(\text{SDSOS})$ is a relaxation of \eqref{sec4-uksos}. Hence $f_{\ts}^1 \ge f_{\text{sdsos}}$ by Proposition \ref{chap7:prop1}.
\qed
\end{proof}


\section{The TSSOS hierarchy for constrained POPs}\label{sec:tssos-cons}
In this section, we describe an iterative procedure to exploit \ac{TS} for the primal-dual \ac{moment-SOS} hierarchy of constrained \ac{POP}s:
\begin{align}\label{eq:popts}
\P: \quad f_{\min}\coloneqq\inf\,\{f(\x): \x\in\X\},
\end{align}
with
\begin{equation}\label{eq:defeX}
\X = \{\x \in \R^n \mid g_1(\x)\ge0, \ldots, g_m(\x)\ge0\}.
\end{equation}
Let $\sA$ denote the union of supports involved in $\P$, i.e.,
\begin{equation}\label{supp}
\sA = \supp(f)\cup\bigcup_{j=1}^{m}\supp(g_j).
\end{equation}

Let $r_{\min}\coloneqq\max\,\{\lceil\deg(f)/2\rceil,d_1,\ldots,d_{m}\}$ with $d_j\coloneqq\lceil\deg(g_j)/2\rceil$ for $j\in[m]$.
Fix a relaxation order $r \ge r_{\min}$.
Let $g_0=1$, $d_0=0$ and $\Basis_{r,j}=\N^n_{r-d_j}$ be the standard monomial basis for $j=0,1,\ldots,m$. We define a graph $G^{\text{tsp}}$ with nodes $\Basis_{r,0}$ and edges
\begin{equation}\label{sec4-eq1}
E(G^{\text{tsp}})=\left\{\{\b,\g\}\mid\b\ne\g\in\Basis_{r,0},\,\b+\g\in\sA\cup(2\Basis_{r,0})\right\},
\end{equation}
which is called the \ac{tsp} graph associated with $\P$ or essentially $\sA$.

Now let us initialize with $G_{r,0}^{(0)}\coloneqq G^{\text{tsp}}$ and $G_{r,j}^{(0)}$ being an empty graph for $j\in[m]$. Then for each $j\in\{0\}\cup[m]$, we define a sequence of graphs $(G_{r,j}^{(s)})_{s\ge1}$ by iteratively performing two successive operations:\\
(1) {\bf support extension}. Let $F_{r,j}^{(s)}$ be the graph with nodes $\Basis_{r,j}$ and edges
\begin{equation}\label{sec4-eq2}
\begin{split}
E(F_{r,j}^{(s)})=\,&\Bigl\{\{\b,\g\}\mid\b\ne\g\in\Basis_{r,j},\\
&(\supp(g_j)+\b+\g)\cap\bigcup_{i=0}^{m}\bigl(\supp(g_i)+\supp(G_{r,i}^{(s-1)})\bigr)\ne\emptyset\Bigr\}.
\end{split}
\end{equation}
(2) {\bf chordal extension}. Let
\begin{equation}\label{sec4-graphtssos}
	G_{r,j}^{(s)}\coloneqq\bigl(F_{r,j}^{(s)}\bigr)',\quad j\in\{0\}\cup[m].
\end{equation}

Recall that the dense moment relaxation of order $r$ for $\P$ is given by
\begin{equation}\label{momconstr}
\P^r:\quad
\begin{array}{rll}
f_{\text{mom}}^r\coloneqq&\inf\limits_{\y} &L_\y(f)\\
&\rm{ s.t.}&\M_{r-d_j}(g_j\,\y)\succeq0,\quad j\in\{0\}\cup[m]\\
&&y_{\mathbf{0}}=1
\end{array}
\end{equation}
Let $t_j\coloneqq|\Basis_{r,j}|=\binom{n+r-d_j}{r-d_j}$. Therefore by replacing $\M_{r-d_j}(g_j\y)\succeq0$ with the weaker condition $\B_{G_{r,j}^{(s)}}\circ\M_{r-d_j}(g_j\y)\in\Pi_{G_{r,j}^{(s)}}(\Sbb^+_{t_j})$ for $j\in\{0\}\cup[m]$ in \eqref{momconstr}, we obtain the following sparse moment relaxation of $\P^r$ and $\P$ for each $s\ge1$:
\begin{equation}\label{chap7:seq1}
\P_{\ts}^{r,s}:\quad
\begin{array}{rll}
f_{\ts}^{r,s}\coloneqq&\inf\limits_{\y} &L_{\y}(f)\\
&\rm{ s.t.}&\B_{G_{r,0}^{(s)}}\circ\M_{r}(\y)\in \Pi_{G_{r,0}^{(s)}}(\Sbb^+_{t_0})\\
&&\B_{G_{r,j}^{(s)}}\circ\M_{r-d_j}(g_j\y)\in\Pi_{G_{r,j}^{(s)}}(\Sbb^+_{t_j}),\quad j\in[m]\\
&&y_{\mathbf{0}}=1
\end{array}
\end{equation}
As in the unconstrained case, we call $s$ the sparse order.
By construction, one has $G_{r,j}^{(s)}\subseteq G_{r,j}^{(s+1)}$ for all $j,s$. Therefore, for each $j\in\{0\}\cup[m]$, the sequence of graphs $(G_{r,j}^{(s)})_{s\ge1}$ stabilizes after a finite number of steps. We denote the stabilized graph by $G_{r,j}^{(\bullet)}$ for all $j$ and the corresponding moment relaxation by $\P_{\ts}^{r,\bullet}$ with optimum $f_{\ts}^{r,\bullet}$.

For each $s\ge1$, the dual \ac{SDP} of $\P_{\ts}^{r,s}$ reads as
\begin{equation}\label{chap7:seq2}
\begin{cases}
\sup\limits_{\G_j,b} & b\\
\rm{ s.t.}&\sum_{j=0}^{m}\langle \Cb_{\a}^j , \G_j \rangle =f_{\a} - b \oneb_{\a = \mathbf{0}}, \quad\forall\a\in\bigcup_{j=0}^{m}\bigl(\supp(g_j)+\supp(G_{r,j}^{(s)})\bigr)\\
&\G_j\in\Sbb^+_{t_j}\cap\Sbb(G_{r,j}^{(s)}),\quad j\in\{0\}\cup[m]\\
\end{cases}
\end{equation}
where $\Cb_{\a}^j$ is defined after \eqref{dualj}. The primal-dual \ac{SDP} relaxations \eqref{chap7:seq1}--\eqref{chap7:seq2}
are called the \ac{TSSOS} hierarchy associated with $\P$, which is indexed by two parameters: the relaxation order $r$ and the sparse order $s$.

\begin{theoremf}\label{sec6-thm1}
With the above notation, the following hold:
\begin{enumerate}[(i)]
  \item Assume that $\X$ has a nonempty interior. Then there is no duality gap between $\P_{\ts}^{r,s}$ and its dual \eqref{chap7:seq2} for any $r\ge r_{\min}$ and $s\ge1$.
  \item Fixing a relaxation order $r \ge r_{\min}$, the sequence $(f_{\ts}^{r,s})_{s\ge1}$ is monotonically nondecreasing and $f_{\ts}^{r,s}\le f_{\rm{mom}}^r$ for all $s \geq 1$.
  \item When the maximal chordal extension is used for the chordal extension operation, the sequence $(f_{\ts}^{r,s})_{s\ge1}$ converges to $f_{\rm{mom}}^r$ in finitely many steps.
  \item Fixing a sparse order $s\ge1$, the sequence $(f_{\ts}^{r,s})_{r\ge r_{\min}}$ is monotonically nondecreasing.
\end{enumerate}
\end{theoremf}
\begin{proof}
(i). This easily follows from the fact that $\P_{\ts}^{r,s}$ satisfies Slater's condition by Theorem 4.2 of \citets{Las01sos} and Theorem \ref{th:sparsesdpproj}.

(ii). For all $j,s$, the inclusion $G_{r,j}^{(s)}\subseteq G_{r,j}^{(s+1)}$ implies that each maximal clique of $G_{r,j}^{(s)}$ is a subset of some maximal clique of $G_{r,j}^{(s+1)}$. Hence by Theorem \ref{th:sparsesdpproj}, $\P_{\ts}^{r,s}$ is a relaxation of $\P_{\ts}^{r,s+1}$ (and also a relaxation of $\P^{r}$) from which we have that $(f_{\ts}^{r,s})_{s\ge1}$ is monotonically nondecreasing and $f_{\ts}^{r,s}\le f_{\rm{mom}}^r$ for all $s\ge1$.

(iii). Let $\y=(y_{\a})$ be an arbitrary feasible solution of $\P_{\ts}^{r,\bullet}$. We note that $\bigl\{y_{\a}\mid\a\in\bigcup_{j\in \{0\}\cup[m]}\bigl(\supp(g_j)+\supp(G_{r,j}^{(\bullet)})\bigr)\bigr\}$ is the set of decision variables involved in $\P_{\ts}^{r,\bullet}$, and $\{y_{\a}\mid\a\in\N^n_{2r}\}$ is the set of decision variables involved in $\P^{r}$ \eqref{momconstr}.
We then define a vector $\overline{\y}=(\overline{y}_{\a})_{\a\in\N^n_{2r}}$ as follows:
\begin{equation*}
\overline{y}_{\a}=\begin{cases}y_{\a},\quad&\text{if }\a\in\bigcup_{j\in \{0\}\cup[m]}\bigl(\supp(g_j)+\supp(G_{r,j}^{(\bullet)})\bigr),\\
0,\quad&\text{otherwise}.
\end{cases}
\end{equation*}
By construction and since $G_{r,j}^{(\bullet)}$ stabilizes under support extension for all $j$, we immediately have $\M_{r-d_j}(g_j\overline{\y})=\B_{G_{r,j}^{(\bullet)}}\circ \M_{r-d_j}(g_j\y)$.
As we use the maximal chordal extension for the chordal extension operation, the matrix $\B_{G_{r,j}^{(\bullet)}}\circ \M_{r-d_j}(g_j\y)$ is block-diagonal up to permutation. So from $\B_{G_{r,j}^{(\bullet)}}\circ \M_{r-d_j}(g_j\y)\in\Pi_{G_{r,j}^{(\bullet)}}(\Sbb_+^{t_{j}})$ it follows $\M_{r-d_j}(g_j\overline{\y})\succeq0$ for $j\in \{0\}\cup[m]$. Therefore $\overline{\y}$ is a feasible solution of $\P^{r}$ and so $L_{\y}(f)=L_{\overline{\y}}(f)\ge f_{\text{mom}}^{r}$.
Hence $f_{\ts}^{r,\bullet}\ge f_{\text{mom}}^{r}$ as $\y$ is an arbitrary feasible solution of $\P_{\ts}^{r,\bullet}$.
By (ii), we already have $f_{\ts}^{r,\bullet}\le f_{\text{mom}}^{r}$. Therefore, $f_{\ts}^{r,\bullet}=f_{\text{mom}}^{r}$ as desired.

(iv). The conclusion follows if we can show that $G_{r,j}^{(s)}\subseteq G_{r+1,j}^{(s)}$ for all $j,r$ since by Theorem \ref{th:sparsesdpproj} this implies that $\P_{\ts}^{r,s}$ is a relaxation of $\P_{\ts}^{r+1,s}$. Let us prove $G_{r,j}^{(s)}\subseteq G_{r+1,j}^{(s)}$ by induction on $s$. For $s=1$, from $\eqref{sec4-eq1}$, we have $G_{r,0}^{(0)}\subseteq G_{r+1,0}^{(0)}$, which implies $G_{r,j}^{(1)}\subseteq G_{r+1,j}^{(1)}$ for $j\in\{0\}\cup[m]$. Now assume that $G_{r,j}^{(s)}\subseteq G_{r+1,j}^{(s)}$, $j\in\{0\}\cup[m]$ hold for a given $s\geq 1$. Then from $\eqref{sec4-eq2}$ and by the induction hypothesis, we have $G_{r,j}^{(s+1)}\subseteq G_{r+1,j}^{(s+1)}$ for $j\in\{0\}\cup[m]$, which completes the induction and also completes the proof.
\qed
\end{proof}

By Theorem \ref{sec6-thm1}, we have the following two-level hierarchy of lower bounds for the optimum $f_{\min}$ of $\P$:
\begin{equation}\label{cliquehierc}
\begin{matrix}
f_{\ts}^{r_{\min},1}&\le& f_{\ts}^{r_{\min},2} &\le&\cdots&\le& f_{\text{mom}}^{r_{\min}} \\
\vge&&\vge&&&&\vge\\
f_{\ts}^{r_{\min}+1,1}&\le&f_{\ts}^{r_{\min}+1,2}&\le&\cdots&\le&f_{\text{mom}}^{r_{\min}+1}\\
\vge&&\vge&&&&\vge\\
\vdots&&\vdots&&\vdots&&\vdots\\
\vge&&\vge&&&&\vge\\
f_{\ts}^{r,1}&\le&f_{\ts}^{r,2}&\le&\cdots&\le&f_{\text{mom}}^r\\
\vge&&\vge&&&&\vge\\
\vdots&&\vdots&&\vdots&&\vdots\\
\end{matrix}
\end{equation}

The \ac{TSSOS} hierarchy entails a trade-off between the computational cost and the quality of the obtained lower bound via the two tunable parameters $r$ and $s$. Besides, one has the freedom to choose a specific chordal extension in \eqref{sec4-graphtssos} (e.g., maximal chordal extensions, approximately smallest chordal extensions and so on). This choice could affect the resulting sizes of PSD blocks and the quality of the related lower bound. Intuitively, chordal extensions with smaller clique numbers would lead to PSD blocks of smaller sizes and lower bounds of (possibly) lower quality while chordal extensions with larger clique numbers would lead to PSD blocks with larger sizes and lower bounds of (possibly) higher quality.

\begin{remark}\label{ts-qcqp}
If $\P$ is a \ac{QCQP}, then $\P_{\text{ts}}^{1,1}$ and $\P^{1}$ yield the same lower bound, i.e., $f_{\rm{ts}}^{1,1}=f_{\rm{mom}}^1$.
Indeed, for a \ac{QCQP}, the moment relaxation $\P^{1}$ reads as
\begin{equation*}
\begin{cases}
\inf\limits_{\y}& L_{\y}(f)\\
\rm{ s.t.}&\M_{1}(\y)\succeq0\\
&L_{\y}(g_j)\ge0,\quad j\in[m]\\
&y_{\mathbf{0}}=1
\end{cases}
\end{equation*}
Note that the objective function and the affine constraints of $\P^{1}$ involve only the decision variables $\{y_{\mathbf{0}}\}\cup\{y_{\a}\}_{\a\in\sA}$ with $\sA=\supp(f)\cup\bigcup_{j=1}^{m}\supp(g_j)$. Hence there is no discrepancy of optima in replacing $\P^{1}$ with $\P_{\ts}^{1,1}$ by construction.
\end{remark}

%


\section{Obtaining a possibly smaller monomial basis}\label{sec:tssos-basis}
The size of \ac{SDP}s arising form the \ac{TSSOS} hierarchy heavily depends on the chosen monomial basis $\Basis$ or $\Basis_{r,0}$. As we already saw in Chapter~\ref{sec:tssos-uncons}, for unconstrained \ac{POP}s the Newton polytope method usually produces a monomial basis smaller than the standard monomial basis. However, this method does not apply to constrained \ac{POP}s. Here as an optional pre-treatment,
we present an iterative procedure which not only allows us to obtain a monomial basis smaller than the one given by the Newton polytope method for unconstrained \ac{POP}s in many cases, but can also be applied to constrained \ac{POP}s.

We start with the unconstrained case. Let $f\in\R[\x]$ with $\sA=\supp(f)$ and $\Basis$ be the monomial basis given by the Newton polytope method. Initializing with $\Basis_0\coloneqq\emptyset$, we iteratively define a sequence of monomial sets $(\Basis_p)_{p\ge1}$ by
\begin{equation}\label{chap7:eq3}
    \Basis_{p}\coloneqq\{\b\in\Basis\mid\exists\g\in\Basis\text{ s.t. }\b+\g\in\sA\cup(2\Basis_{p-1})\}.
\end{equation}
Consequently, we obtain an ascending chain of monomial sets:
$$\Basis_1\subseteq\Basis_2\subseteq\Basis_3\subseteq\cdots\subseteq\Basis.$$
This procedure is formulated in Algorithm \ref{alg1}.
Each $\Basis_p$ in the chain can serve as a candidate monomial basis. In practice, if indexing the unknown Gram matrix by $\Basis_p$ leads to an infeasible \ac{SDP}, then we turn to $\Basis_{p+1}$ until a feasible \ac{SDP} is retrieved.

\begin{algorithm}\caption{${\tt GenerateBasis}$}\label{alg1}
	\begin{algorithmic}[1]
		\Require A support set $\sA$ and an initial monomial basis $\Basis$
		\Ensure An ascending chain of potential monomial bases $(\Basis_p)_{p\ge1}$
		\State $\Basis_0\leftarrow\emptyset$
		\State $p\leftarrow0$
		\While{$p=0$ or $\Basis_{p}\ne\Basis_{p-1}$}
		\State $p\leftarrow p+1$
		\State $\Basis_{p}\leftarrow\emptyset$
		\For{each pair $\b,\g$ in $\Basis$}
		\If{$\b+\g\in\sA\cup(2\Basis_{p-1})$}
		\State $\Basis_{p}\leftarrow\Basis_{p}\cup\{\b,\g\}$
		\EndIf
		\EndFor
		\EndWhile
		\State \Return $(\Basis_p)_{p\ge1}$
	\end{algorithmic}
\end{algorithm}

\begin{proposition}\label{sec5-prop1}
Let $f\in\R[\x]$ and $\Basis_*=\cup_{p\ge1}\Basis_p$ with $\Basis_p$ being defined by \eqref{chap7:eq3}. If $f\in\SDSOS$, then $f$ is an SDSOS polynomial in the monomial basis $\Basis_*$.
\end{proposition}
\begin{proof}
Let $\Basis$ be the monomial basis given by the Newton polytope method. If $f\in\SDSOS$, then there exists a scaled diagonally dominant Gram matrix $\G\in\Sbb_+^{|\Basis|}$ indexed by $\Basis$ such that $f=(\x^{\Basis})^\intercal \G\x^{\Basis}$. We then construct a Gram matrix $\tilde{\G}\in\Sbb_+^{|\Basis_*|}$ indexed by $\Basis_*$ for $f$ as follows:
\begin{equation*}
\tilde{\G}_{\b\g}=
\begin{cases}
\G_{\b\g},\quad&\text{if } \b+\g\in\sA\cup(2\Basis_*),\\
0,\quad&\text{otherwise}.
\end{cases}
\end{equation*}
One can easily check that we still have $f=(\x^{\Basis_*})^\intercal\tilde{\G}\x^{\Basis_*}$. Let $\hat{\G}$ be the principal submatrix of $\G$ by deleting the rows and columns whose indices are not in $\Basis_*$, which is also a scaled diagonally dominant matrix. By construction, $\tilde{\G}$ is obtained from $\hat{\G}$ by replacing certain off-diagonal entries by zeros.
Since replacing off-diagonal entries by zeros does not affect the scaled diagonal dominance of a matrix, $\tilde{\G}$ is also a scaled diagonally dominant matrix. It follows that $f$ is an SDSOS polynomial in the monomial basis $\Basis_*$.
\qed
\end{proof}

\begin{remark}
By Proposition \ref{sec5-prop1}, if we use the monomial basis $\Basis_*$ for \eqref{sec4-ukmom}--\eqref{sec4-uksos}, we still have the hierarchy of optima: $$f_{\rm{sdsos}}\le f_{\ts}^1\le f_{\ts}^2\le\cdots\le f_{\rm{mom}}\le f_{\min}.$$
\end{remark}

The algorithm ${\tt GenerateBasis}$ may provide a smaller monomial basis than the one given by the Newton polytope method as the following example illustrates.
\begin{example}
Consider the polynomial $f=1+x+x^8$. The monomial basis given by the Newton polytope method is $\Basis=\{1,x,x^2,x^3,x^4\}$.
By the algorithm ${\tt GenerateBasis}$, we obtain $\Basis_1=\{1,x,x^4\}$ and $\Basis_2=\{1,x,x^2,x^4\}$. It turns out that $f$ admits no \ac{SOS} decomposition with $\Basis_1$ while $\Basis_2$ can serve as a monomial basis to represent $f$ as an \ac{SOS}.
\end{example}

For the constrained case we follow the notation of Chapter~\ref{sec:tssos-cons}. Fix a relaxation order $r$ and a sparse order $s$ of the \ac{TSSOS} hierarchy. Initializing with $\Basis_{r,0}=\N^n_r$, we iteratively perform the following two steps:

For Step 1, let the maximal cliques of $G_{r,j}^{(s)}$ be $C_{j,1},C_{j,2},\ldots,C_{j,t_j}$ for $j\in\{0\}\cup[m]$.
Let
\begin{equation}\label{sec5-eq1}
\mathscr{F}=\supp(f)\cup\bigcup_{j=1}^{m}\left(\supp(g_j)+\bigcup_{i=1}^{t_j}(C_{j,i}+C_{j,i})\right).
\end{equation}
Then call the algorithm {\tt GenerateBasis} with $\sA=\mathscr{F}$ and $\Basis=\Basis_{r,0}$ to generate a new monomial basis $\Basis_{r,0}'$.

For Step 2, with the new monomial basis $\Basis_{r,0}'$, recompute the graph $G_{r,j}^{(s)}$ for $j\in \{0\} \cup [m]$ as in Chapter~\ref{sec:tssos-cons} and then go back to Step 1.

Continue the iterative procedure until $\Basis_{r,0}'=\Basis_{r,0}$, which is the desired monomial basis.

\section[Sign symmetries and a sparse representation theorem]{Sign symmetries and a sparse representation theorem for positive polynomials}
\label{sec:signtssos}
The exploitation of \ac{TS} developed in the previous sections is closely related to {\em sign symmetries}. Intuitively, a polynomial is said to have sign symmetries if it is invariant when we change signs of some variables. For instance, the polynomial $f(x_1,x_2)=x_1^2+x_2^2+x_1x_2$ has the sign symmetry associated to $(x_1,x_2) \mapsto (-x_1,-x_2)$ as $f(-x_1,-x_2)=f(x_1,x_2)$. To be more precise, we give the following definition of sign symmetries in terms of support sets.
\begin{definition}[sign symmetry]
Given a finite set $\sA\subseteq\N^n$, the {\em sign symmetries} of $\sA$ are defined by all vectors $\bs\in\Z_2^n\coloneqq\{0,1\}^n$ such that $\bs^\intercal\a\equiv0$ $(\rm{mod}\,2)$ for all $\a\in\sA$.
\end{definition}
%



Assume that the maximal chordal extension is chosen for the chordal extension operation in Chapter~\ref{sec:tssos-cons}.
As mentioned earlier, for any $j$ the sequence of graphs $(G_{r,j}^{(s)})_{s\ge1}$ ends up with $G_{r,j}^{(\bullet)}$ in finitely many steps. Note that the graph $G_{r,j}^{(\bullet)}$ induces a partition of the monomial basis $\N^n_{r-d_j}$: two monomials $\b,\g\in\N^n_{r-d_j}$ belong to the same block if and only if they belong to the same connected component of $G_{r,j}^{(\bullet)}$. The following theorem provides an interpretation of this partition in terms of sign symmetries.
\begin{theoremf}\label{sec6-thm2}
Notations are as in the previous sections.
Fix the relaxation order $r\ge r_{\min}$.
Assume that the maximal chordal extension is chosen for the chordal extension operation and the sign symmetries of $\sA$ are given by the columns of a binary matrix denoted by $\Rb$.
Then for each $j\in\{0\}\cup[m]$, $\b,\g$ belong to the same block in the partition of $\N^n_{r-d_j}$ induced by $G_{r,j}^{(\bullet)}$ if and only if $\Rb^\intercal(\b+\g)\equiv0$ $(\rm{mod}\,2)$. In other words, for a fixed relaxation order the block structures arising from the \ac{TSSOS} hierarchy converge to the block structure determined by the sign symmetries of the \ac{POP} assuming that the maximal chordal extension is used for the chordal extension operation.
\end{theoremf}

Theorem \ref{sec6-thm2} is applied for the standard monomial basis $\N^n_{r-d_j}$. If a smaller monomial basis is chosen, then we only have the "only if" part of the conclusion in Theorem \ref{sec6-thm2}.
\begin{example}\label{sec6-ex1}
Let $f=1+x_1^2x_2^4+x_1^4x_2^2+x_1^4x_2^4-x_1x_2^2-3x_1^2x_2^2$ and $\sA=\supp(f)$. The monomial basis given by the Newton polytope method is $\Basis=\{1,x_1x_2,x_1x_2^2,x_1^2x_2,$ $x_1^2x_2^2\}$.
The sign symmetries of $\sA$ consist of two elements: $(0,0)$ and $(0,1)$.
According to the sign symmetries, $\Basis$ is partitioned into $\{1,x_1x_2^2,$ $x_1^2x_2^2\}$ and $\{x_1x_2,x_1^2x_2\}$. On the other hand, the partition of $\Basis$ induced by $G^{(\bullet)}$ is $\{1,x_1x_2^2,x_1^2x_2^2\}$, $\{x_1x_2\}$ and $\{x_1^2x_2\}$, which is a refinement of the partition determined by the sign symmetries.
\end{example}



As a corollary of Theorem \ref{sec6-thm2}, we can prove a sparse representation theorem for positive polynomials on compact basic semialgebraic sets.
\begin{theoremf}\label{sec6-thm3}
Let $\X$ be defined as in \eqref{eq:defeX}. Assume that the quadratic module $\cM(\frakg)$ is Archimedean and that the polynomial $f$ is positive on $\X$. Let $\sA=\supp(f)\cup\bigcup_{j=1}^{m}\supp(g_j)$ and let the sign symmetries of $\sA$ be given by the columns of the binary matrix $R$.
Then $f$ can be decomposed as
\begin{equation*}
f= \sigma_0+\sum_{j=1}^m \sigma_j g_j,
\end{equation*}
for some \ac{SOS} polynomials $\sigma_0,\sigma_1,\ldots,\sigma_m$ satisfying $R^\intercal\a\equiv0$ $(\rm{mod}\,2)$ for any $\a\in\supp(\sigma_j),j=0,\ldots,m$.
\end{theoremf}
\begin{proof}
By Putinar's Positivstellensatz (Theorem \ref{th:putinar}), there exist \ac{SOS} polynomials $\sigma_0',\sigma_1',\ldots,\sigma_m'$ such that
\begin{equation}\label{sec7-eq1}
f= \sigma_0'+\sum_{j=1}^m \sigma_j' g_j.
\end{equation}
Let $d_j=\lceil\deg(g_j)/2\rceil,j=0,1,\ldots,m$ and
$$r=\max\,\{\lceil\deg(\sigma_j' g_j)/2\rceil:j=0,1,\ldots,m\}$$ with $g_0=1$. Let $\G_j$ be a Gram matrix associated with $\sigma_j'$ and indexed by the monomial basis $\N^n_{r-d_j},j=0,1,\ldots,m$. Then define $\sigma_j=(\x^{\N^n_{r-d_j}})^\intercal(\B_{G_{r,j}^{(\bullet)}}\circ \G_j)\x^{\N^n_{r-d_j}}$ for $j=0,1,\ldots,m$, where $G_{r,j}^{(\bullet)}$ is defined in Chapter~\ref{sec:tssos-cons}. For any $j\in\{0\}\cup[m]$, since $\B_{G_{r,j}^{(\bullet)}}\circ \G_j$ is block-diagonal (up to permutation) and $\G_j$ is positive semidefinite, we see that $\sigma_j$ is an \ac{SOS} polynomial.

Suppose $\a\in\supp(\sigma_j)$ with $j\in\{0\}\cup[m]$. Then we can write $\a=\a'+\b+\g$ for some $\a'\in\supp(g_j)$ and some $\b,\g$ belonging to the same connected component of $G_{r,j}^{(\bullet)}$. By Theorem \ref{sec6-thm2}, we have $\Rb^\intercal(\b+\g)\equiv0$ $(\rm{mod}\,2)$ and therefore, $\Rb^\intercal\a\equiv0$ $(\rm{mod}\,2)$.
Moreover, for any $\a'\in\supp(g_j)$ and $\b,\g$ not belonging to the same connected component of $G_{r,j}^{(\bullet)}$, we have $\Rb^\intercal(\b+\g)\not\equiv0$ $(\rm{mod}\,2)$ by Theorem \ref{sec6-thm2} and so $\Rb^\intercal(\a'+\b+\g)\not\equiv0$ $(\rm{mod}\,2)$. From these facts we deduce that substituting $\sigma_j'$ with $\sigma_j$ in \eqref{sec7-eq1} is just removing the terms whose exponents $\a$ do not satisfy $\Rb^\intercal\a\equiv0$ $(\rm{mod}\,2)$ from the right-hand side of \eqref{sec7-eq1}. Doing so, one does not change the match of coefficients on both sides of the equality. Thus we have
\begin{equation*}
f= \sigma_0+\sum_{j=1}^m \sigma_j g_j,
\end{equation*}
with the desired property.
\qed
\end{proof}

\section{Numerical experiments}\label{sec:benchtssos}
We present some numerical results of the proposed \ac{TSSOS} hierarchy in this section.
The related algorithms were implemented in the Julia package {\tt TSSOS}\footnote{\url{https://github.com/wangjie212/TSSOS}}. For the numerical experiments, we use $\mosek$ \citets{mosek} as an \ac{SDP} solver.
%
All examples were computed on an Intel Core i5-8265U@1.60GHz CPU with 8GB RAM memory. The timing (in seconds) includes the time for pre-processing (to get the block structure), the time for modeling \ac{SDP} and the time for solving \ac{SDP}.


\subsection{Unconstrained polynomial optimization}
We first consider Lyapunov functions emerging from some networked systems.
The following polynomial is from Example 2 in \citets{han}:
\begin{equation*}
f=\sum_{i=1}^n a_i(x_i^2+x_i^4)-\sum_{i=1}^n\sum_{k=1}^n b_{ik}x_i^2x_k^2,
\end{equation*}
where $a_i$ are randomly chosen from $[1,2]$ and $b_{ik}$ are randomly chosen from $[\frac{0.5}{n},\frac{1.5}{n}]$. Here, $n$ is the number of nodes in the network. The task is to determine whether $f$ is globally nonnegative.

The sizes of \ac{SDP}s corresponding to the \ac{TSSOS} (with sparse order $s=1$) and dense relaxations are listed in Table \ref{cc}. In the column ``\#PSD blocks'', $i\times j$ means $j$ PSD blocks of size $i$. The column ``\#equality constraints'' indicates the number of equality constraints involved in \ac{SDP}s.
\begin{table}[htbp]
\caption{The sizes of \ac{SDP}s for the sparse and dense relaxations.}\label{cc}
\renewcommand\arraystretch{1.2}
\centering
\begin{tabular}{|c|c|c|}
\hline
&\#PSD blocks&\#equality constraints\\
\hline
TSSOS&$3\times\frac{n(n-1)}{2},1\times n,(n+1)\times 1$&$\frac{3n(n-1)}{2}+2n+1$\\
\hline
Dense&$\binom{n+2}{2}\times 1$&$\binom{n+4}{4}$\\
\hline
\end{tabular}
\end{table}

We solve the \ac{TSSOS} relaxation with $s=1$ for $n\in\{10,20,\ldots,80\}$. The results are displayed in Table \ref{tb:network1} in which ``mb'' stands for the maximal size of PSD blocks.
\begin{table}[htbp]
\caption{Results for the first networked system.}\label{tb:network1}
\renewcommand\arraystretch{1.2}
\centering
\begin{tabular}{|c|c|c|c|c|c|c|c|c|}
\hline
$n$&$10$&$20$&$30$&$40$&$50$&$60$&$70$&$80$\\
\hline
mb&$11$&$31$&$31$&$41$&$51$&$61$&$71$&$81$\\
\hline
time&$0.006$&$0.03$&$0.10$&$0.34$&$0.92$&$1.9$&$4.7$&$12$\\
\hline
\end{tabular}
\end{table}

For this example, the size of the system that can be handled in \citets{han} is up to $n=50$ nodes while {\tt TSSOS} can easily handle the system with up to $n=80$ nodes.

The following polynomial is from Example 3 in \citets{han}:
\begin{equation}\label{sec8-eq1}
V=\sum_{i=1}^na_i(\frac{1}{2}x_i^2-\frac{1}{4}x_i^4)+\frac{1}{2}\sum_{i=1}^n\sum_{k=1}^nb_{ik}\frac{1}{4}(x_i-x_k)^4,
\end{equation}
where $a_i$ are randomly chosen from $[0.5,1.5]$ and $b_{ik}$ are randomly chosen from $[\frac{0.5}{n},\frac{1.5}{n}]$.
The task is to analyze the domain on which the Hamiltonian function $V$ for a network of Duffing oscillators is positive
definite. We use the following condition to establish an inner approximation of the domain on which $V$ is positive definite:
\begin{equation}\label{sec8-eq2}
f=V-\sum_{i=1}^n\lambda_ix_i^2(g-x_i^2)\ge0 ,
\end{equation}
where $\lambda_i>0$ are scalar decision variables and $g$ is a fixed positive scalar.
Clearly, the condition \eqref{sec8-eq2} ensures that $V$ is positive definite when $x_i^2<g$.

We illustrate the \ac{tsp} graph $G^{(0)}$ of this example in Figure \ref{cost}, which has $1$ maximal clique of size $n+1$ (involving the nodes $1,x_1^2,\ldots,x_n^2$), $\frac{n(n-1)}{2}$ maximal cliques of size $3$ (involving the nodes $x_i^2,x_j^2,x_ix_j$ for each pair $\{i,j\},i\ne j$) and $n$ maximal cliques of size $1$ (involving the node $x_i$ for each $i$). Since $G^{(0)}$ is already a chordal graph, we have $G^{(1)}=G^{(0)}$.

\begin{figure}[!ht]
	\centering
	{\tiny
		\begin{tikzpicture}[main node/.style={circle, draw=blue!50, thick, minimum size=6mm}]
			\node[main node] (n1) at (1,0) {$x_j^2$};
			\node[main node] (n2) at (-1,0) {$x_k^2$};
			\node[main node] (n3) at (0,-1.73) {$x_i^2$};
			\node[main node] (n10) at (0,-0.58) {$1$};
			\draw (n10)--(n2);
			\draw (n10)--(n3);
			\draw (n10)--(n1);
			\draw (n1)--(n2);
			\draw (n1)--(n3);
			\draw (n2)--(n3);
			\node[main node] (n4) at (2,-1.73) {$x_ix_j$};
			\draw (n1)--(n4);
			\draw (n3)--(n4);
			\node[main node] (n5) at (-2,-1.73) {$x_ix_k$};
			\draw (n2)--(n5);
			\draw (n3)--(n5);
			\node[main node] (n6) at (0,1.73) {$x_jx_k$};
			\draw (n1)--(n6);
			\draw (n2)--(n6);
			\node[main node] (n7) at (3,0) {$x_1$};
			\node[main node] (n8) at (4,0) {$x_2$};
			\node[main node] (n9) at (6,0) {$x_n$};
			\path (n8) -- node[auto=false]{$\cdots$} (n9);
	\end{tikzpicture}}
	\caption{The \ac{tsp} graph $G^{(0)}$ for the second networked system. What is displayed is merely a subgraph of $G^{(0)}$. The whole graph $G^{(0)}$ can be obtained by putting all such subgraphs together.}\label{cost}
\end{figure}
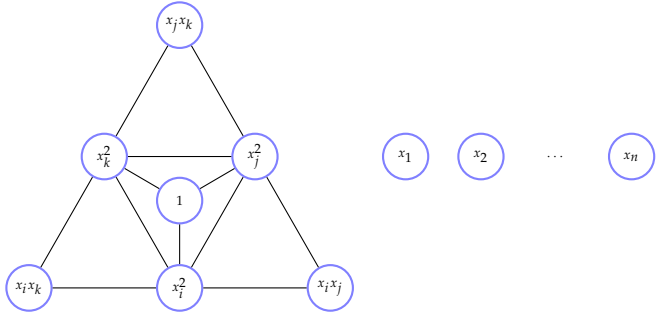

Here we solve the \ac{TSSOS} relaxation with $s=1$ for $n\in\{10,20,\ldots,50\}$.
This example was also examined in \citets{maj} to demonstrate the advantage of SDSOS programming compared to dense \ac{SOS} programming. The approach based on SDSOS programming was implemented in $\spot$ \citets{me} with $\mosek$ as a second-order cone programming solver. We display the results for both $\tssos$ and SDSOS in Table \ref{tb:network2} in which the row ``\#\ac{SDP} vars" indicates the number of decision variables used in \ac{SDP}s.

\begin{table}[htbp]
\caption{Results for the second networked system.}\label{tb:network2}
\renewcommand\arraystretch{1.2}
\centering
\resizebox{\linewidth}{!}{
\begin{tabular}{|c|c|c|c|c|c|c|}
\hline
\multicolumn{2}{|c|}{$n$}&$10$&$20$&$30$&$40$&$50$\\
\hline
\multirow{4}*{\#PSD blocks}&\multirow{3}*{TSSOS}&$3\times45,$&$3\times190,$&$3\times435,$&$3\times780,$&$3\times1225,$\\
&&$1\times10,$&$1\times20,$&$1\times30,$&$1\times40,$&$1\times50,$\\
&&$11\times1$&$21\times1$&$31\times1$&$41\times1$&$51\times1$\\
\cline{2-7}
&SDSOS&$2\times2145$&$2\times26565$&$2\times122760$&$2\times370230$&$2\times878475$\\
\hline
\multirow{2}*{\#SDP vars}&TSSOS&$346$&$1391$&$3136$&$5581$&$8726$\\
\cline{2-7}
&SDSOS&$6435$&$79695$&$368280$&$1110690$&$2635425$\\
\hline
\multirow{2}*{time}&TSSOS&$0.01$&$0.06$&$0.17$&$0.50$&$0.89$\\
\cline{2-7}
&SDSOS&$0.47$&$1.14$&$5.47$&$20$&$70$\\
\hline
\end{tabular}}
\end{table}

From the table we see that the $\tssos$ approach uses much less decision variables than the SDSOS approach, and hence is more efficient. On the other hand,
the $\tssos$ approach computes a positive definite form $V$ after selecting a value for $g$ up to $2$ (which is the same as the maximal value obtained by the dense \ac{SOS} approach) while the method in \citets{han} can select $g$ up to $1.8$ and the SDSOS approach only works out for a maximal value of $g$ up to around $1.5$.

\subsection{Constrained polynomial optimization}
Now we present the numerical results for constrained polynomial optimization problems.
As a first example, we minimize randomly generated\footnote{These polynomials can be downloaded at https://wangjie212.github.io/jiewang/code.html.} sparse polynomials $H_2,H_4,H_6$ over the unit ball $$\mathbf{X}=\left\{(x_1,\ldots,x_n)\in\R^n\,\middle|\, g_1=1-(x_1^2+\cdots+x_n^2)\ge0\right\}.$$
We solve these instances with {\tt TSSOS} as well as $\gloptipoly$. The related results are shown in Table \ref{unitball}. Note that approximately smallest chordal extensions are used for \ac{TS} and only the results of the first three steps of the \ac{TSSOS} hierarchy (for a fixed relaxation order) are displayed. From the table it can be seen that for each instance {\tt TSSOS} is significantly faster than $\gloptipoly$\footnote{$\gloptipoly$ also uses $\mosek$ as an \ac{SDP} solver.} without compromising accuracy.

\begin{table}[!ht]
\caption{Results for minimizing randomly generated polynomials over the unit ball; $d$ denotes the polynomial degree and $t$ denotes the number of terms.}\label{unitball}
\renewcommand\arraystretch{1.2}
\centering
\begin{tabular}{|c|c|c|c|c|c|c|c|c|c|c|}
\hline
&\multirow{2}*{$(n,d,t)$}&\multirow{2}*{$r$}&\multicolumn{4}{c|}{{\tt TSSOS}}&\multicolumn{3}{c|}{$\gloptipoly$}\\
\cline{4-10}
&&&$s$&mb&opt&time&mb&opt&time\\
\hline
\multirow{6}*{$H_2$}&$\multirow{6}*{(7,8,12)}$&\multirow{3}*{$4$}&$1$&$36$&\multirow{6}*{$0.1373$}&$0.36$&\multirow{3}*{$330$}&\multirow{3}*{$0.1373$}&\multirow{3}*{$34$}\\
\cline{4-5}\cline{7-7}
&&&$2$&$36$&&$0.52$&&&\\
\cline{4-5}\cline{7-7}
&&&$3$&$38$&&$1.6$&&&\\
\cline{3-5} \cline{7-10}
&&\multirow{3}*{$5$}&$1$&$36$&&$1.9$&\multirow{3}*{$792$}&\multirow{3}*{-}&\multirow{3}*{-}\\
\cline{4-5}\cline{7-7}
&&&$2$&$45$&&$3.9$&&&\\
\cline{4-5}\cline{7-7}
&&&$3$&$59$&&$34$&&&\\
\hline
\multirow{6}*{$H_4$}&$\multirow{6}*{(9,6,15)}$&\multirow{3}*{$3$}&$1$&$10$&\multirow{6}*{$0.8704$}&$0.15$&\multirow{3}*{$220$}&\multirow{3}*{$0.8704$}&\multirow{3}*{$16$}\\
\cline{4-5}\cline{7-7}
&&&$2$&$10$&&$0.22$&&&\\
\cline{4-5}\cline{7-7}
&&&$3$&$10$&&$0.25$&&&\\
\cline{3-5} \cline{7-10}
&&\multirow{3}*{$4$}&$1$&$55$&&$1.3$&\multirow{3}*{$715$}&\multirow{3}*{-}&\multirow{3}*{-}\\
\cline{4-5}\cline{7-7}
&&&$2$&$55$&&$2.0$&&&\\
\cline{4-5}\cline{7-7}
&&&$3$&$56$&&$2.8$&&&\\
\hline
\multirow{6}*{$H_6$}&$\multirow{6}*{(11,6,20)}$&\multirow{3}*{$3$}&$1$&$12$&\multirow{6}*{$0.1171$}&$0.28$&\multirow{3}*{$364$}&\multirow{3}*{$0.1171$}&\multirow{3}*{$115$}\\
\cline{4-5}\cline{7-7}
&&&$2$&$15$&&$0.36$&&&\\
\cline{4-5}\cline{7-7}
&&&$3$&$16$&&$0.60$&&&\\
\cline{3-5} \cline{7-10}
&&\multirow{3}*{$4$}&$1$&$78$&&$4.4$&\multirow{3}*{$1365$}&\multirow{3}*{-}&\multirow{3}*{-}\\
\cline{4-5}\cline{7-7}
&&&$2$&$78$&&$4.7$&&&\\
\cline{4-5}\cline{7-7}
&&&$3$&$78$&&$7.5$&&&\\
\hline
\end{tabular}
\end{table}

Next we present the numerical results (Table \ref{tb:gRfunction}) for minimizing the generalized Rosenbrock function over the unit ball:
\begin{equation*}
    f_{\text{gR}}(\x)=1+\sum_{i=1}^n\left(100(x_i-x_{i-1}^2)^2+(1-x_i)^2\right).
\end{equation*}


We approach this problem by solving the \ac{TSSOS} hierarchy with $r=2$. Here we compare the results obtained with approximately smallest chordal extensions ({\bf min}) and maximal chordal extensions ({\bf max}). In Table \ref{tb:gRfunction}, we can see that {\tt TSSOS} scales much better with approximately smallest chordal extensions than with maximal chordal extensions while providing the same optimum.

\begin{table}[htbp]
\caption{Results for the generalized Rosenbrock function.}\label{tb:gRfunction}
\renewcommand\arraystretch{1.2}
\centering
\begin{tabular}{c|cccc|cccc}
\hline
\multirow{2}*{$n$}
&\multicolumn{4}{c|}{{\bf min}}&\multicolumn{4}{c}{{\bf max}}\\
\cline{2-9}
&$s$&mb&opt&time&$s$&mb&opt&time\\
\hline
\multirow{2}*{$20$}&\multirow{2}*{$1$}&\multirow{2}*{$21$}&\multirow{2}*{$18.25$}&\multirow{2}*{$0.19$}&$1$&$58$&$18.25$&$8.2$\\
&&&&&$2$&$211$&$18.25$&$45$\\
\hline
\multirow{2}*{$60$}&\multirow{2}*{$1$}&\multirow{2}*{$61$}&\multirow{2}*{$57.85$}&\multirow{2}*{$6.6$}&$1$&$178$&-&-\\
&&&&&$2$&$1831$&-&-\\
\hline
\multirow{2}*{$100$}&\multirow{2}*{$1$}&\multirow{2}*{$101$}&\multirow{2}*{$97.45$}&\multirow{2}*{$85$}&$1$&$308$&-&-\\
&&&&&$2$&$5051$&-&-\\
\hline
\multirow{2}*{$140$}&\multirow{2}*{$1$}&\multirow{2}*{$141$}&\multirow{2}*{$137.05$}&\multirow{2}*{$448$}&$1$&$428$&-&-\\
&&&&&$2$&$9871$&-&-\\
\hline
\multirow{2}*{$180$}&\multirow{2}*{$1$}&\multirow{2}*{$181$}&\multirow{2}*{$176.65$}&\multirow{2}*{$1495$}&$1$&$548$&-&-\\
&&&&&$2$&$16291$&-&-\\
\hline
\end{tabular}
\end{table}

\section{Notes and sources}
Besides the Newton polytope method and the approach given in Chapter~\ref{sec:tssos-basis}, there are also other algorithms that provide a possibly smaller monomial basis; see for instance \citets{kojima2005sparsity} and \citets{yang2020one}.

The results on the \ac{TSSOS} hierarchy presented in this chapter have been published in \citets{tssos,chordaltssos}.
The idea of exploiting \ac{TS} in \ac{SOS} decompositions was initially proposed in \citets{wang2019new} for the unconstrained case and sooner after extended to the constrained case in \citets{tssos,chordaltssos}.
The exploitation of sign symmetries in \ac{SOS} decompositions was firstly discussed in \citets{lo1} in the unconstrained setting. Theorem \ref{sec6-thm2}, stated in \citets[Theorem 6.5]{tssos}, relates the convergence of block structures arising from the \ac{TSSOS} hierarchy (for a fixed relaxation order) to sign symmetries.
For more extensive comparison of {\tt TSSOS} with the polynomial optimization solvers $\gloptipoly$ \citets{gloptipoly}, $\yalmip$ \citets{YALMIP}, and $\sparsepop$ \citets{waki08}, the reader is referred to \citets{tssos,chordaltssos}.

\input{ts.bbl}

%% file: ts.bbl
\providecommand{\etalchar}[1]{$^{#1}$}

%% file: cstssos.tex
\chapter{Exploiting both correlative and term sparsity}
\label{chap:cstssos}

In previous chapters, we have seen that how \ac{CS} or \ac{TS} of \ac{POP}s helps to reduce the size of \ac{SDP} relaxations arising from the \ac{moment-SOS} hierarchy. For large-scale \ac{POP}s, it is natural to ask whether one can exploit \ac{CS} and \ac{TS} simultaneously to further reduce the size of \ac{SDP} relaxations. As we shall see in this chapter, the answer is affirmative.

\section{The CS-TSSOS hierarchy}\label{chap8:sec1}
The underlying idea to exploit \ac{CS} and \ac{TS} simultaneously in the \ac{moment-SOS} hierarchy consists of the following two steps:

\begin{enumerate}[(1)]
\item decomposing the set of variables into a tuple of cliques $\{I_k\}_{k\in[p]}$ by exploiting \ac{CS} as in Chapter~\ref{chap:cs};
\item applying the iterative procedure for exploiting \ac{TS} to each decoupled subsystem involving variables $\x(I_k)$ as in Chapter~\ref{chap:tssos}.
\end{enumerate}

More concretely, let us fix a relaxation order $r\ge r_{\min}$. Suppose that $G^{\text{csp}}$ is the \ac{csp} graph associated to \ac{POP} \eqref{eq:pop} defined as in Chapter~\ref{chap:cs}, $(G^{\text{csp}})'$ is a chordal extension of $G^{\text{csp}}$, and $I_k,k\in[p]$ are the maximal cliques of $(G^{\text{csp}})'$ with cardinality being denoted by $n_k,k\in[p]$. As in Chapter~\ref{chap:cs}, the set of variables $\x$ is decomposed into $\x(I_1), \x(I_2), \ldots, \x(I_p)$ by exploiting \ac{CS}. In addition, assume that the constraints are assigned to the variable cliques according to $J_1,\ldots,J_p,J'$ as defined in Chapter~\ref{chap:cs}.

Now we apply the iterative procedure for exploiting \ac{TS} to each subsystem involving variables $\x(I_k)$, $k\in[p]$ in the following way. Let
\begin{equation}\label{sec4-eq0cstssos}
    \sA\coloneqq \supp(f)\cup\bigcup_{j=1}^m\supp(g_j)\text{ and }
    \sA_k\coloneqq \{\a\in\sA\mid\supp(\a)\subseteq I_k\}
\end{equation}
for $k\in[p]$. Let $\N^{n_k}_{r-d_j}$ be the standard monomial basis for $j\in\{0\}\cup J_k,k\in[p]$. 
Let $G_{r,k}^{\text{tsp}}$ be the \ac{tsp} graph with nodes $\N^{n_k}_{r}$ associated to the support $\sA_k$ defined as in Chapter~\ref{chap:tssos}, i.e., its node set is $\N^{n_k}_{r}$ and $\{\b,\g\}$ is an edge if $\b+\g\in\sA_k\cup2\N^{n_k}_{r}$. 
Note that here we embed $\N^{n_k}_{r}$ into $\N^{n}_{r}$ via the map
$\a=(\alpha_i)\in\N^{n_k}_{r}\mapsto\a'=(\alpha'_i)\in\N^{n}_{r}$ satisfying
\begin{equation*}
    \alpha'_i=\begin{cases}
    \alpha_i,\quad&\text{if } i\in I_k,\\
    0,\quad&\text{otherwise.}
    \end{cases}
\end{equation*}

Let us define $G_{r,k,0}^{(0)}\coloneqq G_{r,k}^{\text{tsp}}$ and $G_{r,k,j}^{(0)},j\in J_k, k\in[p]$ are all empty graphs. Next for each $j\in\{0\}\cup J_k$ and each $k\in[p]$, we iteratively define an ascending chain of graphs $(G_{r,k,j}^{(s)}(V_{r,k,j},E_{r,k,j}^{(s)}))_{s\ge1}$ with $V_{r,k,j}\coloneqq \N^{n_k}_{r-d_j}$ via two successive operations:\\
(1) {\bf support extension}. Define $F_{r,k,j}^{(s)}$ to be the graph with nodes $V_{r,k,j}$ and with edges
\begin{equation}\label{sec4-eqcstssos}
E(F_{r,k,j}^{(s)})=\left\{\{\b,\g\}\mid\b\ne\g\in V_{r,k,j},\left(\b+\g+\supp(g_j)\right)\cap\sC_{r}^{(s-1)}\ne\emptyset\right\},	
\end{equation}
where	
\begin{equation}\label{sec4-eq11cstssos}
\sC_{r}^{(s-1)}\coloneqq \bigcup_{k=1}^p\left(\bigcup_{j\in \{0\}\cup J_k}\left(\supp(g_j)+\supp(G_{r,k,j}^{(s-1)})\right)\right).
\end{equation}
(2) {\bf chordal extension}. Let
\begin{equation}\label{sec4-graphcstssos}
G_{r,k,j}^{(s)}\coloneqq (F_{r,k,j}^{(s)})',\quad j\in\{0\}\cup J_k, k\in[p].
\end{equation}
It is clear by construction that the sequences of graphs $(G_{r,k,j}^{(s)})_{s\ge1}$ stabilize for all $j\in\{0\}\cup J_k,k\in[p]$ after finitely many steps.

\begin{example}\label{ex:ex1cstssos}
Let $f=1+x_1^2+x_2^2+x_3^2+x_1x_2+x_2x_3+x_3$ and consider the unconstrained \ac{POP}: $\min\{f(\x): \x\in\R^n\}$. Take the relaxation order $r=r_{\min}=1$. The variables are decomposed into two cliques: $\{x_1,x_2\}$ and $\{x_2,x_3\}$. The \ac{tsp} graphs with respect to these two cliques are illustrated in Figure \ref{fig:ex1cstssos}. The left graph corresponds to the first clique: $x_1$ and $x_2$ are connected because of the term $x_1x_2$. The right graph corresponds to the second clique: $1$ and $x_3$ are connected because of the term $x_3$; $x_2$ and $x_3$ are connected because of the term $x_2x_3$. It is not hard to see that the graph sequences $(G_{1,k}^{(s)})_{s\ge1},k=1,2$ (the subscript $j$ is omitted here since there is no constraint) stabilize at $s=2$ if the maximal chordal extension is used in \eqref{sec4-graphcstssos}.
	
\begin{figure}[htbp]
\centering
\begin{minipage}{0.3\linewidth}
{\tiny
\begin{tikzpicture}[every node/.style={circle, draw=blue!50, thick, minimum size=6mm}]
\node (n1) at (90:1) {$1$};
\node (n2) at (330:1) {$x_2$};
\node (n3) at (210:1) {$x_1$};
\draw (n2)--(n3);
\end{tikzpicture}}
\end{minipage}
\begin{minipage}{0.3\linewidth}
{\tiny
\begin{tikzpicture}[every node/.style={circle, draw=blue!50, thick, minimum size=6mm}]
\node (n1) at (90:1) {$1$};
\node (n2) at (330:1) {$x_3$};
\node (n3) at (210:1) {$x_2$};
\draw (n1)--(n2);
\draw (n2)--(n3);
\end{tikzpicture}}
\end{minipage}
\caption{The \ac{tsp} graphs of Example \ref{ex:ex1cstssos}.}
\label{fig:ex1cstssos}
\end{figure}
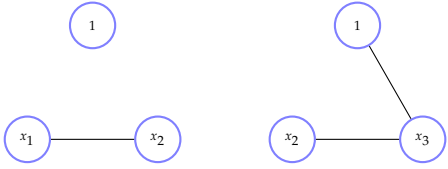
\end{example}

Let $t_{k,j}\coloneqq |\N^{n_k}_{r-d_j}|=\binom{n_k+r-d_j}{r-d_j}$ for all $k,j$. Then with $s\ge1$, the moment relaxation based on correlative-term sparsity for \ac{POP} \eqref{eq:pop} is given by
\begin{equation}\label{eq:cstsmom}
\P_{\csts}^{r,s}:\quad
\begin{cases}
\inf\limits_{\y}&L_{\y}(f)\\
\rm{ s.t.}& \B_{G_{r,k,0}^{(s)}}\circ\M_{r}(\y, I_k)\in\Pi_{G_{r,k,0}^{(s)}}(\Sbb^+_{t_{k,0}}),\quad k\in[p]\\
& \B_{G_{r,k,j}^{(s)}}\circ\M_{r-d_j}(g_j\y, I_k)\in\Pi_{G_{r,k,j}^{(s)}}(\Sbb^+_{t_{k,j}}),\quad j\in J_k, k\in[p]\\
&L_{\y}(g_j)\ge0,\quad j\in J'\\
&y_{\mathbf{0}}=1
\end{cases}
\end{equation}
with optimum denoted by $f_{\csts}^{r,s}$.

For all $k,j$, let us write $\M_{r-d_j}(g_j\y, I_k)=\sum_{\a}\D_{\a}^{k,j}y_{\a}$ for appropriate symmetry matrices $\{\D_{\a}^{k,j}\}$ and $g_j=\sum_{\a}g_{j,\a}\x^{\a}$. Then for each $s\ge1$, the dual of $\P_{\csts}^{r,s}$ \eqref{eq:cstsmom} reads as
\begin{equation}\label{sec4-eq3cstssos}
(\P_{\csts}^{r,s})^*:
\begin{cases}
\sup\limits_{\G_{k,j},\lambda_j,b}&b\\
\,\,\,\,\rm{ s.t.}&\sum_{k=1}^p\sum_{j\in \{0\}\cup J_k}\langle \G_{k,j},\D_{\a}^{k,j}\rangle+\sum_{j\in J'}\lambda_j g_{j,\a}\\
&\quad\quad\quad\quad\quad\quad\quad\quad\quad\quad\quad\quad+b\delta_{\mathbf{0}\a}=f_{\a},\quad \forall\a\in\sC_{r}^{(s)}\\
&\G_{k,j}\in\mathbf{S}_+^{t_{k,j}}\cap\mathbf{S}_{G_{r,k,j}^{(s)}},\quad j\in\{0\}\cup J_k, k\in[p]\\
&\lambda_j\ge0,\quad j\in J'
\end{cases}
\end{equation}
where $\sC_{r}^{(s)}$ is defined in \eqref{sec4-eq11cstssos}.

The primal-dual \ac{SDP} relaxations \eqref{eq:cstsmom}--\eqref{sec4-eq3cstssos} is called the \ac{CS-TSSOS} hierarchy associated with $\P$ \eqref{eq:pop}, which is indexed by two parameters: the relaxation order $r$ and the sparse order $s$.

\begin{example}\label{ex2cstssos}
Let $f=1+\sum_{i=1}^6x_i^4+x_1x_2x_3+x_3x_4x_5+x_3x_4x_6+x_3x_5x_6+x_4x_5x_6$, and consider the unconstrained \ac{POP}: $\min\{f(\x): \x\in\R^n\}$.
Let us apply the \ac{CS-TSSOS} hierarchy (using the maximal chordal extension in \eqref{sec4-graphcstssos}) to this problem by taking $r=r_{\min}=2,s=1$. First, according to the \ac{csp} graph (Figure \ref{ex2-3cstssos}), we decompose the variables into two cliques: $\{x_1,x_2,x_3\}$ and $\{x_3,x_4,x_5,x_6\}$. The \ac{tsp} graphs for the first clique and the second clique are shown in Figure \ref{ex2-1cstssos} and Figure \ref{ex2-2cstssos}, respectively. For the first clique one obtains four blocks of \ac{SDP} matrices with respective sizes $4,2,2,2$. For the second clique one obtains two blocks of \ac{SDP} matrices with respective sizes $5,10$. As a result, the original \ac{SDP} matrix of size $28$ has been reduced to six blocks of maximal size $10$.
	
Alternatively, if one applies the \ac{TSSOS} hierarchy (using the maximal chordal extension in \eqref{sec4-graphtssos}) directly to this problem by taking $r=r_{\min}=2,s=1$ (i.e., without decomposing variables), then the \ac{tsp} graph is shown in Figure \ref{ex2-4cstssos} and one thereby obtains $11$ PSD blocks with respective sizes $7,2,2,2,1,1,1,1,1,1,10$. Compared to the \ac{CS-TSSOS} case, there are six additional blocks of size one and the two blocks with respective sizes $4,5$ are replaced by a single block of size $7$.
	
\begin{figure}[!t]
\centering
{\tiny
\begin{tikzpicture}[every node/.style={circle, draw=blue!50, thick, minimum size=6mm}]
\node (n1) at (-1.732,1) {$1$};
\node (n2) at (-1.732,-1) {$2$};
\node (n3) at (0,0) {$3$};
\node (n4) at (1.414,1.414) {$4$};
\node (n5) at (1.4147,-1.414) {$5$};
\node (n6) at (2.828,0) {$6$};
\draw (n1)--(n2);
\draw (n1)--(n3);
\draw (n2)--(n3);
\draw (n3)--(n4);
\draw (n3)--(n5);
\draw (n3)--(n6);
\draw (n4)--(n5);
\draw (n4)--(n6);
\draw (n5)--(n6);
\end{tikzpicture}\caption{The \ac{csp} graph of Example \ref{ex2cstssos}.}\label{ex2-3cstssos}}
\end{figure}
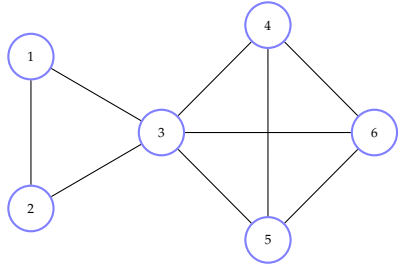
	
\begin{figure}[!t]
\centering
{\tiny
\begin{tikzpicture}[every node/.style={circle, draw=blue!50, thick, minimum size=6mm}]
\node (n1) at (0,0) {$1$};
\node (n8) at (-2,0) {$x_3^2$};
\node (n9) at (-2,-2) {$x_2^2$};
\node (n10) at (0,-2) {$x_1^2$};
\node (n2) at (2,0) {$x_1$};
\node (n3) at (4,0) {$x_2$};
\node (n4) at (6,0) {$x_3$};
\node (n5) at (2,-2) {$x_2x_3$};
\node (n6) at (4,-2) {$x_1x_3$};
\node (n7) at (6,-2) {$x_1x_2$};
\draw (n1)--(n8);
\draw (n1)--(n9);
\draw (n1)--(n10);
\draw (n8)--(n9);
\draw (n8)--(n10);
\draw (n9)--(n10);
\draw (n2)--(n5);
\draw (n3)--(n6);
\draw (n4)--(n7);
\end{tikzpicture}\caption{The \ac{tsp} graph for the first clique of Example \ref{ex2cstssos}.}\label{ex2-1cstssos}}
\end{figure}
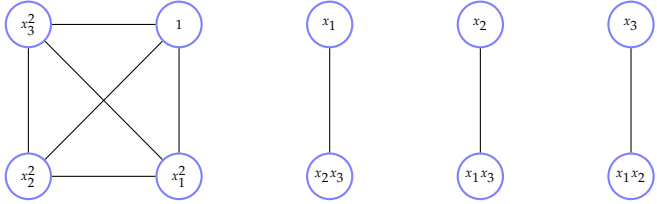
	
\begin{figure}[!t]
\centering
{\tiny
\begin{tikzpicture}[every node/.style={circle, draw=blue!50, thick, minimum size=6mm}]
\node (n1) at (90:2) {$1$};
\node (n2) at (162:2) {$x_6^2$};
\node (n3) at (234:2) {$x_5^2$};
\node (n4) at (306:2) {$x_4^2$};
\node (n5) at (18:2) {$x_3^2$};
\draw (n2)--(n3);
\draw (n2)--(n4);
\draw (n2)--(n5);
\draw (n3)--(n4);
\draw (n3)--(n5);
\draw (n3)--(n1);
\draw (n4)--(n5);
\draw (n4)--(n5);
\draw (n4)--(n1);
\draw (n5)--(n1);
\draw (n1)--(n2);
\node[xshift=150] (n6) at (90:2) {$x_3$};
\node[xshift=150] (n7) at (126:2) {$x_5x_6$};
\node[xshift=150] (n8) at (162:2) {$x_4x_6$};
\node[xshift=150] (n9) at (198:2) {$x_4x_5$};
\node[xshift=150] (n10) at (234:2) {$x_3x_6$};
\node[xshift=150] (n12) at (270:2) {$x_3x_5$};
\node[xshift=150] (n13) at (306:2) {$x_3x_4$};
\node[xshift=150] (n14) at (342:2) {$x_6$};
\node[xshift=150] (n15) at (54:2) {$x_4$};
\node[xshift=150] (n16) at (18:2) {$x_5$};
\draw (n6)--(n9);
\draw (n15)--(n12);
\draw (n16)--(n13);
\draw (n6)--(n8);
\draw (n15)--(n10);
\draw (n14)--(n13);
\draw (n6)--(n7);
\draw (n14)--(n12);
\draw (n16)--(n10);
\draw (n15)--(n7);
\draw (n16)--(n8);
\draw (n14)--(n9);
\end{tikzpicture}\caption{The \ac{tsp} graph for the second clique of Example \ref{ex2cstssos}}\label{ex2-2cstssos}.}
\end{figure}
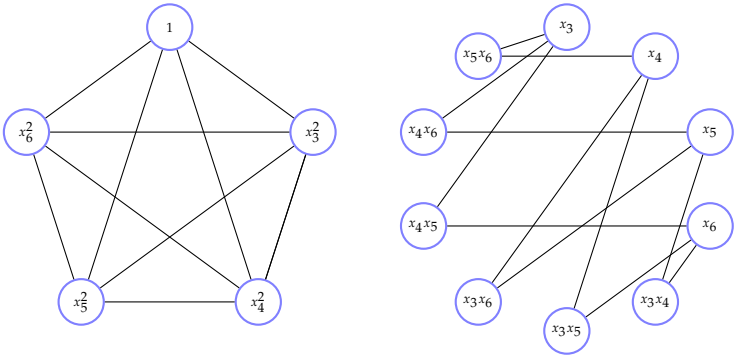 
	
\begin{figure}[!t]
\centering
\begin{minipage}{0.4\linewidth}
{\tiny
\begin{tikzpicture}[every node/.style={circle, draw=blue!50, thick, minimum size=6mm}]
\node (n1) at (90:2) {$1$};
\node (n2) at (38.6:2) {$x_1^2$};
\node (n3) at (347:2) {$x_2^2$};
\node (n4) at (295.6:2) {$x_3^2$};
\node (n5) at (244.2:2) {$x_4^2$};
\node (n6) at (192.8:2) {$x_5^2$};
\node (n7) at (141.4:2) {$x_6^2$};
\draw (n2)--(n3);
\draw (n2)--(n4);
\draw (n2)--(n5);
\draw (n3)--(n4);
\draw (n3)--(n5);
\draw (n3)--(n1);
\draw (n4)--(n5);
\draw (n4)--(n5);
\draw (n4)--(n1);
\draw (n5)--(n1);
\draw (n1)--(n2);
\draw (n1)--(n6);
\draw (n2)--(n6);
\draw (n3)--(n6);
\draw (n4)--(n6);
\draw (n5)--(n6);
\draw (n1)--(n7);
\draw (n2)--(n7);
\draw (n3)--(n7);
\draw (n4)--(n7);
\draw (n5)--(n7);
\draw (n6)--(n7);
\end{tikzpicture}}
\end{minipage}%
\begin{minipage}{0.25\linewidth}
{\tiny
\begin{tikzpicture}[every node/.style={circle, draw=blue!50, thick, minimum size=6mm}]
\node (n2) at (1,2) {$x_1$};
\node (n3) at (2,2) {$x_2$};
\node (n4) at (3,2) {$x_3$};
\node (n5) at (1,0.5) {$x_2x_3$};
\node (n6) at (2,0.5) {$x_1x_3$};
\node (n7) at (3,0.5) {$x_1x_2$};
\draw (n2)--(n5);
\draw (n3)--(n6);
\draw (n4)--(n7);
\node (n8) at (1,-0.5) {$x_1x_4$};
\node (n9) at (2,-0.5) {$x_1x_5$};
\node (n10) at (3,-0.5) {$x_1x_6$};
\node (n11) at (1,-1.5) {$x_2x_4$};
\node (n12) at (2,-1.5) {$x_2x_5$};
\node (n13) at (3,-1.5) {$x_2x_6$};
\end{tikzpicture}}
\end{minipage}%
\begin{minipage}{0.42\linewidth}
{\tiny
\begin{tikzpicture}[every node/.style={circle, draw=blue!50, thick, minimum size=6mm}]
\node (n6) at (90:2) {$x_3$};
\node (n7) at (126:2) {$x_5x_6$};
\node (n8) at (162:2) {$x_4x_6$};
\node (n9) at (198:2) {$x_4x_5$};
\node (n10) at (234:2) {$x_3x_6$};
\node (n12) at (270:2) {$x_3x_5$};
\node (n13) at (306:2) {$x_3x_4$};
\node (n14) at (342:2) {$x_6$};
\node (n15) at (54:2) {$x_4$};
\node (n16) at (18:2) {$x_5$};
\draw (n6)--(n9);
\draw (n15)--(n12);
\draw (n16)--(n13);
\draw (n6)--(n8);
\draw (n15)--(n10);
\draw (n14)--(n13);
\draw (n6)--(n7);
\draw (n14)--(n12);
\draw (n16)--(n10);
\draw (n15)--(n7);
\draw (n16)--(n8);
\draw (n14)--(n9);
\end{tikzpicture}}
\end{minipage}\caption{The \ac{tsp} graph without decomposing variables of Example \ref{ex2cstssos}.}\label{ex2-4cstssos}
\end{figure}
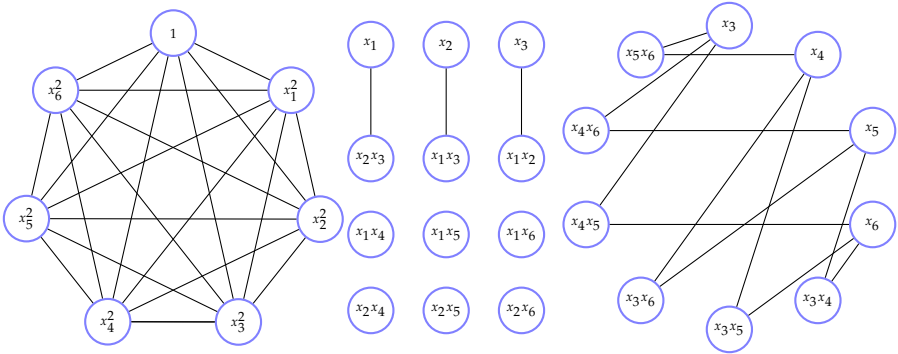
\end{example}

We summarize the basic properties of the \ac{CS-TSSOS} hierarchy in the next theorem.
\begin{theoremf}\label{sec4-prop1cstssos}
Let $f\in\R[\x]$ and $\X$ be defined as in \eqref{eq:defX}. Then the following hold:
\begin{enumerate}
\item If $\X$ has a nonempty interior, then there is no duality gap between $\P_{\csts}^{r,s}$ and $(\P_{\csts}^{r,s})^*$ for any $r\ge r_{\min}$ and $s\ge1$.
\item For any $r\ge r_{\min}$, the sequence $(f_{\csts}^{r,s})_{s\ge1}$ is monotonically non-decreasing and $ f_{\csts}^{r,s}\le f_{\cs}^{r}$ for all $s$ with $f_{\cs}^{r}$ being defined in Section~\ref{sec:cshierarchy}.
\item For any $s\ge 1$, the sequence $(f_{\csts}^{r,s})_{r\ge r_{\min}}$ is monotonically non-decreasing.
\end{enumerate}
\end{theoremf}
\begin{proof}
1. By the duality theory of convex programming, this easily follows from Theorem 3.6 of \citecstssos{Las06} and Theorem \ref{th:sparsesdpsum}.

2. By construction, we have $G_{r,k,j}^{(s)}\subseteq G_{r,k,j}^{(s+1)}$ for all $r,k,j$ and for all $s$. It follows that each maximal clique of $G_{r,k,j}^{(s)}$ is contained in some maximal clique of $G_{r,k,j}^{(s+1)}$. Hence by Theorem \ref{th:sparsesdpsum}, $\P_{\csts}^{r,s}$ is a relaxation of $\P_{\csts}^{r,s+1}$ and is clearly also a relaxation of $\P_{\cs}^{r}$. Therefore, $(f_{\csts}^{r,s})_{s\ge1}$ is monotonically non-decreasing and $f_{\csts}^{r,s}\le f_{\cs}^{r}$ for all $s$.

3. The conclusion follows if we can show that the inclusion $G_{r,k,j}^{(s)}\subseteq G_{r+1,k,j}^{(s)}$ holds for all $r,k,j,s$, since by Theorem \ref{th:sparsesdpsum} this implies that $\P_{\csts}^{r,s}$ is a relaxation of $\P_{\csts}^{r+1,s}$. Let us prove $G_{r,k,j}^{(s)}\subseteq G_{r+1,k,j}^{(s)}$ by induction on $s$. For $s=1$, we have $G_{r,k,0}^{(0)}=G_{r,k}^{\text{tsp}}\subseteq G_{r+1,k}^{\text{tsp}}=G_{r+1,k,0}^{(0)}$, which together with $\eqref{sec4-eqcstssos}$-$\eqref{sec4-eq11cstssos}$ implies that $F_{r,k,j}^{(1)}\subseteq F_{r+1,k,j}^{(1)}$ for $j\in\{0\}\cup J_k,k\in[p]$. It then follows that $G_{r,k,j}^{(1)}=(F_{r,k,j}^{(1)})'\subseteq (F_{r+1,k,j}^{(1)})'=G_{r+1,k,j}^{(1)}$. Now assume that $G_{r,k,j}^{(s)}\subseteq G_{r+1,k,j}^{(s)}$, $j\in\{0\}\cup J_k,k\in[p]$, hold for some $s\geq 1$. Then by $\eqref{sec4-eqcstssos}$-$\eqref{sec4-eq11cstssos}$ and by the induction hypothesis, we have $F_{r,k,j}^{(s+1)}\subseteq F_{r+1,k,j}^{(s+1)}$ for $j\in\{0\}\cup J_k,k\in[p]$. Thus $G_{r,k,j}^{(s+1)}=(F_{r,k,j}^{(s+1)})'\subseteq (F_{r+1,k,j}^{(s+1)})'=G_{r+1,k,j}^{(s+1)}$ which completes the induction. 
\qed
\end{proof}

From Theorem \ref{sec4-prop1cstssos}, we deduce the following two-level hierarchy of lower bounds for the optimum $f_{\min}$ of $\P$ \eqref{eq:pop}:

\begin{equation}\label{mixhierc}
\begin{matrix}
f_{\csts}^{r_{\min},1}&\le&f_{\csts}^{r_{\min},2}&\le&\cdots&\le&f_{\csts}^{r_{\min}}\\
\vge&&\vge&&&&\vge\\
f_{\csts}^{r_{\min}+1,1}&\le&f_{\csts}^{r_{\min}+1,2}&\le&\cdots&\le&f_{\csts}^{r_{\min}+1}\\
\vge&&\vge&&&&\vge\\
\vdots&&\vdots&&\vdots&&\vdots\\
\vge&&\vge&&&&\vge\\
f_{\csts}^{r,1}&\le&f_{\csts}^{r,2}&\le&\cdots&\le&f_{\csts}^{r}\\
\vge&&\vge&&&&\vge\\
\vdots&&\vdots&&\vdots&&\vdots\\
\end{matrix}
\end{equation}

As we have known for the \ac{TSSOS} hierarchy, the block structure arising from the \ac{CS-TSSOS} hierarchy is consistent with the sign symmetries of the \ac{POP}. More precisely, we have the following theorem.

\begin{theoremf}\label{prop-sscstssos}
Let $\sA$ be defined as in \eqref{sec4-eq0cstssos}, $\sC_{r}^{(s)}$ be defined as in \eqref{sec4-eq11cstssos}, and assume that the sign symmetries of $\sA$ are represented by the column vectors of the binary matrix $\Rb$. Then for any $r\ge r_{\min}$, $s\ge1$ and any $\a\in\sC_{r}^{(s)}$, it holds $\Rb^\intercal\a\equiv\mathbf{0}\,(\rm{mod}\,2)$. As a consequence, if $\b,\g$ belong to the same block in the \ac{CS-TSSOS} relaxations, then $\Rb^\intercal(\b+\g)\equiv\mathbf{0}\,(\rm{mod}\,2)$.
\end{theoremf}

\section{Global convergence}
We next show that if the chordal extension in \eqref{eq:cstsmom} is chosen to be \emph{maximal}, then 
for any relaxation order $r\ge r_{\min}$, the sequence of optima $(f_{\csts}^{r,s})_{s\ge1}$ arising from the \ac{CS-TSSOS} hierarchy converges to the optimum $f^{r}_{\cs}$ of the \ac{CSSOS} relaxation. 

It is clear by construction that the sequences of graphs $(G_{r,k,j}^{(s)})_{s\ge1}$ stabilize for all $j\in\{0\}\cup J_k,k\in[p]$ after finitely many steps. We denote the resulting stabilized graphs by $G_{r,k,j}^{(\bullet)},j\in \{0\}\cup J_k,k\in[p]$ and the corresponding \ac{SDP} \eqref{eq:cstsmom} by $\P_{\csts}^{r,\bullet}$.

\begin{theoremf}\label{sec4-thm2cstssos}
If one uses the maximal chordal extension in \eqref{sec4-graphcstssos}, then for any $r\ge r_{\min}$, the sequence $(f_{\csts}^{r,s})_{s\ge1}$ converges to $f^{r}_{\cs}$ in finitely many steps.
\end{theoremf}
\begin{proof}
Let $\y=(y_{\a})$ be an arbitrary feasible solution of $\P_{\csts}^{r,\bullet}$ and $f_{\csts}^{r,\bullet}$ be the optimum of $\P_{\csts}^{r,\bullet}$. Note that $\{y_{\a}\mid\a\in\bigcup_{k=1}^p(\cup_{j\in\{0\}\cup J_k}(\supp(g_j)+\supp(G_{r,k,j}^{(\bullet)})))\}$ is the set of decision variables involved in $\P_{\csts}^{r,\bullet}$. Let $\RR$ be the set of decision variables involved in $\P_{\cs}^{r}$ \eqref{eq:csmom}.
We then define a vector $\overline{\y}=(\overline{y}_{\a})_{\a\in\RR}$ as follows:
$$\overline{y}_{\a}=\begin{cases}y_{\a},\quad&\text{if }\a\in\bigcup_{k=1}^p(\cup_{j\in\{0\}\cup J_k}(\supp(g_j)+\supp(G_{r,k,j}^{(\bullet)}))),\\
0,\quad\quad&\text{otherwise}.
\end{cases}$$
By construction and since $G_{r,k,j}^{(\bullet)}$ stabilizes under support extension for all $k,j$, we have $\M_{r-d_j}(g_j\overline{\y},I_k)=\B_{G_{r,k,j}^{(\bullet)}}\circ \M_{r-d_j}(g_j\y,I_k)$.
As the maximal chordal extension is chosen for \eqref{sec4-graphcstssos}, the matrix $\B_{G_{r,k,j}^{(\bullet)}}\circ \M_{r-d_j}(g_j\y,I_k)$ is block diagonal up to permutation. It follows from $\B_{G_{r,k,j}^{(\bullet)}}\circ \M_{r-d_j}(g_j\y, I_k)\in\Pi_{G_{r,k,j}^{(\bullet)}}(\mathbf{S}_+^{t_{k,j}})$ that $\M_{r-d_j}(g_j\overline{\y},I_k)\succeq0$ for $j\in \{0\}\cup J_k, k\in[p]$. Therefore $\overline{\y}$ is a feasible solution of $\P_{\cs}^{r}$ and so $L_{\y}(f)=L_{\overline{\y}}(f)\ge f^{r}_{\cs}$.
Hence $f_{\csts}^{r,\bullet}\ge f^{r}_{\cs}$ since $\y$ is an arbitrary feasible solution of $\P_{\csts}^{r,\bullet}$. 
By Theorem \ref{sec4-prop1cstssos}, we already have $f_{\csts}^{r,\bullet}\le f^{r}_{\cs}$. Therefore, $f_{\csts}^{r,\bullet}=f^{r}_{\cs}$.
\qed
\end{proof}

By Theorem 3.6 in \citecstssos{Las06}, the sequence $(f^{r}_{\cs})_{r\ge r_{\min}}$ converges to the global optimum $f_{\min}$ of \ac{POP} \eqref{eq:pop} (after adding some redundant quadratic constraints). Therefore, this together with Theorem \ref{sec4-thm2cstssos} offers the global convergence of the \ac{CS-TSSOS} hierarchy.

Proceeding along Theorem \ref{sec4-prop1cstssos}, we are able to provide a \emph{sparse representation} theorem based on both \ac{CS} and \ac{TS} for a polynomial positive on a compact basic semialgebraic set. 
\begin{theoremf}\label{th:cstssos}
Let $f\in\R[\x]$, $\X\subseteq\R^n$ and $\{I_k\}_{k=1}^p,\{J_k\}_{k=1}^p$ be defined in Assumption \eqref{hyp:cs}.
Assume that the sign symmetries of $\sA=\supp(f)\cup\bigcup_{j=1}^m\supp(g_j)$ are represented by the columns of the binary matrix $\Rb$. If $f$ is positive on $\X$, then $f$ admits a representation of form
\begin{equation}\label{sparse-certificate}
f=\sum_{k=1}^p\left(\sigma_{k,0}+\sum_{j\in J_k}\sigma_{k,j}g_j\right),
\end{equation}
for some polynomials $\sigma_{k,j}\in\Sigma[\x(I_k)],j\in \{0\}\cup J_k, k\in[p]$, satisfying $\Rb^{\intercal}\a\equiv\mathbf{0}$ $(\rm{mod}\,2)$ for any $\a\in\supp(\sigma_{k,j})$. 
\end{theoremf}
\begin{proof}
By Corollary 3.9 of \citecstssos{Las06} (or Theorem 5 of \citecstssos{grimm2007note}), there exist polynomials $\sigma'_{k,j}\in\Sigma[\x(I_k)],j\in\{0\}\cup J_k,k\in[p]$ such that
\begin{equation}\label{sec4-eq7cstssos}
	f=\sum_{k=1}^p\left(\sigma'_{k,0}+\sum_{j\in J_k}\sigma'_{k,j}g_j\right).
\end{equation}
Let $r=\max\{\lceil\deg(\sigma'_{k,j}g_j)/2\rceil: j\in \{0\}\cup J_k, k\in[p]\}$. Let $\G'_{k,j}$ be a \ac{PSD} Gram matrix associated with $\sigma'_{k,j}$ and indexed by the monomial basis $\N^{n_k}_{r-d_j}$. Then for all $k,j$, we define $\G_{k,j}\in\mathbf{S}^{t_{k,j}}$ (indexed by $\N^{n_k}_{r-d_j}$) by
\begin{equation*}
	[\G_{k,j}]_{\b\g}\coloneqq \begin{cases}
	[Q'_{k,j}]_{\b\g}, \quad&\text{if }\Rb^\intercal(\b+\g)\equiv\mathbf{0}$ $(\rm{mod}\,2),\\
	0, \quad&\text{otherwise,}
\end{cases}
\end{equation*}
and let $\sigma_{k,j}=(\x^{\N^{n_k}_{r-d_j}})^\intercal \G_{k,j}\x^{\N^{n_k}_{r-d_j}}$.
One can easily verify that $\G_{k,j}$ is block diagonal up to permutation (see also \citecstssos{tssos}) and each block is a principal submatrix of $\G'_{k,j}$. Then the positive semidefiniteness of $\G'_{k,j}$ implies that $\G_{k,j}$ is also positive semidefinite. Thus $\sigma_{k,j}\in\Sigma[\x(I_k)]$.
	
By construction, substituting $\sigma'_{k,j}$ with $\sigma_{k,j}$ in \eqref{sec4-eq7cstssos} boils down to removing the terms with exponents $\a$ that do not satisfy $\Rb^\intercal\a\equiv\mathbf{0}$ $(\rm{mod}\,2)$ from the right hand side of \eqref{sec4-eq7cstssos}. Since any $\a\in\supp(f)$ satisfies $\Rb^\intercal\a\equiv\mathbf{0}$ $(\rm{mod}\,2)$, this does not change the match of coefficients on both sides of the equality. Thus we obtain
\begin{equation*}
	f=\sum_{k=1}^p\left(\sigma_{k,0}+\sum_{j\in J_k}\sigma_{k,j}g_j\right)
\end{equation*}
with the desired property.
\qed
\end{proof}

\section{Extracting a solution}

In the case of the dense \ac{moment-SOS} hierarchy, there is a standard procedure described in \citecstssos{Henrion05} to extract globally optimal solutions when the moment matrix satisfies the so-called flatness condition. This procedure was generalized to the correlative sparse setting in \citecstssos[\S~3.3]{Las06} and \citecstssos{nie2009sparse}. In the term sparse setting, however, the corresponding procedure cannot be applied because the information on the moment matrix is incomplete. 
In order to extract a solution in this case, we may add an order-one (dense) moment matrix for each clique in \eqref{eq:cstsmom}:
\begin{equation}\label{sec4-eq8cstssos}
\begin{cases}
\inf\limits_{\y}&L_{\y}(f)\\
\rm{ s.t.}&\M_{r}(\y, I_k)\in\Pi_{G_{r,k,0}^{(s)}}(\Sbb^+_{t_{k,0}}),\quad k\in[p]\\
&\M_{1}(\y, I_k)\succeq0,\quad k\in[p]\\
&\M_{r-d_j}(g_j\y, I_k)\in\Pi_{G_{r,k,j}^{(s)}}(\Sbb^+_{t_{k,j}}),\quad j\in J_k, \quad k\in[p]\\
&L_{\y}(g_j)\ge0,\quad j\in J'\\
&y_{\mathbf{0}}=1
\end{cases}
\end{equation}

Let $\y^{\opt}$ be an optimal solution of \eqref{sec4-eq8cstssos}.
Typically, $\M_{1}(\y^{\opt}, I_k)$ (after identifying sufficiently small entries with zeros) is a block diagonal matrix (up to permutation). If for all $k$ every block of $\M_{1}(\y^{\opt}, I_k)$ is of rank one, then a globally optimal solution $\x^{\opt}$ to $\P$ \eqref{eq:pop} which is unique up to sign symmetries can be extracted (\citecstssos[Theorem~3.7]{Las06}), and the global optimality is certified. Otherwise, the relaxation might be not exact or yield multiple global solutions. 

\begin{remark}
Note that \eqref{sec4-eq8cstssos} is a tighter relaxation of $\P$ \eqref{eq:pop} than $\P_{\csts}^{r,s}$ \eqref{eq:cstsmom}, and so might provide a better lower bound for $\P$. In particular,
if $\P$ is a \ac{QCQP}, then \eqref{sec4-eq8cstssos} is always tighter than Shor's relaxation of $\P$.
\end{remark}

\section{A minimal initial relaxation step}\label{mir}

For \ac{POP} \eqref{eq:pop}, suppose that $f$ is not a homogeneous polynomial or the constraint polynomials $\{g_j\}_{j\in[m]}$ are of different degrees. Then instead of using the uniform minimum relaxation order $r_{\min}$, it might be more beneficial, from the computational point of view, to assign different relaxation orders to different subsystems obtained from the \ac{csp} for the initial relaxation step of the \ac{CS-TSSOS} hierarchy. To this end, we redefine the \ac{csp} graph $G^{\text{icsp}}(V,E)$ as follows: $V=[n]$ and $\{i,j\}\in E$ whenever there exists $\a\in\sA$ such that $\{i,j\}\subseteq\supp(\a)$. This is clearly a subgraph of $G^{\text{csp}}$ defined in Chapter~\ref{chap:cs} and hence typically admits a smaller chordal extension. Let $(G^{\text{icsp}})'$ be a chordal extension of $G^{\text{icsp}}$ and $\{I_k\}_{k\in[p]}$ be the list of maximal cliques of $(G^{\text{icsp}})'$ with
$n_k\coloneqq |I_k|$. Now we partition the constraint polynomials $\{g_j\}_{j\in[m]}$ into groups $\{g_j\mid j\in J_k\}_{k\in[p]}$ and $\{g_j\mid j\in J'\}$ which satisfy
\begin{enumerate}[(1)]
	\item $J_1,\ldots,J_p,J'\subseteq[m]$ are pairwise disjoint and $\bigcup_{k=1}^pJ_k\cup J'=[m]$;
	\item for any $j\in J_k$, $\bigcup_{\a\in\supp(g_j)}\supp(\a)\subseteq I_k$, $k\in[p]$;
	\item for any $j\in J'$, $\bigcup_{\a\in\supp(g_j)}\supp(\a)\nsubseteq I_k$ for all $k\in[p]$.
\end{enumerate}

Suppose $f$ decomposes as $f=\sum_{k\in[p]}f_k$ such that $\bigcup_{\a\in\supp(f_k)}\supp(\a)\subseteq I_k$ for $k\in[p]$. We define the vector of minimum relaxation orders $\o=(o_k)_k\in\N^{p}$ with $o_k\coloneqq \max\{\{d_j:j\in J_k\}\cup\{\lceil\deg(f_k)/2\rceil\}\}$. Then with $s\ge1$, we define the following minimal initial relaxation step of the \ac{CS-TSSOS} hierarchy:
\begin{equation}\label{cts}
	\begin{cases}
		\inf\limits_{\y} &L_{\y}(f)\\
		\rm{ s.t.}&\B_{G_{o_k,k,0}^{(s)}}\circ \M_{o_k}(\y, I_k)\in\Pi_{G_{o_k,k,0}^{(s)}}(\Sbb_+^{t_{k,0}}),\quad k\in[p]\\
		&\B_{G_{o_k,k,j}^{(s)}}\circ \M_{o_k-d_j}(g_j\y, I_k)\in\Pi_{G_{o_k,k,j}^{(s)}}(\Sbb_+^{t_{k,j}}),\quad j\in J_k, k\in[p]\\
		&L_{\y}(g_j)\ge0,\quad j\in J'\\
		&y_{\mathbf{0}}=1
	\end{cases}
\end{equation}
where $G_{o_k,k,j}^{(s)},j\in J_k,k\in[p]$ are defined in the same spirit with Chapter~\ref{chap8:sec1} and $t_{k,j}\coloneqq \binom{n_k+o_k-d_j}{o_k-d_j}$ for all $k,j$.

\section{Numerical experiments}
In this section, we report some numerical results of the proposed \ac{CS-TSSOS} hierarchy. All numerical examples were computed on an Intel Core i5-8265U@1.60GHz CPU with 8GB RAM memory.

\subsection{Benchmarks for constrained POPs}

Consider the following \ac{POP}:
\begin{equation}\label{cons:cstssos}
\begin{cases}
\inf\limits_{\x}&f_{\text{gR}}\quad(\text{resp. } f_{\text{Bt}}\text{ or } f_{\text{cW}})\\
\rm{ s.t.}&1-(\sum_{i=20j-19}^{20j}x_i^2)\ge0,\quad j=1,2,\ldots,n/20
	\end{cases}
\end{equation}
with $20|n$, where the objective function is respectively given by\\
$\bullet$ the generalized Rosenbrock function
\begin{equation*}
	f_{\text{gR}}(\x)=1+\sum_{i=2}^n(100(x_i-x_{i-1}^2)^2+(1-x_i)^2),
\end{equation*}
$\bullet$ the Broyden tridiagonal function
\begin{align*}
	f_{\text{Bt}}(\x)=\,&((3-2x_1)x_1-2x_2+1)^2+\sum_{i=2}^{n-1}((3-2x_i)x_i-x_{i-1}-2x_{i+1}+1)^2\\&+((3-2x_n)x_n-x_{n-1}+1)^2,
\end{align*}
$\bullet$ the chained Wood function
\begin{align*}
	f_{\text{cW}}(\x)=\,&1+\sum_{i\in J}(100(x_{i+1}-x_{i}^2)^2+(1-x_i)^2+90(x_{i+3}-x_{i+2}^2)^2\\
	&+(1-x_{i+2})^2+10(x_{i+1}+x_{i+3}-2)^2+0.1(x_{i+1}-x_{i+3})^2),
\end{align*}
where $J=\{1,3,5,\ldots,n-3\}$ and $4|n$.

We solve the \ac{CSSOS} relaxation with $r=2$, the \ac{TSSOS} relaxation with $r=2,s=1$, and the \ac{CS-TSSOS} relaxation with $r=2,s=1$, where approximately smallest chordal extensions are used for \ac{TS}. The results are presented in Tables \ref{gr}--\ref{cw}, in which ``mb'' denotes the maximal size of PSD blocks involved in the relaxations, ``opt'' denotes the optimum, ``time'' denotes running time in seconds, and ``-'' indicates an out of memory error. We see that \ac{CSSOS}, \ac{TSSOS} and \ac{CS-TSSOS} yield almost the same optimum while \ac{CS-TSSOS} is the most efficient and scalable approach among them.

\begin{table}[htbp]
	\caption{Results for the generalized Rosenbrock function ($r=2$).}\label{gr}
	\renewcommand\arraystretch{1.2}
	\centering
		\begin{tabular}{c|ccc|ccc|ccc}
			\hline
			\multirow{2}*{$n$}&\multicolumn{3}{c|}{CSSOS}&\multicolumn{3}{c|}{TSSOS}&\multicolumn{3}{c}{CS-TSSOS}\\
			\cline{2-10}
			&mb&opt&time&mb&opt&time&mb&opt&time\\
			\hline
			100&$231$&97.445&377&101&97.436&31.3&$21$&$97.436$&$0.54$\\
			\hline
			200&$231$&-&-&201&196.41&1327&$21$&$196.41$&$1.27$\\
			\hline
			300&$231$&-&-&-&-&-&$21$&$295.39$&$2.26$\\
			\hline
			400&$231$&-&-&-&-&-&$21$&$394.37$&$3.36$\\
			\hline
			500&$231$&-&-&-&-&-&$21$&$493.35$&$4.65$\\
			\hline
			1000&$231$&-&-&-&-&-&$21$&$988.24$&$15.8$\\
			\hline
	\end{tabular}
\end{table}

\begin{table}[htbp]
	\caption{Results for the Broyden tridiagonal function ($r=2$).}\label{bt}
	\renewcommand\arraystretch{1.2}
	\centering
		\begin{tabular}{c|ccc|ccc|ccc}
			\hline
			\multirow{2}*{$n$}&\multicolumn{3}{c|}{CSSOS}&\multicolumn{3}{c|}{TSSOS}&\multicolumn{3}{c}{CS-TSSOS}\\
			\cline{2-10}
			&mb&opt&time&mb&opt&time&mb&opt&time\\
			\hline
			100&$231$&79.834&519&103&79.834&104&$23$&$79.834$&$1.96$\\
			\hline
			200&$231$&-&-&-&-&-&$23$&$160.83$&$4.88$\\
			\hline
			300&$231$&-&-&-&-&-&$23$&$241.83$&$8.67$\\
			\hline
			400&$231$&-&-&-&-&-&$23$&$322.83$&$13.3$\\
			\hline
			500&$231$&-&-&-&-&-&$23$&$403.83$&$19.9$\\
			\hline
			1000&$231$&-&-&-&-&-&$23$&$808.83$&$57.5$\\
			\hline
	\end{tabular}
\end{table}

\begin{table}[htbp]
	\caption{Results for the chained Wood function ($r=2$).}\label{cw}
	\renewcommand\arraystretch{1.2}
	\centering
		\begin{tabular}{c|ccc|ccc|ccc}
			\hline
			\multirow{2}*{$n$}&\multicolumn{3}{c|}{CSSOS}&\multicolumn{3}{c|}{TSSOS}&\multicolumn{3}{c}{CS-TSSOS}\\
			\cline{2-10}
			&mb&opt&time&mb&opt&time&mb&opt&time\\
			\hline
			100&$231$&1485.8&505&101&1485.8&43.2&$21$&$1485.8$&$0.73$\\
			\hline
			200&$231$&-&-&201&3004.5&1238&$21$&$3004.5$&$1.91$\\
			\hline
			300&$231$&-&-&-&-&-&$21$&$4523.6$&$3.39$\\
			\hline
			400&$231$&-&-&-&-&-&$21$&$6042.0$&$5.72$\\
			\hline
			500&$231$&-&-&-&-&-&$21$&$7560.7$&$7.88$\\
			\hline
			1000&$231$&-&-&-&-&-&$21$&$15155$&$23.0$\\
			\hline
	\end{tabular}
\end{table}

\subsection{The Max-Cut problem}
The Max-Cut problem is one of the basic combinatorial optimization problems, which is known to be NP-hard. Let $G(V, E)$ be an undirected graph with $V=\{1,\ldots,n\}$ and with edge weights $w_{ij}$ for $\{i,j\}\in E$. Then the Max-Cut problem for $G$ can be formulated as a \ac{QCQP} in binary variables:
\begin{equation}\label{maxcut}
	\begin{cases}
		\max\limits_{\x}&\frac{1}{2}\sum_{\{i,j\}\in E}w_{ij}(1-x_ix_j)\\
		\rm{ s.t.}&1-x_i^2=0,\quad i=1,\ldots,n.
	\end{cases}
\end{equation}

For the numerical experiments, we construct random Max-Cut instances with a block-band sparsity pattern (illustrated in Figure \ref{block-band}) which consists of $l$ blocks of size $25$ and two bands of width $5$. 
There are ten such Max-Cut instances with $l=20,40,60,80,100,120,140,160,180,200$, respectively\footnote{The instances are available at https://wangjie212.github.io/jiewang/code.html.}.

\begin{figure}[htbp]
\centering
\begin{tikzpicture}
\draw (0,0) rectangle (4,4);
\filldraw[fill=blue, fill opacity=0.3] (0,3) rectangle (1,4);
\filldraw[fill=blue, fill opacity=0.3] (1,2) rectangle (2,3);
\fill (2.375,1.625) circle (0.3ex);
\fill (2.75,1.25) circle (0.3ex);
\fill (3.125,0.875) circle (0.3ex);
\draw (0,0.5)--(3.5,0.5);
\draw (3.5,4)--(3.5,0.5);
\fill[fill=blue, fill opacity=0.3] (0,0) rectangle (3.5,0.5);
\fill[fill=blue, fill opacity=0.3] (3.5,0) rectangle (4,4);
\draw[<->] (-0.1,0) -- (-0.1,0.5);
\node[left] at (-0.1,0.25) {$h$};
\draw[<->] (3.5,4.1) -- (4,4.1);
\node[above] at (3.75,4.1) {$h$};
\draw[<->] (-0.1,3) -- (-0.1,4);
\node[left] at (-0.1,3.5) {$b$};
\draw[<->] (0,4.1) -- (1,4.1);
\node[above] at (0.5,4.1) {$b$};
\node[left] at (2.6,1.25) {$l$ blocks};
\end{tikzpicture}\caption{The block-band sparsity pattern. $l$: the number of blocks, $b$: the size of blocks, $h$: the width of bands.}\label{block-band}
\end{figure}
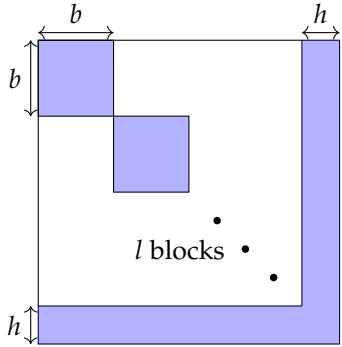

For each instance, we solve Shor's relaxation, the \ac{CSSOS} hierarchy with $r=2$, and the \ac{CS-TSSOS} hierarchy with $r=2,s=1$, where the maximal chordal extension is used for \ac{TS}. The results are reported in Table \ref{max-cut}. We see that for each instance, both \ac{CSSOS} and \ac{CS-TSSOS} significantly improve the bound obtained from Shor's relaxation. Meanwhile, \ac{CS-TSSOS} is several times faster than \ac{CSSOS} at the cost of possibly providing a sightly weaker bound.

\begin{table}[htbp]
\caption{Results for Max-Cut instances. Only integer parts of optima are preserved. ``mc'' denotes the maximal size of variable cliques.}\label{max-cut}
\renewcommand\arraystretch{1.2}
\centering
\resizebox{\linewidth}{!}{
\begin{tabular}{c|c|c|c|c|ccc|ccc}
\hline
\multirow{2}*{name}&\multirow{2}*{node}&\multirow{2}*{edge}&\multirow{2}*{mc}&Shor&\multicolumn{3}{c|}{CSSOS}&\multicolumn{3}{c}{CS-TSSOS}\\
\cline{5-11}
&&&&opt&mb&opt&time&mb&opt&time\\
\hline
g20&$505$&$2045$&$14$&$570$&$120$&$488$&$51.2$&$92$&$488$&$19.6$\\
\hline
g40&$1005$&$3441$&$14$&$1032$&$120$&$885$&$134$&$92$&$893$&$41.1$\\
\hline
g60&$1505$&$4874$&$14$&$1439$&$120$&$1227$&$183$&$92$&$1247$&$71.3$\\
\hline
g80&$2005$&$6035$&$15$&$1899$&$136$&$1638$&$167$&$106$&$1669$&$84.8$\\
\hline
g100&$2505$&$7320$&$14$&$2398$&$120$&$2073$&$262$&$92$&$2128$&$112$\\
\hline
g120&$3005$&$8431$&$14$&$2731$&$120$&$2358$&$221$&$79$&$2443$&$127$\\
\hline
g140&$3505$&$9658$&$13$&$3115$&$105$&$2701$&$250$&$79$&$2812$&$153$\\
\hline
g160&$4005$&$10677$&$14$&$3670$&$120$&$3202$&$294$&$79$&$3404$&$166$\\
\hline
g180&$4505$&$12081$&$13$&$4054$&$105$&$3525$&$354$&$79$&$3666$&$246$\\
\hline
g200&$5005$&$13240$&$13$&$4584$&$105$&$4003$&$374$&$79$&$4218$&$262$\\
\hline
\end{tabular}}
\end{table}

\section{Notes and sources}
The material from this chapter is issued from \citecstssos{cstssos}.
A proof of Theorem \ref{prop-sscstssos} can be found in Section 3.2 of \citecstssos{cstssos}. 

Newton and Papachristodoulou have used the \ac{CS-TSSOS} hierarchy for neural network verification in \citecstssos{newton2022sparse}.


\input{cstssos.bbl}

%% file: opf.tex
\chapter{Application in optimal power flow}\label{chap:opf}
In this chapter, we apply the \ac{CS-TSSOS} hierarchy to the famous alternating current optimal power flow (AC-OPF) problem, which can be formulized as a \ac{POP} either with real variables \citeopf{bienstock2020,ghaddar2015optimal} or with complex variables \citeopf{josz2018lasserre}. To tackle \ac{POP}s in complex variables, we first provide ingredients for extending the \ac{moment-SOS} hierarchy to the complex case.

\section{Extension to complex polynomial optimization}
We start by introducing some notations. Let $\z=(z_1,\ldots,z_n)$ be a tuple of complex variables and $\bar{\z}=(\bar{z}_1,\ldots,\bar{z}_n)$ be its conjugate.
We denote by $\C[\z]\coloneqq \C[z_1,\ldots,z_n]$, $\C[\z,\bar{\z}]\coloneqq \C[z_1,\ldots,z_n,\bar{z}_1,\ldots,\bar{z}_n]$ the complex polynomial ring in $\z$, the complex polynomial ring in $\z,\bar{\z}$, respectively. A polynomial $f\in\C[\z,\bar{\z}]$ can be written as $f=\sum_{(\b,\g)\in\sA}f_{\b,\g}\z^{\b}\bar{\z}^{\g}$ with $\sA\subseteq\N^n\times\N^n$ and $f_{\b,\g}\in\C, \z^{\b}=z_1^{\beta_1}\cdots z_n^{\beta_n},\bar{\z}^{\g}=\bar{z}_1^{\gamma_1}\cdots \bar{z}_n^{\gamma_n}$. The \emph{support} of $f$ is defined by $\supp(f)=\{(\b,\g)\in\sA\mid f_{\b,\g}\ne0\}$. The \emph{conjugate} of $f$ is $\bar{f}=\sum_{(\b,\g)\in\sA}\bar{f}_{\b,\g}\z^{\g}\bar{\z}^{\b}$. A polynomial $\sigma=\sum_{(\b,\g)}\sigma_{\b,\g}\z^{\b}\bar{\z}^{\g}\in\C[\z,\bar{\z}]$ is called an \emph{Hermitian sum of squares (HSOS)} if there exist polynomials $f_i\in\C[\z], i\in[t]$ such that $\sigma=\sum_{i=1}^tf_i\bar{f}_i$. For a positive integer $m$, the set of $m\times m$ Hermitian matrices is denoted by $\Hbb_m$ and the set of $m\times m$ \ac{PSD} Hermitian matrices is denoted by $\Hbb_m^+$.

A \ac{CPOP} is given by
\begin{equation}\label{cpop}
\begin{cases}
\inf\limits_{\z\in\C^n} &f(\z,\bar{\z})\coloneqq \sum_{\a,\b}f_{\a,\b}\z^{\a}\bar{\z}^{\b}\\
\,\,\rm{ s.t.}&g_j(\z,\bar{\z})\coloneqq \sum_{\a,\b}g_{j,\a,\b}\z^{\a}\bar{\z}^{\b}\ge0,\quad j\in[m]
\end{cases}
\end{equation}
where the functions $f,g_1,\ldots,g_m$ are real-valued polynomials and their coefficients satisfy $f_{\a,\b}=\bar{f}_{\b,\a}$, and $g_{j,\a,\b}= \bar{g}_{j,\b,\a}$. There are two ways to construct a ``\ac{moment-SOS}'' hierarchy for \ac{CPOP} \eqref{cpop}. The first way is introducing real variables for both real and imaginary parts of each complex variable in \eqref{cpop}, i.e., letting $z_i=x_i+x_{i+n}\i$ for $i\in[n]$. Then one can convert \ac{CPOP} \eqref{cpop} to a \ac{POP} involving only real variables at the price of doubling the number of variables. Hence the usual real \ac{moment-SOS} hierarchy applies to the resulting real \ac{POP}. On the other hand, as the second way, it might be advantageous to handle \ac{CPOP} \eqref{cpop} directly with the complex moment-HSOS hierarchy introduced in \citeopf{josz2018lasserre}. To this end, we define the \emph{complex} moment matrix $\M^{\text{c}}_{r}(\y)$ ($r\in\N$) by
\begin{equation*}
	[\M^{\text{c}}_r(\y)]_{\b,\g}\coloneqq L^{\text{c}}_{\y}(\z^{\b}\bar{\z}^{\g})=y_{\b,\g}, \quad\forall\b,\g\in\N^n_{r},
\end{equation*}
where $\y=(y_{\b,\g})_{(\b,\g)\in\N^n\times\N^n}\subseteq\C$ is a sequence indexed by $(\b,\g)\in\N^n\times\N^n$ satisfying $y_{\b,\g}=\bar{y}_{\g,\b}$, and $L^{\text{c}}_{\y}:\C[\z,\bar{\z}]\rightarrow\R$ is the linear functional such that
\begin{equation*}
	f=\sum_{(\b,\g)}f_{\b,\g}\z^{\b}\bar{\z}^{\g}\mapsto L^{\text{c}}_{\y}(f)=\sum_{(\b,\g)}f_{\b,\g}y_{\b,\g}.
\end{equation*}
Suppose that $g=\sum_{(\b',\g')}g_{\b',\g'}\z^{\b'}\bar{\z}^{\g'}\in\C[\z,\bar{\z}]$ is an Hermitian polynomial, i.e., $\bar{g}=g$. The \emph{complex} localizing matrix $\M^{\text{c}}_{r}(g\y)$ associated with $g$ and $\y$ is defined by
\begin{equation*}
	[\M^{\text{c}}_{r}(g\,\y)]_{\b,\g}\coloneqq L^{\text{c}}_{\y}(g\,\z^{\b}\bar{\z}^{\g})=\sum_{(\b',\g')}g_{\b',\g'}y_{\b+\b',\g+\g'}, \quad\forall\b,\g\in\N^n_{r}.
\end{equation*}
Both the complex moment matrix and the complex localizing matrix are Hermitian matrices.
Note that a distinguished difference between the usual (real) moment matrix and the complex moment matrix is that the former has the Hankel property (i.e., $[\M_{r}(\y)]_{\b,\g}$ is a function of $\b+\g$), whereas the latter does not have.

Let $d_j=\lceil\deg(g_j)/2\rceil,j\in[m]$ and $r_{\min}=\max\{\lceil\deg(f)/2\rceil,d_1,\ldots,d_m\}$ as before. Then the complex moment hierarchy indexed by $r\ge r_{\min}$ for \ac{CPOP} \eqref{cpop} is given by
\begin{equation}\label{sec2-eq1opf}
\begin{cases}
\inf\limits_{\y}& L^{\text{c}}_{\y}(f)\\
\rm{ s.t.}&\M^{\text{c}}_{r}(\y)\succeq0\\
&\M^{\text{c}}_{r-d_j}(g_j\y)\succeq0,\quad j\in[m]\\
&y_{\mathbf{0},\mathbf{0}}=1
\end{cases}
\end{equation}
The dual of \eqref{sec2-eq1opf} can be formulized as the following HSOS relaxation:
\begin{equation}\label{sec2-eq2opf}
\begin{cases}
\sup\limits_{\sigma_j,b}&b\\
\rm{ s.t.}&f-b=\sigma_0+\sigma_1g_1+\ldots+\sigma_mg_m\\
&\sigma_j\text{ is an HSOS},\quad j\in\{0\}\cup[m]\\
&\deg(\sigma_0)\le2r,\deg(\sigma_jg_j)\le2r,\quad j\in[m]
\end{cases}
\end{equation}

\begin{remark}
In \eqref{sec2-eq1opf}, the expression ``$X\succeq0$" means an Hermitian matrix $X$ to be \ac{PSD}. Since popular \ac{SDP} solvers deal with only real \ac{SDP}s, it is then necessary to convert this constraint to a constraint involving only real matrices. This can be done by introducing the real part $A$ and the imaginary part $B$ of $X$ respectively such that $X=A+B\i$. Then,
\begin{equation*}
X\succeq0\quad \iff \quad\begin{bmatrix}A&-B\\B&A
\end{bmatrix}\succeq0.
\end{equation*}
\end{remark}

\begin{remark}
The first-order moment-(H)SOS relaxation for \ac{QCQP}s is also known as Shor's relaxation. It was proved in \citeopf{josz2015moment} that the real Shor's relaxation and the complex Shor's relaxation for homogeneous \ac{QCQP}s yield the same bound. However, in general the complex hierarchy is weaker (i.e., producing looser bounds) than the real hierarchy at the same relaxation order $r>1$ as Hermitian sums of squares are a special case of real sums of squares; see \citeopf{josz2018lasserre}.
\end{remark}

\begin{remark}
By the complex Positivstellensatz theorem due to D'Angelo and Putinar \citeopf{d2009polynomial}, global convergence of the complex hierarchy is guaranteed when a sphere constraint is present.
\end{remark}

As for the usual \ac{moment-SOS} hierarchy, we can reduce the size of \ac{SDP} relaxations arising from the complex moment-HSOS hierarchy by exploiting \ac{CS} and/or \ac{TS}. The procedures are quite similar. The only significant difference is on the definitions of \ac{tsp} graphs: in the real case, we use $\sA\cup2\N^n_{r}$ while in the complex case we use $\sA$ instead due to the absence of the Hankel structure of complex moment matrices; see \citeopf{cpop}.

\section{Applications to optimal power flow}
The AC-OPF problem aims to minimize the generation cost of an alternating current transmission network under the physical constraints (Kirchhoff’s laws, Ohm’s law) as well as operational constraints, which can be formulated as the following \ac{POP} in complex variables:
\begin{equation}\label{opf}
	\begin{cases}
		\inf\limits_{V_i,S_k^g\in\C}&\sum_{k\in G}(\mathbf{c}_{2k}(\Re(S_{k}^g))^2+\mathbf{c}_{1k}\Re(S_{k}^g)+\mathbf{c}_{0k})\\
		\,\,\,\,\,\rm{s.t.}&\angle V_{\text{ref}}=0\\
		&\mathbf{S}_{k}^{gl}\le S_{k}^{g}\le \mathbf{S}_{k}^{gu},\quad\forall k\in G\\
		&\boldsymbol{\upsilon}_{i}^l\le|V_i|\le \boldsymbol{\upsilon}_{i}^u,\quad\forall i\in N\\
		&\sum_{k\in G_i}S_k^g-\mathbf{S}_i^d-\mathbf{Y}_i^s|V_{i}|^2=\sum_{(i,j)\in E_i\cup E_i^R}S_{ij},\quad\forall i\in N\\
		&S_{ij}=(\bar{\mathbf{Y}}_{ij}-\mathbf{i}\frac{\mathbf{b}_{ij}^c}{2})\frac{|V_i|^2}{|\mathbf{T}_{ij}|^2}-\bar{\mathbf{Y}}_{ij}\frac{V_i\bar{V}_j}{\mathbf{T}_{ij}},\quad\forall (i,j)\in E\\
		&S_{ji}=(\bar{\mathbf{Y}}_{ij}-\mathbf{i}\frac{\mathbf{b}_{ij}^c}{2})|V_j|^2-\bar{\mathbf{Y}}_{ij}\frac{\bar{V_i}V_j}{\bar{\mathbf{T}}_{ij}},\quad\forall (i,j)\in E\\
		&|S_{ij}|\le \mathbf{s}_{ij}^u,\quad\forall (i,j)\in E\cup E^R\\
		&\boldsymbol{\theta}_{ij}^{\Delta l}\le \angle (V_i\bar{V_j})\le \boldsymbol{\theta}_{ij}^{\Delta u},\quad\forall (i,j)\in E\\
	\end{cases}
\end{equation}
In \eqref{opf}, $V_i$ denotes the voltage, $S_k^g$ denotes the power generation, $N$ denotes the set of buses, and $G$ denotes the set of generators. Besides, $\Re(\cdot)$, $\angle(\cdot)$, $|\cdot|$ stand for the real part, the angle, the magnitude of a complex number, respectively. All symbols in boldface are constants. For a full description on the AC-OPF problem, the reader is referred to \citeopf{baba2019}. Note that by introducing real variables for both real and imaginary parts of each complex variable, we can convert the AC-OPF problem to a \ac{POP} involving only real variables\footnote{The expressions involving angles of complex variables can be converted to polynomials by using $\tan(\angle z)=y/x$ for $z=x+\mathbf{i}y\in\C$.}.

To tackle an AC-OPF instance, we first compute a locally optimal solution with nonlinear programming tools whose global optimality is however unknown. And we then rely on certain convex relaxation of \eqref{opf} to certify global optimality of the local solution. Suppose that the optimum reported by the local solver is AC and the optimum of the convex relaxation is opt. The \emph{optimality gap} between the locally optimal solution and the convex relaxation is defined by
\begin{equation*}
	\text{gap}\coloneqq \frac{\text{AC}-\text{opt}}{\text{AC}}\times100\%.
\end{equation*}
If the optimality gap is less than $1.00\%$, then we accept the locally optimal solution as globally optimal.

We perform two classes of numerical experiments on AC-OPF instances issued from \href{https://github.com/power-grid-lib/pglib-opf}{PGLiB}. For the first class, we compare the complex hierarchy with the real hierarchy in terms of strength and efficiency. The results are reported in Table \ref{c-opf} where ``mb'' denotes the maximal size of PSD blocks involved in the relaxations, ``opt'' denotes the optimum, ``time'' denotes running time in seconds, and ``-'' indicates an out of memory error. We refer to Shor's relaxation (which applies when we convert \eqref{opf} to a \ac{QCQP}) as the 1st order relaxation and refer to the minimal initial relaxation defined in Chapter~\ref{mir} as the 1.5th order relaxation.

As one can see from Table \ref{c-opf}, the complex 1st order relaxation and the real 1st order relaxation give the same lower bound (up to a given precision) while the former runs slightly faster. The complex 1.5th order relaxation typically gives a looser bound than the real 1.5th order relaxation whereas it runs faster by a factor of $1\sim8$. In addition, the 1st order relaxation is able to certify global optimality for 4 out of all 9 instances. For the remaining 5 instances, the complex 1.5th order relaxation is able to certify global optimality for 3 instances and the real 1.5th order relaxation is able to certify global optimality for 4 instances.

\begin{table}[htbp]
\caption{The complex versus real hierarchy on AC-OPF instances under typical operating conditions.}\label{c-opf}
\renewcommand\arraystretch{1.2}
\centering
\resizebox{\linewidth}{!}{
\begin{tabular}{c|c|cccc|cccc}
		\hline
		\multirow{2}*{case name}&\multirow{2}*{order}&\multicolumn{4}{c|}{{\bf Complex}}&\multicolumn{4}{c}{{\bf Real}}\\
		\cline{3-10}
		&&mb&opt&time&gap&mb&opt&time&gap\\
		\hline
		\multirow{2}*{30\_ieee}&1st&8&$7.5472\text{e}3$&0.12&$8.06\%$&8&$7.5472\text{e}3$&0.15&$8.06\%$\\
		&1.5th&12&$8.2073\text{e}3$&0.66&$0.02\%$&22&$8.2085\text{e}3$&0.97&$0.00\%$\\
		\hline
		\multirow{2}*{39\_epri}&1st&8&$1.3565\text{e}4$&0.17&$2.00\%$&8&$1.3565\text{e}4$&0.22&$2.00\%$\\
		\cline{2-10}
		&1.5th&14&$1.3765\text{e}4$&1.08&$0.55\%$&25&$1.3842\text{e}4$&1.12&$0.00\%$\\
		\hline
		\multirow{2}*{89\_pegase}&1st&24&1.0670e5&0.72&0.55\%&24&1.0670e5&0.74&0.55\%\\
		&1.5th&96&$1.0709\text{e}5$&263&0.19\%&184&$1.0715\text{e}5$&1232&0.13\%\\
		\hline
		\multirow{2}*{118\_ieee}&1st&10&9.6900e4&0.49&0.32\%&10&9.6901e4&0.57&0.32\%\\
		&1.5th&20&$9.7199\textrm{e}4$&5.22&0.02\%&37&9.7214e4&8.78&0.00\%\\
		\hline
		\multirow{2}*{162\_ieee\_dtc}&1st&28&$1.0164\text{e}5$&1.49&$5.96\%$&28&$1.0164\text{e}5$&1.51&$5.96\%$\\
		&1.5th&40&$1.0249\text{e}5$&17.1&$5.17\%$&74&$1.0645\text{e}5$&87.5&$1.51\%$\\
		\hline
		\multirow{2}*{179\_goc}&1st&10&7.5016e5&0.72&0.55\%&10&7.5016e5&0.77&0.55\%\\
		&1.5th&20&$7.5078\text{e}5$&6.77&0.46\%&37&$7.5382\text{e}5$&10.6&0.06\%\\
		\hline
		\multirow{2}*{300\_ieee}&1st&14&5.5424e5&1.41&1.94\%&16&5.5424e5&1.49&1.94\%	\\
		&1.5th&22&$5.6455\text{e}5$&19.1&0.12\%&40&$5.6522\text{e}5$&27.3&0.00\%\\
		\hline
		\multirow{2}*{1354\_pegase}&1st&26&1.2172e6&10.9&3.30\%&26&1.2172e6&13.1&3.30\%\\
		&1.5th&26&$1.2304\text{e}6$&255&2.29\%&49&1.2514e6&392&0.59\%\\
		\hline
		\multirow{2}*{2869\_pegase}&1st&26&2.4387e6&47.2&0.98\%&26&2.4388e6&67.3&0.97\%\\
		&1.5th&98&$2.4586\text{e}6$&1666&0.17\%&191&-&-&-\\
		\hline
\end{tabular}}
\end{table}

As the 1st order relaxation is already able to certify global optimality for a large number of AC-OPF instances, we now focus on more challenging AC-OPF instances for which the 1st order relaxation yields an optimality gap greater than $1.00\%$.
The related data of these selected AC-OPF instances are displayed in Table \ref{ac-opf1}, in which ``var'' denotes the number of variables, ``cons'' denotes the number of constraints, and ``mc'' denotes the maximal size of variable cliques.

Since the real relaxations typically yield tighter lower bounds when the relaxation order is greater than one, we only investigate the real relaxations here.
Particularly, we solve the \ac{CSSOS} hierarchy with $r=2$ and the \ac{CS-TSSOS} hierarchy with $r=2,s=1$ for these AC-OPF instances, and report the results in Table \ref{ac-opf2}. As the table shows,
\ac{CS-TSSOS} is more efficient and scales much better with the problem size than \ac{CSSOS}.
In particular, \ac{CS-TSSOS} succeeds in reducing the optimality gap to less than $1.00\%$ for all instances.

\begin{table}[htbp]
\caption{The data of selected AC-OPF instances.}\label{ac-opf1}
\renewcommand\arraystretch{1.2}
\centering
\resizebox{\linewidth}{!}{
\begin{tabular}{c|c|c|c|c|cc}
\hline
\multirow{2}*{case name}&\multirow{2}*{var}&\multirow{2}*{cons}&\multirow{2}*{mc}&\multirow{2}*{AC}&\multicolumn{2}{c}{Shor}\\
\cline{6-7}
&&&&&opt&gap\\
\hline
3\_lmbd\_api&$12$&$28$&$6$&$1.1242\text{e}4$&$1.0417\text{e}4$&$7.34\%$\\
\hline
5\_pjm&$20$&$55$&$6$&$1.7552\text{e}4$&$1.6634\text{e}4$&$5.22\%$\\
\hline
24\_ieee\_rts\_sad&$114$&$315$&$14$&$7.6943\text{e}4$&$7.3592\text{e}4$&$4.36\%$\\
\hline
30\_as\_api&$72$&$297$&$8$&$4.9962\text{e}3$&$4.9256\text{e}3$&$1.41\%$\\
\hline
73\_ieee\_rts\_sad&$344$&$971$&$16$&$2.2775\text{e}5$&$2.2148\text{e}5$&$2.75\%$\\
\hline
162\_ieee\_dtc\_api&$348$&$1809$&$21$&$1.2100\text{e}5$&$1.1928\text{e}5$&$1.42\%$\\
\hline
240\_pserc&$766$&$3322$&$16$&$3.3297\text{e}6$&$3.2818\text{e}6$&$1.44\%$\\
\hline
500\_tamu\_api&$1112$&$4613$&$20$&$4.2776\text{e}4$&$4.2286\text{e}4$&$1.14\%$\\
\hline
793\_goc&$1780$&$7019$&$18$&$2.6020\text{e}5$&$2.5636\text{e}5$&$1.47\%$\\
\hline
1888\_rte&$4356$&$18257$&$26$&$1.4025\text{e}6$&$1.3748\text{e}6$&$1.97\%$\\
\hline
3022\_goc&$6698$&$29283$&$50$&$6.0143\text{e}5$&$5.9278\text{e}5$&$1.44\%$\\
\hline
\end{tabular}}
\end{table}

\begin{table}[htbp]
\caption{The CSSOS versus CS-TSSOS hierarchy on AC-OPF instances.}\label{ac-opf2}
\renewcommand\arraystretch{1.2}
\centering
\resizebox{\linewidth}{!}{
\begin{tabular}{c|cccc|cccc}
\hline
\multirow{2}*{case name}&\multicolumn{4}{c|}{CSSOS}&\multicolumn{4}{c}{CS-TSSOS}\\
\cline{2-9}
&mb&opt&time&gap&mb&opt&time&gap\\
\hline
3\_lmbd\_api&$28$&$1.1242\text{e}4$&$0.21$&$0.00\%$&$22$&$1.1242\text{e}4$&$0.09$&$0.00\%$\\
\hline
5\_pjm&$28$&$1.7543\text{e}4$&$0.56$&$0.05\%$&$22$&$1.7543\text{e}4$&$0.30$&$0.05\%$\\
\hline
24\_ieee\_rts\_sad&$120$&$7.6943\text{e}4$&$94.9$&$0.00\%$&$39$&$7.6942\text{e}4$&$14.8$&$0.00\%$\\
\hline
30\_as\_api&$45$&$4.9927\text{e}3$&$4.43$&$0.07\%$&$22$&$4.9920\text{e}3$&$2.69$&$0.08\%$\\
\hline
73\_ieee\_rts\_sad&$153$&$2.2775\text{e}5$&$504$&$0.00\%$&$44$&$2.2766\text{e}5$&$71.5$&$0.04\%$\\
\hline
162\_ieee\_dtc\_api&$253$&-&-&-&$34$&$1.2096\text{e}5$&$201$&$0.03\%$\\
\hline
240\_pserc&$153$&$3.3072\text{e}6$&$585$&$0.68\%$&$44$&$3.3042\text{e}6$&$33.9$&$0.77\%$\\
\hline
500\_tamu\_api&$231$&$4.2413\text{e}4$&$3114$&$0.85\%$&$39$&$4.2408\text{e}4$&$46.6$&$0.86\%$\\
\hline
793\_goc&$190$&$2.5938\text{e}5$&$563$&$0.31\%$&$33$&$2.5932\text{e}5$&$66.1$&$0.34\%$\\
\hline
1888\_rte&$378$&-&-&-&$27$&$1.3953\text{e}6$&$934$&$0.51\%$\\
\hline
3022\_goc&$1326$&-&-&-&$76$&$5.9858\text{e}5$&$1886$&$0.47\%$\\
\hline
\end{tabular}}
\end{table}

\section{Notes and sources}
The complex moment-HSOS hierarchy was initially introduced and studied
in \citeopf{josz2015moment,josz2018lasserre}, which has been shown to have advantages over the real hierarchy for certain \ac{CPOP}s (e.g., a simplified version of the AC-OPF problem).

The AC-OPF is a fundamental problem in power systems, which has been extensively studied in recent years; for a detailed introduction and recent developments, the reader is referred to the survey \citeopf{bienstock2020} and references therein. Since 2006, several convex relaxation schemes (e.g., second-order cone relaxations (SOCR) \citeopf{jabr2006radial}, quadratic convex relaxations (QCR) \citeopf{cof}, tight-and-cheap conic
relaxations (TCR) \citeopf{bingane2018tight} and semidefinite relaxations (SDR) \citeopf{bai2008semidefinite}) have been proposed to provide lower bounds for the AC-OPF which can be then used to certify global optimality of local optimal solutions. While these relaxations (SOCR, QCR, TCR, SDR) could be scalable to problems of large size and prove to be tight for quite a few cases \citeopf{baba2019,cof,eltved2019robustness}, they yield significant optimality gaps on a large number of cases\footnote{The reader may find related results on benchmarking SOCR and QR at \href{https://github.com/power-grid-lib/pglib-opf/blob/master/BASELINE.md}{https://github.com/power-grid-lib/pglib-opf/blob/master/BASELINE.md}.}. To tackle these more challenging cases, it is then mandatory to go to higher steps of the \ac{moment-SOS} hierarchy which can provide tighter lower bounds.
Along with this line recently in \citeopf{gopinath2020proving}, Gopinath et al. certified $1$\% global optimality for all AC-OPF instances with up to $300$ buses from the AC-OPF library \href{https://github.com/power-grid-lib/pglib-opf}{PGLiB} using an \ac{SDP}-based bound tightening approach.
Relying on the complex \ac{moment-SOS} hierarchy combined with a multi-order technique, Josz and Molzahn certified $0.05$\% global optimality for certain $2000$-bus cases on a simplified AC-OPF model \citeopf{josz2018lasserre}.
A comprehensive numerical study on AC-OPF instances from \href{https://github.com/power-grid-lib/pglib-opf}{PGLiB} with up to tens of thousands of variables and constraints via the \ac{CS-TSSOS} hierarchy could be found in \citeopf{WANG2022108683}.

\input{opf.bbl}

%% file: opf.bbl
\providecommand{\etalchar}[1]{$^{#1}$}

%% file: nctssos.tex
\chapter[Exploiting term sparsity in noncommutative optimization]{Exploiting term sparsity in noncommutative polynomial optimization}\label{chap:ncts}
In this chapter, we generalize the methodology of exploiting \ac{TS} to noncommutative polynomial optimization. For the sake of conciseness, we consider the problem of eigenvalue optimization over noncommutative polynomials and omit the proofs.

\section{Eigenvalue optimization with term sparsity}\label{sec3:ncts}
Recall that the eigenvalue optimization problem is defined by
\begin{align}\label{constr_eigmin:ncts}
\lambda_{\min}(f, \frakg)\coloneqq \inf\{\langle f(\underline{A}) \vb,\vb \rangle:\underline{A}\in{\cD_{\frakg}^\infty},\|\vb\|=1\},
\end{align}
for $f \in \SymRX$ and $\frakg = \{g_1,\dots,g_{m} \} \subseteq \SymRX$. Let
\begin{equation}\label{csupp:ncts}
	\sA = \supp(f)\cup\bigcup_{j=1}^m\supp(g_j).
\end{equation}
As before, we set $g_0\coloneqq 1$, and let $d_j=\lceil\deg(g_j)/2\rceil$ for $j\in\{0\}\cup[m]$ and $r_{\min}=\max\{\lceil\deg(f)/2\rceil,d_1,\ldots,d_m\}$. Fixing a relaxation order $r\ge r_{\min}$, we define a graph $G_{r}^{\text{tsp}}$ with nodes $\W_{r}$\footnote{If $\frakg=\emptyset$, then we may replace the monomial basis $\W_{r}$ with the one returned by the Newton chip method; see \citencts[\textsection 2.3]{burgdorf16}.} and edges
\begin{equation}\label{sec3-eq1:ncts}
	E(G_{r}^{\text{tsp}})=\{\{u,v\}\mid(u,v)\in \W_{r}\times \W_{r},\,u\ne v,\,u^{\star}v\in\sA\cup\W_{r}^2\},
\end{equation}
where $\W_{r}^2\coloneqq \{u^{\star}u\mid u\in\W_{r}\}$. We call $G_{r}^{\text{tsp}}$ the \ac{tsp} graph associated with the support $\sA$.

For a graph $G(V,E)$ with $V\subseteq\langle\underline{x}\rangle$ and $g\in\R\langle\underline{x}\rangle$, let us define
\begin{equation}
	\supp_{g}(G)\coloneqq \{u^{\star}wv\mid u=v\in V\text{ or }\{u,v\}\in E, w\in\supp(g)\}.
\end{equation}
Let $G_{r,0}^{(0)}=G_{r}^{\text{tsp}}$ and $G_{r,j}^{(0)}$ be the empty graph with nodes $V_{r,j}\coloneqq \W_{r-d_j}$ for $j\in[m]$. Then for each $j\in\{0\}\cup[m]$, we iteratively define a sequence of graphs $(G_{r,j}^{(s)}(V_{r,j},E_{r,j}^{(s)}))_{s\ge1}$ via two successive operations:\\
(1) {\bf support extension}. Let $F_{r,j}^{(s)}$ be the graph with nodes $V_{r,j}$ and
\begin{equation}\label{sec3-eq2:ncts}
	\begin{split}
	E(F_{r,j}^{(s)})=&\{\{u,v\}\mid(u,v)\in V_{r,j}\times V_{r,j},\,u\ne v,\\
		&u^{\star}\supp(g_j)v\cap\bigcup_{j=0}^m\supp_{g_j}(G_{r,j}^{(s-1)})\ne\emptyset\},
	\end{split}
\end{equation}
where $u^{\star}\supp(g_j)v\coloneqq \{u^{\star}wv\mid w\in\supp(g_j)\}$.\\
(2) {\bf chordal extension}. Let
\begin{equation}\label{sec3-graph:ncts}
	G_{r,j}^{(s)}\coloneqq (F_{r,j}^{(s)})'.
\end{equation}
By construction, one has $G_{r,j}^{(s)}\subseteq G_{r,j}^{(s+1)}$ for all $j,s$. Therefore, for every $j$, the sequence of graphs $(G_{r,j}^{(s)})_{s\ge1}$ stabilizes after a finite number of steps.

Let $t_j=|\W_{r-d_j}|$ for $j\in\{0\}\cup[m]$.
Then by replacing the \ac{PSD} constraint $\M_{r-d_j}(g_j\y)\succeq0$ with the weaker constraint $\B_{G_{r,j}^{(s)}}\circ \M_{r-d_j}(g_j\y)\in\Pi_{G_{r,j}^{(s)}}(\Sbb^+_{t_j})$ for $j\in\{0\}\cup[m]$ in \eqref{eq:constr_eigmin_primal}, we obtain the following series of sparse moment relaxations for \eqref{constr_eigmin:ncts} indexed by $s\ge1$:
\begin{equation}\label{cpop-seigen1:ncts}
\begin{array}{rll}
\lambda_{\text{ts}}^{r,s}(f, \frakg)\coloneqq &\inf\limits_{\y} &L_{\y}(f)\\
&\rm{ s.t.}&\B_{G_{r,0}^{(s)}}\circ \M_{r}(\y)\in\Pi_{G_{r,0}^{(s)}}(\Sbb^+_{t_0})\\
&&\B_{G_{r,j}^{(s)}}\circ \M_{r-d_j}(g_j\y)\in\Pi_{G_{r,j}^{(s)}}(\Sbb^+_{t_j}),\quad j\in[m]\\
&&y_{1}=1
\end{array}
\end{equation}
We call $s$ the {\em sparse order}. For each $s\ge1$, the dual of \eqref{cpop-seigen1:ncts} reads as
\begin{equation}\label{cpop-seigen2:ncts}
\begin{cases}
\sup\limits_{\G_j,b}&b\\
\rm{ s.t.}&\sum_{j=0}^m\langle \G_j,\D_{w}^j\rangle+b\delta_{1w}=f_{w},\quad\forall w\in\bigcup_{j=0}^m\supp_{g_j}(G_{r,j}^{(s)})\\
&\G_j\in\Sbb^+_{t_j}\cap\Sbb_{G_{r,j}^{(s)}},\quad j\in\{0\}\cup[m]
\end{cases}
\end{equation}
where $\{\D_{w}^j\}_{j,w}$ are appropriate matrices satisfying $\M_{r-d_j}(g_j\y)=\sum_{w}\D_{w}^jy_{w}$.
We call the \ac{TS}-adapted moment-\ac{SOHS} relaxations \eqref{cpop-seigen1:ncts}--\eqref{cpop-seigen2:ncts} the \emph{NCTSSOS hierarchy} associated with \eqref{constr_eigmin:ncts}.


\begin{theoremf}\label{cts-thm1:ncts}
Let $\{f\}\cup\frakg\subseteq\Sym\,\R\langle\underline{x}\rangle$. Then the following hold:
\begin{enumerate}[(1)]
\item Suppose that $\cD_{\frakg}$ contains an nc $\varepsilon$-neighborhood of $0$. Then for all $r,s$, there is no duality gap between \eqref{cpop-seigen1:ncts} and its dual \eqref{cpop-seigen2:ncts}.
\item Fixing a relaxation order $r\ge r_{\min}$, the sequence $(\lambda_{\text{ts}}^{r,s}(f, \frakg))_{s\ge1}$ is monotonically non-decreasing and $\lambda_{\text{ts}}^{r,s}(f, \frakg)\le\lambda^{r}(f, \frakg)$ for all $s$ (with $\lambda^{r}(f, \frakg)$ being defined in \eqref{eq:constr_eigmin_primal}).
\item Fixing a sparse order $s\ge1$, the sequence $(\lambda_{\text{ts}}^{r,s}(f, \frakg))_{r\ge r_{\min}}$ is monotonically non-decreasing.
\item If the maximal chordal extension is chosen in \eqref{sec3-graph:ncts}, then $(\lambda_{\text{ts}}^{r,s}(f, \frakg))_{s\ge 1}$ converges to $\lambda^{r}(f, \frakg)$ in finitely many steps.
\end{enumerate}
\end{theoremf}

Following from Theorem \ref{cts-thm1:ncts}, we have the following two-level hierarchy of lower bounds for the optimum $\lambda_{\min}(f, \frakg)$ of \eqref{constr_eigmin:ncts}:
\begin{equation}\label{cliquehierc:ncts}
\begin{matrix}
\lambda_{\text{ts}}^{r_{\min},1}(f, \frakg)&\le&\lambda_{\text{ts}}^{r_{\min},2}(f, \frakg)&\le&\cdots&\le&\lambda^{r_{\min}}(f, \frakg)\\
\vge&&\vge&&&&\vge\\
\lambda_{\text{ts}}^{r_{\min}+1,1}(f, \frakg)&\le&\lambda_{\text{ts}}^{r_{\min}+1,2}(f, \frakg)&\le&\cdots&\le&\lambda^{r_{\min}+1}(f, \frakg)\\
\vge&&\vge&&&&\vge\\
\vdots&&\vdots&&\vdots&&\vdots\\
\vge&&\vge&&&&\vge\\
\lambda_{\text{ts}}^{r,1}(f, \frakg)&\le&\lambda_{\text{ts}}^{r,2}(f, \frakg)&\le&\cdots&\le&\lambda^{r}(f, \frakg)\\
\vge&&\vge&&&&\vge\\
\vdots&&\vdots&&\vdots&&\vdots\\
\end{matrix}
\end{equation}

\begin{example}\label{sec3-ex4:ncts}
Consider $f=2-x^2+xy^2x-y^2+xyxy+yxyx+x^3y+yx^3+xy^3+y^3x$ and $\frakg=\{1-x^2,1-y^2\}$. The graph sequence $(G_{2,0}^{(s)})_{s\ge1}$ for $f$ and $\frakg$ is given in Figure~\ref{ex4-1:ncts}. In fact the graph sequence $(G_{2,j}^{(s)})_{s\ge1}$ stabilizes at $s=2$ for $j=0,1,2$ (with approximately smallest chordal extensions). Using {\tt TSSOS}, we obtain that $\lambda_{\text{ts}}^{2,1}(f, \frakg)\approx-2.55482$, $\lambda_{\text{ts}}^{2,2}(f, \frakg)=\lambda^{2}(f, \frakg)\approx-2.05111$.

\begin{figure}[htbp]
\centering
\begin{minipage}{0.45\linewidth}
{\tiny
\begin{tikzpicture}[every node/.style={circle, draw=blue!50, thick, minimum size=6mm}]
\node (n1) at (90:2) {$1$};
\node (n2) at (162:2) {$x^2$};
\node (n3) at (234:2) {$xy$};
\node (n4) at (306:2) {$yx$};
\node (n5) at (18:2) {$y^2$};
\draw (n1)--(n2);
\draw (n2)--(n3);
\draw (n3)--(n4);
\draw (n4)--(n5);
\draw (n5)--(n1);
\draw[dashed] (n3)--(n5);
\draw[dashed] (n2)--(n5);
\node[yshift=90] (n6) at (180:1) {$x$};
\node[yshift=90] (n7) at (0:1) {$y$};
\end{tikzpicture}}
\end{minipage}
\begin{minipage}{0.45\linewidth}
{\tiny
\begin{tikzpicture}[every node/.style={circle, draw=blue!50, thick, minimum size=6mm}]
\node (n1) at (90:2) {$1$};
\node (n2) at (162:2) {$x^2$};
\node (n3) at (234:2) {$xy$};
\node (n4) at (306:2) {$yx$};
\node (n5) at (18:2) {$y^2$};
\draw (n1)--(n2);
\draw (n2)--(n3);
\draw (n3)--(n4);
\draw (n4)--(n5);
\draw (n5)--(n1);
\draw (n3)--(n5);
\draw (n2)--(n5);
\draw (n1)--(n3);
\draw (n1)--(n4);
\node[yshift=90] (n6) at (180:1) {$x$};
\node[yshift=90] (n7) at (0:1) {$y$};
\draw (n6)--(n7);
\end{tikzpicture}}
\end{minipage}
\caption{The graph sequence $(G_{2,0}^{(s)})_{s\ge1}$ in Example \ref{sec3-ex4:ncts}: left for $s=1$; right for $s=2$. The dashed edges are added after a chordal extension.}\label{ex4-1:ncts}
\end{figure}
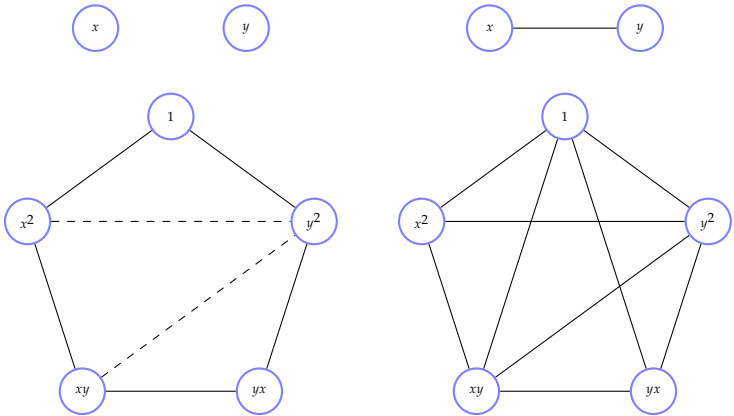
\end{example}

\section{Combining correlative and term sparsity}
\label{cts:ncts}

Combining \ac{CS} with \ac{TS} for eigenvalue optimization proceeds in a similar manner as for the commutative case in Chapter~\ref{chap8:sec1}.

Let $f=\sum_{w}f_{w}w\in\Sym\,\R\langle\underline{x}\rangle$ and $\frakg=\{g_1,\ldots,g_m\}\subseteq\Sym\,\R\langle\underline{x}\rangle$. Suppose that $G^{\text{csp}}$ is the \ac{csp} graph associated with $f$ and $\frakg$, and $(G^{\text{csp}})'$ is a chordal extension of $G^{\text{csp}}$. Let $\{I_k\}_{k\in[p]}$ be the maximal cliques of $(G^{\text{csp}})'$ with cardinality being denoted by $n_k,k\in[p]$. Then the set of variables $\underline{x}$ is decomposed into $\underline{x}(I_1), \underline{x}(I_2), \ldots, \underline{x}(I_p)$. Let $J_1,\ldots,J_p$ be defined as in Chapter~\ref{chap6:sec2}.

Now we consider the \ac{tsp} for each subsystem involving the variables $\underline{x}(I_k)$, $k\in[p]$ respectively as follows. Let
\begin{equation}\label{cts-eq1:ncts}
\sA\coloneqq  \supp(f)\cup\bigcup_{j=1}^m\supp(g_j)\text{ and }
\sA_k\coloneqq  \{w\in\sA\mid\var(w)\subseteq \underline{x}(I_k)\} ,
\end{equation}
for $k\in[p]$. As before, let $g_0=1$, $d_j=\lceil\deg(g_j)/2\rceil$, $j\in\{0\}\cup[m]$ and $r_{\min}=\max\{\lceil\deg(f)/2\rceil,d_1,\ldots,d_m\}$. Fix a relaxation order $r\ge r_{\min}$. Let $\W_{r-d_j,k}$ be the standard monomial basis of degree $\le r-d_j$ with respect to the variables $\underline{x}(I_k)$ and $G_{r,k}^{\text{tsp}}$ be the \ac{tsp} graph with nodes $\W_{r,k}$ associated with $\sA_k$ defined as in Chapter~\ref{sec3:ncts}.
Let $G_{r,k,0}^{(0)}=G_{r,k}^{\text{tsp}}$ and $G_{r,k,j}^{(0)}$ be the empty graph with nodes $V_{r,k,j}\coloneqq \W_{r-d_j,k}$ for $j\in J_k, k\in[p]$. Letting
\begin{equation}\label{cts-eq2:ncts}
\sC_{r}^{(s)}\coloneqq \bigcup_{k=1}^p\bigcup_{j\in \{0\}\cup J_k}\supp_{g_j}(G_{r,k,j}^{(s)}),
\end{equation}
we iteratively define a sequence of graphs $(G_{r,k,j}^{(s)}(V_{r,k,j},E_{r,k,j}^{(s)}))_{s\ge1}$ for each $j\in\{0\}\cup J_k,k\in[p]$ by
\begin{equation}\label{cts-eq3:ncts}
	G_{r,k,j}^{(s)}\coloneqq (F_{r,k,j}^{(s)})',
\end{equation}
where $F_{r,k,j}^{(s)}$ is the graph with nodes $V_{r,k,j}$ and edges
\begin{equation}\label{cts-eq4:ncts}
E(F_{r,k,j}^{(s)})=\{\{u,v\}\mid u\ne v\in V_{r,k,j}, u^{\star}\supp(g_j)v\cap\sC_{r}^{(s-1)}\ne\emptyset\}.
\end{equation}

Let $t_{k,j}=|\W_{r-d_j,k}|$ for all $k,j$. Then for each $s\ge1$ (called the {\em sparse order}), the moment relaxation based on correlative-term sparsity for \eqref{constr_eigmin:ncts} is given by
\begin{equation}\label{cts-eq5:ncts}
\begin{array}{rll}
\lambda_{\text{cs-ts}}^{r,s}(f, \frakg)\coloneqq 
&\inf\limits_{\y}&L_{\y}(f)\\
&\rm{ s.t.}&\B_{G_{r,k,0}^{(s)}}\circ \M_{r}(\y, I_k)\in\Pi_{G_{r,k,0}^{(s)}}(\Sbb^+_{r_{k,0}}), \quad k\in[p]\\
&&\B_{G_{r,k,j}^{(s)}}\circ \M_{r-d_j}(g_j\y, I_k)\in\Pi_{G_{r,k,j}^{(s)}}(\Sbb^+_{r_{k,j}}), \quad j\in J_k,k\in[p]\\
&&y_1=1
\end{array}
\end{equation}

For all $k,j$, let us write $\M_{r-d_j}(g_j\y, I_k)=\sum_{w}\D_{w}^{k,j}y_{w}$ for appropriate matrices $\{\D_{w}^{k,j}\}_{k,j,w}$. Then for each $s\ge1$, the dual of \eqref{cts-eq5:ncts} reads as
\begin{equation}\label{cts-eq6:ncts}
\begin{cases}
\sup\limits_{\G_{k,j},b}&b\\
\,\rm{ s.t.}\, &\sum_{k=1}^p\sum_{j\in \{0\}\cup J_k}\langle \G_{k,j},\D_{w}^{k,j}\rangle+b\delta_{1w}=f_{w},\quad\forall w\in\sC_{r}^{(s)}\\
&\G_{k,j}\in\Sbb^+_{t_{k,j}}\cap\Sbb_{G_{r,k,j}^{(s)}},\quad j\in \{0\}\cup J_k, k\in[p]
\end{cases}
\end{equation}
where $\sC_{r}^{(s)}$ is defined in \eqref{cts-eq2:ncts}.

The properties of the relaxations \eqref{cts-eq5:ncts}--\eqref{cts-eq6:ncts} are summarized in the following theorem.
\begin{theoremf}\label{cts-prop1:ncts}
Assume that $\{f\}\cup \frakg\subseteq\Sym\,\R\langle\underline{x}\rangle$. Then the following hold:
\begin{enumerate}[(1)]
\item Fixing a relaxation order $r\ge r_{\min}$, the sequence $(\lambda_{\text{cs-ts}}^{r,s}(f, \frakg))_{s\ge1}$ is monotonically non-decreasing and $\lambda_{\text{cs-ts}}^{r,s}(f, \frakg)\le\lambda_{\text{cs}}^{r}(f, \frakg)$ for all $s\ge1$ (with $\lambda_{\text{cs}}^{r}(f, \frakg)$ being defined in Chapter~\ref{eq:sparse_constr_eigmin_primal}).
\item Fixing a sparse order $s\ge 1$, the sequence $(\lambda_{\text{cs-ts}}^{r,s}(f, \frakg))_{r\ge r_{\min}}$ is monotonically non-decreasing.
\item If the maximal chordal extension is chosen in \eqref{cts-eq3:ncts}, then $(\lambda_{\text{cs-ts}}^{r,s}(f, \frakg))_{s\ge 1}$ converges to $\lambda_{\text{cs}}^{r}(f, \frakg)$ in finitely many steps.
\end{enumerate}
\end{theoremf}

\section{Numerical experiments}
In this section, we present numerical results of the proposed NCTSSOS hierarchies. The tool {\tt NCTSSOS} to implement these hierarchies is available at 
\vspace{2pt}
\centerline{https://github.com/wangjie212/NCTSSOS}
\vspace{2pt}
{\tt NCTSSOS} employs $\mosek$ as an \ac{SDP} solver.
All numerical examples were computed on an Intel Core i5-8265U@1.60GHz CPU with 8GB RAM memory. In the following, ``mb'' denotes the maximal size of PSD blocks, ``opt'' denotes the optimum, ``time'' denotes running time in seconds, and ``-'' indicates an out of memory error.

Let $\D$ be the semialgebraic set defined by $\frakg=\{1-X_1^2,\ldots,1-X_n^2,X_1-1/3,\ldots,X_n-1/3\}$, and consider the optimization problem of minimizing the eigenvalue of the nc Broyden banded function on $\D$, where the nc Broyden banded function is defined by
\begin{equation*}
f_{\text{Bb}}(\x)=\sum_{i=1}^n(2x_i+5x_i^3+1-\sum_{j\in J_i}(x_j+x_j^2))^{\star}(2x_i+5x_i^3+1-\sum_{j\in J_i}(x_j+x_j^2)),
\end{equation*}
where $J_i=\{j\mid j\ne i, \max(1,i-5)\le j\le\min(n,i+1)\}$.

We compute $\lambda_{\text{cs-ts}}^{3,1}(f,\frakg)$ using approximately smallest chordal extensions for \ac{TS} and present the results in Table \ref{ceigen_Bb}, indicated by ``CS+TS''.
To show the benefits of NCTSSOS against the \ac{CS}-based approach developed in Chapter~\ref{chap:ncsparse}, we also display the results for the latter approach in the table, indicated by ``CS''.
It is evident from the table that NCTSSOS is much more scalable than the \ac{CS}-based approach. Actually, the latter can never be executed due to the memory limit even when the problem has only $6$ variables.

\begin{table}[htbp]
\caption{The eigenvalue minimization for the nc Broyden banded function on $\D$ with $r=3,s=1$.}\label{ceigen_Bb}
\renewcommand\arraystretch{1.2}
\centering
\begin{tabular}{c|ccc|ccc}
\hline
\multirow{2}*{$n$}&\multicolumn{3}{c|}{CS+TS}&\multicolumn{3}{c}{CS}\\
\cline{2-7}
&mb&opt&time&mb&opt&time\\
\hline
$5$&$11$&$3.113$&$0.50$&$156$&$3.113$&$70.7$\\
\hline
$10$&$15$&$3.011$&$2.78$&$400$&-&-\\
\hline
$20$&$15$&$9.658$&$11.4$&$400$&-&-\\
\hline
$40$&$15$&$22.93$&$38.1$&$400$&-&-\\
\hline
$60$&$15$&$36.21$&$80.5$&$400$&-&-\\
\hline
$80$&$15$&$49.49$&$138$&$400$&-&-\\
\hline
$100$&$15$&$62.77$&$180$&$400$&-&-\\
\hline
\end{tabular}
\end{table}

Now we construct randomly generated examples whose \ac{csp} graph consists of $p$ maximal cliques of size $15$ as follows: let $f=\sum_{k=1}^p(h_k+h_k^{\star})/2$ where $h_k\in\R\langle x_{10k-9},\ldots,x_{10k+5}\rangle$ is a
random quartic nc polynomials with $15$ terms and coefficients being taken from $[-1,1]$, and let $\frakg=\{g_k\}_{k=1}^p$ where $g_k=1-x_{10k-9}^2-\cdots-x_{10k+5}^2$. We consider the eigenvalue minimization problem for $f$ on the multi-ball defined by $\frakg$. Let $p=100,200,300,400$ so that we obtain $4$ such instances\footnote{The polynomials are available at \url{https://wangjie212.github.io/jiewang/code.html}.}. 
We compute the sequence $(\lambda_{\text{cs-ts}}^{r,s}(f, \frakg))_{s\ge1}$ with $r=2$ and present the results of the first three steps (where we use the maximal chordal extension for the first step and use approximate smallest chordal extensions for the second and third steps, respectively) in Table \ref{ceigen_rge}. 
Again we see that NCTSSOS is much more scalable than the \ac{CS}-based approach.

\begin{table}[htbp]
\caption{The eigenvalue minimization for randomly generated examples over multi-balls with $r=2$.}\label{ceigen_rge}
\renewcommand\arraystretch{1.2}
\centering
\begin{tabular}{c|cccc|ccc}
\hline
\multirow{2}*{$n$}&\multicolumn{4}{c|}{CS+TS}&\multicolumn{3}{c}{CS}\\
\cline{2-8}
&$s$&mb&opt&time&mb&opt&time\\
\hline
\multirow{3}*{$1005$}&$1$&$25$&$-32.58$&$9.71$&\multirow{3}*{$241$}&\multirow{3}*{-}&\multirow{3}*{-}\\
&$2$&$25$&$-31.91$&$24.5$&&&\\
&$3$&$25$&$-31.71$&$40.9$&&&\\
\hline
\multirow{3}*{$2005$}&$1$&$25$&$-63.58$&$33.7$&\multirow{3}*{$241$}&\multirow{3}*{-}&\multirow{3}*{-}\\
&$2$&$25$&$-62.05$&$85.8$&&&\\
&$3$&$25$&$-61.76$&$149$&&&\\
\hline
\multirow{3}*{$3005$}&$1$&$23$&$-95.73$&$74.8$&\multirow{3}*{$241$}&\multirow{3}*{-}&\multirow{3}*{-}\\
&$2$&$23$&$-93.13$&$212$&&&\\
&$3$&$23$&$-92.71$&$396$&&&\\
\hline
\multirow{3}*{$4005$}&$1$&$25$&$-131.1$&$122$&\multirow{3}*{$241$}&\multirow{3}*{-}&\multirow{3}*{-}\\
&$2$&$25$&$-127.5$&$375$&&&\\
&$3$&$25$&$-126.8$&$687$&&&\\
\hline
\end{tabular}\\
\end{table}

\section{Polynomial Bell inequalities}
\label{sec:polbell}
The above framework can be extended to minimize the (normalized) trace of either noncommutative polynomials or a so-called \emph{trace polynomial}.
Let us denote by $\Trace$ the normalized trace operator.
A trace polynomial is a polynomial in symmetric noncommutative variables $x_1,\dots,x_n$ and traces of their products.
Thus naturally each trace polynomial $f$ has an adjoint $f^\star$.
A {\em pure trace polynomial} is a trace polynomial that is made only of traces, i.e., has no free variables $x_j$. For instance,
the trace of a trace polynomial
is a pure trace polynomial, e.g.,
\[
\begin{split}
f&=x_1x_2x_1^2-\Trace(x_2)\Trace(x_1x_2)\Trace(x_1^2x_2)x_2x_1, \\
\Trace(f)&=\Trace(x_1^3x_2)-\Trace(x_2)\Trace(x_1x_2)^2\Trace(x_1^2x_2),\\
f^\star&= x_1^2x_2x_1-\Trace(x_2)\Trace(x_1x_2)  \Trace(x_1^2x_2) x_1x_2.
\end{split}
\]
In this section we connect trace polynomial optimization to violations of nonlinear Bell inequalities, outline a few examples and prove the optimal bound on maximal violation of a covariance Bell inequality considered in the quantum information literature.

We already introduced classical (linear) Bell inequalities in Chapter \ref{sec:bell}, and saw that  $v^* (A_1\otimes B_1+A_1\otimes B_2+A_2\otimes B_1-A_2\otimes B_2)v$
is at most $2$
for all separable states $v \in \C^k\otimes \C^k$ and
$A_j,B_j \in \C^{k\times k}$ satisfying $A_j^*=A_j$, $A_j^2=I$, $B_j^*=B_j$, $B_j^2=I$.
Tsirelson's bound implies that the value is at most $2\sqrt{2}$, which is attained in particular
when $k=2$ and
$v= \frac{1}{\sqrt{2}}(e_1\otimes e_1+e_2\otimes e_2)$, with $(e_1, e_2)$ being an orthonormal basis of $\C^2$.
In general, if $v_k$ is the generalized Bell state,
$$v_k= \frac{1}{\sqrt{k}}\sum_{j=1}^k e_j\otimes e_j \in \R^k\otimes \R^k,$$
which is a maximally entangled bipartite state on $\C^k\otimes \C^k$, unique up to bipartite unitary equivalence,
then
\begin{equation}\label{e:ghz}
v_k^*(X\otimes Y) v_k = \Trace(XY)
\end{equation}
for all $X,Y\in \Sbb_k$.

While linear Bell inequalities are linear in expectation values of (products of) observables, polynomial Bell inequalities contain multivariate polynomials in expectation values of (products of) observables.
For this reason, noncommutative polynomial optimization is not suitable for studying violations of nonlinear Bell inequalities.
In contrast, trace polynomial optimization gives upper bounds on violations of polynomial Bell inequalities, at least for certain families of states, e.g., the maximally entangled bipartite states via \eqref{e:ghz}.
Consider a simple quadratic Bell inequality
\begin{equation}\label{e:bell2}
\big(v^* (A_1\otimes B_2+A_2\otimes B_1)v\big)^2+
\big(v^* (A_2\otimes B_1-A_2\otimes B_2)v \big)^2
\le4 .
\end{equation}
An automatized proof of \eqref{e:bell2} for maximally entangled states of arbitrary dimension can be obtained by solving the optimization problem:
\begin{equation}\label{e:bell3}
	\begin{cases}
		\sup 
		&\left(\Trace(a_1b_2+a_2b_1)\right)^2+
		\left(\Trace(a_1b_1-a_2b_2)\right)^2 \\
		\rm{ s.t.}&a_j^2=1,b_j^2=1 \text{ for }j=1,2.
	\end{cases}
\end{equation}
We compare the value of the dense relaxation of \eqref{e:bell3} with the ones obtained after exploiting \ac{TS}. 
At relaxation order $r = 2$, we obtain a bound of $4$ (the optimal one) in the dense setting. 
The number of \ac{SDP} equality constraints is $222$ and the size of the \ac{SDP} matrix is $53$.
At the sparse order $s=1$ and using the maximal chordal extension, we obtain the same bound but the maximal block size is only 6 and the number of equality constraints is only $30$.

Another class of polynomial Bell inequalities arises from covariances of quantum correlations. Let
$\cov_v(X,Y) \coloneqq 
v^*(X\otimes Y) v - v^*(X\otimes I) v \cdot v ^*(I\otimes Y) v$.
and let us show that the value of
\begin{equation}\label{e:bell4}
\begin{split}
&\cov_v(A_1,B_1)+\cov_v(A_1,B_2)+\cov_v(A_1,B_3) \\
+&\cov_v(A_2,B_1)+\cov_v(A_2,B_2)-\cov_v(A_2,B_3)\\
+&\cov_v(A_3,B_1)-\cov_v(A_3,B_2)
\end{split}
\end{equation}
is at most 5 for every maximally entangled state.
Let
\begin{align*}
t=\,&\Trace(a_1b_1)-\Trace(a_1)\Trace(b_1)+\Trace(a_1b_2)-\Trace(a_1)\Trace(b_2)+\Trace(a_1b_3)+\Trace(a_2b_1)\\
&-\Trace(a_1)\Trace(b_3)-\Trace(a_2)\Trace(b_1)+\Trace(a_2b_2)-\Trace(a_2)\Trace(b_2)-\Trace(a_2b_3)\\
&+\Trace(a_2)\Trace(b_3)+\Trace(a_3b_1)-\Trace(a_3)\Trace(b_1)-\Trace(a_3b_2)+\Trace(a_3)\Trace(b_2).
\end{align*}
The dense relaxation of
\begin{equation}\label{e:bell5}
\sup\ t\ \text{ s.t. }\  a_j^2=1,b_j^2=1 \text{ for }j=1,2,3
\end{equation}
with $r=2$ returns 5. Therefore the value of \eqref{e:bell4} is at most 5 for every maximally entangled state, regardless of the local dimension $k$. 
This dense relaxation involves $1010$ \ac{SDP} equality constraints and an \ac{SDP} matrix of size $115$.
Exploiting \ac{TS} at the sparse order $s=1$ with the maximal chordal extension yields the same bound but requires to solve an \ac{SDP} with maximal block size $83$ and $797$ equality constraints. 
We can obtain an even better computational gain by relying on approximately smallest chordal extensions as it yields an \ac{SDP} with maximal block size $25$ and only $115$ equality constraints, $10$ times less compared to the dense case.
We refer the interested programmer to the Julia scripts displayed in Appendix \ref{sec:tssos_nc}, allowing one to retrieve these results.

\section{Notes and sources}
The material from this chapter is mainly issued from \citencts{nctssos}.
Our framework can be extended to minimize the trace of a noncommutative polynomial over a noncommutative semialgebraic set; see \citencts[\S~5]{nctssos}.
Optimization over trace polynomials has been recently developed in \citencts{klep2022optimization}, where the authors present a novel Positivstellensatz  certifying positivity of trace polynomials subject to trace constraints, and
a hierarchy of semidefinite relaxations
converging monotonically to the optimum of a trace polynomial subject to tracial constraints is provided.
Trace polynomial optimization can be used to detect entanglement in multipartite Werner states  \citencts{huber2021dimension}.
The interested reader can find more information about bilocal models in \citencts{BRGP,Chaves}, covariance of quantum correlations in \citencts{PHBB} and detection of partial separability in \citencts{Uffink}.
Inequality \eqref{e:bell2} is given in \citencts{Uffink}, where it is shown to hold for all separable states, and for all 2-dimensional states.
In \citencts{NKI}, \eqref{e:bell2} is shown to hold for arbitrary states, meaning it admits no quantum violations.

In \citencts{PHBB} it is shown that while the value of \eqref{e:bell4} is at most $\frac92$ for separable states, it attains the value 5 with the Bell state $v_2$.
The authors also performed extensive numerical search within entangled states for local dimensions $k\le 5$, but no higher value of \eqref{e:bell4} was found.
They leave it as an open question whether higher dimensional entangled states could lead to larger violations \citencts[Appendix D.1(b)]{PHBB}.


\input{ncts.bbl}

%% file: jsr.tex
\chapter{Application in stability of control-systems}\label{chap:sparsejsr}
The concept of \ac{JSR} can be viewed as a generalization of the usual spectral radius to the case of multiple matrices. 
The exact computation and even the approximation of the \ac{JSR} for a set of matrices are, however, notoriously difficult. In this chapter, we discuss how to efficiently compute upper bounds on \ac{JSR} via the \ac{SOS} approach when the matrices possess certain sparsity.

\section{Approximating JSR via SOS relaxations}
The \ac{JSR} for a set of matrices $\AA=\{A_1,\ldots,A_m\}\subseteq\R^{n\times n}$ is given by
\begin{equation}\label{jsf:def}
	\rho(\AA)\coloneqq \lim_{k\rightarrow\infty}\max_{\sigma\in\{1,\ldots,m\}^k}\|A_{\sigma_1}A_{\sigma_2}\cdots A_{\sigma_k}\|^{\frac{1}{k}} .
\end{equation}
Note that the value of $\rho(\AA)$ is independent of the choice of the norm used in \eqref{jsf:def}. 
Parrilo and Jadbabaie proposed to compute a sequence of upper bounds for $\rho(\AA)$ via \ac{SOS} relaxations. The underlying idea is based on the following theorem.
\begin{theorem}[\citejsr{parrilo}, Theorem 2.2]\label{sec2-thm1jsr}
Let $\AA=\{A_1,\ldots,A_m\}\subseteq\R^{n\times n}$ be a set of matrices, and let $p$ be a strictly positive form of degree $2r$ that satisfies
\begin{equation*}
	p(A_i\x)\le\gamma^{2r}p(\x),\quad\forall\x\in\R^n,\quad i=1,\ldots,m.
\end{equation*}
Then, $\rho(\AA)\le\gamma$.
\end{theorem}

By replacing positive forms with more tractable \ac{SOS} forms, Theorem \ref{sec2-thm1jsr} immediately suggests the following \ac{SOS} relaxations indexed by $r\in\N^*$ to compute a sequence of upper bounds for $\rho(\AA)$: 
\begin{equation}\label{densesosjsr}
\begin{array}{rll}
\rho_{\text{SOS},2r}(\AA)\coloneqq &\inf\limits_{p\in\R[\x]_{2r},\gamma} &\gamma\\
&\quad\,\,\rm{ s.t.} &p(\x)-\|\x\|_2^{2r}\in\Sigma_{n,2r}\\
&&\gamma^{2r}p(\x)-p(A_i\x)\in\Sigma_{n,2r},\quad i=1,\ldots,m
\end{array}
\end{equation}
The term ``$\|\x\|_2^{2r}$'' appearing in the first constraint of \eqref{densesosjsr} is added to make sure that $p$ is strictly positive.
The optimization problem \eqref{densesosjsr} can be solved via \ac{SDP} by bisection on $\gamma$. It was shown in \citejsr{parrilo} that the upper bound $\rho_{\text{SOS},2r}(\AA)$ satisfies the inequalties stated in the following theorem.
\begin{theorem}[\citejsr{parrilo}]\label{sec2-thm2jsr}
Let $\AA=\{A_1,\ldots,A_m\}\subseteq\R^{n\times n}$. For any integer $r\ge1$, one has $m^{-\frac{1}{2r}}\rho_{\text{SOS},2r}(\AA)\le\rho(\AA)\le\rho_{\text{SOS},2r}(\AA)$.
\end{theorem}
It is immediate from Theorem \ref{sec2-thm2jsr} that $(\rho_{\text{SOS},2r}(\AA))_{r\ge1}$ converges to $\rho(\AA)$ as $r$ goes to infinity.

\section{The SparseJSR Algorithm}\label{sec:sparsejsr}
In this section, we propose an algorithm SparseJSR for bounding \ac{JSR} from above based on the sparse \ac{SOS} decomposition when the matrices $\AA$ possess certain sparsity. The first step is to establish a hierarchy of sparse supports for the auxiliary form $p(\x)$ used in the \ac{SOS} program \eqref{densesosjsr}.

Let us fix a relaxation order $r$. Let $p_0(\x)=\sum_{i=1}^nc_ix_i^{2r}$ with random coefficients $c_{i}\in (0,1)$ and let $\sA^{(0)}=\supp(p_0)$. Then for $s\in\N^*$, we iteratively define
\begin{equation}\label{suppjsr}
	\sA^{(s)}\coloneqq \sA^{(s-1)}\cup\bigcup_{i=1}^m\supp(p_{s-1}(A_i\x)),
\end{equation}
where $p_{s-1}(\x)=\sum_{\a\in\sA^{(s-1)}}c_{\a}\x^{\a}$ with random coefficients $c_{\a}\in (0,1)$. Note that here the particular form of $p_0(\x)$ is chosen such that $\sA^{(s)}$ contains all possible homogeneous monomials of degree $2r$ that are ``compatible'' with the couplings between variables $x_1,\ldots, x_n$ introduced by the mappings $\x\mapsto A_i\x$ for all $i$. It is clear that
\begin{equation}\label{hsuppjsr}
	\sA^{(1)}\subseteq\cdots\subseteq\sA^{(s)}\subseteq\sA^{(s+1)}\subseteq\cdots\subseteq\N^n_{2r}
\end{equation}
and the sequence $(\sA^{(s)})_{s\ge1}$ stabilizes in finitely many steps. We emphasis that it is not guaranteed a hierarchy of sparse supports is always retrieved by \eqref{hsuppjsr} even if all $A_i$ are sparse. For instance,
if some matrix $A_i\in\AA$ has a fully dense row, then by definition, one immediately has $\sA^{(1)}=\N^n_{2r}$. In this case, the sparsity of $\AA$ cannot be exploited by the present method.

On the other hand, if the matrices in $\AA$ have some common zero columns, then a hierarchy of sparse supports must be retrieved by \eqref{hsuppjsr} as shown in the next proposition.

\begin{proposition}\label{propjsr}
Let $\AA=\{A_1,\ldots,A_m\}\subseteq\R^{n\times n}$ and assume that the matrices in $\AA$ have common zero columns indexed by $J\subseteq[n]$. Let $\tilde{\N}^{n-|J|}_{2r}\coloneqq \{(\alpha_i)_{i\in[n]}\in\N^n\mid(\alpha_i)_{i\in[n]\setminus J}\in\N^{n-|J|}_{2r},\alpha_i=0\text{ for }i\in J\}$ and $\mathbf{b}_j\coloneqq \{(\alpha_i)_{i\in[n]}\in\N^n\mid\alpha_j=2r,\alpha_i=0\text{ for }i\ne j\}$ for $j\in[n]$. Then $\sA^{(s)}\subseteq\tilde{\N}^{n-|J|}_{2r}\cup\{\mathbf{b}_j\}_{j\in J}$ for all $s\ge1$. 
\end{proposition}
\begin{proof}
Let us do induction on $s$. It is obvious that $\sA^{(0)}\subseteq\tilde{\N}^{n-|J|}_{2r}\cup\{\mathbf{b}_j\}_{j\in J}$.
Now assume $\sA^{(s)}\subseteq\tilde{\N}^{n-|J|}_{2r}\cup\{\mathbf{b}_j\}_{j\in J}$ for some $s\ge0$. Since the variables effectively involved in $p_s(A_j\x)$ are contained in $\{x_i\}_{i\in[n]\setminus J}$, we have $\supp(p_s(A_j\x))\subseteq\tilde{\N}^{n-|J|}_{2r}$ for $j=1,\ldots,m$. This combined with the induction hypothesis yields $\sA^{(s+1)}\subseteq\tilde{\N}^{n-|J|}_{2r}\cup\{\mathbf{b}_j\}_{j\in J}$ as desired.
\qed
\end{proof}

For each $s\ge1$, by restricting $p(\x)$ to forms with the sparse support $\sA^{(s)}$, \eqref{densesosjsr} now reads as
\begin{equation}\label{ssosjsr}
	\begin{cases}
	\inf\limits_{p\in\R[\sA^{(s)}],\gamma}&\gamma\\
	\quad\,\,\,\,\rm{ s.t.}&p(\x)-\|\x\|_2^{2r}\in\Sigma_{n,2r}\\
		&\gamma^{2r}p(\x)-p(A_i\x)\in\Sigma_{n,2r},\quad i\in[m]
	\end{cases}
\end{equation}

Let $\sA_i^{(s)}=\sA^{(s)}\cup\supp(p_s(A_i\x)$ for $i=1,\ldots,m$. In order to exploit the sparsity present in \eqref{ssosjsr}, for a sparse support $\sA\subseteq\N^n_{2r}$ we define
\begin{equation}
	\Sigma(\sA)\coloneqq \left\{f\in\R[\sA]\mid\exists \G\in\Sbb^+_{|\Basis|}\cap\Sbb_{G^{\text{tsp}}}\,\text{ s.t. }f=(\x^{\Basis})^\intercal \G\x^{\Basis}\right\},
\end{equation}
where $\Basis$ is a monomial basis and $G^{\text{tsp}}$ is the \ac{tsp} graph with respect to $\Basis$ defined as in Chapter~\ref{sec:tssos-uncons}. Then by replacing $\Sigma_{n,2d}$ with $\Sigma(\sA^{(s)})$ or $\Sigma(\sA_i^{(s)})$ in \eqref{ssosjsr}, we therefore obtain a hierarchy of sparse \ac{SOS} relaxations indexed by $s$ for a fixed $r$:
\begin{equation}\label{tssosjsr}
\begin{array}{rll}
	\rho_{s,2r}(\AA)\coloneqq &\inf\limits_{p\in\R[\sA^{(s)}],\gamma}&\gamma \\
	&\quad\,\,\,\,\rm{ s.t.}& p(\x)-\|\x\|_2^{2r}\in\Sigma({\sA^{(s)}})\\
	&&	\gamma^{2r}p(\x)-p(A_i\x)\in\Sigma({\sA^{(s)}_i}),\quad i\in[m]
\end{array}
\end{equation}
As in the dense case, the optimization problem \eqref{tssosjsr} can be solved via \ac{SDP} by bisection on $\gamma$.
We call the index $s$ the \emph{sparse order} of \eqref{tssosjsr}, and we get a hierarchy of upper bounds on the \ac{JSR} indexed by the sparse order $s$ when the relaxation order $r$ is fixed.
\begin{theoremf}\label{thm2jsr}
Let $\AA=\{A_1,\ldots,A_m\}\subseteq\R^{n\times n}$. For any integer $r\ge1$, one has $\rho_{\text{SOS},2r}(\AA)\le\cdots\le\rho_{s,2r}(\AA)\le\cdots\le\rho_{2,2r}(\AA)\le\rho_{1,2r}(\AA)$.
\end{theoremf}
\begin{proof}
For any fixed $r\in\N^*$, because of \eqref{hsuppjsr}, it is clear that the feasible set of \eqref{tssosjsr} with the sparse order $s$ is contained in the feasible set of \eqref{tssosjsr} with the sparse order $s+1$, which is in turn contained in the feasible set of \eqref{densesosjsr}. This yields the desired conclusion.
\end{proof}

Therefore, the algorithm SparseJSR computes a non-increasing sequence of upper bounds for the \ac{JSR} of a tuple of matrices via solving \eqref{tssosjsr} for any fixed $r$. By tuning the relaxation order $r$ and the sparse order $s$, SparseJSR offers a trade-off between the computational cost and the quality of the obtained upper bound.



\section{Numerical Experiments}\label{sec:benchs}
In this section, we present numerical experiments for the proposed algorithm SparseJSR, which was implemented in the Julia package also named \href{https://github.com/wangjie212/SparseJSR}{\tt SparseJSR}.
In Appendix \ref{sec:tssosdyn}, we provide a Julia script to illustrate how to use {\tt SparseJSR}.

The examples were computed on an Intel Core i5-8265U@1.60GHz CPU with 8GB RAM memory, where the sparse order $s$ was set to $1$, the tolerance for bisection was set to $\epsilon=1\times10^{-5}$, and the initial interval for bisection was set to $[0,2]$. To measure the quality of upper bounds (ub) provided by {\tt SparseJSR}, we also compute lower bounds (lb) on \ac{JSR} using Gripenberg's algorithm \citejsr{gripenberg1996computing}.
In the following, ``mb'' denotes the maximal size of PSD blocks, ``time'' denotes running time in seconds, $*$ indicates running time exceeding one hour, and ``-'' indicates an out of memory error.

\subsection{Randomly generated examples}
We generate random sparse matrices as follows\footnote{Available at https://wangjie212.github.io/jiewang/code.html.}: generate a random directed graph $G$ with $n$ nodes and $n+10$ edges; for each edge $(i,j)$ of $G$, put a random number in $[-1,1]$ on the position $(i,j)$ of the matrix and put zeros on the other positions. We compute an upper bound on \ac{JSR} for pairs of such matrices using the first-order \ac{SOS} relaxation and report the results in Table \ref{random}. It is evident that the sparse approach is much more efficient than the dense approach. 
For instance, the dense approach takes over $3600\,$s when the size of matrices is greater than $100$ while the sparse approach can handle matrices of size $120$ within $12\,$s. The upper bound produced by the sparse approach is slightly weaker, but is still close to the lower bound.


\begin{table}[htbp]
\caption{Randomly generated examples with $r=1$ and $m = 2$.}\label{random}
\renewcommand\arraystretch{1.2}
\centering
\begin{tabular}{cc|ccc|ccc}
\hline
&&\multicolumn{3}{c|}{Sparse}&\multicolumn{3}{c}{Dense}\\
\cline{3-8}
\multirow{-2}*{$n$}&\multirow{-2}*{lb}&time&ub&mb&time&ub&mb\\
\hline
$20$&$0.7894$&$0.74$&$0.8192$&$10$&$1.88$&$0.7967$&$20$\\
\hline
$40$&$0.9446$&$2.68$&$0.9446$&$14$&$25.6$&$0.9446$&$40$\\
\hline
$60$&$0.7612$&$3.64$&$0.7843$&$13$&$171$&$0.7612$&$60$\\
\hline
$80$&$0.9345$&$5.95$&$0.9399$&$15$&$743$&$0.9345$&$80$\\
\hline
$100$&$0.8642$&$8.15$&$0.9132$&$13$&$2568$&$0.8659$&$100$\\
\hline
$120$&$0.7483$&$11.7$&$0.7735$&$16$&$*$&$*$&$*$\\
\hline
\end{tabular}
\end{table}

\subsection{Examples from control systems}

Here we consider examples from \citejsr{maggio2020control}, where the dynamics of closed-loop systems are given by the combination of a plant and a one-step delay controller that stabilizes the plant. The closed-loop system evolves according to either a completed or a missed computation. In the case of a deadline hit, the closed-loop state matrix is $A_H$. In the case of a deadline miss, the associated closed-loop state matrix is $A_M$. The computational platform (hardware and software) ensures that no more than $m-1$ deadlines are missed consecutively.
The set of possible realisations $\AA$ of such a system contains either a single hit or at most $m-1$ misses followed by a hit, namely $\AA \coloneqq  \{A_H A_M^i \mid 0 \leq  i \leq m-1 \}$.
Then, the closed-loop system that can switch between the realisations included in $\AA$  is asymptotically stable if and only if $\rho (\AA) < 1$. This gives an indication for scheduling and control co-design, in which the hardware and software platform must guarantee that the maximum number of deadlines missed consecutively does not interfere with stability requirements.

In Tables \ref{control1} and \ref{control2}, we report the results obtained for various control systems with $n$ states, under $m-1$ deadline misses, by applying the dense and sparse \ac{SOS} approaches with relaxation orders $r=1$ and $r=2$, respectively.
The examples are randomly generated, i.e., our script generates a
random system and then tries to control it\footnote{Available at https://wangjie212.github.io/jiewang/code.html.}. 

In Table \ref{control1}, we fix $m=5$ and vary $n$ from $20$ to $110$. For these examples, surprisingly the dense and sparse approaches with the relaxation order $r=1$ always produce the same upper bounds. As one can see from the table, the sparse approach is more efficient and scalable than the dense one.

In Table \ref{control2}, we vary $m$ from $3$ to $11$ and vary $n$ from $8$ to $24$.
The column ``ub'' indicates the upper bound given by the dense approach with the relaxation order $r=1$. For these examples, with the relaxation order $r=2$, the sparse approach produces upper bounds that are very close to those given by the dense approach. And again the sparse approach is more efficient and scalable than the dense one.

\begin{table}[htbp]
\caption{Results for control systems with $r=1$ and $m = 5$.}\label{control1}
\renewcommand\arraystretch{1.2}
\centering
\begin{tabular}{cc|ccc|ccc}
\hline
&&\multicolumn{3}{c|}{Sparse}&\multicolumn{3}{c}{Dense}\\
\cline{3-8}
\multirow{-2}*{$n$}&\multirow{-2}*{lb}&time&ub&mb&time&ub&mb\\
\hline
$30$&$1.4682$&$4.30$&$1.5132$&$14$&$57.8$&$1.5131$&$30$\\
\hline
$30$&$1.0924$&$4.42$&$1.0961$&$14$&$65.4$&$1.0961$&$30$\\
\hline
$50$&$1.3153$&$17.3$&$1.3248$&$18$&$660$&$1.3248$&$50$\\
\hline
$50$&$1.1884$&$17.5$&$1.1884$&$18$&$680$&$1.1884$&$50$\\
\hline
$70$&$1.8135$&$54.2$&$1.8578$&$22$&$*$&$*$&$*$\\
\hline
$70$&$1.2727$&$53.9$&$1.2727$&$22$&$*$&$*$&$*$\\
\hline
$90$&$1.8745$&$133$&$1.9020$&$26$&$*$&$*$&$*$\\
\hline
$90$&$1.4452$&$132$&$1.4452$&$26$&$*$&$*$&$*$\\
\hline
$110$&$2.3597$&$280$&$2.3943$&$30$&-&-&-\\
\hline
$110$&$1.5753$&$287$&$1.5753$&$30$&-&-&-\\
\hline
\end{tabular}
\end{table}

\begin{table}[htbp]
\caption{Results for control systems with $r=2$.}\label{control2}
\renewcommand\arraystretch{1.2}
\centering
\begin{tabular}{p{0.2cm}p{0.2cm}p{1cm}p{1cm}|p{0.5cm}p{1cm}c|p{0.5cm}p{1cm}c}
\hline
&&&&\multicolumn{3}{c|}{Sparse}&\multicolumn{3}{c}{Dense}\\
\cline{5-10}
\multirow{-2}*{$m$}&\multirow{-2}*{$n$}&\multirow{-2}*{\quad lb}&\multirow{-2}*{\quad ub}&time&\multicolumn{1}{c}{ub}&mb&time&\multicolumn{1}{c}{ub}&mb\\
\hline
$3$&$8$&$0.7218$&$0.7467$&$0.60$&$0.7310$&$10$&$13.4$&$0.7305$&$36$\\
\hline
$4$&$10$&$0.7458$&$0.7738$&$0.75$&$0.7564$&$10$&$107$&$0.7554$&$55$\\
\hline
$5$&$12$&$0.8601$&$0.8937$&$1.08$&$0.8706$&$10$&$1157$&$0.8699$&$78$\\
\hline
$6$&$14$&$0.7875$&$0.8107$&$1.32$&$0.7958$&$10$&$*$&$*$&$*$\\
\hline
$7$&$16$&$1.1110$&$1.1531$&$1.81$&$1.1182$&$10$&-&-&-\\
\hline
$8$&$18$&$1.0487$&$1.0881$&$2.05$&$1.0569$&$10$&-&-&-\\
\hline
$9$&$20$&$0.7570$&$0.7808$&$2.52$&$0.7660$&$10$&-&-&-\\
\hline
$10$&$22$&$0.9911$&$1.0315$&$2.70$&$1.0002$&$10$&-&-&-\\
\hline
$11$&$24$&$0.7339$&$0.7530$&$3.67$&$0.7418$&$10$&-&-&-\\
\hline
\end{tabular}
\end{table}

\section{Notes and sources}
The material from this chapter is issued from \citejsr{wang2021sparsejsr}. The \ac{JSR} was first introduced by Rota and Strang in \citejsr{rota} and since then has found applications in many areas such as the stability of switched linear dynamical systems, the continuity of wavelet functions, combinatorics and language theory, the capacity of some codes, the trackability of graphs. We refer the reader to \citejsr{jungers} for a survey of the theory and applications of \ac{JSR}.


\input{jsr.bbl}

%% file: misc.tex
\chapter{Miscellaneous}\label{chap:misc}
The goal of this chapter is to propose alternative schemes to methods based on sparse \ac{SOS} polynomials. 
First, we focus in Section \ref{sec:sonc} on \ac{SONC} polynomials, a set of new nonnegativity  certificates, independent of the set of \ac{SOS} polynomials described earlier in this book.
We present a characterization of \ac{SONC} polynomials in terms of sums of binomial squares with rational
exponents. Next, we present in Section \ref{sec:firstsdp} a framework to speed-up the resolution of \ac{SDP} relaxations arising from the \ac{moment-SOS} hierarchy. This framework is based on the use of first-order methods. 
\section{Nonnegative circuits and binomial squares}\label{sec:sonc}
A lattice point $\a\in\N^n$ is said to be {\em even} if it is in $(2\N)^n$.
A subset $\TT=\{\a_1,\ldots,\a_m\} \subseteq(2\N)^n$ is called a {\em trellis} when $\TT$ comprises the vertices of a simplex. 
Given a trellis $\TT$, a {\em circuit polynomial} is of the form $\sum_{\a\in\TT}c_{\a}\x^{\a}-d\x^{\b}\in\R[\x]$, 
where $c_{\a}>0$ for all $\a\in\TT$, and $\b$ lies in the relative interior of the simplex associated to $\TT$. The name ``circuit polynomial'' stems from the fact that $(\TT,\b)$ consists of a circuit.
Given a circuit polynomial $f=\sum_{\a\in\TT}c_{\a}\x^{\a}-d\x^{\b}\in\R[\x]$, there exist unique barycentric coordinates $(\lambda)_{j=1}^m$ satisfying
\begin{align}\label{equ:BarycentricCoordinates}
\beta =\sum_{j=1}^m\lambda_j\alpha(j)\text{ with }\lambda_j>0\text{ and }\sum_{j=1}^m \lambda_j = 1.
\end{align}
We then define the related circuit number as $\Theta_f=\prod_{j=1}^m\left(c_{\alpha(j)}/\lambda_j\right)^{\lambda_j}$.

Circuit polynomials are proper building blocks for nonnegativity certificates since the circuit number alone determines whether they are nonnegative.
\begin{theorem}\label{theorem:CircuitPolynomialNonnegativity}
A circuit polynomial $f$ is nonnegative if and only if $f$ is a sum of monomial squares or $|d| \leq \Theta_f$.
\end{theorem}
\begin{example}\label{ex:motzkin}
Let $f=x_1^4x_2^2+x_1^2x_2^4+1-3x_1^2x_2^2$ be the Motzkin polynomial and $\TT=\{\a_1=(0,0),\a_2=(4,2),\a_3=(2,4)\}$, $\b=(2,2)$. 
Then $\b=\frac{1}{3}\a_1+\frac{1}{3}\a_2+\frac{1}{3}\a_3$. 
\begin{center}
\begin{tikzpicture}
\draw (0,0)--(2,4);
\draw (0,0)--(4,2);
\draw (2,4)--(4,2);
\fill (0,0) circle (2pt);
\node[above left] (1) at (0,0) {$(0,0)$};
\node[below left] (1) at (0,0) {$\a_1$};
\fill (2,4) circle (2pt);
\node[above right] (2) at (2,4) {$(2,4)$};
\node[above left] (2) at (2,4) {$\a_3$};
\fill (4,2) circle (2pt);
\node[above right] (3) at (4,2) {$(4,2)$};
\node[below right] (3) at (4,2) {$\a_2$};
\fill (2,2) circle (2pt);
\node[above right] (4) at (2,2) {$(2,2)$};
\node[below left] (4) at (2,2) {$\b$};
\end{tikzpicture}
\end{center}
One easily checks that $|-3| \leq  \Theta_{f} = 3$, proving that $f$ is nonnegative. 
\end{example}

An explicit representation of a polynomial being a \ac{SONC} provides a certificate for its nonnegativity, which is called a \ac{SONC} decomposition. Polynomial optimization via \ac{SONC} decompositions can be solved by means of geometric programming.
Here we introduce a potentially cheaper approach based on second-order cone programming.

For a subset of points $M\subseteq\N^n$, let
$$\overline{A}(M)\coloneqq \left\{\frac{1}{2}(\bv+\bw)\mid\bv\ne\bw,\bv,\bw\in M\cap(2\N)^n\right\}$$
be the set of averages of distinct even points in $M$. 
For a trellis $\TT$, we call $M$ a {\em $\TT$-mediated set} if $\TT\subseteq M\subseteq\overline{A}(M)\cup\TT$.
\begin{theoremf}\label{th:mediated}
Let $f=\sum_{\a\in\TT} c_{\a}\x^{\a}-d\x^{\b}\in\R[\x]$ with $d\ne0$ be a nonnegative circuit polynomial. Then $f$ is a sum of binomial squares if and only if there exists a $\TT$-mediated set containing $\b$. Moreover, suppose that $\b$ belongs to a $\TT$-mediated set $M$ and for each $\bu\in M\setminus\TT$, let us write $\bu=\frac{1}{2}(\bv_\bu+\bw_\bu)$ for some $\bv_\bu \ne\bw_\bu \in M\cap(2\N)^n$. 
Then we can rewrite $f$ as $f=\sum_{\bu\in M\setminus\TT}(a_{\bu}\x^{\frac{1}{2}\bv_\bu}-b_{\bu}\x^{\frac{1}{2}\bw_\bu})^2$ for some $a_{\bu},b_{\bu}\in\R$.
\end{theoremf}
By Theorem \ref{th:mediated}, to represent a nonnegative circuit polynomial as a sum of binomial squares, we need to first decide if there exists a $\TT$-mediated set containing a given lattice point, and then compute one if there exists. However, such a $\TT$-mediated set may not exist in general.
In order to circumvent this obstacle, we introduce the concept of $\TT$-rational mediated sets as a replacement of $\TT$-mediated sets by admitting rational numbers in coordinates.

Concretely, for a subset of points $M\subseteq\Q^n$, let us define
$$\widetilde{A}(M)\coloneqq \left\{\frac{1}{2}(\bv+\bw)\mid\bv\ne\bw,\bv,\bw\in M\right\}$$
as the set of averages of distinct rational points in $M$. 
For a trellis $\TT\subseteq\N^n$, we say that $M$ is a {\em $\TT$-rational mediated set} if $\TT\subseteq M\subseteq\widetilde{A}(M)\cup\TT$. Given a trellis $\TT$ and a lattice point $\b\in\Conv(\TT)^{\circ}$, there always exists a $\TT$-rational mediated set containing $\b$ as stated in the following Proposition.

\begin{proposition}\label{lemma:mediated}
Given a trellis $\TT$ and a lattice point $\b\in\Conv(\TT)^{\circ}$, there exists a $\TT$-rational mediated set $M_{\TT\b}$ containing $\b$ such that the denominators (resp. numerators) of coordinates of points in $M_{\TT\b}$ are odd (resp. even) numbers.
\end{proposition}

The proof of Proposition \ref{lemma:mediated} is constructive, and yields an algorithm to compute such a $\TT$-rational mediated set $M_{\TT\b}$. We refer the reader to \citemisc{magron2020sonc} for the details. By virtue of Proposition \ref{lemma:mediated}, we are able to decompose \ac{SONC} polynomials into sums of binomial squares with rational exponents.

\begin{theoremf}\label{th:ratmediated}
Let $f=\sum_{\a\in\TT} c_{\a}\x^{\a}-d\x^{\b}\in\R[\x]$ with $d\ne0$ be a circuit polynomial. Assume that $M_{\TT\b}$ is a $\TT$-rational mediated set containing $\b$ provided by Proposition \ref{lemma:mediated}.
For each $\bu\in M_{\TT\b}\setminus\TT$, let $\bu=\frac{1}{2}(\bv_{\bu}+\bw_{\bu})$ for some $\bv_{\bu}\ne\bw_{\bu}\in M_{\TT\b}$. Then $f$ is nonnegative if and only if $f$ can be written as $f=\sum_{\bu\in M_{\TT\b}\setminus\TT}(a_{\bu}\x^{\frac{1}{2}\bv_{\bu}}-b_{\bu}\x^{\frac{1}{2}\bw_{\bu}})^2$ for some $a_{\bu},b_{\bu}\in\R$.
\end{theoremf}

\begin{example}
As in Example \ref{ex:motzkin}, let $f=x_1^4x_2^2+x_1^2x_2^4+1-3x_1^2x_2^2$ be the Motzkin polynomial, $\TT=\{\a_1=(0,0),\a_2=(4,2),\a_3=(2,4)\}$ and $\b=(2,2)$. 
Let $\b_1=\frac{1}{3}\a_1+\frac{2}{3}\a_2$ and $\b_2=\frac{1}{3}\a_1+\frac{2}{3}\a_3$ such that $\b=\frac{1}{2}\b_1+\frac{1}{2}\b_2$. 
Let $\b_3=\frac{2}{3}\a_1+\frac{1}{3}\a_2$ and $\b_4=\frac{2}{3}\a_1+\frac{1}{3}\a_3$.
Then  $M=\{\a_1,\a_2,\a_3,\b,\b_1,\b_2,\b_3,\b_4\}$ is a $\TT$-rational mediated set containing $\b$.
By Theorem \ref{th:ratmediated}, one has  
\begin{align*}
	f=\,&(a_1 x_1^{\frac{2}{3}} x_2^{\frac{4}{3}}-b_1x_1^{\frac{4}{3}}x_2^{\frac{2}{3}})^2+(a_2 x_1 x_2^2-b_2x_1^{\frac{1}{3}}x_2^{\frac{2}{3}})^2+(a_3x_1^{\frac{2}{3}}x_2^{\frac{4}{3}}-b_3)^2\\
	&+(a_4x_1^2 x_2-b_4x_1^{\frac{2}{3}}x_2^{\frac{1}{3}})^2+(a_5x_1^{\frac{4}{3}}x_2^{\frac{2}{3}}-b_5)^2.
\end{align*}
\begin{center}
\begin{tikzpicture}
\draw (0,0)--(2,4);
\draw (0,0)--(4,2);
\draw (2,4)--(4,2);
\draw (4/3,8/3)--(8/3,4/3);
\fill (0,0) circle (2pt);
\node[above left] (1) at (0,0) {$(0,0)$};
\node[below left] (1) at (0,0) {$\a_1$};
\fill (2,4) circle (2pt);
\node[above right] (2) at (2,4) {$(2,4)$};
\node[above left] (2) at (2,4) {$\a_3$};
\fill (4,2) circle (2pt);
\node[above right] (3) at (4,2) {$(4,2)$};
\node[below right] (3) at (4,2) {$\a_2$};
\fill (2,2) circle (2pt);
\node[above right] (4) at (2,2) {$(2,2)$};
\node[below left] (4) at (2,2) {$\b$};
\fill (4/3,8/3) circle (2pt);
\node[above left] (5) at (4/3,8/3) {$(\frac{4}{3},\frac{8}{3})$};
\node[right] (5) at (4/3,8/3) {$\b_2$};
\fill (8/3,4/3) circle (2pt);
\node[below right] (6) at (8/3,4/3) {$(\frac{8}{3},\frac{4}{3})$};
\node[above] (6) at (8/3,4/3) {$\b_1$};
\fill (2/3,4/3) circle (2pt);
\node[above left] (5) at (2/3,4/3) {$(\frac{2}{3},\frac{4}{3})$};
\node[right] (5) at (2/3,4/3) {$\b_4$};
\fill (4/3,2/3) circle (2pt);
\node[below right] (6) at (4/3,2/3) {$(\frac{4}{3},\frac{2}{3})$};
\node[above] (6) at (4/3,2/3) {$\b_3$};
\end{tikzpicture}
\end{center}
Comparing coefficients yields
\begin{align*}
	f=\,&\frac{3}{2}(x_1^{\frac{2}{3}}x_2^{\frac{4}{3}}-x_1^{\frac{4}{3}}x_2^{\frac{2}{3}})^2+(x_1 x_2^2-x_1^{\frac{1}{3}}x_2^{\frac{2}{3}})^2+\frac{1}{2}(x_1^{\frac{2}{3}}x_2^{\frac{4}{3}}-1)^2\\
	&+(x_1^2 x_2-x_1^{\frac{2}{3}}x_2^{\frac{1}{3}})^2+\frac{1}{2}(x_1^{\frac{4}{3}}x_2^{\frac{2}{3}}-1)^2,
\end{align*}
a sum of five binomial squares with rational exponents.
\end{example}

For a polynomial $f=\sum_{\a\in\sA}f_{\a}\x^{\a}\in\R[\x]$, let
$\Lambda(f)\coloneqq \{\a\in\sA\mid\a\in(2\N)^n\textrm{ and }f_{\a}>0\}$
and $\Gamma(f)\coloneqq \supp(f)\setminus\Lambda(f)$ so that
$f=\sum_{\a\in\Lambda(f)}c_{\a}\x^{\a}-\sum_{\b\in\Gamma(f)}d_{\b}\x^{\b}.$
For each $\b\in\Gamma(f)$, let
\begin{equation}\label{fb}
\CC(\b)\coloneqq \{\TT\mid\TT\subseteq\Lambda(f)\textrm{ and }(\TT,\b)\textrm{ consists of a circuit}\}.
\end{equation}
As a consequence of Theorem 5.5 from \citemisc{wangsiaga}, if $f$ is a \ac{SONC} polynomial, then it admits a decomposition
\begin{equation}\label{sec2-eq1}
f=\sum_{\b\in\Gamma(f)}\sum_{\TT\in\CC(\b)}f_{\TT\b}+\sum_{\a\in \tilde{\sA}}c_{\a}\x^{\a},
\end{equation}
where $f_{\TT\b}$ is a nonnegative circuit polynomial supported on $\TT\cup\{\b\}$ and $\tilde{\sA}=\{\a\in\Lambda(f)\mid\a\notin\cup_{\b\in\Gamma(f)}\cup_{\TT\in\CC(\b)}\TT\}$.

\begin{theoremf}\label{th:soncsocp}
Let $f=\sum_{\a\in\Lambda(f)}c_{\a}\x^{\a}-\sum_{\b\in\Gamma(f)}d_{\b}\x^{\b}\in\R[\x]$. For every $\b\in\Gamma(f)$ and every trellis $\TT\in\CC(\b)$, let $M_{\TT\b}$ be a $\TT$-rational mediated set containing $\b$ provided by Proposition \ref{lemma:mediated}.
Let $M=\cup_{\b\in\Gamma(f)}\cup_{\TT\in\CC(\b)}M_{\TT\b}$. 
For each $\bu\in M\setminus\Lambda(f)$, let $\bu=\frac{1}{2}(\bv_{\bu}+\bw_{\bu})$ for some $\bv_{\bu}\ne\bw_{\bu}\in M$.
Then $f$ is a \ac{SONC} polynomial if and only if $f$ can be written as $f=\sum_{\bu\in M\setminus\Lambda(f)}(a_{\bu}\x^{\frac{1}{2}\bv_{\bu}}-b_{\bu}\x^{\frac{1}{2}\bw_{\bu}})^2+\sum_{\a\in \tilde{\sA}}c_{\a}\x^{\a}$ for some $a_{\bu},b_{\bu}\in\R$.
\end{theoremf}

In order to obtain a \ac{SONC} decomposition of $f$, we use all simplices covering $\b$ for each $\b\in\Gamma(f)$ in Theorem \ref{th:ratmediated}. 
In practice, we do not need that many simplices as illustrated by the following example.
\begin{example}
Let $f=50x_1^4x_2^4+x_1^4+3x_2^4+800-100x_1x_2^2
-100x_1^2x_2$. Let $\a_1=(0,0),\a_2=(4,0),\a_3=(0,4),\a_4=(4,4)$ and $\b_1=(2,1),\b_2=(1,2)$. 
There are two simplices covering $\b_1$: the one with vertices $\{\a_1,\a_2\,\a_3\}$ (denoted by $\Delta_1$), and the one with vertices $\{\a_1,\a_2,\a_4\}$ (denoted by $\Delta_2$). 
There are two simplices covering $\b_2$: $\Delta_1$ and the one with vertices $\{\a_1,\a_3,\a_4\}$ (denoted by $\Delta_3$). One can check that $f$ admits a \ac{SONC} decomposition $f=g_1+g_2$, where $g_1=20x_1^4x_2^4+x_1^4+400-100x_1^2x_2$ supported on $\Delta_2$ and $g_2=30x_1^4x_2^4+3x_2^4+400-100x_1x_2^2$ supported on $\Delta_3$ are both nonnegative circuit polynomials. So the simplex $\Delta_1$ is not needed in this \ac{SONC} decomposition of $f$.
\begin{center}
\begin{tikzpicture}
\draw (0,0)--(0,2);
\draw (0,0)--(2,0);
\draw (2,0)--(2,2);
\draw (0,2)--(2,2);
\draw (0,0)--(2,2);
\draw (0,2)--(2,0);
\fill[fill=green,fill opacity=0.3] (0,0)--(2,0)--(2,2)--(0,0);
\fill[fill=blue,fill opacity=0.3] (0,0)--(0,2)--(2,2)--(0,0);
\fill (0,0) circle (2pt);
\node[below left] (1) at (0,0) {$\a_1$};
\fill (2,0) circle (2pt);
\node[below right] (2) at (2,0) {$\a_2$};
\fill (0,2) circle (2pt);
\node[above left] (3) at (0,2) {$\a_3$};
\fill (2,2) circle (2pt);
\node[above right] (4) at (2,2) {$\a_4$};
\fill (1,0.5) circle (2pt);
\node[right] (5) at (1,0.5) {$\b_1$};
\fill (0.5,1) circle (2pt);
\node[above] (6) at (0.5,1) {$\b_2$};
\end{tikzpicture}
\end{center}
\end{example}
Here we rely on a heuristics to compute a set of simplices with vertices coming from $\Lambda(f)$ and that covers $\Gamma(f)$.  
For $\b\in\Gamma(f)$ and $\a_0\in\Lambda(f)$, let us define an auxiliary linear program:
\begin{equation*}
	\text{SimSel}(\b,\Lambda(f),\a_0)\coloneqq\begin{cases}
		\argmax\limits_{\lambda_{\a}}&\lambda_{\a_0}\\
		\rm{s.t.}&\sum_{\a\in\Lambda(f)}  \lambda_{\a}\cdot\a=\b\\
		&\sum_{\a\in\Lambda(f)}\lambda_{\a}=1\\
		&\lambda_{\a}\ge0, \quad\forall\a\in\Lambda(f)
	\end{cases}
\end{equation*}
If $\b$ and $\a_0$ lie on the same face of $\New(f)$, then the output of $\text{SimSel}(\b,\Lambda(f),\a_0)$ corresponds to a trellis which contains $\a_0$ and covers $\b$. 

Suppose $f=\sum_{\a\in\Lambda(f)}c_{\a}\x^{\a}-\sum_{\b\in\Gamma(f)}d_{\b}\x^{\b}\in\R[\x]$ and assume that $\{\a\in\Lambda(f)\mid\a\notin\cup_{\b\in\Gamma(f)}\cup_{\TT\in\CC(\b)}\TT\}=\varnothing$.
We first compute a simplex cover $\{(\TT_k,\b_k)\}_{k=1}^l$ for $f$ by repeatedly running the program $\rm{SimSel}$ with appropriate $\b$ and $\a_0$. 
Then, for each $k$ let $M_{k}$ be a $\TT_{k}$-rational mediated set containing $\b_k$ and $s_k$ be the cardinality of $M_k\setminus\TT_{k}$. 
For each $\bu_i^k\in M_k\setminus\TT_{k}$, let us write  $\bu_i^k=\frac{1}{2}(\bv_i^k+\bw_i^k)$. 
Let $\mathbf{K}$ be the $3$-dimensional rotated second-order cone, i.e.,
\begin{equation}
\mathbf{K}\coloneqq \{(a,b,c)\in\R^3\mid2ab\ge c^2,a\ge0,b\ge0\}.
\end{equation}
Then we can approximate $f_{\min}$ from below with the following  second-order cone program:
\begin{equation}\label{eq:soncsocp}
\begin{cases}
\sup&b\\
\rm{s.t.}&f(\x)-b=\sum_{k=1}^l\sum_{i=1}^{s_k}(2a_i^k\x^{\bv_i^k}+b_i^k\x^{\bw_i^k}-2c_i^k\x^{\bu_i^k})\\
&(a_i^k,b_i^k,c_i^k)\in\mathbf{K},\quad\forall i,k
\end{cases}
\end{equation}

\section{First-order SDP solvers}\label{sec:firstsdp}
In the previous chapters, we have explained how to reduce the size of the \ac{moment-SOS} relaxations by
exploiting certain sparsity structures induced by the input polynomials. 
A complementary framework consists of exploiting the specific properties of the matrices involved in the moment \ac{SDP} relaxations, in order to speed-up their resolution via specific first-order algorithms.
Here, we prove that every moment relaxation of a \ac{POP} with a sphere or ball constraint can be reformulated as an \ac{SDP}  involving a \ac{PSD} matrix with \ac{CTP}. 
As a result, such moment relaxations can be solved efficiently by first-order methods that exploit \ac{CTP}, e.g., the conditional gradient-based augmented Lagrangian method.

First let us define \ac{CTP} for a \ac{POP}. 
Given $f,g_1,\dots,g_m, h_1, \dots, h_l\in\R[\x]$, let us consider the following \ac{POP} with $n$ variables, $m$ inequality constraints and $l$ equality constraints:
\begin{equation}\label{eq:POP.def.intro}
    f_{\min} = \min\,\{f(\mathbf x):g_j(\mathbf x)\ge 0,j\in[m],h_i(\mathbf x)= 0,i\in[l]\}.
\end{equation}
Let $r_{\min}\coloneqq\max\,\{\lceil \frac{\deg(f)}{2} \rceil, \lceil \frac{\deg(g_1)}{2} \rceil,\dots, \lceil \frac{\deg(g_m)}{2} \rceil, \lceil \frac{\deg(h_1)}{2} \rceil,\dots, \lceil \frac{\deg(h_l)}{2} \rceil\}$.
For each $r \geq r_{\min}$, recall that the moment relaxation associated to \ac{POP} \eqref{eq:POP.def.intro} is
\begin{equation}\label{eq:moment.hierarchy}
f^r=\inf\limits_{\y} \left\{ L_{\mathbf y}(f)\ \left|\begin{array}{rl}
& \mathbf M_r(\mathbf y)\succeq 0,y_{\mathbf{0}}=1\\
&\mathbf M_{r - \lceil \deg(g_j)/2 \rceil }(g_j\mathbf y)   \succeq 0,j\in[m]\\
&\mathbf M_{r - \lceil \deg(h_i)/2 \rceil }(h_i\mathbf y)= 0,i\in[l]
\end{array}
\right. \right\}.
\end{equation}
%
%
%
%
If we denote
$$\mathbf D_r(\mathbf y)\coloneqq \diag(\mathbf M_r(\mathbf y),\mathbf M_{r-\lceil \deg(g_1)/2\rceil}(g_1\mathbf y),\dots, \mathbf M_{r-\lceil \deg(g_m)/2\rceil}(g_m\mathbf y)),$$ 
then \ac{SDP} \eqref{eq:moment.hierarchy} can be rewritten as
\begin{equation}\label{eq:dual.diag.moment.mat}
f^r=\inf\limits_\y\left\{L_{\mathbf y}(f)\left|\begin{array}{rl}
&\mathbf D_r(\mathbf y)\succeq0,y_{\mathbf{0}}=1\\
&\mathbf M_{r-\lceil \deg(h_i) / 2\rceil}(h_i\mathbf y)=0,i\in[l]
\end{array}
\right.\right\}.
\end{equation}
\begin{definition}\label{def:ctp}
(\ac{CTP} for a \ac{POP})
We say that \ac{POP} \eqref{eq:POP.def.intro} has \ac{CTP} if for every $r \geq r_{\min}$, there exists $a_r>0$ and a positive definite matrix $\mathbf T_{r}$ such that
\begin{equation}
    \left.
\begin{array}{rl}
&\mathbf M_{r-\lceil \deg(h_i)/2\rceil}(h_i\mathbf y)=0,i\in[l]\\
&y_{\mathbf{0}}=1
\end{array}
\right\}\Rightarrow  \trace(\mathbf T_{r} \mathbf D_r(\mathbf y) \mathbf T_{r})=a_r.
\end{equation}
\end{definition}
In other words, we say that a \ac{POP} has \ac{CTP} if each moment relaxation \eqref{eq:dual.diag.moment.mat} has an equivalent form involving a \ac{PSD} matrix whose trace is constant. 
In this case, we call $a_r$ the constant trace and $\mathbf T_r$ the basis transformation matrix.
%
We illustrate this conversion to an \ac{SDP} with \ac{CTP}.
\begin{example}
Consider the following univariate \ac{POP}
\[-1=\inf\,\{x:1-x^2=0\}.\]
Then the second-order moment relaxation is
\[\begin{array}{rll}
f^2 = &\inf \limits_{\mathbf y} & y_1\\
&\rm{s.t.}
&\begin{bmatrix}
y_0 & y_1 & y_2\\
y_1 & y_2 & y_3\\
y_2 & y_3 & y_4
\end{bmatrix}\succeq 0,
\begin{bmatrix}
y_0-y_2 & y_1-y_3 \\
y_1-y_3 & y_2-y_4 
\end{bmatrix}= 0,
y_{0}=1.
\end{array}\]
It can be rewritten as
\[\begin{array}{rll}
f^2 = &\inf \limits_{\mathbf y} & y_1\\
&\rm{s.t.}
&\begin{bmatrix}
1 & y_1 & 1\\
y_1 & 1 & y_1\\
1 & y_1 & 1
\end{bmatrix}\succeq 0,\\
\end{array}\]
by removing equality constraints, in which the \ac{PSD} matrix has trace 3.
Alternatively, with $\D_2(\y)=\M_2(\y)$ and 
\[
\mathbf T_2=\begin{bmatrix}
1 & 0 & 0\\
0 & \sqrt{2} & 0\\
0 & 0 & 1
\end{bmatrix}, 
\quad 
\mathbf Y= \mathbf T_2 \D_2(\y) \mathbf T_2 = \mathbf T_2 \M_2(\y) \mathbf T_2,\]
we have
\[-f^2 = \sup_{\mathbf Y}\,\{ \left< \mathbf C,\mathbf Y\right>:\left< \mathbf A_i,\mathbf Y\right>=b_i,i\in [5], \mathbf Y \succeq 0\},\]
where $b_1=\dots=b_4=0$, $b_5=1$ and
\[\begin{array}{rl}
&\mathbf C=-\frac{\sqrt{2}}4\begin{bmatrix}
0 &1& 0\\
1 &0& 0\\ 
0& 0& 0
\end{bmatrix},
\mathbf A_1=\frac{\sqrt{2}}2\begin{bmatrix}
0& 0& 1\\
0& -1& 0\\
1& 0& 0
\end{bmatrix},
\mathbf A_2=\frac{1}2\begin{bmatrix}
2& 0& -1\\
0& 0& 0\\ 
-1& 0& 0
\end{bmatrix},\\
&\mathbf A_3=\frac{\sqrt{2}}4\begin{bmatrix}
0& 1& 0\\
1& 0& -1\\
0& -1& 0
\end{bmatrix},
\mathbf A_4=\frac{1}2\begin{bmatrix}
0& 0& 1\\
0& 0& 0\\
1& 0& -2
\end{bmatrix},
\mathbf A_5=\begin{bmatrix}
1& 0& 0\\
0& 0& 0\\
0& 0& 0
\end{bmatrix}.
\end{array}\]
We then obtain that 
\[\left< \mathbf A_i,\mathbf Y\right>=b_i,i\in [5]\Rightarrow\trace(\mathbf Y)=4.\]
\end{example}

For the minimization of a polynomial on the unit sphere, one can show that \ac{POP}~\eqref{eq:POP.def.intro} has \ac{CTP} 
with $a_r=2^r$ and $\mathbf T_{r}= \diag((\theta^{1/2}_{r,\alpha})_{\alpha\in\N^n_r})$,
where $(\theta_{r,\alpha})_{\alpha\in\N^n_r}\subseteq \R^{>0}$ satisfies $(1+\|\mathbf
x\|_2^2)^r=\sum_{\alpha\in\N^n_r}\theta_{r,\alpha}\mathbf x^{2\alpha}$, for all $r \geq 1$.

Next, we provide a sufficient condition for \ac{POP} \eqref{eq:POP.def.intro} to have \ac{CTP}.
For $r \geq r_{\min}$, let $\cM(\frakg)_r$ be the truncated quadratic module associated to $\frakg = \{g_1,\dots,g_m\}$. 
Let $\cM^\circ(\frakg)_r$ be the interior of the truncated quadratic module $\cM(\frakg)_r$, which is defined by
\begin{align*}
\cM^\circ(\frakg)_r\coloneqq \{\mathbf v_r^\intercal \mathbf G_0 \mathbf v_r+&\sum_{j\in[m]} g_j \mathbf v_{r-\lceil \deg(g_j)/2\rceil}^\intercal  \mathbf G_j \mathbf v_{r-\lceil \deg(g_j)/2\rceil}\\
&\mid\mathbf G_j\succ0, j\in\{0\}\cup[m]\}.
\end{align*}

\begin{theoremf}\label{theo:suff.cond.CTP}
If $1\in \cM^\circ(\frakg)_r$ for all $r \geq r_{\min}$, then \ac{POP} \eqref{eq:POP.def.intro} has \ac{CTP}.
This condition holds in particular when the set of constraints includes either ball or annulus constraints.
\end{theoremf}
Once we have the knowledge of the constant $a_r$, the $r$-th order moment relaxation can be cast as follows:
\begin{equation}\label{eq:moment.hierarchy.intro}
- f^r = \sup _{\mathbf Y} \{ \left< \mathbf C_r,\mathbf Y\right>:\mathcal{A}_r \mathbf Y=\mathbf b_r,\mathbf Y \succeq 0, \trace (\Y) = a_r \},
\end{equation}
where $\mathcal{A}_r$ is a linear operator (used here to encode affine constraints of the \ac{SDP}).
Afterwards, it turns out that \ac{SDP} \eqref{eq:moment.hierarchy.intro} is equivalent to minimizing the largest eigenvalue of a matrix pencil:
\begin{equation}\label{eq:nonsmooth.hierarchy.intro}
\begin{array}{r}
-f^r = \inf\limits_{\mathbf z}\,\{a_r\lambda_{\max}(\mathbf C_r-\mathcal{A}_r^\intercal \mathbf z)+\mathbf b_r^\intercal \mathbf z\},
\end{array}
\end{equation}
where $\mathcal A_k^\intercal $ denotes the adjoint operator of $\mathcal A_k$.
Hence \eqref{eq:nonsmooth.hierarchy.intro} forms what we call a hierarchy of (nonsmooth, convex) spectral relaxations.

To solve large-scale instances of this maximal eigenvalue minimization problem, a plethora of first-order methods are available, including subgradient descent or variants of the mirror-prox algorithm, spectral bundle methods, the conditional gradient based augmented Lagrangian (CGAL) algorithm, and variants relying on limited memory and arithmetic.


\section{Notes and sources}
The material presented in Chapter \ref{sec:sonc} has been published in \citemisc{wang2020second,magron2020sonc}. 
The set of \ac{SONC} polynomials was introduced by \citemisc{iw}.
The condition linking nonnegativity of a circuit polynomial with its circuit number, stated in Theorem \ref{theorem:CircuitPolynomialNonnegativity}, can be found in \citemisc[Theorem 3.8]{iw}. 
The interested reader can find more details on mediated sets in \citemisc{pow,re}.
Theorem \ref{th:mediated}, Proposition \ref{lemma:mediated}, Theorem \ref{th:ratmediated} correspond to Theorem 5.2 in \citemisc{iw}, Lemma 3.5 and Theorem 3.6 in \citemisc{wang2020second}, respectively.
The heuristics used to compute the simplex cover can be found, e.g., in \citemisc{se}.
Detailed numerical experiments comparing the usual geometric programming approach and the second-order cone programming approach stated in Chapter \ref{sec:sonc} can be found in \citemisc[\S~6]{wang2020second}.
The corresponding Julia package is available at \href{https://github.com/wangjie212/SONCSOCP}{github:{\tt SONCSOCP}}.
A recent study by \citemisc{papp} proposes a systematic method to compute an optimal simplex cover.

The fact that \ac{SONC} cones admit a second-order cone characterization was firstly stated in \citemisc[Theorem 17]{abe}, but the related proof does not provide an explicit construction.
Another recently introduced alternative certificates by \citemisc{ca} are sums of arithmetic geometric exponentials (SAGE), which can be obtained via relative entropy programming.
The connection between \ac{SONC} and SAGE polynomials has been recently studied in \citemisc{mu,wangsiaga,ka19}. 
It happens that \ac{SONC} polynomials and SAGE polynomials are actually equivalent, and that both have a cancellation-free representation in terms of generators.

The framework by \citemisc{dressler2020global} relies on the dual \ac{SONC} cone to compute lower bounds of polynomials by means of linear programming instead of geometric programming. 
Note that there are no general guarantees that the bounds obtained with this framework are always more or less accurate than the approach based on geometric programming from \citemisc{se}, and the same holds for performance.
One of the similar features shared by \ac{SOS}/\ac{SONC}-based frameworks is their intrinsic connections with conic programming: \ac{SOS} decompositions are computed via semidefinite programming and \ac{SONC} decompositions via geometric programming or second-order cone programming. 
In both cases, the resulting optimization problems are solved with interior-point algorithms, thus output approximate nonnegativity certificates. 
However, one can still obtain an exact certificate from such output via hybrid numerical-symbolic algorithms when the input polynomial lies in the interior of the \ac{SOS}/\ac{SONC} cone.
One way is to rely on rounding-projection algorithms adapted to the \ac{SOS} cone by \citemisc{pe} and the \ac{SONC} cone by \citemisc{ma19}, or alternatively on perturbation-compensation schemes as in \citemisc{univsos,multivsos18,csos,gradsos,magron2020sonc}. 

The material presented in Chapter \ref{sec:firstsdp} is issued from  \citemisc{mai2020hierarchy,mai2020exploiting}.
Theorem \ref{theo:suff.cond.CTP} is proved in  \citemisc[\S~3]{mai2020exploiting}.
The equivalence between \ac{SDP} \eqref{eq:moment.hierarchy.intro} and \eqref{eq:nonsmooth.hierarchy.intro} follows from the framework by Helmberg and Rendl \citemisc{helmberg2000spectral}.
The \ac{CTP} has been also studied in the context of eigenvalue optimization; see \citemisc{mai2021constant}.

CGAL is a first-order method that exploits \ac{CTP}. 
In \citemisc{yurtsever2021scalable}, the authors combined CGAL with the Nystr\"om sketch (named SketchyCGAL), which requires dramatically less storage than other methods and is very efficient for solving the first-order relaxation of large-scale Max-Cut instances.

When solving the second and higher order relaxations, \ac{SDP} solvers often encounter the following issues:
\begin{itemize}
\item\textbf{Storage}: {Interior-point methods} are often chosen by users because of their highly accurate output. 
These methods are efficient for solving medium-scale \ac{SDP}s.
However they frequently fail due to lack of memory when solving large-scale \ac{SDP}s. 
{First-order methods} (e.g., ADMM, SBM, CGAL) provide alternatives to interior-point methods to avoid the memory issue.
This is due to the fact that the cost per iteration of first-order methods is much {lower} than that of {interior-point methods}. 
At the price of losing convexity one can also rely on heuristic methods and  replace the full matrix $\mathbf Y$ in the moment relaxation by a simpler one, in order to save memory. 
For instance, the Burer-Monteiro method \citemisc{burer2005local} considers a low rank factorization of $\mathbf Y$. However, to get correct results the rank cannot be too low \citemisc{waldspurger2020rank} and therefore this limitation makes it useless for the second and higher order relaxations of \ac{POP}s.
Not suffering from such a limitation, CGAL not only maintains the convexity of the moment relaxation but also possibly runs with an implicit matrix $\Y$; see Remarks A.12 and A.17 in \citemisc{mai2020exploiting}.
\item \textbf{Accuracy}: Nevertheless, first-order methods have low convergence rates compared to interior-point methods.
Their performance depends heavily on the problem scaling and conditioning.
As a result, when solving large-scale \ac{SDP}s with first-order methods it is often difficult to obtain numerical results with high accuracy.         
\end{itemize}
Extensive numerical comparisons between some of these methods have been performed in \S~4 and Appendix A.3 from \citemisc{mai2020exploiting}.

We close this chapter by emphasizing that there are other issues to be addressed to improve the scalability of polynomial optimization methods.
A first complementary research track is to overcome the issue that \ac{SDP} relaxations arising from the \ac{moment-SOS} hierarchy can often be ill-conditioned. 
Possible remedies include the use of other bases of vector spaces of polynomials, for instance Chebyshev polynomials instead of the standard monomial basis as done in the univariate case in \citemisc[\S~5]{henrion2012semidefinite}, and encoding polynomial identities by evaluating them on suitably (perhaps randomly) chosen points, seeing \citemisc{lofberg2004coefficients,lasserre2017bounded} for preliminary attempts. 

Another complementary research track is to exploit symmetries in polynomial optimization. 
For instance the  invariance of all input polynomials under the action of a subgroup of the general linear group has been studied in \citemisc{riener2013exploiting}.
Previous works focused on exploiting the knowledge of the group action at the \ac{SDP} level \citemisc{gatermann2004symmetry}.
These frameworks have been applied to compute correlation bounds for finite-dimensional quantum systems \citemisc{tavakoli2019enabling} and bounds of packing problems  \citemisc{de2015semidefinite,dostert2017new,dostert2021exact}. 
Exploiting symmetries has also been investigated for polynomial optimization based on sums of arithmetic-geometric-exponentials in \citemisc{moustrou2021symmetry}.
It would be definitely worth extending these frameworks to other (non-)discrete groups, dynamical systems and to noncommutative polynomial optimization in the future.
This research direction of symmetry exploitation is still to be pursued and shall hopefully lead to publishing another complementary book in the upcoming years!
%
\input{misc.bbl}

%% file: misc.bbl
\providecommand{\etalchar}[1]{$^{#1}$}

%% file: appendix.tex
\appendix
\part{Appendix: software libraries}
\chapter{Programming with MATLAB}\label{chap:matlab}

\section{Sparse moment relaxations with $\gloptipoly$}\label{sec:gloptipoly}

$\gloptipoly$ \citeapp{gloptipoly} is a MATLAB library designed to solve the generalized problem of moments. 
We can rely on this library to solve the (primal) moment relaxation arising after exploiting \ac{CS} for a given \ac{POP}. 
Next we give two scripts to obtain the dense and sparse bounds given in Example~\ref{ex:sparsepop}.
In this example, we approximate the minimum of $f = f_1 + f_2 + f_3$ on the basic compact semialgebraic set $\X$ with
\begin{align*}
f_1 &= - x_1 x_4,\\
f_2 &= -x_1^2 + x_1 x_2 + x_1 x_3 - x_2 x_3 + x_2 x_5,\\
f_3 &= - x_5 x_6 + x_1 x_5 + x_1 x_6 + x_3 x_6,
\end{align*} 
and 
\[\X = \{\x \in \R^n \mid (6.36 - x_1) (x_1 - 4) \geq 0, \dots, (6.36 - x_6) (x_6 - 4) \geq 0\}.\]
By exploiting \ac{CS}, the set of variables is decomposed as $I_1 = \{1, 4\}$, $I_2 = \{1, 2, 3, 5\}$, $I_3 = \{1, 3, 5, 6\}$.
The first script below allows one to retrieve the dense bound $f^2$ of the dense primal moment relaxation \eqref{primalj} at the second relaxation order $r=2$. 
Our problem involves six polynomial variables and a measure $\mu$ (defined in Line 3) supported on $\X$ (defined in Lines 8-13).
At Line 15, $\gloptipoly$ calls the \ac{SDP} solver SeDuMi \citeapp{sedumi} and returns the objective value $f^2 = 20.8608$. 
In addition, $\gloptipoly$ is able to extract a minimizer of $f$ on $\X$ through Algorithm \ref{alg:extract}, which then certifies that $f_{\min} = f^2$. 
\begin{lstlisting}
r = 2 % relaxation order
mpol x1 x2 x3 x4 x5 x6
mu = meas([x1 x2 x3 x4 x5 x6]);
f1 = mom(-x1*x4);
f2 = mom(-x1^2 + x1*x2 + x1* x3 - x2*x3 + x2*x5);
f3 = mom(-x5*x6 + x1*x5 + x1*x6 + x3*x6);
f = f1 + f2 + f3; % objective function
X = [(6.36 - x1)*(x1-4)>=0,...
(6.36 - x2)*(x2-4)>=0,...
(6.36 - x3)*(x3-4)>=0,...
(6.36 - x4)*(x4-4)>=0,...
(6.36 - x5)*(x5-4)>=0,...
(6.36 - x6)*(x6-4)>=0]; % support
Pdense = msdp(min(f), X, mass(mu)==1, r);
[stat,obj] = msol(Pdense);
\end{lstlisting}
For more details on the six $\gloptipoly$ commands (\texttt{mpol}, \texttt{meas}, \texttt{mom}, \texttt{msdp}, \texttt{mass}, \texttt{msol}), we refer the interested reader to the online tutorial:
\begin{center}
 \href{https://homepages.laas.fr/henrion/papers/gloptipoly3.pdf}{https://homepages.laas.fr/henrion/papers/gloptipoly3.pdf}.
\end{center}
The second script below allows one to retrieve the sparse bound $f_{\cs}^2$ corresponding to the second relaxation order $r=2$.
By contrast with the dense case, one needs to define three measures $\mu_1, \mu_2, \mu_3$ (Lines 6-8), associated to $I_1$, $I_2$ and $I_3$, respectively, and ensure that their marginals satisfy the equality constraints given in \eqref{eq:csmultimom}. 
For instance, the marginals of $\mu_1$ and $\mu_2$ must have equal moments for monomials in $x_1$ since $I_{12} = I_1 \cap I_2 = \{1\}$ (see Lines 14-15).
At Line 34, $\gloptipoly$ returns the objective value $f_{\cs}^2 = f^2 = 20.8608$ and is again able to  extract a minimizer of $f$ on $\X$.
\begin{lstlisting}
r = 2 % relaxation order
mpol x1 3
mpol x2 x4 x6
mpol x3 2
mpol x5 2
mu(1) = meas([x1(1) x4]); % first measure on I1 = {1, 4}
mu(2) = meas([x1(2) x2 x3(1) x5(1)]); % second measure on I2 = {1, 2, 3, 5}
mu(3) = meas([x1(3) x3(2) x5(2) x6]); % third measure on I3 = {1, 3, 5, 6}
f1 = mom(-x1(1)*x4) ;
f2 = mom(-x1(2)^2 + x1(2)*x2(1) + x1(2)* x3(1) - x2(1)*x3(1) + x2(1)*x5(1));
f3 = mom(-x5(2)*x6 + x1(3)*x5(2) + x1(3)*x6 + x3(2)*x6);
f = f1 + f2 + f3; % objective function

m12_1 = mom(mmon(x1(1),2*r)); % moments of the marginal of mu1 on monomials in x1, corresponding to I12 = {1}
m12_2 = mom(mmon(x1(2),2*r)); % moments of the marginal of mu2 on monomials in x1, corresponding to I12 = {1}

m23_2 = mom(mmon([x1(2) x3(1) x5(1)],2*r)); % moments of the marginal of mu2 on monomials in x1, x3, x5, corresponding to I23 = {1, 3, 5}
m23_3 = mom(mmon([x1(3) x3(2) x5(2)],2*r)); % moments of the marginal of mu3 on monomials in x1, x3, x5, corresponding to I23 = {1, 3, 5}

m13_1 = mom(mmon(x1(1),2*r)); % moments of the marginal of mu1 on monomials in x1, corresponding to I13 = {1} 
m13_3 = mom(mmon(x1(3),2*r)); % moments of the marginal of mu3 on monomials in x1, corresponding to I13 = {1}

K = [(6.36 - x1(1))*(x1(1)-4)>=0,...
(6.36 - x1(2))*(x1(2)-4)>=0,...
(6.36 - x1(3))*(x1(3)-4)>=0,...
(6.36 - x2)*(x2-4)>=0,...
(6.36 - x4)*(x4-4)>=0,...
(6.36 - x6)*(x6-4)>=0,...
(6.36 - x3(1))*(x3(1)-4)>=0,... 
(6.36 - x3(2))*(x3(2)-4)>=0,...
(6.36 - x5(1))*(x5(1)-4)>=0,...
(6.36 - x5(2))*(x5(2)-4)>=0]; % supports
Psparse = msdp(min(f),m12_1-m12_2==0,m23_2-m23_3==0, m13_1-m13_3==0, K,mass(mu)==1);
[stat,obj] = msol(Psparse);
\end{lstlisting}
Other similar scripts for approximating minima of rational functions are given in \citeapp[\S~7]{bugarin2016minimizing}.

\section{Sparse SOS relaxations with Yalmip}
\label{sec:yalmip}

Yalmip \citeapp{yalmip} is a MATLAB toolbox to model optimization problems and to solve them using external solvers.
In particular, we can rely on this toolbox to solve the \ac{SOS} program \eqref{eq:cssos}, which is the dual of the moment relaxation \eqref{eq:csmom}  encoded in the prior section. 
Next we give two scripts to obtain the dense and sparse \ac{SOS} bounds given in Example~\ref{ex:sparsepop}.

The first script below allows to retrieve the dense bound, corresponding to \ac{SDP} \eqref{dualjsos} at the second relaxation order $r=2$. 
For each of the six box constraints, we need to define an \ac{SOS} multiplier \texttt{si} (see Line 11), corresponding to $\sigma_i$ in \eqref{dualjsos}. 
Note that the first argument \texttt{si} refers to the \ac{SOS} polynomial multiplier itself while the second argument \texttt{ci} refers to its (unknown) vector of coefficients. 
Here, we want to maximize the lower bound $b$ of $f$ such that $f-b$ has a Putinar's representation of degree $2 r = 4$, and thus the degree of each \ac{SOS} multiplier is $2 r -2 = 2$. 
As with $\gloptipoly$, Yalmip calls at Line 15 the same \ac{SDP} solver SeDuMi and returns the objective value $f^2 = 20.8608$. 
\begin{lstlisting}
r = 2; % relaxation order
sdpvar x1 x2 x3 x4 x5 x6 b;
x = [x1; x2; x3; x4; x5; x6];
f1 = -x1*x4;
f2 = -x1^2 + x1*x2 + x1* x3 - x2*x3 + x2*x5;
f3 = -x5*x6 + x1*x5 + x1*x6 + x3*x6;
f = f1 + f2 + f3; % objective function
g = (6.36 - x).*(x - 4); % the 6 polynomials used to define the box constraints
s = []; c = [b]; F = [];
for i=1:6
  [si,ci] = polynomial(x,2*r-2); % SOS multiplier si of degree 2 in the dense Putinar's representation
  s = [s si]; c = [c; ci]; F = [F, sos(si)];
end
F = [F, sos(f - b - s*g)]; % SOS multiplier s0 = f - b - sum_i si gi of degree 4
solvesos(F,-b,[],c); % solves the SOS program: maximize b such that f - b = s0 + sum_i si gi
\end{lstlisting}
For more details on the four Yalmip commands (\texttt{sdpvar}, \texttt{polynomial}, \texttt{sos}, \texttt{solvesos}), we refer the interested reader to the online tutorial
\href{https://yalmip.github.io/tutorial/sumofsquaresprogramming}{github:yalmip}.

The second script below allows one to retrieve the sparse bound $f_{\cs}^2$ corresponding to the second relaxation order $r=2$. 
By contrast with the dense case, one needs to define three \ac{SOS} multipliers \texttt{s01}, \texttt{s02}, \texttt{s03} (Lines 20-22), corresponding to $\sigma_{01}, \sigma_{02}, \sigma_{03}$ in \eqref{dualjsos}, associated to $I_1$, $I_2$ and $I_3$, respectively.

As in the dense case, one needs a single \ac{SOS} multiplier for each constraint but the support is more restricted.
For instance, the \ac{SOS} multiplier associated to the polynomial $(6.36 - x_1) (4 - x_1)$ (defined in Line 13) depends only on the variables $x_1, x_4$, related to the index subset $I_1 = \{1,4\}$.
At Line 28, Yalmip returns the objective value $f_{\cs}^2 = f^2 = 20.8608$.
\begin{lstlisting}
r = 2; % relaxation order
sdpvar x1 x2 x3 x4 x5 x6 b;
x = [x1; x2; x3; x4; x5; x6];
xI1 = [x1; x4];
xI2 = [x1; x2; x3; x5];
xI3 = [x1; x3; x5; x6];
f1 = -x1*x4;
f2 = -x1^2 + x1*x2 + x1* x3 - x2*x3 + x2*x5;
f3 = -x5*x6 + x1*x5 + x1*x6 + x3*x6;
f = f1 + f2 + f3; % objective function
g = (6.36 - x).*(x - 4); % the 6 polynomials used to define the box constraints

[s1,c1] = polynomial(xI1,2*r-2); % SOS multiplier of degree 2 in the sparse Putinar's representation, depending only on x1, x4
[s2,c2] = polynomial(xI2,2*r-2); % SOS multiplier of degree 2 in the sparse Putinar's representation, depending only on x1, x2, x3, x5
[s3,c3] = polynomial(xI2,2*r-2); % SOS multiplier of degree 2 in the sparse Putinar's representation, depending only on x1, x2, x3, x5
[s4,c4] = polynomial(xI1,2*r-2); % SOS multiplier of degree 2 in the sparse Putinar's representation, depending only on x1, x4
[s5,c5] = polynomial(xI2,2*r-2); % SOS multiplier of degree 2 in the sparse Putinar's representation, depending only on x1, x2, x3, x5
[s6,c6] = polynomial(xI3,2*r-2); % SOS multiplier of degree 2 in the sparse Putinar's representation, depending only on x1, x3, x5, x6

[s01, c01] = polynomial(xI1,2*r); % SOS multiplier of degree 4 in the sparse Putinar's representation, depending only on x1, x4
[s02, c02] = polynomial(xI2,2*r); % SOS multiplier of degree 4 in the sparse Putinar's representation, depending only on x1, x2, x3, x5
[s03, c03] = polynomial(xI3,2*r); % SOS multiplier of degree 4 in the sparse Putinar's representation, depending only on x1, x3, x5, x6

s = [s1 s2 s3 s4 s5 s6]; c = [b; c1; c2; c3; c4; c5; c6; c01; c02; c03];
F = [sos(s1), sos(s2), sos(s3), sos(s4), sos(s5), sos(s6), sos(s01), sos(s02), sos(s03)];

F = [F, coefficients(f - b - s01 - s02 - s03 - s*g,x) == 0]; 
solvesos(F,-b,[],c); % solves the SOS program: maximize b such that f - b = s01 + s02 + s03 + sum_i si gi
\end{lstlisting}

\chapter{Programming with Julia}\label{chap:julia}

The goal of this chapter is to present our Julia library, called \texttt{TSSOS}, which aims at helping polynomial optimizers to solve large-scale problems with sparse input data. 
The underlying algorithmic framework is based on exploiting \ac{TS} (see Chapter \ref{chap:tssos}) as well as the combination of \ac{CS} and \ac{TS} (see Chapter \ref{chap:cstssos}). 
As emphasized in the different chapters of Part \ref{part:ts}, \ac{TS} can be applied to numerous problems ranging from power networks to eigenvalue optimization of noncommutative polynomials, involving up to tens of thousands of variables and constraints.
A complete documentation of the \texttt{TSSOS} library is available at  \href{https://github.com/wangjie212/TSSOS}{github:{\tt TSSOS}}.
\texttt{TSSOS} depends on the following Julia packages:
\begin{itemize}
    \item {\tt MultivariatePolynomials} to manipulate multivariate polynomials;
    \item {\tt JuMP} \citeapp{jump} to model the \ac{SDP} problem;
    \item {\tt Graphs} \citeapp{Graphs2021} to handle graphs;
    \item {\tt MetaGraphs} to handle weighted graphs;
    \item {\tt ChordalGraph} \citeapp{ChordalGraph} to generate approximately smallest chordal extensions;
    \item {\tt SemialgebraicSets} to compute Gr\"obner bases. 
\end{itemize}
Besides, {\tt TSSOS} requires an \ac{SDP} solver, which can be {\tt MOSEK} \citeapp{mosek} or {\tt COSMO} \citeapp{cosmo}. 
Once one of the \ac{SDP} solvers has been installed, the installation of {\tt TSSOS} is straightforward:

\begin{minted}[breaklines,escapeinside=||,mathescape=true, numbersep=3pt, gobble=0, frame=lines, fontsize=\small, framesep=2mm]{csharp}
Pkg.add("https://github.com/wangjie212/TSSOS")
\end{minted}

\section{{\tt TSSOS} for polynomial optimization}
\label{sec:tssos_pop}
\texttt{TSSOS} provides an easy way to define a \ac{POP} and to solve it by sparsity-adapted \ac{SDP} relaxations, including the relaxation $\P^{r,s}_{\ts}$ given in \eqref{chap7:seq1} exploiting \ac{TS}, the relaxation $\P^{r}_{\cs}$ given in \eqref{eq:csmom} exploiting \ac{CS}, as well as the relaxation $\P^{r,s}_{\csts}$ given in \eqref{eq:cstsmom} exploiting both \ac{CS} and \ac{TS}.
The tunable parameters (e.g., the relaxation order $r$, the sparse order $s$, the types of chordal extensions) allow the user to find the best compromise between the computational cost and the solution accuracy.
The following Julia script is a simple example to illustrate the usage of \texttt{TSSOS}.

\begin{minted}[breaklines,escapeinside=||,mathescape=true, linenos, numbersep=3pt, gobble=0, frame=lines, fontsize=\small, framesep=2mm]{csharp}
using TSSOS
using DynamicPolynomials
r = 2 /* relaxation order */
@polyvar x[1:6]
f = x[1]^4 + x[2]^4 - 2x[1]^2*x[2] - 2x[1] + 2x[2]*x[3] - 2x[1]^2*x[3] - 2x[2]^2*x[3] - 2x[2]^2*x[4] - 2x[2] + 2x[1]^2 + 2.5x[1]*x[2] - 2x[4] + 2x[1]*x[4] + 3x[2]^2 + 2x[2]*x[5] + 2x[3]^2 + 2x[3]*x[4] + 2x[4]^2 + x[5]^2 - 2x[5] + 2 /* objective function */
g = 1 - sum(x[1:2].^2) /* inequality constraints */
h = 1 - sum(x[3:5].^2) /* equality constraints */
nh = 1 /* number of equality constraints */
\end{minted}
To solve the first step of the TSSOS hierarchy with approximately smallest chordal extensions (option {\tt TS="MD"}), run
\begin{minted}[mathescape,               
               numbersep=5pt,
               gobble=0,
               frame=lines,
               framesep=2mm]{csharp}
opt,sol,data = tssos_first([f;g;h], x, r, numeq=nh, TS="MD")
\end{minted}
We obtain $f^{2,1}_{\ts}= 0.2096$.\\
To solve higher steps of the TSSOS hierarchy, repeatedly run
\begin{minted}[mathescape,               
               numbersep=5pt,
               gobble=0,
               frame=lines,
               framesep=2mm]{csharp}
opt,sol,data = tssos_higher!(data, TS="MD")
\end{minted}
For instance, at the second step of the TSSOS hierarchy we obtain $f^{2,2}_{\ts}=  0.2123$.\\
To solve the first step of the CS-TSSOS hierarchy, run
\begin{minted}[mathescape,               
               numbersep=5pt,
               gobble=0,
               frame=lines,
               framesep=2mm]{csharp}
opt,sol,data = cs_tssos_first([f;g;h], x, r, numeq=nh, TS="MD")
\end{minted}
to obtain the lower bound $f^{2,1}_{\csts}=  0.2092$. \\
To solve higher steps of the CS-TSSOS hierarchy, repeatedly run
\begin{minted}[mathescape,               
               numbersep=5pt,
               gobble=0,
               frame=lines,
               framesep=2mm]{csharp}
opt,sol,data = cs_tssos_higher!(data, TS="MD")
\end{minted}
For instance, at the second step of the CS-TSSOS hierarchy we obtain the lower bound $f^{2,2}_{\csts}=  0.2097$.\\
\texttt{TSSOS} also employs other techniques to gain more speed-up.
\subsection*{Binary variables}
{\tt TSSOS} supports binary variables. By setting ${\tt nb}=a$, one can specify that the first $a$ variables (i.e., $x_1,\dots,x_a$) are binary variables which satisfy the equations $x_i^2=1$, $i \in [a]$. The specification is helpful to reduce the number of decision variables of \ac{SDP} relaxations since one can identify $x^j$ with $x^{j\,(\textrm{mod }2)}$ for a binary variable $x$.

\subsection*{Equality constraints}
If there are equality constraints in the description of POP \eqref{eq:pop}, then one can reduce the number of decision variables of \ac{SDP} relaxations by working in the quotient ring $\R[\x]/(h_1,\ldots,h_t)$, where $\{h_1=0,\ldots,h_t=0\}$ is the set of equality constraints. To conduct the elimination, we need to compute a Gr\"obner basis $GB$ of the ideal $(h_1,\ldots,h_t)$. Then any monomial $\x^{\alpha}$ can be replaced by its normal form $\NF(\x^{\alpha}, GB)$ with respect to the Gr\"obner basis $GB$ when constructing \ac{SDP} relaxations. 

\subsection*{Adding extra first-order moment matrices}
When POP \eqref{eq:pop} is a \ac{QCQP}, the first-order \ac{moment-SOS} relaxation is also known as Shor's relaxation. 
In this case, $\P^1$, $\P^1_{\cs}$ and $\P^{1,1}_{\ts}$ yield the same optimum. 
To ensure that any higher order \ac{CS-TSSOS} relaxation (i.e., $\P^{r,s}_{\csts}$ with $r>1$) provides a tighter lower bound compared to the one given by Shor's relaxation, we may add an extra first-order moment matrix for each variable clique in $\P^{r,s}_{\csts}$. 
In {\tt TSSOS}, this is accomplished by setting ${\tt MomentOne=true}$.

\subsection*{Chordal extensions} 
For \ac{TS}, {\tt TSSOS} supports two types of chordal extensions: the maximal chordal extension (option {\tt TS="block"}) and approximately smallest chordal extensions. 
{\tt TSSOS} generates approximately smallest chordal extensions through two heuristics: the Minimum Degree heuristic (option {\tt TS="MD"}) and the Minimum Fillin heuristic (option {\tt TS="MF"}).
To use relaxations without \ac{TS}, set {\tt TS = false}.
Similarly for \ac{CS} there are options for the field {\tt CS} which can be {\tt "MF"} by default, or {\tt "MD"}, or {\tt "NC"} (without chordal extension), or {\tt false} (without \ac{CS}).

See \citeapp{treewidth} for a full description of the two chordal extension heuristics. The Minimum Degree heuristic is slightly faster in practice, whereas the Minimum Fillin heuristic yields on average slightly smaller clique numbers. Hence for \ac{CS}, the Minimum Fillin heuristic is recommended and for \ac{TS}, the Minimum Degree heuristic is recommended.

\subsection*{Merging PSD blocks}
In case that two \ac{PSD} blocks have a large portion of overlaps, it might be beneficial to merge these two blocks into a single block for efficiency. See Figure~\ref{fig:merge} for such an example. {\tt TSSOS} supports \ac{PSD} block merging inspired by the strategy proposed in \citeapp{cosmo}. To activate the merging process, one just needs to set the option ${\tt Merge=True}$. The parameter ${\tt md}$ ($=3$ by default) can be used to tune the merging strength.

\begin{figure}
\begin{equation*}
    \begin{bmatrix}
    \begin{array}{ccccc}
    \bullet&\bullet&\bullet&\bullet&\\
    \bullet&\bullet&\bullet&\bullet&\bullet\\
    \bullet&\bullet&\bullet&\bullet&\bullet\\
    \bullet&\bullet&\bullet&\bullet&\bullet\\
    &\bullet&\bullet&\bullet&\bullet\\
    \end{array}
    \end{bmatrix}
    \quad\longrightarrow\quad
    \begin{bmatrix}
    \begin{array}{ccccc}
    \bullet&\bullet&\bullet&\bullet&\bullet\\
    \bullet&\bullet&\bullet&\bullet&\bullet\\
    \bullet&\bullet&\bullet&\bullet&\bullet\\
    \bullet&\bullet&\bullet&\bullet&\bullet\\
    \bullet&\bullet&\bullet&\bullet&\bullet\\
    \end{array}
    \end{bmatrix}
\end{equation*}
\caption{Merge two $4\times4$ blocks into a single $5\times5$ block.}\label{fig:merge}
\end{figure}

\subsection*{Representing polynomials with supports and coefficients}
The Julia package {\tt DynamicPolynomials} provides an efficient way to define polynomials symbolically. But for large-scale polynomial optimization (say, $n>500$), it is more efficient to represent polynomials by their supports and coefficients. For instance, we can represent   $f=x_1^4+x_2^4+x_3^4+x_1x_2x_3$  in terms of its support and coefficients as follows:

\begin{minted}[mathescape,               
               numbersep=5pt,
               gobble=0,
               frame=lines,
               framesep=2mm]{csharp}
supp = [[1; 1; 1; 1], [2; 2; 2 ;2], [3; 3; 3; 3], [1; 2; 3]] 
/* support array of f */
coe = [1; 1; 1; 1] /* coefficient vector of f */
\end{minted}
The above representation of polynomials is natively supported by {\tt TSSOS}. Hence the user can define the polynomial optimization problem directly by the support data and the coefficient data to speed up the modeling process.\\
\subsection*{SOS + sparse + RIP $\nRightarrow$ sparse SOS}
To finish this section, we provide a {\tt TSSOS} script showing that the relaxations based on \ac{CS} can be strictly more conservative than the dense ones, even when the \ac{RIP} holds.
Let $f_1 = x_1^4 + (x_1 x_2 - 1)^2$, $f_2 = x_2^2 x_3^2 + (x_3^2 - 1)^2$ and $f = f_1+ f_2$. 
Here the \ac{RIP} trivially holds as we have only two subsets of variables. 
After running the following commands:
\begin{minted}[breaklines,escapeinside=||,mathescape=true, linenos, numbersep=3pt, gobble=0, frame=lines, fontsize=\small, framesep=2mm]{csharp}
using TSSOS, DynamicPolynomials
@polyvar x[1:3]
f1 = x[1]^4 + (x[1]*x[2] - 1)ˆ2 
f2 = x[2]^2*x[3]^2 + (x[3]^2 - 1)^2
f = f1 + f2
opt,sol,data = cs_tssos_first([f], x, 2, CS=false, TS=false)
opt,sol,data = cs_tssos_first([f], x, 2, TS=false)
\end{minted}
we obtain $f^2_{\cs} = 0.0005 < 0.8498 = f^2$.
One can also compute lower bounds based on \ac{TS} as follows:
\begin{minted}[breaklines,escapeinside=||,mathescape=true, linenos, numbersep=3pt, gobble=0, frame=lines, fontsize=\small, framesep=2mm]{csharp}
opt,sol,data = tssos_first([f], x, 2, TS="block")
opt,sol,data = tssos_higher!(data, TS="block")
\end{minted}
to obtain $f^{2,1}_{\ts} = 0.0004 < 0.8498 = f^{2,2}_{\ts} = f^2$.
\section{{\tt TSSOS} for noncommutative optimization}\label{sec:tssos_nc}
As seen previously in Chapters  \ref{chap:ncsparse} and \ref{chap:ncts}, the whole framework of exploiting \ac{CS} and \ac{TS} for (commutative) polynomial optimization can be extended to handle noncommutative polynomial optimization (including eigenvalue and trace optimization), which leads to the submodule {\tt NCTSSOS} in {\tt TSSOS}, available at  \href{https://github.com/wangjie212/NCTSSOS}{github:{\tt NCTSSOS}}.
The corresponding commands are similar to {\tt TSSOS}. 
To illustrate the use of {\tt NCTSSOS}, we consider the eigenvalue minimization problem given in Example \ref{ex:nosparseEigBall}.
After running the following commands:
\begin{minted}[breaklines,escapeinside=||,mathescape=true, linenos, numbersep=3pt, gobble=0, frame=lines, fontsize=\small, framesep=2mm]{csharp}
@ncpolyvar x1 x2 x3 x4
x = [x1; x2; x3; x4]
f1 = 4 - x1 + 3 * x2 - 3 * x3 - 3 * x1^2 - 7 * x1 * x2 + 6 * x1 * x3 - x2 * x1 -5 * x3 * x1 + 5 * x3 * x2 - 5 * x1^3 - 3 * x1^2 * x3 + 4 * x1 * x2 * x1 - 6 * x1 * x2 * x3 + 7 * x1 * x3 * x1 + 2 * x1 * x3 * x2 - x1 * x3^2 - x2 * x1^2 + 3 * x2 * x1 * x2 - x2 * x1 * x3 - 2 * x2^3 - 5 * x2^2 * x3 - 4 * x2 * x3^2 - 5 * x3 * x1^2 + 7 * x3 * x1 * x2 + 6 * x3 * x2 * x1 - 4 * x3 * x2 * x2 - x3^2 * x1 - 2 * x3^2 * x2 + 7 * x3^3
f2 = -1 + 6 * x2 + 5 * x3 + 3 * x4 - 5 * x2^2 + 2 * x2 * x3 + 4 * x2 * x4 - 4 * x3 * x2 + x3^2 - x3 * x4 + x4 * x2 - x4 * x3 + 2 * x4^2 - 7 * x2^3 + 4 * x2 * x3^2 + 5 * x2 * x3 * x4 - 7 * x2 * x4 * x3 - 7 * x2 * x4^2 + x3 * x2^2 + 6 * x3 * x2 * x3 - 6 * x3 * x2 * x4 - 3 * x3^2 * x2 - 7 * x3^2 * x4 + 6 * x3 * x4 * x2  - 3 * x3 * x4 * x3 - 7 * x3 * x4^2 + 3 * x4 * x2^2 - 7 * x4 * x2 * x3 - x4 * x2 * x4 - 5 * x4 * x3^2 + 7 * x4 * x3 * x4 + 6 * x4^2 * x2 - 4 * x4^3
f = f1 + f2
ncball = [1 - x1^2 - x2^2 - x3^2; 1 - x2^2 - x3^2 - x4^2] 
cs_nctssos_first([f; ncball], x, 2, CS=false, TS=false, obj="eigen")
cs_nctssos_first([f; ncball], x, 2, TS=false, obj="eigen")
cs_nctssos_first([f; ncball], x, 3, TS=false, obj="eigen")
\end{minted}
we obtain $
\lambda^2_{\sparse} (f, \frakg ) \simeq
-13.7680 <  -13.7333 \simeq \lambda^3_{\sparse} (f,\frakg)
\simeq \lambda^2 (f,\frakg) = \lambda_{\min} (f,\frakg)$. \\
One can also compute lower bounds based on \ac{TS} as follows:
\begin{minted}[breaklines,escapeinside=||,mathescape=true, linenos, numbersep=3pt, gobble=0, frame=lines, fontsize=\small, framesep=2mm]{csharp}
opt,data = nctssos_first([f; ncball], x, 2, TS="MD", obj="eigen")
opt,data = nctssos_higher!(data, TS="MD")
\end{minted}
After repeating the last command twice, we obtain $\lambda_{\ts}^{2,1}(f, \frakg ) \simeq -14.0922$,  $\lambda_{\ts}^{2,2}(f, \frakg ) \simeq-13.7618$,  $\lambda_{\ts}^{2,3}(f, \frakg ) \simeq-13.7393$, $\lambda^{2,4}_{\ts}(f, \frakg ) = \lambda^2 (f,\frakg) = \lambda_{\min} (f,\frakg)\simeq-13.7333$.

{\tt NCTSSOS} can also handle optimization over trace polynomials. 
To illustrate this, let us consider the two  polynomial Bell inequalities given in Chapter \ref{sec:polbell}. 
For the first one, we use the set of noncommutative variables $(x_1,x_2,x_3,x_4) \coloneqq (a_1,a_2,b_1,b_2)$, and so the objective function is 
\begin{align*}
&\left(\Trace(a_1b_2+a_2b_1)\right)^2+
\left(\Trace(a_1b_1-a_2b_2)\right)^2 \\
 = & \left(\Trace(x_1x_4+x_2x_3)\right)^2+
\left(\Trace(x_1x_3-x_2x_4)\right)^2\\
 = & \left(\Trace(x_1x_4)\right)^2 
+ \left(\Trace(x_2x_3)\right)^2
+ 2 \Trace(x_1x_4) \Trace(x_2x_3) \\
& + \left(\Trace(x_1x_3)\right)^2
+ \left(\Trace(x_2x_4)\right)^2
+ 2 \Trace(x_1x_3) \Trace(x_2x_4).
\end{align*}
As seen earlier in the commutative setting, we represent trace polynomials by their supports and coefficients. 
Note that we minimize the opposite expression, so the resulting vector of coefficients is $[-1; -1; -2; -1; -1; 2]$ (see Line 4). 
The option {\tt constraint="unipotent"} (Line 5) allows us to encode the constraints $x_i^2=1$ for $i\in[4]$.
The next script provides an upper bound for the maximization problem given in \eqref{e:bell3} at relaxation order $r=2$ and sparse order $s=1$ with maximal chordal extensions.
\begin{minted}[breaklines,escapeinside=||,mathescape=true, linenos, numbersep=3pt, gobble=0, frame=lines, fontsize=\small, framesep=2mm]{csharp}
n = 4
r = 2
tr_supp = [[[1;4], [1;4]], [[2;3], [2;3]], [[1;4], [2;3]], [[1;3], [1;3]], [[2;4], [2;4]], [[1;3], [2;4]]]
coe = [-1; -1; -2; -1; -1; 2]
opt,data = ptraceopt_first(tr_supp, coe, n, r, TS="block", constraint="unipotent")
\end{minted}
Similarly, the next script provides an upper bound for the maximization problem \eqref{e:bell5} at relaxation order $r=2$ and sparse order $s=1$ with approximately smallest chordal extensions.
\begin{minted}[breaklines,escapeinside=||,mathescape=true, linenos, numbersep=3pt, gobble=0, frame=lines, fontsize=\small, framesep=2mm]{csharp}
n = 6
r = 2
tr_supp = [[[1;4]], [[1], [4]], [[1;5]], [[1], [5]], [[1;6]], [[1], [6]], [[2;4]], [[2], [4]], [[2;5]], [[2], [5]], [[2;6]], [[2], [6]], [[3;4]], [[3], [4]], [[3;5]], [[3], [5]]]
coe = [-1; 1; -1; 1; -1; 1; -1; 1; -1; 1; 1; -1; -1; 1; 1; -1]
opt,data = ptraceopt_first(tr_supp, coe, n, r, TS="MD", constraint="unipotent")
\end{minted}
To obtain bounds at higher sparse orders, one should run repeatedly
\begin{minted}[mathescape,               
               numbersep=5pt,
               gobble=0,
               frame=lines,
               framesep=2mm]{csharp}
opt,data = ptraceopt_higher!(data, TS="block")
\end{minted}

\section{{\tt TSSOS} for dynamical systems}\label{sec:tssosdyn}

{\tt SparseJSR} is an efficient tool to compute bounds on the \ac{JSR} of a set of matrices, based on the sparse \ac{SOS} relaxations provided in Chapter \ref{chap:sparsejsr}. 
To use it in Julia, run
\begin{minted}[mathescape,               
               numbersep=5pt,
               gobble=0,
               frame=lines,
               framesep=2mm]{csharp}
add https://github.com/wangjie212/SparseJSR
\end{minted}
The following simple example illustrates how to compute upper bounds:
\begin{minted}[breaklines,escapeinside=||,mathescape=true, linenos, numbersep=3pt, gobble=0, frame=lines, fontsize=\small, framesep=2mm]{csharp}
A = [[1 -1 0; -0.5 1 0; 1 1 0], [0.5 1 0; -1 1 0; -1 -0.5 0]]
r = 2 /* relaxation order */
ub = SparseJSR(A, r, TS = "block") /* computing an upper bound on the JSR of A via sparse SOS */
ub = JSR(A, r) /* computing an upper bound on the JSR of A via dense SOS */
\end{minted}
As explained in Chapter \ref{sec:sparsejsr}, the \ac{JSR} upper bounds are obtained via a bisection procedure. 
By default, the initial lower bound and the initial upper bound for bisection are 0 and 2,  respectively.
The default tolerance is $1\text{e-}5$.
The default sparse order is 1.
One can also set {\tt TS="MD"} to use approximately smallest chordal extensions, which is recommended for the relaxation order $r = 1$.
In addition, Gripenberg's algorithm can be employed to produce a lower bound and an upper bound on the \ac{JSR} with difference at most $\delta$.
\begin{minted}[breaklines,escapeinside=||,mathescape=true, linenos, numbersep=3pt, gobble=0, frame=lines, fontsize=\small, framesep=2mm]{csharp}
lb,ub = gripenberg(A, |$\delta=0.2$|)
\end{minted}
Finally, the {\tt SparseDynamicSystem} library allows one to approximate regions of attraction, maximum positively invariant sets, global attractors for polynomial dynamic systems via the \ac{TS}-based \ac{moment-SOS} hierarchy. 
For more details, the reader is referred to the dedicated webpage: 
\href{https://github.com/wangjie212/SparseDynamicSystem}{github:{\tt SparseDynamicSystem}}.

\input{app.bbl}

%% file: app.bbl
\providecommand{\etalchar}[1]{$^{#1}$}